\documentclass[a4paper,10pt,reqno]{amsart}

\usepackage[margin=25.4mm]{geometry}

\usepackage{hyperref}
\usepackage{amssymb,amsmath,amsthm,amsfonts,datetime,bm}
\numberwithin{equation}{section}
\numberwithin{figure}{section}
\usepackage[american]{babel}
\usepackage{mathrsfs}

\usepackage{xcolor}

\usepackage[utf8]{inputenc}

\usepackage{eucal,paralist,enumerate}
\usepackage[textsize=footnotesize]{todonotes}
\usepackage[eulergreek]{sansmath}

\makeatletter
\renewcommand*\env@cases[1][1.2]{%
  \let\@ifnextchar\new@ifnextchar
  \left\lbrace
  \def\arraystretch{#1}%
  \array{@{\,}c@{\ }l@{}}%
}
\makeatother

\newcommand{\R}{\mathbb{R}}
\newcommand{\N}{\mathbb{N}}

\mathchardef\emptyset="001F

\newtheorem{theorem}{Theorem}[section]

\newtheorem{lemma}[theorem]{Lemma}
\newtheorem{remark}[theorem]{Remark}

\newtheorem{definition}[theorem]{Definition}
\newtheorem{proposition}[theorem]{Proposition}

\newtheorem{example}[theorem]{Example}
\newtheorem{notation}[theorem]{Notation}

\newtheorem{hypothesis}[theorem]{Hypothesis}

\newcommand{\eps}{\varepsilon}
\newcommand{\epsk}{{\varepsilon_k}}
\newcommand{\tauk}{{\tau_k}}
\newcommand{\weakto}{\rightharpoonup} 

\newcommand{\aein}{\text{a.e.\ in }}

\newcommand{\down}{\downarrow}
\newcommand{\weaksto}{\overset{*}{\rightharpoonup}}

\newcommand{\AC}{\mathrm{AC}}

 \def\calE{{\mathcal E}} \def\calF{{\mathcal F}}

\def\calM{{\mathcal M}}  
  \def\calR{{\mathcal R}}
\def\calS{{\mathcal S}}  \def\calU{{\mathcal U}}
\def\calV{{\mathcal V}} \def\calW{{\mathcal W}} 
 \def\calZ{{\mathcal Z}}
  
\def\rmd{{\mathrm d}} \def\rme{{\mathrm e}}

 \def\rmB{{\mathrm B}} \def\rmC{{\mathrm C}}
\def\rmD{{\mathrm D}}  
 \def\rmH{{\mathrm H}} 
  \def\rmL{{\mathrm L}}
\def\rmM{{\mathrm M}} \def\rmN{{\mathrm N}} 
  \def\rmR{{\mathrm R}}
  
\def\rmV{{\mathrm V}}

\def\FG{\mathbf}

 \def\bfQ{{\FG Q}} 
  \def\bfU{{\FG U}}
  
 \def\bfZ{{\FG Z}}
\def\BS{\boldsymbol} 

          \def\bftheta{{\BS\theta}}

\def\bfPsi{{\BS\Psi}}

\newcommand{\dd}{\:\,\!\mathrm{d}} 
\renewcommand{\rmd}{\mathrm{d}} 

\newcommand{\pairing}[4]{ \sideset{_{ #1 }}{_{ #2 }}  {\mathop{\langle #3 , #4
\rangle}}}

\newcommand\GC{\Gamma_{\rmC}}
\newcommand\GDir{\Gamma_{\rmD}}
\newcommand\GNeu{\Gamma_{\rmN}}
\newcommand{\teta}{\vartheta}

\definecolor{violet}{rgb}{0.4,0,0.9}
\definecolor{ddcyan}{rgb}{0,0.1,0.9}
\definecolor{dcyan}{rgb}{0,0.4,0.7}
\definecolor{ddmagenta}{rgb}{0.8,0,0.8}
\definecolor{lmagenta}{rgb}{0.7,0,0.6}
\definecolor{vgreen}{rgb}{0.1,0.5,0.2}
\definecolor{dred}{rgb}{.8,0,0}
\definecolor{Turk}{rgb}{0,0.7,0.4}
\definecolor{purpler}{rgb}{0.4,0,0.9}

\newcommand{\piecewiseConstant}[2]{\overline{#1}_{\kern-1pt#2}}

\newcommand{\upiecewiseConstant}[2]{\underline{#1}_{\kern-1pt#2}}

\newcommand{\piecewiseLinear}[2]{{#1}_{\kern-1pt#2}}

\newcommand{\piecewiseVariational}[2]{\tilde{#1}_{\kern-1pt#2}}

\newcommand{\foraa}{\text{for a.a. }}

\newcommand{\sfr}{\mathsf{r}}
\newcommand{\sfs}{\mathsf{s}}
\newcommand{\sft}{\mathsf{t}}
\newcommand{\sfu}{\mathsf{u}}
\newcommand{\sfq}{\mathsf{q}}
\newcommand{\sfx}{\mathsf{x}}
\newcommand{\sfz}{\mathsf{z}}

\newcommand{\Spq}{\mathbf{Q}}
\newcommand{\Iof}{\mathbf{S}}
\newcommand{\eneq}[2]{\mathcal{E}(#1,#2)}
\newcommand{\ene}[3]{\mathcal{E}(#1,#2, #3)}
\newcommand{\en}[2]{\mathcal{E}(#1,#2)}
\newcommand{\frsub}[4]{\partial_{\mathsf{#1}}\calE (#2,#3, #4)}
\newcommand{\frsubq}[3]{\partial_{\mathsf{#1}}\calE (#2,#3)}
\newcommand{\frname}[1]{\partial_{\mathsf{#1}}}

\newcommand{\subl}[1]{\mathcal{S}_{#1}}

\newcommand{\pet}[2]{\partial_t \mathcal{E}(#1,#2)}

\newcommand{\domene}[1]{\mathrm{D}_{\mathsf{#1}}}

\newcommand{\domq}{\mathrm{D}}

\newcommand\JUMP[1]{\mathchoice
                  {\big[\hspace*{-.3em}\big[#1\big]\hspace*{-.3em}\big]}
                   {[\hspace*{-.15em}[#1]\hspace*{-.15em}]}
                   {[\![#1]\!]}
                   {[\![#1]\!]}}

\newcommand{\Spu}{\bfU}

\newcommand{\Spw}{\Spu_{\mathrm{e}}}
\newcommand{\Spz}{\bfZ}
\newcommand{\Spx}{\Spz_{\mathrm{e}}}
\newcommand{\Spy}{\Spz_{\mathrm{ri}}}

\newcommand{\disv}[1]{\calV_{\mathsf{#1}}}

\newcommand{\sigmav}[1]{\sigma_{\mathsf{#1}}}
\newcommand{\sigmave}[2]{\sigma_{\mathsf{#1}, #2}}
\newcommand{\disve}[2]{\calV_{\mathsf{#1}}^{#2}}

\newcommand{\conj}[1]{\mathcal{W}_{\mathsf{#1}}^*}
\newcommand{\SetG}[3]{\mathscr{G}^{#1}[#2,#3]}

\newcommand{\la}{\langle}
\newcommand{\ra}{\rangle}
\newcommand{\norm}[2]{\| #1\|_{#2}}

\newcommand{\BV}{\text{BV}}
\newcommand{\Var}{\mathrm{Var}}
\newcommand{\pBV}{\text{pBV}}

\newcommand{\Varname}[1]{\mathrm{Var}_{#1}}

\newcommand{\Variq}[4]{\Var_{#1}(#2;[#3,#4])}

\newcommand{\wh}[1]{\widehat{#1}}

\newcommand{\bbA}{\mathbb{A}}
\newcommand{\bbB}{\mathbb{B}}
\newcommand{\bbC}{\mathbb{C}}
\newcommand{\bbG}{\mathbb{G}}
\newcommand{\bsA}{\boldsymbol{A}}
\newcommand{\bsJ}{\boldsymbol{J}}
\newcommand{\bsD}{\boldsymbol{D}}
\newcommand{\bsC}{\boldsymbol{C}}

\newcommand{\bbD}{\mathbb{D}}

\newcommand{\meq}[6]{\mathfrak{M}_{#1}^{#2}(#3,#4, #5,#6)} 
\newcommand{\mename}[2]{\mathfrak{M}_{#1}^{#2}}
\newcommand{\mredq}[6]{\mathfrak{M}_{#1}^{#2,\mathrm{red}}(#3,#4, #5,#6)}
\newcommand{\mredname}[2]{\mathfrak{M}_{#1}^{#2,\mathrm{red}}}
\newcommand{\slov}[3]{\mathscr{S}_{\mathsf{#1}}^*(#2,#3)}
\newcommand{\slovname}[1]{\mathscr{S}_{\mathsf{#1}}^*}
\newcommand{\argminSlo}[3]{\mathfrak{A}_{\mathsf{#1}}^{*}(#2,#3)}

\newcommand{\RIS}{(\Spu {\times} \Spz,\calE,\calV_{\mathsf{u}}^{\epsalpha}%
            \!{+}\calR{+}\calV_{\mathsf{z}}^{\eps})_{\eps\down0}}

\newcommand{\tht}[1]{\teta_{\mathsf{#1}}}
\newcommand{\tetaopt}{\serifTeta^{\mathrm{opt}}}

\newcommand{\Ctc}{\Sigma}
\newcommand{\thn}[1]{\lambda_{\mathsf{#1}}}

\newcommand{\rgs}[2]{\mathrm{#1}_{\mathsf{#2}}}
\newcommand{\Reg}[3]{\mathrm{R}(#1,#2;#3)}
\newcommand{\llim}[2]{{#1}(#2^-)}
\newcommand{\rlim}[2]{{#1}(#2^+)}

\newcommand{\costq}[4]{\mathrm{cost}_{#1}(#2;#3,#4)}
\newcommand{\costname}[1]{\mathrm{cost}_{#1}}

\newcommand{\admtcq}[3]{\mathcal{A}_{#1}(#2,#3)}

\newcommand{\sfS}{\mathsf{S}}

\newcommand{\vertiii}[1]{{\vert\kern-0.25ex \vert\kern-0.25ex \vert #1 
    \vert\kern-0.25ex\vert\kern-0.25ex\vert}}

\newcommand{\Bfun}{\mathcal{B}}
\newcommand{\Bfu}[4]{\mathcal{B}_{#1}(#2,#3,#4)}
\newcommand{\RCBfu}[7]{\mathfrak{B}_{#1}^{#2}(#3,#4,#5, #6, #7)}
\newcommand{\tMfu}[5]{\widetilde{\mathfrak{M}}_{#1}^{#2}(#3,#4,#5)}
\newcommand{\RJMF}{rescaled joint M-function}

\newcommand{\ti}{{\times}}
\newcommand{\indic}{{\boldsymbol1}}
\newcommand{\mfb}{\mathfrak b}
\newcommand{\calB}{\mathcal{B}} 
 
\newcommand{\bigset}[2]{\big\{\, #1\, \big| \, #2 \,\big\}}
\newcommand{\Bigset}[2]{\Big\{\: #1\: \Big| \, #2 \:\Big\}}
\newcommand{\STEP}[1]{\noindent\underline{\emph{Step #1}}}
\newcommand{\Sign}{\mathop{\mathrm{Sign}}}

\renewcommand{\uppercasenonmath}[1]{}

\newcommand{\epsalpha}{\eps^\alpha} 
\newcommand{\scrX}{X}
\newcommand{\mfB}{\mathfrak{B}}

\newcommand{\mfE}{\mathfrak{E}}
\newcommand{\mfM}{\mathfrak{M}}

\newcommand{\pl}{\partial}
\newcommand{\dk}{{\delta_k}}

\newcommand{\serifteta}{{\sansmath \theta}}

\newcommand{\serifTeta}{{\sansmath \Theta}}

\newcommand{\db}[2]{{#1}_{\mathsf{#2}}}
\newcommand{\Loptimal}[3]{\Lambda_{\mathsf{#1}}(#2,#3)}

\makeatletter
\renewcommand\section{\@startsection{section}{1}%
  \z@{.7\linespacing\@plus\linespacing}{.5\linespacing}%
  {\normalfont\huge\scshape\centering}%
 }%
\renewcommand\subsection{\@startsection{subsection}{2}%
  \z@{.5\linespacing\@plus.7\linespacing}{-.5em}%
  {\normalfont\LARGE\scshape}}
\renewcommand\subsubsection{\@startsection{subsubsection}{3}%
  \z@{.5\linespacing\@plus.7\linespacing}{-.5em}%
  {\Large\scshape}}
\makeatother

\newcommand{\Section}[1]{\section{#1} \markleft{\thesection. #1}}
\newcommand{\Subsection}[1]{\subsection{#1} \markright{\thesubsection. #1}}

\usepackage{longtable}

\begin{document}

\title[Balanced-Viscosity solutions for multi-rate systems]%
 {\Large Balanced-Viscosity solutions to\\ infinite-dimensional
         multi-rate systems\thanks{A.M. was partially supported by DFG within
           SPP\,2256 (no.\,441470105) under grant Mi\,459/9-1.}}

\author[Alexander Mielke]{\large Alexander Mielke}

\address{\upshape A.\ Mielke, Weierstra\ss-Institut f\"ur Angewandte Analysis und
Stochastik, Mohrenstr.\ 39, D--10117 Berlin and
 Institut f\"ur Mathematik, Humboldt-Universit\"at 
 zu Berlin, Rudower Chaussee 25, D--12489 Berlin (Adlershof) --  Germany\newline
\indent ORCID 0000-0002-4583-3888}
\email{alexander.mielke\,@\,wias-berlin.de}

\author[Riccarda Rossi]{\large Riccarda Rossi}

\address{\upshape R.\ Rossi, DIMI, Universit\`a degli studi di Brescia,
via Branze 38, I--25133 Brescia -- Italy}
\email{riccarda.rossi\,@\,unibs.it} 

\date{December 2, 2021}

\begin{abstract}
We consider generalized gradient systems with rate-independent and
rate-dependent dissipation potentials. We provide a general framework for
performing a vanishing-viscosity limit leading to the notion of parametrized
and true Balanced-Viscosity solutions that include a precise description of the
jump behavior developing in this limit. Distinguishing an elastic variable $u$
having a viscous damping with relaxation time $\eps^\alpha$ and an internal
variable $z$ with relaxation time $\eps$ we obtain different limits for the
three cases $\alpha \in (0,1)$, $\alpha=1$ and $\alpha>1$.  
An application to a delamination problem shows that the theory is general
enough to treat nontrivial models in continuum mechanics. 
\\[0.2em]
\\[0.2em]
\end{abstract} 

\maketitle

\setlength{\leftmargini}{2em}
\setlength{\leftmargin}{1.3em}

{\small\tableofcontents}

\subsubsection*{List of symbols}
\mbox{}\\[-1.7em]
{\small 
\begin{center}\bigskip
\begin{longtable}{lll}
$\Spu, \Spz$  &  state spaces & \eqref{intro-state_spaces}  
\smallskip
\\
$\Spq = \Spu \ti \Spz$  & overall state space & 
\smallskip
\\
$  \Spw \subset \Spu,   \  \Spx \Subset \Spz $  & energy spaces & Hyp.\ \ref{hyp:setup}
\smallskip
\\
$\Spy \supset \Spz$  & space for $1$-homogeneous dissipation potential  &  Hyp.\  \ref{hyp:setup}
\smallskip
\\
$\calR: \Spy \to [0,\infty) $& $1$-homogeneous dissipation potential &  Hyp.\ \ref{hyp:diss-basic}
\smallskip
\\
$\disv u :  \Spu \to [0,\infty)$, \ $\disv z :  \Spz \to [0,\infty)$ & viscous dissipation potentials &  Hyp.\  \ref{hyp:diss-basic}
\smallskip
\\
 $\disv x^*: \mathbf{X}^* \to [0,\infty)$, $ \mathbf{X} \in \{ \Spu, \Spz\}$
 & Legendre-Fenchel conjugate of   $\disv x$  for $\mathsf{x} \in \{ \sfu, \sfz\} $ 
 & Def.\ \ref{def:DissPotential}
\smallskip
\\
$\conj z: \Spz^* \to [0,\infty)$ & conjugate of $\calR{+}\disv z$ & \eqref{def:conj}
\smallskip
\\
$\disve x{\lambda}, \ \mathsf{x} \in \{ \sfu,\sfz\}, \ \lambda \in (0,\infty) $ & rescaled  viscous dissipation potentials & \eqref{eq:Def.Vx.lambda}
\smallskip
\\
$\Psi_{\eps,\alpha} = \disve u{\eps^\alpha} +  \calR+  \disve z{\eps}$ & overall viscous potential  & \eqref{eq:def.Psi.e.a}
\smallskip
\\
$\calE: [0,T]\times \Spu \ti \Spz \to (-\infty,+\infty]$ & driving energy functional &  Hyp.\  \ref{hyp:1}
\smallskip
\\
$\mathcal{S}_E$, $E>0$ & energy sublevels & \eqref{Esublevels}
\smallskip
\\
$\pl_q \calE$ & Fr\'echet subdifferential of $\calE(t,\cdot)$ & \eqref{Frsubq}
\smallskip
\\
  $\slovname{x}: [0,T]\ti \domq \to [0,\infty] $
 & generalized slope functional for $\mathsf{x} \in \{ \sfu, \sfz\} $ 
 & \eqref{def:GeneralSlope}
\smallskip
\\
 $\argminSlo xtq, \   (t,q)\in [0,T]\ti\domq$  
 & set of minimizers for the slope $ \slov {x}tq, \ \mathsf{x} \in \{ \sfu, \sfz\}$  & \eqref{not-empty-mislo}
 \smallskip
 \\
  $  \SetG\alpha tq, \  (t,q)\in [0,T]\ti\domq  $  &  sets of  positivity for
  the slopes at $(t,q)$ &     \eqref{setGalpha}   
 \smallskip
 \\
$\calB_\psi$ & B-function associated with a dissipation potential $\psi$  &
\eqref{Bfunctions-2} 
\smallskip
\\
$\mfb_\psi$ & vanishing-viscosity contact potential assoc.\ with $\psi$ & \eqref{bipotentials-1}
\smallskip
\\
$\mfB_\eps^\alpha$, $\eps\geq 0$,  & rescaled joint B-function & \eqref{eq:def.B.al.eps}
\smallskip
\\
$\mfM_\eps^\alpha$, $\eps >0$, & rescaled joint M-function & \eqref{def-Me}
\smallskip
\\
$\mfM_0^\alpha$ & (limiting) rescaled joint M-function & \eqref{mename-0}
\smallskip
\\
$\mfM_0^{\alpha, \mathrm{red}}$ & reduced rescaled joint M-function &  \eqref{decomposition-M-FUNCTION} 
\smallskip
\\
$\mathscr{A}([a,b];[0,T]\ti \Spq)$ &   admissible parametrized curves from $[a,b]$ to $[0,T]\ti \Spq$ &   Def.\ \ref{def:adm-p-c}
\\
$ \admtcq{t}{q_0}{q_1} $ & admissible transition curves betw.\ $q_0$ and $q_1$ at time $t$  &  Def.\ \ref{def:adm-p-c}
\\
 $\Sigma_\alpha$ &   contact set &  \eqref{ctc-set}
\\
 $\rgs Au \rgs Cz  = \rgs Au \cap \rgs Cz$&  evolution regimes with $  
   \mathrm{A \in \{E,V,B\}} \text{ and }  \mathrm{C \in \{R,V,B\} }$ & \eqref{regime-sets}
\\ 
 $\Varname \calR$ &   $\calR$-variation &   \eqref{R-variation}
\\
 $  \mathrm{J}[q] $ &   jump set of a true $\BV$ solution &  Def.\ \ref{def:jumpset}
\\
$  \costname{\mename 0\alpha}$ &  Finsler cost induced by  $ \mename 0\alpha$   &  \eqref{Finsler-cost}
\\
 $  \Varname{\mename 0\alpha} $ &  total variation induced by $\mename 0
\alpha$ &  \eqref{Finsler-variation} 
 \end{longtable}
\end{center}
} 

\vspace{-1em}

\Section{Introduction}
\label{se:Intro}

In this paper we address rate-independent  limits of viscous evolutionary
systems that are motivated by applications  in solid mechanics.  These
systems  can be described in terms of two variables $u\in \Spu $ and $
z\in \Spz$; throughout, we shall assume that the state spaces 
\begin{equation}
\label{intro-state_spaces}
\text{$\Spu$ and $\Spz$ are (separable) reflexive Banach spaces.}
\end{equation}
Typically, $u$ is the displacement, or the deformation of the body, whereas $z$
is an internal variable specific of the phenomenon under investigation, in
accordance with the theory of \emph{generalized standard materials}, see
\cite{HalNgu75MSG}.

\Subsection{Rate-independent systems}
\label{su:RIS} 

Under very slow loading rates, one often assumes that  $u$ satisfies a
\emph{static} balance law that  arises as Euler-Lagrange equation from
minimizing  the energy functional $\calE$ with respect to $u$. The 
evolution of $z$ is governed by a (doubly nonlinear) subdifferential inclusion
featuring the $z$-derivative of the energy and the  viscous force in form
of the subdifferential $\pl \calR$ of a  dissipation potential $\calR$:
\begin{subequations}
\label{grad-structure}
\begin{align} & 
\label{grad-structure-u}
\rmD_u \ene t{u(t)}{z(t)} =0   && \text{ in } \Spu^*,  && t \in (0,T), 
\\
\label{grad-structure-z}
\mbox{}\qquad\pl  \calR(z'(t)) + {}& 
\rmD_z \ene t{u(t)}{z(t)} \ni 0  && \text{ in } \Spz^*, &&  t \in (0,T)\,. 
\qquad\mbox{}
\end{align}
\end{subequations}
 If $\calR : \Spz \to [0,\infty]$ is positively homogeneous of degree $1$,
i.e.\ $\calR(\lambda z')=\lambda \calR(z')$ for all $\lambda>0$, then the
system \eqref{grad-structure} is called \emph{rate-independent}, and the triple 
$(\Spu\ti\Spz,\calE,\calR)$ is called a rate-independent system, cf.\
\cite{MieRouBOOK}.

Here, $\pl  \calR: \Spz \rightrightarrows \Spz^*$ denotes  the subdifferential
of convex analysis for  the nonsmooth functional $\calR$, whereas,
throughout this introduction, for simplicity we will assume that the (proper)
energy functional $\calE: [0,T]\ti \Spu \ti \Spz\to (-\infty,\infty]$, 
which is  smooth
with respect to\ time, is additionally smooth with respect to\ both variables $u$ and $z$.
System \eqref{grad-structure} reflects the ansatz that energy is dissipated
through changes of the internal variable $z$ only: in particular, the doubly
nonlinear evolution inclusion \eqref{grad-structure-z} balances the dissipative
frictional forces from $\pl  \calR(z')$ with the restoring force
$\rmD_z \ene t{u}{z}$.  Despite the assumed smoothness of
$(u,z) \mapsto \ene tuz$, system \eqref{grad-structure} is only formally
written: due to the $1$-homogeneity of $\calR$, one can in general expect only
$\BV$-time regularity for $z$. Thus $z$ may have jumps as a function of time
and the pointwise derivative $z'$ in the subdifferential inclusion
\eqref{grad-structure-z} need not be defined. This has motivated the
development of various weak solution concepts for system
\eqref{grad-structure}.
\par
\emph{Energetic  solutions} were advanced in the late '90s in
\cite{MieThe99MMRI, MiThLe02, MieThe04RIHM} for \emph{abstract}
rate-independent systems, and in the context of phase transformations in
solids.
In the realm of crack propagation, an analogous notion of evolution was
pioneered in \cite{Francfort-Marigo98} and later further developed in
\cite{DM-Toa2002} with the concept of `quasistatic evolution'. Due to its
flexibility, the energetic concept has been successfully applied to a wide
scope of problems, see e.g.\ \cite{MieRouBOOK} for a survey. 

\Subsection{The vanishing-viscosity approach}
\label{su:I.VVA}
However,  it has been observed that the energetic notion may fail to provide a
feasible description of the system behavior at jumps, in the case of a
\emph{nonconvex} driving energy.  This fact has motivated the introduction of
an alternative weak solvability concept, first suggested in \cite{EfeMie06RILS}
and based on the vanishing-viscosity regularization of the rate-independent
system as a selection criterion for mechanically feasible weak solutions.  In
the context of system \eqref{grad-structure}, this `viscous regularization'
involves a second (lower semicontinuous, convex) dissipation potential
$ \disv z : \Spz \to [0,+\infty)$, with superlinear growth at infinity; to fix
ideas, we may think of a \emph{quadratic} potential.  The vanishing-viscosity
approach then consists in performing the asymptotic analysis of solutions to
the \emph{rate-dependent} system
\begin{subequations}
\label{naive-vanvisc}
\begin{align} 
\label{naive-van-visc-u}
&\rmD_u \ene t{u(t)}{z(t)} =0   && \text{ in } \Spu^*,  && t \in (0,T), 
\\
\label{naive-van-visc-z}
\mbox{}\qquad\pl  \calR(z'(t)) +  \pl  \disv z (\eps z'(t))  +{}
& 
\rmD_z \ene t{u(t)}{z(t)} \ni 0  && \text{ in } \Spz^*, &&  t \in (0,T)\,,
\qquad\mbox{}
\end{align}
\end{subequations}
as the \emph{viscosity} parameter $\eps \to 0^+$.  System \eqref{naive-vanvisc}
now features \emph{two rates}: in addition to that of the external loading,
scaling as $\eps^0=1$, the internal rate of the system, set on the faster scale
$\eps$, is revealed.  In diverse (finite-dimensional, infinite-dimensional, and
even metric) setups, cf.\ \cite{EfeMie06RILS,MRS09,MRS12,Mielke-Zelik,MRS13}
 (see also \cite{KZ21}  and \cite{RiScVe21TSSN}), 
solutions to the `viscous system' have been shown to converge to a
different type of solution of \eqref{grad-structure}, which we shall
refer to as \emph{Balanced-Viscosity} solution, featuring a better description
of the jumps of the system.  In parallel, the vanishing-viscosity approach has
proved to be a robust method in manifold applications, ranging from plasticity
(cf.\ e.g.\ \cite{DalDesSol11, BabFraMor12, FrSt2013}), to fracture
\cite{KnMiZa08ILMC, Lazzaroni-Toader, Negri14}, damage and fatigue \cite{KRZ13,
  Crismale-Lazzaroni,ACO2018}, and to optimal control \cite{SteWachWach2017} to
name a few.

This paper revolves around a different, but still of \emph{vanishing-viscosity
  type}, solution notion for system \eqref{grad-structure}. Indeed, we are
going to regularize it by considering a viscous approximation of
\eqref{grad-structure-u}, besides the viscous approximation
\eqref{naive-van-visc-z} of \eqref{grad-structure-z}.  Therefore, we will
address the asymptotic analysis as $\eps\to 0^+$ of the system of doubly
nonlinear differential inclusions
\begin{subequations}
\label{van-visc-intro}
\begin{align}
\label{van-visc-intro-u}
&
\pl  \disve{u}{\epsalpha}(u'(t)) +\rmD_u \ene t{u(t)}{z(t)}  
 \ni 0 && \text{ in } \Spu^*  && \foraa\, t \in (0,T), 
\\
&
\label{van-visc-intro-z}
\pl  \calR (z'(t)) + \pl  \disve{z}{\eps}(z'(t)) +\rmD_z 
\ene t{u(t)}{z(t)}   \ni 0  &&  \text{ in } \Spz^*  && \foraa\, t \in (0,T), 
\end{align}
\end{subequations}
where  for $\mathsf x\in \{\mathsf u, \mathsf z\}$ we have set
\begin{subequations}
  \label{eq:Def.Vx.la}
\begin{equation}
  \label{eq:Def.Vx.lambda}
  \disve x{\lambda}(w): = \frac{1}{\lambda} \disv x (\lambda w) \text{ for }
  \lambda\in (0,\infty) \  \text{ and } \ \disve x{\infty} (w): = 
  \begin{cases} 0 &\text{for } w=0, \\ \!\infty &\text{for }w\neq 0\end{cases}
\end{equation}
 (the functional $\disve x{\infty}$ will indeed come into play later on, cf.\ 
\eqref{e:diff-char-intro}). 
Throughout we assume that $\disv x$ satisfies $\disv x(0)=0$, \ $\pl \disv
x(0)=\{0\}$, and has superlinear growth, which implies that $\disve x0$ and $\disve
x\infty$ are indeed the Mosco limits of $\disve x\lambda$ for
$\lambda\to 0^+$ and $\lambda \to \infty$, respectively. We will use that the
subdifferentials take the form
\begin{equation}
  \label{eq:SubDiff.Vx}
  \pl  \disve x{\lambda}(w)=\pl \disv x(\lambda w)  
  \text{ for }\lambda\in [0,\infty)  \ \text{ and } \ 
  \pl  \disve x{\infty}(w)=
    \begin{cases}  \mathbf{X}^*  &\text{for } w=0, \\   
               \emptyset &\text{for }w\neq 0.\end{cases}
\end{equation}
\end{subequations}

The parameter $\alpha$ in \eqref{van-visc-intro-u} determines which of the two
variables $u$ and $z$ relaxes faster to \emph{equilibrium} and
\emph{rate-independent} evolution, respectively.  Hence, following
the finite-dimensional work \cite{MRS14}  we shall refer to
\eqref{van-visc-intro} as a \emph{multi-rate} system, with the time scale
$\eps^0=1$ of the external loading and the (possibly different) relaxation
times $\eps$ and $\eps^\alpha$ of the variables $z$ and $u$.

From a broader perspective, with our analysis we aim to contribute to the
investigation of \emph{coupled} rate-dependent/rate-independent phenomena, a
topic that has attracted some attention over the last decade.  In this
connection, we may mention the study of systems with a \emph{mixed}
rate-dependent/rate-independent character (typically, a rate-independent flow
rule for the internal variable coupled with the momentum balance, with
viscosity and inertia, for the displacements, and possibly with the heat
equation),  see the series of papers by T.\ Roub\'{\i}\v{c}ek
\cite{Roub08, Roub10TRIP, Roub-PP, Roubicek-defect,
  Roub-Tomassetti},  among others. There, a weak solvability notion, still of
\emph{energetic 
  type}, was advanced, cf.\ also \cite{Rossi-Thomas-SIAM, Maggiani-Mora}.

However, unlike in those contributions, in our `modeling' approach the balanced
interplay of rate-dependent and rate-independent behavior does not stem from
coupling equations with a rate-dependent and a rate-independent character.  Instead, it
emerges through the asymptotic analysis as $\eps \to 0^+$ of the `viscous'
system \eqref{van-visc-intro}, which leads to a solution of the
rate-independent one \eqref{grad-structure} that is `reminiscent of viscosity',
\emph{in both variables} $u$ and $z$, in the description of the system behavior
at jumps.  This `full' vanishing-viscosity approach, also involving the
displacement variable $u$, has been already carried out for a model for
fracture evolution with pre-assigned crack path in \cite{Racca}, as well as in
the context of perfect plasticity \cite{DMSca14QEPP,Rossi2018} and delamination
\cite{Scala14}.  With different techniques, based on an alternating
minimization scheme, the emergence of viscous behavior both for the
displacement and for the internal variable is demonstrated in
\cite{Knees-Negri} for a phase-field type fracture model.

In this mainstream, in \cite{MRS14} we have addressed the vanishing-viscosity
analysis of \eqref{van-visc-intro} in a preliminary {finite-dimensional}
setting, with $\Spu = \R^n$ and $\Spz = \R^m$, and for a {smooth} energy
$\calE \in \rmC^1([0,T]{\ti}\R^n{\ti}\R^m)$, with the aim of emphasizing
the role of viscosity in the description of the jump behavior of the limiting
rate-independent system.  Even in this significantly simplified setup, the
analysis in \cite{MRS14} conveyed how the balanced interplay of the different relaxation
rates in \eqref{van-visc-intro} enters in the description of the jump dynamics
of the rate-independent system. In particular, it showed that viscosity in $u$
and viscosity $z$  determine the jump transition path 
in different ways depending on whether the parameter $\alpha$ is
strictly bigger than, or equal to, or strictly smaller than $1$.

The aim of this paper is to thoroughly extend the results from \cite{MRS14} to
an \emph{infinite-dimensional} and \emph{non-smooth} setting, suited for the
application of this vanishing-viscosity approach to models in solid mechanics.
What is more, we will also broaden the analysis in \cite{MRS14},  which is
 confined to the case of \emph{quadratic} `viscous' dissipation potentials,
to a fairly general class of potentials $\disv u$ and $\disv z$.

\Subsection{Our results}
\label{ss:1.1}

Throughout most of this paper, we will confine the discussion to the
\emph{abstract rate-independent} system $\RIS$ arising in the
vanishing-viscosity limit of \eqref{van-visc-intro}.  The notation looks a bit
extensive, but has the advantage of emphasizing the dependence of the solution
concept on the energy functional $\calE$, the three different types of
dissipation $\disv u$, $\calR$, and $\disv z$, and the parameter
$\alpha>0$. This also explains the name ``Balanced-Viscosity solution'' that
suggests the appearance of the viscous effects by balancing the influence of
$\calR$, $\disv x$, and $\disv z$ in such a way that the energy-dissipation
balance remains true. Of course, using the abbreviation ``BV solution'' should
remind us about the fact that these solutions may not be continuous but may
have jumps as functions of time.

In our opinion, in that general framework the main ideas underlying the
vanishing-viscosity approach are easier to convey. Indeed, we aim to provide
some possible  recipes  for the application of this approach to concrete
rate-independent  limiting  processes, where of course the `abstract
techniques' may have 
to be suitably adjusted to the specific situation.   For this,  we will strive
to work in a fairly general setup,
\begin{compactenum}
\item encompassing nonsmoothness of the energies $u\mapsto \ene tuz$ and
  $z\mapsto \ene tuz$ through the usage of suitable subdifferentials
  $\frname u\calE :[0,T]\ti \Spu \ti \Spz \rightrightarrows \Spu^*$ and
  $\frname z\calE :[0,T]\ti \Spu \ti \Spz \rightrightarrows \Spz^*$ in place of
  the G\^ateau derivatives $\rmD_u \calE$ and $\rmD_z \calE$, and
\item allowing for a wide range of `viscous dissipation potentials' $\disv u$
  and $\disv z$.  In particular, we shall allow for a much broader class of
  dissipation potentials $\disv z$ than those considered in \cite{MRS13}.
\end{compactenum}

The first cornerstone of our vanishing-viscosity analysis is the observation
that the viscous system \eqref{van-visc-intro} has the structure of a
\emph{generalized gradient system}  (cf.\ \cite{Miel16EGCG}):  indeed,
it rewrites as 
\begin{equation}
\label{GGS-structure}
\pl  \Psi_{\eps,\alpha}(q'(t)) +\rmD_q \eneq t{q(t)} \ni 0 \qquad 
\text{in } \Spq^* \quad  \foraa\, t \in (0,T)
\end{equation}
with $q=(u,z)\in \Spq = \Spu\ti \Spz$ and 
\begin{equation}
 \label{eq:def.Psi.e.a}
 \Psi_{\eps,\alpha}(q')  = ( \disve u{\epsalpha}{\oplus} (\calR {+}  \disve
 z{\eps}))(q') := \disve{u}{\epsalpha}(u') + \calR(z') + \disve z{\eps}(z').
\end{equation}
In turn, \eqref{GGS-structure} can be equivalently formulated using the single
\emph{energy-dissipation balance} 
\begin{equation}
 \label{EnDissBal-intro}
   \eneq t{q(t)}+ \int_s^t \tMfu\eps\alpha {r}{q(r)}{q'(r)} \dd r 
   =  \eneq s{q(s)}+ \int_s^t \pl_t \eneq r{q(r)} \dd r  
\end{equation}
for all $0 \leq s\leq t \leq T$,  featuring the M-function 
\begin{equation}
  \label{tildeMF}
  \tMfu{\eps}{\alpha} {t}{q}{q'}: = \Psi_{\eps,\alpha} (q') +
  \Psi_{\eps,\alpha}^* ({-}\rmD_q \eneq tq) 
\end{equation}
with the Legendre-Fenchel conjugate $\Psi_{\eps,\alpha}^*$ of
$\Psi_{\eps,\alpha}$. This reformulation is often referred to
\emph{energy-dissipation principle}; the germs of this idea trace back to
E.\ De Giorgi's \emph{variational theory for gradient flows}
in \cite{Ambr95MM},  see also \cite[Prop.\,1.4.1]{AGS08} and
\cite[Thm.\,3.2]{Miel16EGCG}. In  our setup, it is based on the validity of
a suitable chain rule for $\calE$, which will be
thoroughly discussed in the sequel.  From \eqref{EnDissBal-intro} we obtain the
basic a priori estimates on a sequence $(u_\eps,z_\eps)_\eps$ of
solutions to \eqref{van-visc-intro}. Together with the additional bound
\begin{equation}
 \label{u-bound-intro}
 \int_0^T \| u_\eps'(t)\|_{\Spu} \dd t \leq C,
\end{equation} 
we are able to reparametrize the curves
$(q_\eps)_\eps = (u_\eps,z_\eps)_\eps$ by their  ``dissipation arclength''
$ \mathsf{s}_\eps: [0,T] \to [0,\mathsf{S}_\eps]$ given by
\[
 \mathsf{s}_\eps(t) : =\int_0^t \left( 1{+} \tMfu{\eps}\alpha{r}{q_\eps(t)}{q_\eps'(r)}
  {+} \| u_\eps'(r)\|_{\Spu}\right) \dd r \,.
\]
Reparametrization was first advanced in \cite{EfeMie06RILS} as a tool to
capture the viscous transition paths, at jumps, in the rate-independent limit.
With this aim, first of all we observe that, setting
$\sft_\eps: = \sfs_\eps^{-1}: [0, ,\mathsf{S}_\eps]\to [0,T]$ and
$\sfu_\eps: = u_\eps \circ \sft_\eps$, $\sfz_\eps: = z_\eps \circ \sft_\eps$,
the rescaled curves $(\sft_\eps, \sfu_\eps,\sfz_\eps )_\eps$ satisfy a
reparametrized version of \eqref{EnDissBal-intro}.  Using the first main
results of this paper presented in Theorems \ref{thm:existBV} and
  \ref{thm:exist-enh-pBV}, we are able to pass  to the limit in this
reparametrized energy balance as $\eps\to0^+$ and obtain a triple
$(\sft,\sfq) = (\sft,\sfu,\sfz): [0,\mathsf{S}]\to [0,T]\ti \Spu\ti \Spz$
satisfying the energy-dissipation balance
\begin{equation}
\label{lim-EDbal-intro}
\begin{aligned}
  & \eneq {\sft(s_2)}{\sfq(s_2)} +\int_{s_1}^{s_2} \meq
  0\alpha{\sft(s)}{\sfq(s)}{\sft'(s)}{\sfq'(s)} \dd s
  \\
  & = \eneq {\sft(s_1)}{\sfq(s_1)} + \int_{s_1}^{s_2} \pl_t \ene
  s{\sfq(s)} {\sft'(s)} \dd s \quad \text{for all } 0 \leq s_1\leq s_2 \leq
  \mathsf{S}\,,
  \end{aligned}
\end{equation}
which encodes all the information on the behavior of the limiting
rate-independent system in the expression of the `time-space dissipation
function' $\mename 0\alpha$, thoroughly investigated in Section \ref{su:ReJoMFcn}.  We
shall call a triple $(\sft,\sfu,\sfz)$ complying with \eqref{lim-EDbal-intro} a
\emph{parametrized Balanced-Viscosity} ($\pBV$, for short) solution to the
rate-independent system $\RIS$.

We highlight two main properties of this solution concept  that
follow from the special form of $\mename 0\alpha$: 
\begin{compactitem}
\item[\textbullet] When  a solution  does not jump, i.e.\ when the function
  $\sft$ of the artificial time $s$, recording the (slow) external time scale,
  fulfills $\sft'(s)>0$, the term
  $\meq 0\alpha{\sft}{\sfq}{\sft'}{\sfq'} $ is finite if and only
  if $\sfu$ is \emph{stationary} and $\sfz$ is \emph{locally stable}, i.e.\
 \[
   -\rmD_u \ene {\sft(s)} {\sfu(s)} {\sfz(s)} =0 \text{ in } \Spu^* \quad
   \text{ and } \quad -\rmD_z \ene {\sft(s)} {\sfu(s)} {\sfz(s)} \in
   \pl \calR(0) \text{ in } \Spz^*.
 \]
 Because of the local character of the second condition, the unfeasible jumps that
 may occur in `energetic solutions' via their `global stability' are
 thus avoided.
\item[\textbullet] The function $\mename 0\alpha$ in \eqref{lim-EDbal-intro}
  comprises  the contributions of the dissipation potentials  $\calR$, 
  $\disv u$ and $\disv z$ by  condensing the viscous effects into a 
  description of the  limiting jump behavior that can occur only if
  $\sft'(s)=0$,  i.e.\ the slow external time is frozen. 
  For example, if the dissipation potentials $\disv u$ and
  $\disv z$ are  $p$-homogeneous (i.e.\ $\disv x(\lambda \sfx')=\lambda^p \disv
  x(\sfx')$ for $\lambda>0$), then for $\alpha=1$ and $\sft'=0$ we have
\begin{equation}
 \label{2-homog-mename-intro}
 \begin{aligned}
&\meq 01{\sft}{(\sfu,\sfz)}{0}{(\sfu',\sfz')} = \calR(\sfz') \\
&\qquad + \widehat c_p \: \big(\disv z
(\sfz'){+} \disv u (\sfu')\big)^{1/p} \:\big( \disv u^*({-}\rmD_u 
  \ene {\sft}{\sfu}{\sfz}) {+} \conj z({-}\rmD_z \ene {\sft}{\sfu}{\sfz}) 
  \big)^{1{-}1/p} 
 \end{aligned}
\end{equation}
(see Example \ref{ex:p-homog}). The symmetric role of $\disv u $ and $\disv z$
in \eqref{2-homog-mename-intro} arises because of $\alpha=1$ and reflects the
fact that, at a jump, the system may switch to a viscous regime where
\emph{both} dissipation mechanisms intervene in the evolution of $u$ and $z$,
respectively. In contrast, for $\alpha>1$ and $\alpha<1$, the $M$-function
$\mename 0\alpha$ shows the different roles of $\disv u $ and $\disv z$, cf.\
\eqref{l:partial}.
\end{compactitem}

These features are even more apparent in the characterization of a suitable
class of $\pBV$ solutions in terms of a system of subdifferential inclusions
that has the very same structure as the original viscous system
\eqref{van-visc-intro} as provided by Theorem \ref{thm:diff-charact}.
 This result shows  that a triple
$(\sft,\sfu,\sfz) : [0,\mathsf{S}]\to [0,T]  \ti \Spu\ti \Spz$ is an
\emph{enhanced} $\pBV$ solution 
if and only if  there exist measurable
functions $\lambda_\sfu , \lambda_\sfz:[0,\mathsf{S}] \to [0,\infty]$ such that for
almost all $s\in (0,\sfS)$ we have
\begin{subequations}
\label{e:diff-char-intro}
\begin{equation}
\label{param-subdif-incl-intro}
\begin{aligned}
\pl  \disve u{\thn u(s) } (\sfu'(s)) +\rmD_u \ene {\sft(s)}{\sfu(s)}{\sfz(s)}
\ni 0 & \text{ in } \Spu^*, 
\\
\pl \calR(\sfz'(s))  + \pl  \disve z{\thn z(s)} ( \sfz'(s)) +\rmD_z \ene
{\sft(s)}{\sfu(s)}{\sfz(s)} \ni 0 & \text{ in } \Spz^*,  
\end{aligned}
\end{equation}
\begin{equation}
\label{switch-intro-1}
\sft'(s)\:\frac{\thn u(s)}{1{+}\thn u(s)} =0 \quad  \text{and} \quad
\sft'(s)\:\frac{\thn z(s)}{1{+}\thn z(s)} =0, 
\end{equation}
\begin{equation}
  \label{switch-intro-2}
  \begin{cases}
    \thn u(s) \: \dfrac{1}{1{+}\thn z(s)} =0& \text{ for }\alpha>1 , 
    \\[0.6em]
    \thn u(s) =\thn z(s) &\text{ for }\alpha=1, 
    \\[0.3em]
   \dfrac{1}{1{+}\thn u(s)}\,\thn z(s) =0& \text{ for }\alpha\in (0,1).
\end{cases}
\end{equation}
\end{subequations}
In \eqref{switch-intro-1} and \eqref{switch-intro-2} we use the obvious
conventions $\frac{\infty}{1+\infty}=1$ and $\frac1{1+\infty}=0$, respectively.
Condition \eqref{switch-intro-1} entails that the coefficients $\thn u(s)$ and
$\thn z(s)$ of the `viscous terms' in \eqref{param-subdif-incl-intro} are
allowed to be nonzero only when $\sft'(s)=0$, i.e.\ viscous behavior may
manifest itself only at jumps happening now at a fixed time $t_*=\sft(s)$ for
$s\in [s_0,s_1]$.  Conditions \eqref{switch-intro-2} reveal that the onset of
viscous effects in $u$ and/or in $z$ depends on whether $u$ relaxes to
equilibrium faster (case $\alpha>1$), with the same speed (case $\alpha =1$),
or more slowly (case $\alpha<1$), than $z$ relaxes to local stability.  In
particular, the case $\lambda_\sfx=\infty$ leads to a blocking of the variable
$\sfx\in \{\sfu,\sfz\}$, i.e.\ $\sfx'(s)=0$ and
$\pl  \disve x{\infty}(0) =\Spx^*$.  These aspects will be thoroughly
explored in Sections \ref{s:EDI} and \ref{ss:6.3-diff-charact}.

Finally, in analogy with the case of the `single-rate'
vanishing-viscosity approach developed in \cite{MRS12, MRS13}, here as well we
introduce ``true \emph{Balanced-Viscosity solutions}''  (shortly referred to as
\emph{BV solutions})  as the
\emph{non-parametrized counterpart} to $\pBV$ solutions, see Definition
\ref{def:trueBV}.  These solutions are functions of the \emph{original time
  variable} $t\in [0,T]$ and fulfill an energy balance that again encompasses
the contribution of the viscous dissipation potentials $\disv u$ and $\disv z$
to the description of energy dissipation at jump times of the solution.  We are
going to show that true $\BV$ solutions are related to $\pBV$ solutions in a
canonical way, see Theorem \ref{th:pBV.v.BVsol}.  What is more, in Theorems
\ref{thm:exist-trueBV} and \ref{thm:exist-nonpar-enh} we provide general
assumptions that guarantee that all pointwise-in-time limits of a family of
(\emph{non-parametrized}) viscous solutions $q_\epsk:[0,T]\to \Spq$, for
$\epsk \to 0^+$, is indeed a $\BV$ solution.

We emphasize that the definition of BV solutions is independent of the
vanishing-viscosity approach. This  independence   guarantees that the
solution concept is indeed stable under parameter variations in the way shown
in   \cite[Thm.\,4.8]{MRS2013} for generalized gradient systems (cf.\ also
 \cite[Thm.\ 4.2]{MiRoSa_VARRIS12}).    Otherwise,  doing the limit $\eps\to 0^+$ first and
then a parameter limit $\delta\to \delta_*$ it is not possible to show that the
obtained limit curve is a vanishing-viscosity limit for fixed $\delta_*$, see
Remark \ref{rm:VVAvsBV}.  In principle, our general definition of
(parametrized) BV solutions for limiting rate-independent systems can be used
and analyzed independently of the vanishing-viscosity approach. However, to
avoid overburdening the present work we do not following this line and restrict
ourselves to situations where existence of solutions can be established exactly
by these methods. After all, this is the mechanical motivation for considering
such solution classes.   %

\subsection{Application to a model for delamination}
In Section \ref{s:appl-dam} we show that our existence results for $\pBV$
solutions, characterized by \eqref{e:diff-char-intro}, and (true) $\BV$
solutions apply to a rate-independent process modeling delamination between two
elastic bodies in adhesive contact along a prescribed interface.  For a first
approach to energetic solutions for this delamination problem, we refer to
\cite{KoMiRo06RIAD}. A systematic approach to BV solutions for a  multi-rate system 
involving elastoplasticity and damage is given in \cite{Crismale-Rossi19}. 

The vanishing-viscosity analysis for the viscously regularized 
delamination model  poses nontrivial challenges due to the presence of
various maximal monotone nonlinearities, in the displacement equation and in
the flow rule for the  delamination variable $z$, which for instance render the
constraints  $z(t,x) \in [0,1]$ and the 
unidirectionality of the  evolution. In particular, the main  challenge is to 
obtain the a priori estimate \eqref{u-bound-intro} uniformly in $\eps$
when  taking the vanishing-viscosity limit. For this, it is necessary to
carefully regularize the viscous system. Because of the relatively weak
coupling between the displacement equation and the flow rule for $z$, the
smoothened system possesses a \emph{semilinear structure} that allows us to
apply the techniques developed in \cite[Sec.\,4.4]{Miel11DEMF} and
\cite[Sec.\,2]{Mielke-Zelik}, see Section \ref{su:DelamSmooth}.

\Subsection{Plan of the paper}
\label{su:PlanPaper}

In Section \ref{s:EDI} we introduce a prototype of the coupled
systems that we aim to mathematically model through the Balanced-Viscosity
concept. In this simplified context, avoiding technicalities we illustrate the
notion of (parametrized) $\BV$ solution and its mechanical interpretation.

Section \ref{se:Tools} contains some auxiliary tools on that will be central
for the rest of the paper.  It revolves around the construction of
\emph{vanishing-viscosity contact potential} that will be relevant for
describing the dissipative behavior of the viscously regularized system in the
multi-rate case with $1$, $\eps$, and $\eps^\alpha$. In fact, it will enter
into the definition of the function $\mename 0\alpha$ in
\eqref{lim-EDbal-intro}. Since in this paper we will extend the analysis of
\cite{MRS14} to general viscous dissipation potentials, we will not be able to
explicitly calculate the related vanishing-viscosity contact potential except
for particular cases. Thus, a large part of Section \ref{se:Tools} will focus on
the derivation of general properties of contact potentials that will lay the
ground for the study of the dissipation function $\mename 0\alpha$.  

In Section \ref{s:setup} we thoroughly establish the setup for our analysis,
specifying the basic conditions on the spaces, on the energy functional, and on
the dissipation potentials.  Moreover, Theorem \ref{th:exist} recalls the
existence result from \cite{MRS2013} for the viscous system
\eqref{van-visc-intro}.  Section \ref{su:AprioViscSol} is devoted to the
derivation of a priori estimates for the solutions $(u_\eps,z_\eps)_\eps$ to
\eqref{van-visc-intro} that are uniform with respect to the parameter $\eps$.

Section \ref{su:ReJoMFcn} entirely revolves around the functional
$\mename 0\alpha$ that has a central role in the definition of both $\pBV$ and
true $\BV$ solutions. In particular, (i) we motivate its definition as the
Mosco limit of the family of the time-integrated dissipation
functional appearing in \eqref{EnDissBal-intro}, and (ii) relying on the
results from Section \ref{se:Tools} we compute the limit $\mename 0\alpha$
explicitly and investigate its properties. In the subsequent subsections we
give the definition of parametrized Balanced-Viscosity solution to the
rate-independent system $\RIS$, state our existence results in Theorem
\ref{thm:existBV} (and Theorem \ref{thm:exist-enh-pBV} for enhanced $\pBV$
solutions), and present the characterizations of $\pBV$ in terms of the
subdifferential inclusions \eqref{e:diff-char-intro}, cf.\ Theorem
\ref{thm:diff-charact}.

In Section \ref{ss:4.3} we introduce \emph{true} $\BV$ solutions and state our
existence result Theorem \ref{thm:exist-trueBV} (and Theorem
\ref{thm:exist-nonpar-enh} for enhanced BV solutions). In particular, we show
that these solutions are obtained by taking the vanishing-viscosity limit in
system \eqref{van-visc-intro} written in the \emph{real time} variable
$t\in [0,T]$.  We also gain further insight into the description of the jump
dynamics provided by true $\BV$ solutions.

The proofs of the main results of Sections \ref{s:4+} and \ref{ss:4.3} are
carried out in Section \ref{s:8}.

Section \ref{s:appl-dam} shows that our abstract setup is suitable to
handle a concrete application to in solid mechanics.  In particular,
in Theorem \ref{thm:BV-adh-cont} we prove the existence of enhanced
parametrized and true $\BV$ solutions for a  viscoelastic model with
delamination along a prescribed interface. 

\Subsection{General notations}
\label{su:I.GenNotations}
Throughout the paper, for a given Banach space $X$, we will denote its norm by
$\| \cdot \|_X$. For product spaces $X \ti \cdots \ti X$, we will often (up to
exceptions) simply write $\| \cdot \|_X$ in place of
$\| \cdot\|_{X {\ti \cdots \ti}X}$.  By $\pairing{}{X}{\cdot}{\cdot}$ we shall
denote both the duality pairing between $X$ and $X^*$ and the scalar product in
$X$, if $X$ is a Hilbert space.

We shall use the symbols $c,\,c',\, C,\,C'$, etc., whose meaning may vary even
within the same line, to denote various positive constants depending only on
known quantities. Furthermore, the symbols $I_i$, $i = 0, 1,... $, will be used
as place-holders for terms involved in the various estimates: we warn the
reader that we will not be self-consistent with the numbering, so that, for
instance, the symbol $I_1$ will occur several times with different meanings.

\Section{A prototypical class of coupled systems}
\label{s:EDI}

In this section we illustrate the notion of parametrized $\BV$ solution for a
prototypical  and simple  class of coupled systems to which the
existence and characterization results obtained in the sequel will apply. 
In particular, it contains a model combining linearized viscoelasticity and
viscoplasticity.  We shall confine the discussion to the particular case in
which the ambient spaces
\begin{subequations}
\label{prototypical-setup}
\begin{equation}
\label{prot-Hilb}
\text{$\Spu$  and $\Spz$ are Hilbert spaces,}
\end{equation}
 the viscous  dissipation potentials are \emph{quadratic}, namely
\begin{equation}
\label{prot-pot}
 \disv u : \Spu \to [0,\infty);\ u'\mapsto   \frac12 \la \mathbb V_\sfu u',u'\ra,
\qquad \disv z : \Spz \to [0,\infty);\ z' \mapsto \frac12 \la \mathbb V_\sfz
z', z'\ra ,  
\end{equation}
 with bounded linear and symmetric operators
$\mathbb V_\sfx:\Spx\to \Spx^* $,  and the driving energy functional is of
the form
\begin{equation}
\label{prot-en}
\ene tuz: = \frac12 \pairing{}{\Spu}{\bbA u}{u}  +
\pairing{}{\Spz}{\bbB u}{z} +\frac12
\pairing{}{\Spu}{\bbG u}{u}  -  \pairing{}{\Spu}{f(t)}{u}  -
 \pairing{}{\Spz}{g(t)}{z} ,
\end{equation}
\end{subequations}
where $\bbA : \Spu \to \Spu^* $ and $\bbG: \Spz \to \Spz^*$ are
linear, bounded and self-adjoint, $\bbB : \Spu \to \Spz^*$ is linear and
bounded, and  $(f,g): [0,T]\to \Spu^*\ti \Spz^*$ are smooth time-dependent
applied forces. Moreover, we assume that the block operator 
$\binom{\bbA\ \bbB^*}{\bbB\ \bbG}$  is positive semidefinite.  Together with the
$1$-homogeneous potential $\calR: \Spz \to [0,\infty)$ the viscous system
\eqref{van-visc-intro} reads 
\begin{subequations}
\label{prot-van-visc-syst}
\begin{align}
&
\label{prot-van-visc-u}
\qquad \quad \eps^\alpha \mathbb V_\sfu u' + 
\bbA u+ \bbB^*z = f(t) && \text{in } \Spu 
&& \foraa\, t \in (0,T),
\\
&
\label{prot-van-visc-z}
\pl \calR(z') + \eps  \mathbb V_\sfz z' +
\bbB u +  \bbG z =g(t) && \text{in } \Spz
&& \foraa\, t \in (0,T), 
\end{align}
\end{subequations}
with $\mathbb V_\sfx$ from \eqref{prot-pot}. It will be important to allow for
coercivity of $\calR$ on a Banach space $\Spy$ such that $\Spz \subset \Spy$
continuously and $\calR(z')\geq c\|z'\|_{\Spy}$ for all $z'\in \Spy$.

\begin{example}[Linearized elastoplasticity with hardening]
  \slshape Let the elastoplastic body
  occupy a bounded Lipschitz domain $\Omega\subset\R^d$: linearized
  elastoplasticity is described in terms of the displacement $u:\Omega
  \to \R^d$  with $u(t)\in \Spu=\rmH^1_0(\Omega)$ for simplicity  and in
  terms of the symmetric, trace-free
  plastic strain tensor $z:\Omega \to \R_{\mathrm{dev}}^{d\ti d} 
  := \bigset{z\in \R_{\mathrm{sym}}^{d\ti d}}{ \mathrm{tr}(z)=0}$.  The
  driving energy functional  $\calE: [0,T]\ti \Spu\ti \Spz \to \R$ with
  $\Spz= \rmL^2(\Omega; \R_{\mathrm{dev}}^{d\ti d})$ is defined by
\[
\ene tuz : = \int_\Omega \left( \tfrac12 (e(u){-}z){:}
  \mathbb{C}(e(u){-}z)  {+}\tfrac12z{:} \mathbb{H} z \right) \dd  x -   
 \pairing{}{H_0^1(\Omega;\R^d)}{f(t)}{u}
\]
with $e(u)$ the linearized symmetric strain tensor,
$\mathbb{C} \in \mathrm{Lin}(\R_\mathrm{sym}^{d\ti d}) $ and
$\mathbb{H}\in \mathrm{Lin}(\R_\mathrm{devm}^{d\ti d})$ are the positive
definite and symmetric elasticity and hardening tensors, respectively, and
$f: [0,T]\to \rmH^{-1}(\Omega;\R^d)$ a time-dependent volume loading. The
dissipation potentials are
\[
\calR(z') = \int_\Omega  \sigma_\mathrm{yield}|z'| \dd x , \quad 
\disv u(u'): = \int_\Omega \tfrac12 e(u') : \bbD e(u') \dd x, \quad \disv z
  (z') : = \int_\Omega \tfrac12 z': \mathbb V z' \dd x
\]
where $\sigma_\mathrm{yield}>0$ is the yield stress and $\mathbb{D} \in
\mathrm{Lin}(\R_\mathrm{sym}^{d\ti d})$  and $\mathbb V \in
\mathrm{Lin}(\R_\mathrm{devm}^{d\ti d})$  are  the  symmetric and positive definite
viscoelasticity and viscoplasticity tensors, respectively. 

Hence,  system \eqref{prot-van-visc-syst} translates into
\[
\begin{aligned}
  -\mathrm{div}\big(\eps^\alpha \bbD e(u') + \bbC (e(u){-}z)\big) \ \ \qquad
  & = f(t) && \text{ in } \Omega \ti (0,T), 
  \\
  \sigma_\mathrm{yield} \Sign (z') + \eps \mathbb V  z' + \mathrm{dev}\big(\bbC
  (z{-}e(u))\big) + \mathbb{H} z & \ni\ \ 0 && \text{ in } \Omega \ti (0,T).
\end{aligned}
\]
where ``$\mathrm{dev}$'' projects to the deviatoric part, namely
$\mathrm{dev}\,A = A - \frac1d (\mathrm{tr}\!\; A)\,I$.
\end{example}

For the system $\RIS$ from \eqref{prototypical-setup}, featuring $2$-positively
homogeneous dissipation potentials, the time-space dissipation function
$\mename 0\alpha$ that enters into the definition of (parametrized) Balanced
Viscosity solution can be explicitly computed (cf.\ Example \ref{ex:p-homog}
ahead). Nonetheless, here we can give an even more transparent illustration of
(parametrized) $\pBV$ solutions in terms of their differential
characterization \eqref{e:diff-char-intro}.  The upcoming Theorem
\ref{thm:diff-charact} states that
a triple $(\sft,\sfu,\sfz)$ is an (\emph{enhanced}) parametrized $\BV$
solution if and only if it solves, for almost all $s\in (0,\sfS)$, 
\begin{equation}
\label{param-subdif-incl-2}
\begin{aligned}
   \thn u(s) \mathbb V_\sfu {\sfu'(s)} +{}&\bbA  \sfu(s)+
    \bbB^* \sfz(s)  \ni f(\sft(s)) && \text{in }
  \Spu^*,
  \\
   \pl \calR(\sfz'(s)) + \thn z(s) \mathbb V_\sfz{\sfz'(s)}+ {}&
   \bbB \,\sfu(s) \,+ \bbG \,\sfz(s)   \ni g(\sft(s)) && \text{in } \Spz^*,
\end{aligned}
\end{equation}  
joint  with the `switching conditions'
\eqref{switch-intro-1}--\eqref{switch-intro-2} on the measurable functions
$\thn u, \, \thn z: (0,\sfS)\to [0,\infty]$.  Here ``$\infty \mathbb
V_\sfz{\sfz'}\,$'' has to be interpreted in the sense of $\pl  \disve
z{\infty}(\sfz')$, see \eqref{eq:SubDiff.Vx}. 

We recall that
\eqref{switch-intro-1} simply ensures that, if the system is not jumping (i.e.,
$\sft'(s)>0$), then viscosity does not come into action, i.e.\ 
$\thn u(s)=\thn z(s)=0$.   This means that $\sfu(s)$ is in `E'quilibrium with respect
to $\sfz(s)$ and the loading $f(\sft(s))$, whereas $\sfz$ evolves according to
the truly `R'ate-independent evolution $\pl \calR(\sfz')+\bbB \sfu+
\bbG\sfz   \ni g$, hence we will denote this  evolution  
regime by $\rgs Eu \rgs Rz$ in Section \ref{ss:6.3-diff-charact}.  

Conditions \eqref{switch-intro-2} differ in the
three cases $\alpha=1$, $\alpha>1$ and $\alpha\in (0,1)$ and indeed show how
the (possibly different) relaxation rates of the variables $u$ and $z$
influence the system behavior at jumps,  see Section
\ref{ss:6.3-diff-charact} for a full discussion  of the occurring evolution
regimes.   

For \underline{\emph{$\alpha=1$}} the variables $u$ and $z$ relax with the same
rate: at a jump, the system \emph{may} switch to a viscous regime where the
viscosity in $u$ and in $z$ 
are involved  \emph{equally}, since the
coefficients $\thn u$ and $\thn z$ modulating the  `V'iscosity  terms $\mathbb
V_\sfu{\sfu'}$ and $\mathbb V_\sfz{\sfz'} $ coincide.  This  evolution 
regime will be denoted $\rgs V{uz}$. 

For \underline{\emph{$\alpha>1$}} the switching condition
\eqref{switch-intro-2} imposes that either $\thn z =\infty$ (i.e.\ $\sfz'=0$)
or that $\thn u=0$ (so that $\sfu$ is at equilibrium). Indeed, since $\sfu$
relaxes `V'iscously faster to equilibrium than $\sfz$ to rate-independent
evolution, $\sfz$ is `B'locked until $\sfu$ has reached the equilibrium: and we
call this  evolution  regime $\rgs Vu \rgs Bz$.  After that $\sfu$ is
in `E'quilibrium and $\sfz$ may have a `V'iscous transition with $\thn z>0$, a
regime denoted by $\rgs Eu \rgs Vz$.  Moreover, under suitable conditions on
the operators $\bbA $, $\bbB $, and $\bbG$ which in particular ensure that the
functional $\ene t{\cdot}z$ from \eqref{prot-en} is uniformly convex, the
arguments from \cite[Prop.\,5.5]{MRS14} may be repeated for the system $\RIS$
defined via \eqref{prototypical-setup}. Hence, it is possible to show that the
regime $\rgs Vu \rgs Bz$ can only occur once in the initial phase, while $\sfu$
never leaves equilibrium afterwards, i.e.\ only $\rgs Eu \rgs Rz$ and
$\rgs Eu \rgs Vz$ are possible.

For  \underline{\emph{$\alpha \in (0,1)$}} the variable $\sfz$ relaxes faster
than $\sfu$, which leads to the two viscous regimes: (i) $\rgs Bu \rgs Vz$ where
$\sfu$ is blocked ($\thn u=\infty$) while $\sfz$ evolves viscously, and (ii)
$\rgs Vu \rgs Rz$ where $\sfu$ relaxes to equilibrium while $\sfz$ stays in
locally stable states ($\thn z=0$). For $\alpha \in (0,1)$ these  two
regimes and $\rgs Eu \rgs Rz$  may occur more than once in the evolution of
the system.

\Section{Some auxiliary tools  for dissipation potentials}
\label{se:Tools}

In this section we prepare a series of useful tools  for handling the
balanced effect of the different dissipation potentials. They   will be
essential for the upcoming analysis and may be interesting elsewhere.

\begin{definition}[Primal and dual dissipation potentials]
\label{def:DissPotential} Let $X$ be a
  reflexive Banach space. Then, a function $\psi:X\to [0,\infty]$ is called a
  \emph{(primal) dissipation potential}, if 
\[
\psi \text{ is  convex, lower semicontinuous (lsc, for
short) and } \ \psi(0)=0.
\] 
The \emph{dual  dissipation potential} $\psi^*:X^*\to [0,\infty]$ is defined
via Legendre-Fenchel conjugation  as 
\[
\psi^*(\xi): = \sup\bigset{\langle \xi,v\rangle - \psi(v)}{ v\in X}.
\]
\end{definition} 

Note that $\psi^*$ is indeed again a dissipation potential, and we have
$(\psi^*)^*=\psi$. In this section, 
we allow for functionals $\psi$
taking the value $\infty$ as well as degenerate functionals such that  $\psi(v)=0$
for $v\neq 0$.  With $\psi$ we associate the  \emph{B-function} 
\begin{equation}
 \label{Bfunctions-2}
  \calB_{\psi} : (0,\infty) \ti X \ti [0,\infty) \to [0,\infty], \qquad 
 \Bfu {\psi}\tau v\sigma: =\tau \psi\left(\frac v\tau\right) +\tau\sigma\,. 
\end{equation}
 We highlight the rescaling properties of $\calB_\psi$ as follows
\begin{equation}
  \label{eq:mfB.rescale}
   \Bfu {\psi}\tau v\sigma = \tau \, \Bfu {\psi} 1{\frac1\tau v}\sigma =
   \frac1\delta  \Bfu {\psi}{\delta\tau}{\delta v}\sigma 
\quad \text{for all }\delta>0.
\end{equation}
We will  use  that the functional $ \calB_{\psi}(\cdot,\cdot, \sigma)$ is
convex for all $\sigma \geq 0$. To see this, we consider $\tau_0,\, \tau_1
\in (0,\infty)$, $v_0,v_1 \in X$, and $\theta \in [0,1]$ and 
set $\tau_\theta: = (1{-}\theta)\tau_0+\theta\tau_1 >0$ and
$v_\theta: = (1{-}\theta)v_0+\theta v_1 $.   With this we find 
\begin{equation}
 \label{miracle-HS}
 \begin{aligned}
   \Bfu {\psi}{\tau_\theta}{v_\theta}\sigma  &= \tau_\theta \,\psi\left(\frac
     {v_\theta}{\tau_\theta}\right) +\tau_\theta\sigma
  = \tau_\theta \,\psi\Big( \frac{(1{-}\theta)\tau_0}{\tau_\theta}
   \,\frac1{\tau_0} v_0 + \frac{\theta  \tau_1}{\tau_\theta} \,\frac1{\tau_1}
   v_1 \Big) +  \tau_\theta \sigma
   \\
   & \overset{(1)}{\leq} \tau_\theta\Big( \frac{(1{-}\theta)
     \tau_0}{\tau_\theta} \psi\left(\frac{v_0}{\tau_0} \right) +
   \frac{\theta\tau_1}{\tau_\theta} \psi\left(\frac{v_1}{\tau_1} \right) \Big)
   + \big(1{-}\theta)\tau_0 \sigma + \theta\tau_1 \sigma 
   \\ & 
   =(1{-}\theta)
   \Bfu\psi{\tau_0}{v_0}\sigma + \theta \Bfu\psi{\tau_1}{v_1}\sigma,
    \end{aligned}
\end{equation}
where in $\overset{(1)}{\leq} $  we used the convexity of $\psi$.  
We next define the functional
\begin{align}
&
\label{bipotentials-1} \qquad
\mfb_\psi: X\ti [0,\infty) \to [0,\infty];  && \mfb_{\psi}(v,\sigma)   : =
\inf\bigset{\Bfu {\psi}\tau v\sigma}{\tau>0}. 
\end{align}
We shall refer to the functional $\mfb_\psi $ as \emph{vanishing-viscosity
  contact potential} associated with $\psi$, in accordance with the terminology
used in \cite{MRS12}. As we will see, $\mfb_\psi$ will be handy for describing
the interplay of vanishing viscosity and time rescaling upon taking the limit
of \eqref{van-visc-intro}.

\Subsection{Properties of vanishing-viscosity contact potentials $\mfb_\psi$}
\label{su:mfb.psi}

For arbitrary dissipation potentials $\psi$, we  define the 
\emph{rate-independent part} $\psi_\mathrm{ri}:X \to [0,\infty]$ via
\begin{equation}
  \label{eq:RI.part}
  \psi_\mathrm{ri}(v)=\lim_{\gamma \to 0^+} \frac1\gamma \psi(\gamma
  v) = \sup\bigset{\langle \eta,v\rangle_X }{ \eta \in \pl \psi(0)
  }.
\end{equation}
The following results are slight variants of the results in
\cite[Thm.\,3.7]{MRS12}.

\begin{proposition}[Properties of  vanishing-viscosity contact potentials] 
\label{pr:VVCP}
Assume that the dissipation potential $\psi:X \to [0,\infty]$ is superlinear,
i.e.\ 
\begin{equation}
\label{ass-diss-pot-superl}
\lim_{\|v\|_X\to \infty}\frac{\psi(v)}{\|v\|_X} =  \infty\,.
\end{equation}
Then,  $\mfb_\psi$ 
 has  the following properties: 
\begin{compactenum}
\item[\upshape(b1)] $\mfb_\psi(v,\sigma)=0$ implies $\sigma=0$ or
  $v=0$. Moreover, $\mfb_\psi(0,\sigma)=0$ for all $\sigma\geq 0$.
\item[\upshape(b2)] For all $v\in X$ the function
  $\mfb_\psi(v,\cdot):[0,\infty) \to [0,\infty]$ is nondecreasing and concave.
  For $v\neq 0$ and $\sigma>0$ the  infimum  in the definition of
  $\mfb_\psi$ is  attained  at a value $\tau_{v,\sigma} \in (0,\infty)$.
  Moreover, for all $v \neq 0 $ and $\sigma>0$ we have
  $\mfb_\psi (v,\sigma) > \mfb_\psi (v,0)=\psi_\mathrm{ri}(v)$.
\item[\upshape(b3)] For all $\sigma\geq 0$ the function
  $\mfb_\psi(\cdot,\sigma):X \to {[0,\infty)}$ is positively 1-homogeneous,
  lsc, and convex.
 \item[\upshape(b4)]  If $\psi= \phi+\varphi$ where $\phi$ is 1-homogeneous, then
$\mfb_{\phi+\varphi}(v,\sigma)= \phi(v) + \mfb_\varphi(v,\sigma)$.  
\item[\upshape(b5)] For all $(v,\eta)\in X\ti X^*$ we have
  $ \mfb_\psi(v,\psi^*(\eta)) \geq \langle \eta,v\rangle_X$.
\end{compactenum}
\end{proposition}
\begin{proof} The main observation is that the function
  $g_{v,\sigma}: (0,\infty) \ni \tau \mapsto \tau \psi(\frac1\tau v) + \tau
  \sigma$ is convex  (cf.\ \eqref{miracle-HS})  and takes only
  nonnegative values. For $\sigma>0$ we have $g_{v,\sigma}(\tau) \to \infty$
  for $ \tau \to \infty$, and for $v\neq 0$ we have
  $g_{v,\sigma}(\tau)\to \infty$ for $\tau \to 0^+$  due to superlinearity
  of $\psi$. 

\underline{Part (b1):} If
$\mfb_\psi(v,\sigma) = \inf g_{v,\sigma}(\cdot)=0$, the infimum must either be
realized for $\tau \to 0^+$ or for $\tau \to \infty$. In the first case,
the value of $\sigma $ does not matter, but the superlinearity of $\psi$ gives
$\tau \psi(\frac1\tau v) \to \infty$, unless $v=0$. In the second case we
have $\tau \sigma \to \infty$, unless $\sigma=0$. The relation
$\mfb_\psi(0,\sigma)=0$ is obvious.

\underline{Part (b2):} The first  two  statements  follow because
$\mfb_\psi(v,\cdot)$ is the infimum of a family of functions that are 
increasing and concave in $\sigma$.  For $v\neq 0$ and $\sigma>0$ the
minimum of $g_{v,\sigma}(\tau)$ is achieved at a
$\tau_{v,\sigma} \in (0,\infty)$ as $g_{v,\sigma}(\tau) \to \infty$
on both sides  (i.e., as $\tau \to 0^+$ and $\tau \to \infty$). 
 Thus, $\mfb_\psi(v,\sigma)\geq \mfb_\psi(v,0)+ \sigma
\tau_{v,\sigma} >\mfb_\psi(v,0)$ as desired.  The relation $\mfb_\psi (v,0) =
\psi_\mathrm{ri}(v)$ follows easily from the convexity of $\psi$.  

\underline{Part (b3):} The positive 1-homogeneity $\mfb_\psi(\gamma v,
\sigma)=\gamma \mfb_\psi(v,\sigma)$ for all $\gamma>0$ follows by
replacing $\tau $ by $\tau \gamma$.  Convexity is obtained as
follows. For fixed $v_0,v_1\in X$, $\theta \in (0,1)$, and $\sigma\geq
0$, we choose $\eps>0$ and find $\tau_0,\tau_1>0 $ such that for
$j\in \{0,1\}$ we have 
\begin{equation}
  \label{eq:mfb.psi.eps}
  \tau_j\,\psi\big(\frac1{\tau_j} v_j,\sigma \big) + \tau_j \sigma \leq
\mfb_\psi(v_j,\sigma)+\eps.
\end{equation}
Here we assumed without loss of generality
$\mfb_\psi(v_j,\sigma)<\infty$ since otherwise there is nothing to be
shown. Now we set $v_\theta= (1{-}\theta) v_0 + \theta v_1 $
and $\tau_\theta =(1{-}\theta) \tau_0 + \theta \tau_1>0$.
Using the convexity \eqref{miracle-HS} of the functional
$\Bfu{\psi}{\cdot}{\cdot}{\sigma}$, 
we obtain
\begin{align*}
\mfb_\psi(v_\theta,\sigma) \leq \Bfu\psi{\tau_\theta}{v_\theta}\sigma
 & \leq  (1{-}\theta) \Bfu\psi{\tau_0}{v_0}\sigma +\theta  \Bfu\psi{\tau_1}{v_1}\sigma 
 \\
  & 
\leq (1{-}\theta)\mfb_\psi(v_0,\sigma) + \theta \mfb_\psi(v_1) + \eps,
 \end{align*} 
 with the last inequality due to \eqref{eq:mfb.psi.eps}. 
Since $\eps>0$ was arbitrary, this is the desired result. 

To prove lower semicontinuity, 
 we use the special way $\mfb_\psi$ is
constructed and that $\psi $ is lsc. For all sequences $v_j\to v_*$
 and $\sigma_j\to \sigma_*$ we have to show
\[
 \mfb_\psi (v_*,\sigma_*) \leq  \alpha:=\liminf_{j\to \infty}
 \mfb_\psi  (v_j,\sigma_j)
\]
We may assume $\alpha<\infty$ and $\mfb_\psi(v_*,\sigma_*)>0$, since otherwise
 the desired estimate is trivial. 

The case $\sigma_*=0$ is simple, as we have 
\[
\alpha=\liminf_{j\to \infty}  \mfb_\psi  (v_j,\sigma_j) \geq 
\liminf_{j\to \infty}  \mfb_\psi (v_j,0) \geq 
\liminf_{j\to \infty} \psi_\mathrm{ri}(v_j) \geq 
\psi_\mathrm{ri}(v_*) =\mfb_\psi(v_*,0)= \mfb_\psi(v_*,\sigma_*).
\]
It remains to consider the case $v_*\neq 0$ and $\sigma_*>0$. Since
$\|v_j\|\geq \|v_*\|/2>0$ and $\sigma_j\geq \sigma_*/2>0$ for sufficiently
large $j$, we see that the
optimal $\tau_j=\tau_{v_j,\sigma_j}$ lie in a set $[1/M,M]\Subset
(0,\infty)$. Thus, choosing a subsequence (not relabeled), we may
assume $\tau_j\to \tau_\circ$ and obtain lower semicontinuity
by using $\frac1{\tau_j}v_j \to \frac1{\tau_\circ}v_*$ as follows:
\[
\alpha=\liminf_{j\to \infty}  \mfb_\psi  (v_j,\sigma_j)= 
\liminf_{j\to \infty} \left( \tau_j\psi\big(\frac1{\tau_j}v_j\big) +
\tau_j \sigma_j \right) \geq \tau_\circ\psi\big(\frac1{\tau_\circ}v_*\big) +
\tau_\circ \sigma_* \geq \mfb_\psi(v_*,\sigma_*). 
\]

 \underline{Part (b4):} The formula for $\mfb_{\phi+\varphi}$ follows from a direct calculation.

\underline{Part (b5):} We have 
$g_\tau( v, \psi^*(\eta)) = \tau \Big(\psi\big(\frac1\tau
v\big) + \psi^*(\eta) \Big) \geq \tau\big(\langle
\eta,\frac1\tau v\rangle  \big) = \langle \eta,v\rangle,
$ 
and taking the infimum over $\tau>0$ gives the result. 
Thus, Proposition \ref{pr:VVCP} is proved.
\end{proof}

There is a canonical case in which $\mfb_\psi$ can be given explicitly, namely
the case that $\mfb_\psi(v)$ only depends on the Banach-space norm $\|v\|$.  In
that case we have an explicit expression for  $\mfb_\psi$ and  the
functional $ X \times X^* \ni (v,\eta) \mapsto \mfb_\psi(v,\psi^*(\eta))$.

\begin{lemma}[Dissipation potentials depending on the norm]
\label{le:psi.norm} Assume that $\psi$ is given in the form
$\psi(v)=\zeta(\|v\|)$, where $\zeta:[0,\infty)\to [0,\infty]$ satisfies
$\zeta(0)=0$ and is lsc, nondecreasing, convex, and superlinear. Setting 
$\zeta'(0)= \lim_{h\to 0^+} \frac1 h \zeta(h)$ we have the identities
\begin{equation}
\label{explicit-vvcp}
\begin{aligned}
\mfb_\psi(v,\sigma)&= \| v\| \kappa_\zeta(\sigma) \text{ with }
\kappa_\zeta(\sigma):=\inf\bigset{\tau \zeta(1/\tau)+\tau\sigma}{ \tau>0},
\\
 \mfb_\psi(v,\psi^*(\xi)) &= \| v\| \,\max\big\{ \zeta'(0)\, , \, \|\xi\|_* \big\}.
\end{aligned}
\end{equation}
\end{lemma}
\begin{proof} The first statement is trivial for $v=0$. For $v\neq 0$ we can
replace $\tau$ by $\tau \|v\|$ and obtain the desired product form with
$\|v\|$ as the first factor. 

To obtain  the second statement in  \eqref{explicit-vvcp} we first
observe that $\psi^*(\xi)=\zeta^*(\|\xi\|_*)$ with
$\zeta^*(r)=\sup\bigset{r\rho-\zeta(\rho)}{\rho\geq 0}$. As $\zeta$ is
superlinear $\zeta^*(r)$ is finite for all $r\geq 0$, and $\zeta^*(r)=0$ for
$r\in [0,\zeta'(0)]$. Secondly, we characterize $\kappa_\zeta$ by using the
following estimate
\[
\kappa_\zeta(\zeta^*(r)) =\inf\bigset{\tau\big(\zeta(\tfrac1\tau) 
 {+} \zeta^*(r)\big)}{\tau>0}
\geq \inf\bigset{\tau\big( \tfrac1\tau \,r \big)}{\tau>0} = r. 
\] 
The inequality is even an identity if the infimum is attained, which is the
case of $\frac1\tau \in \pl  \zeta^*(r)$ for some $\tau$. Thus, we have
attainment whenever $\zeta^*(r)>0$, whereas for $r\in [0,\zeta'(0)]$, where
$\zeta^*(r)=0$, we have non-attainment but find $\kappa_\zeta(0)= \zeta'(0)$.
Together we arrive at $\kappa_\zeta(\zeta^*(r))= \max\{\zeta'(0),r\}$ (see also
\cite[Sec.\,2.3]{LiMiSa18OETP}), and
$ \mfb( v,\psi^*(\xi))=\| v\| \kappa_\zeta(\zeta^*(\|\xi\|_*))$ gives the
desired result.
\end{proof}

The above result shows that the estimate
$\mfb_\psi( v,\psi^*(\xi)) \geq \la \xi,v\ra$ in (b4) improves to
\begin{equation}
  \label{used-later-HS}
  \mfb_\psi( v,\psi^*(\xi)) \geq \|\xi\|_*\|v\|
\end{equation}
in certain cases, in particular in the metric approach used in \cite{RMS08,
  MRS09}. As some of the following examples show, the latter estimate is not
true in general, and that is why we will derive general lower bounds on the
vanishing-viscosity contact potential $\mfb_\psi$  in Section
\ref{su:LowerBounds}.

\begin{example}[The function $\mfb_\psi$ for  some special cases] 
\label{ex:VVCP}\slshape 
The following cases 
give some intuition about the  vanishing-viscosity contact potential  $\mfb_\psi$.  
\\[0.3em]
(A) Assume that $\psi$ is positively $p$-homogeneous with $p\in (1,\infty)$,
i.e.\ $\psi(\gamma v) = \gamma^p \psi(v)$ for all $\gamma>0$ and $v\in
X$. Then, we have
\begin{equation}
 \label{p-homo-mfb}
 \mfb_\psi(v,\sigma )=  \big(\psi(v)\big)^{1/p}\: \hat c_p\, \sigma^{1/p'}, \quad
\text{where } \hat c_p = p^{1/p} (p')^{1/p'}\text{ and }  \frac1p+\frac1{p'} =1.
\end{equation}
In particular, for $\psi(v) = \frac1p \|v\|^p$  we find
\[
 \mfb_\psi(v,\sigma ) = \|v\|\big(p'\sigma\big)^{1/p'} \quad \text{ and }\quad 
  \mfb_\psi(v,\psi^*(\eta)) = \|v\|\|\eta\|_*\, . 
\]
\\[0.3em]
(B)  On $X=\R^2$ consider $\psi(v)=\frac12(av_1^2+b v_2^2)$ with $a,b>0$. Then,
\[
\mfb_\psi(v,\sigma) = \big(av_1^2{+}b v_2^2\big)^{1/2} \:(2\sigma)^{1/2} \quad
\text{ and }\quad   \mfb_\psi(v,\psi^*(\xi))=\big(av_1^2{+}b
v_2^2\big)^{1/2}\big(\frac1a \xi_1^2{+}\frac1b \xi_2^2\big)^{1/2}.
\]
If $\R^2$ is equipped with the Euclidean norm $\|\cdot\|$, then $
\mfb_\psi(v,\psi^*(\xi)) \geq \big(\frac{\min\{a,b\}}{\max\{a,b\}}\big)^{1/2}
\| \xi\|_* \|v\|$, but estimate \eqref{used-later-HS} 
fails, while $ \mfb_\psi(v,\psi^*(\xi)) \geq \la \xi , v
\ra$ obviously holds.
\\[0.3em]
(C) Again for $X=\R^2$ consider $\psi(v)=\frac12v_1^2+\phi(v_2)$ with
\[
\phi(s)=\begin{cases} \frac12 s^2 &\text{for } |s|\leq 1,\\ 
  \frac14(|s|{+}1)^2 -\frac12 &\text{for }|s|\geq 1, \end{cases} \quad
\text{and} \quad \phi^*(r)=\begin{cases} \frac12 r^2 &\text{for } |r|\leq 1,\\ 
  r^2 -|r|  +\frac12 &\text{for }|r|\geq 1. \end{cases}
\]
An explicit calculation leads to the expression 
\[
\mfb_\psi(v,\sigma) = 
\begin{cases} \|v\|\sqrt{2\sigma} &\text{for } v_1^2 \geq (2\sigma{-}1) v_2^2,\\ 
 \frac12\sqrt{2v_1^2{+}v_2^2}\,\sqrt{4\sigma{-}1}+ \frac{|v_2|}2 
   &\text{for } \sigma\geq 1/2 \text{ and } v_1^2 \leq (2\sigma{-}1) v_2^2. 
\end{cases}
\]
Using $\psi^*(\xi)=\frac12\xi_1^2 +\phi^*(\xi_2)$ we find
$  \mfb_\psi(v, \psi^*(\xi))  =\|\xi\|_*\|v\|$ whenever
$\|\xi\|_*\leq 1$. However,  the explicit form of $\mfb_\psi$ shows that, in general,   $   \mfb_\psi(v,\psi^*(\xi)) $  cannot be expressed
in terms of $\|v\|$ and $\|\xi\|_*$ alone. With
$\psi^*(\xi)\geq \frac12\|\xi\|_*^2$ and $\mfb_\psi(v,\sigma)\geq
\frac12\big(\sqrt{4\sigma{-}1} +1)\|v\|$ for $\sigma\geq 1/2$ we obtain
$   \mfb_\psi(v,\psi^*(\xi))  \geq \|v\|\frac12\big(\sqrt{2\|\xi\|_*^2{-}1}+1\big)$ for
$\|\xi\|_*\geq 1$.  
\\[0.3em]
(D) We still look at the case $X=\R^2$ with the Euclidean norm
$\|v\|=\big(v_1^2{+} v_2^2 \big)^{1/2}$ and
\[
\psi(v)=\frac12\, v_1^2 + \frac14 \, v_2^4 \quad \text{and} \quad
\psi^*(\xi)=\frac12\, \xi_1^2 + \frac43 \, |\xi_2|^{4/3}.
\]
In principle, we can calculate $\mfb_\psi(v_1,v_2,\sigma)$ explicitly, however,
it suffices to use (A) giving
\[
\mfb_\psi(v_1,0,\sigma)= |v_1|(2\sigma)^{1/2}  \quad \text{and} \quad
\mfb_\psi(0,v_2,\sigma) = |v_2| \big(\tfrac43 \sigma\big)^{3/4}.
\]
Inserting $\sigma =\psi^*(\xi_1,\xi_2)$ and inserting the ``wrong directions''
with $\langle \xi, v \rangle =0$ we find
\[
 \mfb_\psi(v_1,0, \psi^*(\xi_1,\xi_2)) = \big(\tfrac83\big)^{1/2} |v_1|\,|\xi_2|^{2/3}
   \quad \text{and} \quad
\mfb_\psi(0,v_2, \psi^*(\xi_1,\xi_2))  
= \big(\tfrac23\big)^{3/4}
|v_2|\,|\xi_1|^{3/2}.
\]
Clearly, there cannot be a constant $c_0>0$ such that
$ \mfb_\psi(v,\psi^*(\xi)) \geq c_0 \| v\| \, \|\xi\|$ for all $v,\xi\in
\R^2$. Of course, the relations are compatible with (b4)  in Proposition
\ref{pr:VVCP},  i.e.\
$ \mfb_\psi(v,\psi^*(\xi)) \geq \langle \xi, v \rangle$.
\end{example}

As we will see, the vanishing-viscosity contact potentials $\mfb_\psi$, which
were developed for the case of two-rate problems  (with time scales 
$1$ and $\eps$) in \cite{MRS12}, are also relevant to describe the limiting
behavior of B-functions in the multi-rate case with  time scales  $1$,
$\eps$, and $\eps^\alpha$.  For this,  we will use  the concept of 
(sequential)  Mosco convergence, which we recall here  for a sequence
of functionals $\mathscr{F}_n: \scrX \to (-\infty,+\infty] $ defined in a
Banach space $\scrX$.

\begin{definition}[Mosco convergence]
\label{def:Mosco}
We say that $\mathscr{F}: {X} \to (-\infty,+\infty] $ is the \emph{Mosco limit} of the
functionals $(\mathscr{F}_n)_n$ as $n\to\infty$, and write
$ \mathscr{F}_n \overset{\rmM}\to \mathscr{F}$ in ${X}$, if the following two
conditions hold:
\begin{subequations}
 \label{Gamma-convergence}
\begin{align}
\label{Gamma-liminf-Fn}
&\Gamma\text{-}\liminf\ \mathrm{estimate:} \qquad x_n\weakto x \text{ weakly in } \scrX \ \ \Longrightarrow \ \ \mathscr{F}(x) \leq \liminf_{n\to\infty} \mathscr{F}_n(x_n);
\\ \nonumber
&\Gamma\text{-}\liminf\ \mathrm{estimate:}
\\
& \label{Gamma-limsup-Fn}
\qquad \qquad\forall\, x \in \scrX \ \exists\, (x_n)_n \subset \scrX: \ \  x_n\to x
\text{ strongly in } \scrX \ \text{ and }  \mathscr{F}(x) \geq
\limsup_{n\to\infty} \mathscr{F}_n(x_n). 
\end{align}
\end{subequations}
\end{definition}

\Subsection{Mosco convergence for the joint B-functions $\mfB_\eps^\alpha$}
\label{su:MutliRotePot}

In view of the vanishing-viscosity analysis of \eqref{van-visc-intro}, we now
work with two dissipation potentials $\psi_\sfu:\Spu\to [0,\infty]$ and
$\psi_\sfz:\Spz\to [0,\infty]$, with $\Spu$ and $\Spz$ the state spaces from
\eqref{intro-state_spaces}.  In Section \ref{su:ReJoMFcn}, we will indeed
confine the discussion to the choices $\psi_\sfu: = \disv u$ and
$\psi_{\sfz}: = \calR +\disv z$, but here we want to keep the discussion more
general and in particular allow for $\psi_\sfu$ to have a nontrivial
rate-independent part, too.

 When constructing the associated B-function 
 we have to take
care of the different scalings namely $\psi_\sfu^{\eps^\alpha}$ and
$\psi_\sfz^\eps$ in the sense of \eqref{eq:Def.Vx.lambda}, i.e.\
$\psi^\lambda(v)=\frac1\lambda \psi(\lambda v)$.  Indeed,   
 since the conjugate function
$(\psi^\lambda)^*$ satisfies the simple scaling law
$( \psi^\lambda)^*(\xi)= \frac1\lambda \psi^*(\xi)$, the B-function
$\calB_{\psi^\lambda}$ obeys the scaling relations
\begin{equation}
  \label{eq:rescaled.B.eps}
 \calB_{\psi^\lambda} (\tau,v, \frac1\lambda \sigma) = \calB_\psi(\frac\tau\lambda,
v, \sigma) = \frac1\lambda\, \calB_\psi( \tau, \lambda v, \sigma),
\end{equation} 
where we used \eqref{eq:mfB.rescale} for the last step.   Our definition of
 the 
associated B-function for the sum 
\[
  \Psi_{\eps,\alpha}  : \Spu{\ti} \Spz\to [0,\infty]; \quad 
  \Psi_{\eps,\alpha} 
  (u',z'): = \frac1{\eps^\alpha}\psi_\sfu (\eps^\alpha u') + \frac1\eps
  \psi_\sfz(\eps z')
\]
will be denoted by the symbol $\mfB_{\Psi_{\eps,\alpha}}$, see
\eqref{eq:Resc.Bae} below.  We emphasize that we deviate from the construction
set forth in \eqref{Bfunctions-2}, since \eqref{eq:Resc.Bae} applies
\eqref{eq:rescaled.B.eps} for each component individually.
Hence,  we introduce%
\begin{subequations}
  \label{eq:Resc.Bae}
\begin{align}
  \label{eq:Resc.Bae.a}
 \mfB_{\Psi_{\eps,\alpha}}(\tau,u',z',\sigmav u,\sigmav z)&:= 
 \frac1{\eps^\alpha}\,\Bfun_{\psi_\sfu} (\tau,\eps^\alpha u',\sigmav u) 
+ \frac1\eps\:\Bfun_{\psi_\sfz} (\tau,\eps z',\sigmav z) \\
  \label{eq:Resc.Bae.b}
 & \;= \;\Bfun_{\psi_\sfu} \big(\frac\tau{\eps^\alpha}, u',\sigmav u \big) 
+ \Bfun_{\psi_\sfz} \big(\frac\tau\eps , z',\sigmav z \big).
\end{align}
\end{subequations}
Subsequently,  we will use the short-hand notation $\mfB_\eps^\alpha$ in
place of  $\mfB_{\Psi_{\eps,\alpha}}$  and extend $\mfB_\eps^\alpha$ to allow for
the value $\tau=0$, defining  \emph{rescaled joint B-function} 
$\mfB_\eps^\alpha:[0,\infty)\ti \Spu\ti \Spz \ti [0,\infty)^2\to [0,\infty]$ via
\begin{equation}
  \label{eq:def.B.al.eps}
\RCBfu \eps\alpha{\tau}{u'}{z'}{\sigmav u}{\sigmav z}:=
\begin{cases}\displaystyle
  \frac\tau{\eps^\alpha}\, \psi_\sfu\Big(\frac{\eps^\alpha}\tau\,u'\Big) + 
  \frac\tau{\eps^\alpha} \,\sigmav u + 
\frac\tau{\eps} \,\psi_\sfz\Big(\frac{\eps}\tau\,z'\Big) + 
  \frac\tau{\eps} \,\sigmav z\  &\text{for }\tau>0,\\
 \infty&\text{for } \tau=0. 
 \end{cases} 
\end{equation}
We highlight that $\mfB_\eps^\alpha$ is 
relevant for the \emph{coupled} system \eqref{van-visc-intro}, hence the name 
rescaled \emph{joint} B-function. 

The next result shows that the Mosco limit $\mfB^\alpha_0$ of the
B-functions $(\mfB_\eps^\alpha)_\eps$ always exists and can be expressed in
terms of the potentials $\mfb_{\psi_\sfu}$, $\mfb_{\psi_\sfz}$, and
$\mfb_{\psi_\sfu{\oplus}\psi_\sfz}$.  We emphasize that
$(\tau,u',z')\mapsto \mfB^\alpha_0(\tau,u',z',\sigmav u,\sigmav z)$ is
1-homogeneous, which reflects the rate-independent character of the limiting
procedure.

\begin{proposition}[Mosco limit $\mfB^\alpha_0$ of the family $\mfB^\alpha_\eps$]
\label{pr:Mosco.Beps} 
Let $\psi_\sfu$ and $\psi_\sfz$ satisfy \eqref{ass-diss-pot-superl} and assume
$\alpha>0$. Then, $\mfB^\alpha_\eps$ Mosco converge to the limit
$\mfB^\alpha_0:[0,\infty)\ti \Spu\ti \Spz \ti [0,\infty)^2\to [0,\infty]$ that
is given as follows: {\upshape
\begin{align*}
\tau>0:\qquad &\RCBfu 0\alpha \tau{u'}{z'}{\sigmav u}{\sigmav z}= 
  \begin{cases} 
   \big(\psi_\sfu\big)_\mathrm{ri} (u')
    +\big(\psi_\sfz\big)_\mathrm{ri} (z') & \text{for } \sigmav u=\sigmav z=0, \\
    \infty& \text{otherwise};
  \end{cases}  
\\
\tau=0,\ \alpha>1{:}\ \
  &\RCBfu 0\alpha 0{u'}{z'}{\sigmav u}{\sigmav z}
   =\begin{cases} \big(\psi_\sfu\big)_\mathrm{ri}(u')+\mfb_{\psi_\sfz} 
            (z',\sigmav z)  &\text{for }\sigmav u=0, \\
      \mfb_{\psi_\sfu}(u',\sigmav u)& \text{for } \sigmav u>0 
         \text{ and }  z'=0,\\
     \infty& \text{otherwise};
 \end{cases}
\\
\tau=0, \ \alpha=1{:}\ \
   &\RCBfu 01 0{u'}{z'}{\sigmav u}{\sigmav z}\ = \ 
  \mfb_{\psi_\sfu \oplus \psi_\sfz}((u',z'),\sigmav u{+}\sigmav z);
\\
 \tau=0, \ \alpha<1{:}\ \
  &\RCBfu 0\alpha 0{u'}{z'}{\sigmav u}{\sigmav z}
   =\begin{cases} \mfb_{\psi_\sfu} (u',\sigmav u) + 
       \big(\psi_\sfz\big)_\mathrm{ri}(z') &\text{for }\sigmav z=0, \\
      \mfb_{\psi_\sfz}(z',\sigmav z)& \text{for } \sigmav z>0 \text{ and } u'=0,\\
     \infty& \text{otherwise},
 \end{cases}
\end{align*}}%
where $\psi: = \psi_\sfu \oplus  \psi_\sfz:(u',z')\mapsto \psi_\sfu(u') {+}
\psi_\sfz(z')$. Thus, the functional $\RCBfu 0\alpha
{\cdot}{\cdot}{\cdot}{\sigmav u}{\sigmav z}$ is convex and $1$-homogeneous 
for all $(\sigmav u, \sigmav z)  \in [0,\infty)^2$.  
\end{proposition}
\begin{proof} \underline{Case $\tau>0$.} Using $\psi_\sfx(v)\geq
  \big(\psi_\sfx\big)_\mathrm{ri} (v)$ we have
\[
\RCBfu \eps\alpha{\tau}{u'}{z'}{\sigmav u}{\sigmav z} \geq  \big(\psi_\sfu 
 \big)_\mathrm{ri}(u') + \frac\tau{\eps^\alpha}\,\sigmav u+  \big(\psi_\sfz 
 \big)_\mathrm{ri}(z') + \frac\tau\eps\, \sigmav z,
\]
which easily provides the desired liminf estimate. The limsup
estimate follows with  the constant recovery sequence
$(u'_\eps,z'_\eps,\sigmave {u}{\eps},\sigmave z\eps)=(u',z',\sigmav u,\sigmav z)$. 

\underline{Case $\tau=0$ and $\alpha=1$.} 
By definition of $\mfb_\psi = \mfb_{\psi_\sfu {\oplus}  \psi_\sfz}$ we have
\[
  \RCBfu \eps1{\tau}{u'}{z'}{\sigmav u}{\sigmav z} = \frac\tau\eps
  \psi\big(\frac\eps{\tau} (u',z')\big) + \frac\tau\eps\,(\sigmav u{+}\sigmav
  z) \geq \mfb_\psi((u',z'),\sigmav u{+}\sigmav z) \qquad \text{for all }
  \tau>0.
\]
 Hence, the liminf estimate follows
from Proposition \ref{pr:VVCP}. 
 
For the limsup estimate for $\RCBfu 01{0}{u'}{z'}{\sigmav u}{\sigmav z}$ we
choose $\lambda_\eps$ such that
$\lambda_\eps\psi(\frac1{\lambda_\eps} (u',z')) + \lambda_\eps(\sigmav
u{+}\sigmav z) \to \mfb_\psi((u',z'),\sigmav u{+}\sigmav z)$, where we may
assume $\lambda_\eps \leq 1/\sqrt\eps$.  Then, it suffices to set
$\tau_\eps= \lambda_\eps \eps\to 0$ to conclude
$\RCBfu \eps1{\tau}{u'}{z'}{\sigmav u}{\sigmav z} \to \mfb_\psi((u',z'),\sigmav
u{+}\sigmav z)= \RCBfu 01{0}{u'}{z'}{\sigmav u}{\sigmav z}$.

\underline{Case $\tau=0$ and $\alpha>1$.} For the lower bound in the
liminf estimate we only need to consider the case $\sigmav u=0$ and  the 
case $\sigmav u>0$ and $z'=0$. In the latter situation we may drop the two
last terms in the definition of $\mfB^\alpha_\eps$ and the lower bound is
established by the lower semicontinuity of $\mfb_{\psi_\sfu}$. In the
case $\sigmav u=0$, we have the lower bound
\[
\RCBfu \eps\alpha{\tau}{u'}{z'}{\sigmav u}{\sigmav z}\geq
\big(\psi_\sfu\big)_\mathrm{ri}(u') + \mfb_{\psi_\sfz}(z',\sigmav z)
\]
and the liminf again follows by the lsc. 

For the limsup estimates we use the recovery sequence
$(\tau_\eps,u',z',\sigmav u,\sigmav z)$ converging strongly with
$\tau_\eps\to 0$, as in the previous case. For $\sigma=0$ we choose
$\tau_\eps=\lambda_\eps \eps$ where $\lambda_\eps$ realizes the infimum in
$\mfb_{\psi_\sfz}(z',\sigmav z)$. In the case $\sigmav u>0$ and $z'=0$ we
choose $\tau_\eps = \hat\lambda_\eps \eps^\alpha$, where $\hat\lambda_\eps$
realizes the infimum in $ \mfb_{\psi_\sfu}(u',\sigma)$. In the remaining case,
which has $\sigma>0$, we may choose $\tau_\eps=\eps$.

\underline{Case $\tau=0$ and $\alpha<1$.} This case is similar to the
case $\alpha>1$ if we interchange the role of $u'$ and $z'$. Thus,
Proposition \ref{pr:Mosco.Beps} is proved.  
\end{proof}

\Subsection{Lower bounds for the B-function $\mfB_\eps^\alpha$}
\label{su:LowerBounds}

In the subsequent convergence analysis for the vanishing-viscosity limit we
will need $\eps$-uniform a priori bounds for the time derivatives of the solutions
$q_\eps=(u_\eps,z_\eps)$.  They are derived by lower bounds for the
B-functions, however, we have already observed in Example \ref{ex:VVCP} that
the simple lower bound $\mfb_\psi(v,\psi^*(\xi))\geq \|\xi\|_*\|v\|$ in
\eqref{used-later-HS} cannot be expected. The following result provides
suitable surrogates of such estimate. They will play a crucial role in the
vanishing-viscosity analysis, specifically in controlling
  $\|z'\|$ along jump paths, see
Lemma \ref{new-lemma-Alex}.
For this it
will be important that the function $\varkappa$ occurring in
\eqref{eq:1LowBounds} is strictly increasing, which implies
$\varkappa(\sigma)>0$ for $\sigma>0$. 

\begin{lemma}
[Lower bound on $\mfB^\alpha_\eps$]
\label{le:LoBo.Bae} 
 Let $\psi_\sfu$ and $\psi_\sfz$ satisfy \eqref{ass-diss-pot-superl} and
 let $\mfB^\alpha_\eps$ be given as in \eqref{eq:def.B.al.eps}. Then, 
there exists a continuous, convex, nondecreasing, and superlinear function
$\varphi: [0,\infty)\to [0,\infty)$ such that
\begin{subequations}
 \label{eq:1LowBounds}
  \begin{align}
 \nonumber   &\hspace*{-1cm} 
  \forall\, \alpha>0\ \forall\, \eps\in [0,1] \ 
  \forall\, (\tau ,  u',z',\sigmav u,\sigmav z)\in 
     [0,\infty)\ti  \Spu\ti \Spz \ti[0,\infty)^2:
\\
    \label{eq:1LowBo.psi}
   &    \psi_\sfu(u')\geq \varphi(\|u'\|_\Spu) \ \text{ and } \  
    \psi_\sfz(z')\geq \varphi(\|z'\|_\Spz),
\\
    \label{eq:1LowBo.Bae}
 &    \RCBfu \eps\alpha{\tau}{u'}{z'}{\sigmav u}{\sigmav z}
    \geq \|u'\|_\Spu\,\varkappa(\sigmav u) + \|z'\|_\Spz \,\varkappa(\sigmav z), 
\end{align}
where $\varkappa \in \rmC([0,\infty);[0,\infty))$ is given by
$\varkappa(\sigma)= (\varphi^*)^{-1}(\sigma)$, is concave, and
strictly increasing with $\varkappa(0)=0$ and $\varkappa(\sigma)\to \infty$ for
$\sigma\to \infty$.  We  additionally have
\begin{align}
\alpha< 1:\quad    &   \label{eq:1LowBo.Bal1}
    \RCBfu \eps\alpha {\tau}{u'}{z'}{\sigmav u}{\sigmav z}
    \geq \|u'\|_\Spu \, \varkappa\big(\sigmav u{+}\sigmav z\big)  ,
\\
\alpha=1:\quad    &   \label{eq:1LowBo.B1eD}
    \RCBfu \eps1{\tau}{u'}{z'}{\sigmav u}{\sigmav z}
    \geq    \big(  \|u'\|_\Spu {+} \|z'\|_\Spz \big) \, 
     \varkappa\big( \frac12(\sigmav u{+}\sigmav z) \big)   ,
\\
\alpha\geq 1:\quad    &   \label{eq:1LowBo.Bag1}
    \RCBfu \eps\alpha {\tau}{u'}{z'}{\sigmav u}{\sigmav z}
    \geq \|z'\|_\Spz \, \varkappa\big(\sigmav u{+}\sigmav z\big)  .
\end{align}
\end{subequations}
\end{lemma} 
\noindent
\begin{proof} \STEP{1: Construction of $\varphi$.} Since $\psi_\sfu$ and
  $\psi_\sfz$ are superlinear, for each $K\geq 0$ there exists $S_K\geq 0$ such that 
\begin{equation*}
  \forall (u',z') \in \Spu\ti \Spz: \quad \psi_\sfu(u') \geq  K\|u'\|_\Spu - S_K
  \text{ and } \psi_\sfz(z') \geq  K\|z'\|_\Spz - S_K.
\end{equation*}
Hence, the estimates in \eqref{eq:1LowBo.psi} hold for the nonnegative, convex
function $\varphi$ given by 
\[
\varphi(r) := \sup\bigset{ Kr -S_K }{K\geq 0}.
\]
From $\varphi(0)=0$ and non-negativity we conclude that $\varphi$ is
nondecreasing. Moreover, it is superlinear by construction. 

\STEP{2: Lower bound on $\mfb_{\psi_\sfx}$.} In the definition of
$\mfb_\psi$ the dependence on $\psi$ is monotone (because of $\tau>0$)  so 
that  $\psi_1\leq \psi_2$ implies $\mfb_{\psi_1} \leq \mfb_{\psi_2}$. Setting
$\widetilde \varphi(v)=\varphi(\|v\|)$ we obtain $\mfb_{\psi_\sfx}\geq
\mfb_{\widetilde\varphi}$, and using Lemma \ref{le:psi.norm} and the definition
of $\varkappa$ yields 
\[
\mfb_{\psi_\sfx}(v,\sigma) \geq \|v\|\, \varkappa(\sigma) \qquad \text{for } \mathsf{x} \in \{\mathsf{u}, \mathsf{z}\}. 
\]

\STEP{3: Lower bound on $\mfB^\alpha_\eps$.} The definitions of
$\mfB^\alpha_\eps$ in \eqref{eq:Resc.Bae.b} and of $\mfb_\psi$ give, for
$\eps>0$,
\begin{align*}
\RCBfu \eps\alpha{\tau}{u'}{z'}{\sigmav u}{\sigmav z} &=
  \calB_{\psi_\sfu}(\frac\tau{\eps^\alpha} , u',\sigmav u) 
  +  \calB_{\psi_\sfz}(\frac\tau{\eps} , z',\sigmav z)  \\
 & \geq  \mfb_{\psi_\sfu}(u',\sigmav u) +  \mfb_{\psi_\sfz}(z',\sigmav z) 
 \geq \|u'\|_\Spu\,\varkappa(\sigmav u) + \|z'\|_\Spz \,\varkappa(\sigmav z),
\end{align*}
where Step 2 was  invoked for the last estimate. This proves
\eqref{eq:1LowBo.Bae}.

Estimate \eqref{eq:1LowBo.B1eD} follows from the simple observation that,
because of $\alpha=1$, the rescaled B-function $\mfB^1_\eps$ only depends
$\sigmav u {+}\sigmav z$, such that each of $\sigmav u$ and $\sigmav z$ can be
replaced by their arithmetic mean.

For $\alpha\geq 1$ and $\eps\in (0,1]$, we have $\tau/{\eps^\alpha} \geq
\tau/\eps$  so  that
\[
\RCBfu \eps\alpha{\tau}{u'}{z'}{\sigmav u}{\sigmav z}
  \geq \frac\tau{\eps^\alpha}\sigmav u 
        +\calB_{\psi_\sfz}(\frac\tau{\eps} , z', \sigmav z) 
  \geq \calB_{\psi_\sfz}(\frac\tau{\eps} , z', \sigmav u{+}\sigmav z)  
  \geq \|z'\|_\Spz\, \varkappa( \sigmav u{+}\sigmav z). 
\]
This shows estimate \eqref{eq:1LowBo.Bag1}, and
\eqref{eq:1LowBo.Bal1} follows similarly. 

All estimates remain true for $\eps=0$ because $\mfB^\alpha_0$ is the
Mosco limit of $\mfB^\alpha_\eps$.
\end{proof}

\Section{Setup and existence for the viscous system}
\label{s:setup}

In Section \ref{ss:2.1} we will introduce our basic conditions on the ambient
spaces, the energy, and the dissipation potentials, collected in Hypotheses
\ref{hyp:setup}, \ref{hyp:diss-basic}, \ref{hyp:1}, and \ref{h:closedness},
 which will be assumed throughout the paper.  Let us mention in
advance that we will often omit to explicitly recall these assumptions in the
various intermediate statements, with the exception of our main results in
Theorems \ref{thm:existBV}, \ref{thm:exist-enh-pBV}, \ref{thm:exist-trueBV},
and \ref{thm:exist-nonpar-enh}.

Then, in Section \ref{ss:ExistVisc} we will address the existence of solutions
to the viscous system \eqref{van-visc-intro}.  Its main result, Theorem
\ref{th:exist} shows that, under two additional conditions on the driving
energy functional, the existence result from \cite[Thm.\,2.2]{MRS2013} can be
applied to deduce the existence of solutions for the doubly nonlinear system
\eqref{van-visc-intro}. It will be crucial to our analysis that we are able to
show that these solutions satisfy the $(\Psi,\Psi^*)$ energy-dissipation
balance \eqref{EnDissBal-intro}.

\Subsection{Function spaces}
\label{ss:2.1}

Here we state our standing assumptions on the function spaces for the  energy
functionals and for the dissipation potentials.

\begin{hypothesis}[Function spaces] 
\label{hyp:setup}
\slshape In addition to conditions \eqref{intro-state_spaces} on the ambient spaces
$\Spu$ and $\Spz$, our (coercivity) conditions on the energy $\calE$
will involve two other \emph{reflexive} spaces   $\Spw$ and $\Spx$,  \ such
that
\[
  \Spw \subset \Spu \text{ continuously and densely, and } \Spx \Subset \Spz
  \text{ compactly and densely}.
\]
The subscript $\mathrm{e}$ refers to the fact that the latter are `energy
spaces'  relating to $\calE$. 
Furthermore, the $1$-homogeneous dissipation potential $\calR$ will be
in fact defined on a (separable) space   $\Spy$  (where the subscript  $\mathrm{ri}$ accordingly refers to rate-independence),   such that
\[
\Spz \subset \Spy  \text{ continuously and densely}.
\]
\end{hypothesis}

We refer to \eqref{spacesU} for some examples of relevant ambient
spaces. In what follows, we will often use the notation
\begin{equation}
\label{notation}
q:=(u,z) \in  \mathbf{Q} :  = \Spu \ti \Spz.
\end{equation}

\Subsection{Assumption on the dissipation potentials}
\label{su:Dissipation}

We will develop the general theory under the condition that the
viscous dissipation potentials $\disv u$ and $\disv z$ as well as the
$1$-homogeneous potential $\calR$ take only finite values in $[0,\infty)$ and are
thus continuous.  Recall that $\disv x^*$ is the Legendre-Fenchel conjugate
of $\disv x$, see Definition \ref{def:DissPotential}. 

\begin{hypothesis}[Conditions on $\disv u$, $\disv z$, $\calR$]
\label{hyp:diss-basic}
\slshape  Let $\disv u : \Spu \to [0,\infty)$ and
$\disv z : \Spz \to [0,\infty)$ be dissipation potentials with the following
additional conditions: 
\begin{subequations}
\label{hyp:visc-diss}
\begin{align}
\label{h:v-diss-1}
  & \lim_{\| v\|_\Spu \to \infty} \frac{\disv
    u(v)}{\| v\|_\Spu } = \lim_{\|  \mu  \|_{\Spu^*} \to \infty} \frac{\disv
    u^*( \mu) }{\|  \mu \|_{\Spu^*}} =\infty = 
 \lim_{\| \eta\|_\Spz \to \infty} \frac{\disv z(\eta)}{\| \eta\|_\Spz}
  = \lim_{\| \zeta\|_{\Spz^*} \to \infty}
  \frac{\disv z^*(\zeta)}{\| \zeta\|_{\Spz^*}},
\\
  \label{later-added}
  &\lim_{\lambda \to 0^+} \frac1{\lambda} \disv u (\lambda v) =0 \ \text{ for
    all } v \in \Spu,    \quad \text{and} \quad \lim_{\lambda \to 0^+}
  \frac1{\lambda} \disv z  (\lambda \eta) =0 \ \text{ for all } \eta \in \Spz\,. 
\end{align}
\end{subequations}
Let $\calR: \Spy \to  [0,\infty]$ be a  1-homogeneous dissipation potential,  i.e.
\begin{subequations}
\label{hyp:ri-diss}
\begin{equation}
\label{Rzero}
\calR(\lambda \eta) = \lambda \calR (\eta) \quad \text{for all } \eta \in \Spy
\text{ and }  \lambda > 0, 
\end{equation}
 that is additionally $\Spz$-bounded and  $\Spy$-coercive for $\Spz\subset \Spy$, i.e. 
\begin{equation}
\label{R-coerc}
\exists\, C_\calR,\, c_\calR>0:  \quad 
\left\{ 
\begin{array}{ll}
\forall\, \eta \in \Spz:   &  \calR(\eta) \leq  C_\calR \norm{\eta}{\Spz},
\\
 \forall\, \eta \in \Spy:  & \calR(\eta) \geq c_\calR \norm{\eta}{\Spy}\,.
\end{array}
\right.
\end{equation}
\end{subequations}
\end{hypothesis} 

Due to the superlinear growth of $\disv x$ and $\disv x^*$,
$\sfx \in \{ \sfu, \sfz\}, $ both
$\pl  \disv x : \mathbf{X} \rightrightarrows \mathbf{X}^*$ and
$\pl  \disv x^* : \mathbf{X}^* \rightrightarrows \mathbf{X}$,
$ \mathbf{X} \in \{ \Spu, \Spz\}$, are bounded operators, so that, ultimately,
both $\disv u$ and $\disv u^*$ are continuous. Likewise, $\calR$ is
continuous. Indeed, restricting our analysis to the case in which $\calR$
takes only finite values in $[0,\infty)$ excludes the direct application of our
results to systems modeling unidirectional processes in solids such as damage
or delamination. In those cases the existence theory (both for the
rate-dependent, `viscous' system and for $\BV$ solutions of the
rate-independent process) relies on additional estimates not considered here,
see e.g.\ \cite{KnRoZa17}.  Nevertheless, a broad class of models is still
described by \emph{continuous} dissipation functionals.  For instance, the
coercivity and growth conditions \eqref{R-coerc} are compatible with the
following example of dissipation potential, in the ambient spaces
$\Spy =\rmL^1(\Omega)$ and $\Spz=\rmL^2(\Omega)$ (with $\Omega\subset\R^d$ a bounded
domain):
\begin{equation}
 \label{ex:healing}
 \calR: \rmL^1(\Omega) \to [0,\infty]; \quad \calR(\eta): = 
  \begin{cases}
    \|\eta^+\|_{\rmL^2(\Omega)}+ \|\eta^-\|_{\rmL^1(\Omega)}& \text{ if } \eta^+ \in
    \rmL^2(\Omega), \\ \infty & \text{ otherwise.}
 \end{cases}
\end{equation}
Dissipation potentials with this structure occur, for instance, in models for
damage or delamination allowing for possible healing, cf.\ e.g.\
\cite[Sec.\,5.2.7]{MieRouBOOK} and Section \ref{s:appl-dam}. 

 Subsequently, $\pl  \disv u: \Spu\rightrightarrows \Spu^*$, $\pl  \disv
u: \Spz\rightrightarrows \Spz^*$, and $\pl  \calR: \Spz\rightrightarrows
\Spz^*$ will denote the convex subdifferentials of $\disv u$, $\disv z$, and
$\calR$,  respectively.  By the 1-homogeneity \eqref{Rzero} we have 
\begin{equation}
    \label{eq:subdiff.calR}
\pl \calR(\eta) = \bigset{ \omega \in \pl \Spz^*}{  
       \forall\,v{\in} \Spz{:}\; \calR(v) \geq 
     \calR(\eta){+}\!\pairing{}{\Spz\!}{\omega}{v{-}\eta}} 
 = \bigset{\omega \in \pl   \calR(0)}{
            \calR(\eta)=\langle \omega,\eta\rangle } . 
\end{equation}
Thanks to Hypothesis \ref{hyp:setup}, we have $\Spy^* \subset \Spz^*$ densely
and continuously.  As a consequence of \eqref{R-coerc} $\pl  \calR (0)$
turns out to be a bounded subset in $\Spz^*$,  viz.
\begin{equation}
  \label{eq:l:classic}
  \pl  \calR(\eta) \subset \pl  \calR(0) \quad \text{and} \quad 
  \overline{B}_{c_\calR}^{\Spy^*}(0) \subset \pl  \calR(0) \subset
\overline{B}_{C_\calR}^{\Spz^*}(0). 
\end{equation}

\Subsection{Assumptions on the energy $\calE$}
\label{su:Energy}

We now collect our basic requirements on the energy functional $\calE:
[0,T]\ti \Spu \ti \Spz \to (-\infty,\infty]$.  With slight abuse of
notation, we will often write $\calE(t,q) $ in place of $ \ene tuz$, in
accordance with \eqref{notation}.  Recall the embeddings
$\Spw\subset \Spu$ and $\Spx \Subset \Spz\subset \Spy$  and the choice
$\Spq= \Spu\ti \Spz$. 

\begin{hypothesis}[Lower semicontinuity, coercivity, time differentiability of $\calE$]
\label{hyp:1} 
The energy functional \linebreak[4]
$\calE: [0,T]\ti \Spu \ti \Spz \to (-\infty,\infty]$ has the proper domain
$\mathrm{dom}(\calE) = [0,T]\ti \domq $ with $\domq \subset \Spw\ti \Spx$.
Moreover, we require that
\begin{subequations}
\label{h:1}
\begin{equation}
\label{h:1.1}
\forall\, t \in [0,T]: \quad \text{the map }  
q \mapsto \en tq \text{ is weakly  lower semicontinuous on } \Spq,
\end{equation}
and $\calE$ is bounded from below:
\begin{equation}
\label{h:1.2}
\exists\, C_0>0 \ \ \forall\, (t,q) \in 
[0,T]\ti \domq\, : \qquad \en tq \geq C_0\,.
\end{equation}
We set $\mfE( q):= \sup_{t\in [0,T]} \en tq $ and require
that
\begin{equation}
\label{h:2}
\text{the map } q \mapsto \mfE(q)  + \| q \|_{\Spu\ti \Spy} 
 \text{ has sublevels bounded in } \Spw \ti \Spx.
\end{equation}
Finally, we require that $t \mapsto \en tq$
is differentiable for all $q\in \domq $ satisfying the power-control estimate 
\begin{align}
\label{h:1.3d}
&\exists\, C_\#>0 \ \forall\, (t,q)\in [0,T]\ti \domq:\quad 
 \left| \pet tq \right| \leq C_\# \en tq  . 
\end{align}
\end{subequations}
\end{hypothesis}

\noindent
Concerning our conditions on $\mathrm{dom}(\calE)$, the crucial requirement is
that $\mathrm{dom}(\calE(t,\cdot)) \equiv \domq$ is independent of time.  Let
us introduce the energy sublevels
\begin{equation}
\label{Esublevels}
\subl E: = \{ q\in \domq\, : \ \mfE(q) \leq E \} \qquad \text{for }E>0.
\end{equation}
Applying Gr\"onwall's lemma we deduce from \eqref{h:1.3d} that 
\[
\forall\,  (t,q) \in [0,T]\ti \domq \, : \qquad 
 \mfE(q) \leq  \rme^{C_\# T} \,  \calE(t,q)\,. 
\]
Hence, $\calE(t,q) \leq E$ for some $t\in [0,T]$ and $E>0$ guarantees
$q\in \subl {E'}$ with $E' =  \rme^{C_\# T} \,  E $.  Finally, observe
that \eqref{h:2} implies the separate coercivity properties of the functionals
$\mfE(\cdot,z)$ and $\mfE( u,\cdot)$, perturbed by the norm $\| \cdot\|_\Spu$ and 
 $\| \cdot\|_{\Spy}$,  respectively.

Since we are only requiring that $\Spw \subset \Spu$ continuously, our
analysis allows for the following two cases: (i) the energy
$\ene t{\cdot}z$ and the dissipation potential $\disv u$ have sublevels bounded
in the same space and (ii) the energy $\ene t{\cdot}z$ has sublevels compact in
the space $\Spu$ of the dissipation $\disv u$. 
To fix ideas, typical examples for the pairs $(\Spu,\Spw)$  and the triples
$(\Spx, \Spz,\Spy)$  are
\begin{align}
&
\label{spacesU}
\begin{aligned} 
&\text{(i) }\ \Spu =\Spw = \rmH^1(\Omega;\R^d) \quad  \text{ or \quad (ii) } \ 
\Spw = \rmH^1(\Omega;\R^d) \  \Subset  \ \Spu = \rmL^2(\Omega;\R^d), 
\\
&   \text{and} \quad \Spx= \rmH^1(\Omega;\R^m) \ \Subset \   \Spz=
\rmL^2(\Omega;\R^m) \ \subset \  \Spy = \rmL^1(\Omega;\R^d). 
\end{aligned}
\end{align}

As mentioned in the introduction, in our analysis we aim to allow for
nonsmoothness of the energy functional $q=(u,z)\mapsto \en tq$. Accordingly, we
will use the Fr\'echet subdifferential of $\calE$ with respect to\ the variable
$q$, i.e.\ the multivalued operator
$\frname q\calE : [0,T] \ti \Spq \rightrightarrows \Spq^*$ defined for
$(t,q)\in [0,T]\ti \domq$ via
\begin{equation}
\label{Frsubq}
\frsubq qtq :=\bigset{ \xi \in \Spq^*}{\en t{\hat{q}} \geq  \en tq  {+}
 \pairing{}{\Spq}{\eta}{\hat{q}{-}q}  + o(\norm{\hat{q}{-}q}{\Spq}) \text{ as }
\hat{q} \to q \text{ in } \Spq}  
\end{equation}
with domain $\mathrm{dom}(\frname q\calE) : = \bigset{ (t,q) \in [0,T] 
\ti \domq }{ \frsubq qtq \neq \emptyset } $.

Thus, our aim is to solve the subdifferential inclusion
\begin{equation}
\label{dne-q}
\pl \Psi_{\eps,\alpha}(q'(t)) + \pl_q \calE(t,q(t)) \ni 0 \qquad
\text{ in } \mathbf{Q}^* \ \foraa\, t \in (0,T) 
\end{equation}
where the scaled dissipation potential $\Psi_{\eps,\alpha} $ is defined in
\eqref{eq:def.Psi.e.a}.  

\begin{remark}[Partial Fr\'echet subdifferentials]
\label{rmk:Alex}
\slshape
Observe that 
\begin{equation}
\label{just-inclusion}
\frsubq qt{u,z} \subset \frsubq ut{u,z} \ti \frsubq zt{u,z}  \quad \text{for all } (t,q) =
(t,u,z)\in [0,T]\ti\domq, 
\end{equation}
where $\frsubq utq\subset \Spu^*$ and $\frsubq ztq\subset \Spz^*$ are
the `partial' Fr\'echet subdifferentials of $\calE$  with respect to
the variables $u $ and $z$, which are defined as Fr\'echet subdifferentials of  
$\ene t\cdot z:\Spu\to \R$ and $\ene tu\cdot:\Spz\to \R$, respectively. 
However, equality in \eqref{just-inclusion} is false, in general, e.g.\ for
$\Spu=\Spz=\R$ and $\calE(t,u,z)=|u{-}z|$.

In view of the inclusion \eqref{just-inclusion}, any curve
$t \mapsto q(t)=(u(t),z(t))$ solving \eqref{dne-q} also solves the system
\begin{subequations}
\label{DNE-system}
\begin{align}
\label{DNEu}
&
\pl  \disv {u}^{\epsalpha}(u'(t)) +\frsub u t{u(t)}{z(t)}   \ni 0 && \text{
  in } \Spu^*  && \foraa t \in (0,T),  
\\
\label{DNEz}
\pl  \calR (z'(t)) +{} &\pl  \disv {z}^\eps\,(z'(t)) \ + \, 
\frsub z t{u(t)}{z(t)}   \ni 0  &&  \text{ in } \Spz^*  && \foraa t \in (0,T).
\end{align}
\end{subequations}
Nonetheless, let us stress that the `reference viscous system' for the
subsequent discussion will be  the one with the smaller  solution set,
namely \eqref{dne-q}  or \eqref{enid.b} below. 
\end{remark}

The  existence  result from \cite{MRS2013} can be applied provided that
$\calE$ fulfills two further conditions, stated in the following Hypotheses
\ref{h:closedness} and \ref{h:ch-rule}.

\begin{hypothesis}[Closedness  of $(\pl_q\calE,\pl_t\calE)$ on sublevels] 
\label{h:closedness} 
For all sequences  $\big((t_n,q_n, \xi_n) \big)_{n\in \N} $ in the space
$[0,T] \ti \Spq \ti \Spq^*$ with  $t_n\to t$, $q_n \weakto q $ in $\Spq$,
$\xi_n \weakto \xi $ in $\Spq^*$, $\sup_n \mfE(q_n) < \infty $, and
$\xi_n\in \pl_q \en t{q_n}$, we have
\begin{equation}
 \label{h:1.3e}
  \xi \in \pl_q \en tq \quad \text{and} \quad 
 \pet {t_n}{q_n} \to \pet tq.   
\end{equation}
\end{hypothesis} 
\begin{remark}
\label{rmk:closedness-diminished}
\sl
For cases in which the energy space $\Spw$ is compactly embedded into $\Spu$,
the sequences $(q_n)_n$ fulfilling the conditions of Hypothesis
\ref{h:closedness} converge strongly in $\mathbf{Q}$ in view of the
coercivity \eqref{h:2}.  Therefore, in such cases Hypothesis
\ref{h:closedness} turns out to be a closedness condition on the graph of
$\pl_q \calE$ with respect to\ the \emph{strong-weak} topology of
$\mathbf{Q}\ti \mathbf{Q}^*$.

We also mention that, in contrast to what we did in \cite{MRS2013} (cf.\
(2.E$_5$) therein), here in Hypothesis \ref{h:closedness} we omit  the
requirement of energy convergence  $ \en {t_n}{q_n} \to \en tq $  along
the sequence $(t_n,q_n,\xi_n)_n$.  In fact, that additional property was
not strictly needed in the proof of the existence result
\cite[Thm.\,2.2]{MRS2013}, to which we will resort later on to conclude the
existence of solutions for our viscous system \eqref{dne-q}. Rather, in
\cite{MRS2013} the energy-convergence requirement was encompassed in the
closedness assumption in order to pave the way for a weakening of the
chain-rule condition, cf.\ the discussion in \cite[Rmk.\,4.6]{MRS2013}. Such a
weakening is outside the scope of this paper.
\end{remark}

Our final condition on $\calE$ is an abstract \emph{chain rule} that has a
twofold role: First, it is a crucial ingredient in the proof of Theorem
\ref{th:exist}, and secondly, it ensures the validity of the energy-dissipation
balance \eqref{enid}. The latter will be the starting point in the derivation
of our a priori estimates \emph{uniformly} with respect to\ the viscosity
parameter $\eps$.  We refer to Proposition \ref{prop:ch-ruleApp} in Appendix
\ref{s:app-CR} for a discussion of conditions on $\calE$ yielding the validity
of Hypothesis~\ref{h:ch-rule}.

\begin{hypothesis}[Chain rule]
\label{h:ch-rule}
For every absolutely continuous curve $q\in \AC ([0,T]; \mathbf{Q})$ and 
all    measurable selections $\xi: (0,T) \to \Spq^*$   with 
$\xi(t)\in \pl_q \calE(t, {q(t)})$ for a.a.\ $t\in (0,T)$,
 \begin{equation}
\label{conditions-1} \sup_{t \in (0,T)} |\calE(t,q(t))|<\infty,  \quad \text{and} \quad
\int_0^T \| \xi(t)\|_{\mathbf{Q}^*}  \| q'(t)\|_{\mathbf{Q}} 
    \dd t <\infty, 
\end{equation}
we have the following two properties:
\begin{equation}
    \label{eq:48strong}
    \begin{gathered}
     \text{the map  $t\mapsto \calE(t,q(t))$ is absolutely continuous
       on $[0,T]$ and}
     \\
    \frac \dd{\dd t}  \calE(t,q(t)) - \pl_t     \calE(t,q(t)) 
    =   \pairing{}{\bfQ}{\xi(t)}{q'(t)}  \quad 
    \text{for a.a.\ }t\in (0,T).
    \end{gathered}
  \end{equation}
\end{hypothesis}

\Subsection{An existence result  for the viscous problem}
\label{ss:ExistVisc}

We are now in the position to state our existence result for the viscous system
\eqref{dne-q}.  It is based on the 
$(\Psi,\Psi^*)$-formulation of the energy-dissipation balance (cf.\
\eqref{EnDissBal-intro} for the case 
$q\mapsto \calE(t,q)$ is smooth), which we now apply to \eqref{dne-q} using the
Fr\'echet subdifferential $\frsubq q tq$ and the scaled dissipation
potential $\Psi_{\eps,\alpha}$ defined in \eqref{eq:def.Psi.e.a}. The
Legendre-Fenchel conjugate is given by 
\begin{equation}
\label{def:conj}
\Psi_{\eps,\alpha}^*(\mu,\zeta) = \frac1{\eps^\alpha}\disv u^*(\mu) + \frac1\eps\conj z(\zeta) \ \text{ with
}\ \conj z(\zeta): = \min_{\sigma \in \pl \calR(0)} \disv z^*(\zeta{-}\sigma)
\qquad \text{for } \zeta \in  \Spz^*. 
\end{equation}
It  can be straightforwardly checked that the infimum in the definition of
$\calW^*_\sfz$ is attained. 
  
\begin{theorem}[Existence of viscous solutions]
\label{th:exist}
Assume Hypotheses \ref{hyp:diss-basic}, \ref{hyp:1}, \ref{h:closedness}, and 
\ref{h:ch-rule}. Then,  for every
$\eps \in (0,1]$ and  $q_0 = (u_0,z_0)\in \domq$ there exists a curve
$q=(u,z) \in \AC ([0,T];\Spq)$ and  a function $\xi=(\mu,\zeta) 
\in \rmL^1(0,T;\Spu^* \ti \Spz^*)$  fulfilling the initial condition $q(0) = q_0$,
solving the generalized gradient system \eqref{dne-q}  in the sense that
 for a.a.\ $r\in (0,T)$
\begin{subequations}
\label{enid}
\begin{equation}
\label{enid.b}
  (\mu(r),  \zeta(r)) \in \frsubq q r{q(r)} \  \text{ and }  \ 
   \left\{ \begin{array}{l@{\,}l}
              -\mu(r)  &\in  \pl  \disv {u}^{\epsalpha}(u'(r)),
              \\
            -\zeta (r) &\in \pl  \calR (z'(r)) {+} \pl  \disv {z}^\eps(z'(r)),
\end{array}
\right.
\end{equation}
Moreover, for $0\leq s < t \leq T$, these functions satisfy
the energy-dissipation balance
\begin{align}
\label{enid.a}
\en t{q(t)} & + \int_s^t \Big( \disve u{\eps^\alpha}
 (\eps^\alpha u'(r)) + \calR(z'(r)) +  \disve z\eps (\eps\,z'(r)) 
\Big) \dd r 
\\ \nonumber
 & { +  \int_s^t \Big( \frac1{\eps^\alpha}  \disv
u^*  ({-}\mu(r))  +  \frac1\eps  \conj z({-}\zeta(r))\Big)  \dd r 
 = \en s{q(s)}+ \int_s^t\pl_t \en r{q(r)} \dd r.}
\end{align}
\end{subequations}
where $\disve x\lambda$ is defined in 
\eqref{eq:Def.Vx.lambda}. 
\end{theorem}
\begin{proof} Since we are in the simple setting of \cite[Sec.\,2]{MRS2013},
where the dissipation potential $\Psi_{\eps,\alpha}$ does not depend on the
state $q$, we can appeal to \cite[Thm.\,2.2]{MRS2013}. Thus, it suffices to
check the assumptions (2.$\Psi_1$)--(2.$\Psi_3$), (E$_0$), and
(2.E$_1$)--(2.E$_4$) therein.  Our Hypothesis \ref{hyp:diss-basic} clearly
implies (2.$\Psi_1$) and (2.$\Psi_2$). Hypothesis \ref{hyp:1} implies the
assumptions (E$_0$) via \eqref{h:1.1} and \eqref{h:1.2}, and assumption
(2.E$_1$) follows via \eqref{h:2} and Hypotheses \ref{hyp:diss-basic}.
Assumption (2.E$_2$) follows from Hypothesis \ref{h:closedness} via
\cite[Prop.\,4.2]{MRS2013}.  Assumption (2.E$_3$) equals \eqref{h:1.3d} in
Hypothesis \ref{hyp:1}. Finally, leaving out the energy-convergence requirement
assumption (2.E$_5$) follows from Hypothesis \ref{h:closedness}.

Thus, all assumptions are satisfied except for (2.$\Psi_3$) and (2.E$_4$).
Concerning (2.$\Psi_3$), we observe that this technical condition was used for
the proof of \cite[Thm.\,2.2]{MRS2013} only in one place, namely in the proof
of Lemma 6.1 there. In  \cite[Thm.\,3.2.3]{Bach21NADN}  or in
\cite{MieRos20?DL} it is shown that Lemma 6.1, which is also called ``\emph{De
  Giorgi's lemma}'', is also valid if the condition
\cite[Eqn.\,(2.$\Psi_3$)]{MRS2013} is replaced by the condition that the
underlying space Banach space $\Spq$ is reflexive, but this is true by our
Hypothesis \ref{hyp:setup}.  As for the chain rule \cite[(2.E$_4$)]{MRS2013}, a
close perusal of the proof of \cite{MRS2013} shows that our Hypothesis
\ref{h:ch-rule} can replace it, allowing us to conclude the existence
statement.
\end{proof}

\begin{remark}[Energy-dissipation inequality]\slshape 
\label{rmk:GS-used-later}
The analysis from \cite{MRS2013} in fact reveals that, under the chain
rule in Hypothesis \ref{h:ch-rule}, a curve $q\in \AC([0,T];\Spq)$ fulfills
\eqref{enid.b} \emph{if and only if} the  pair $(q,\xi)$  satisfies the
energy-dissipation balance \eqref{enid.a} which, again by the chain rule, is in
turn equivalent to the upper energy-dissipation estimate $\leq$. This
characterization of the viscous system will prove handy for the analysis of the
delamination system from Section \ref{s:appl-dam}.
\end{remark}

\Subsection{Properties of the generalized slopes}
\label{su:GenSlopes}

For the further analysis it is convenient to introduce the
\emph{generalized slope functionals}   $\slovname{x}: [0,T]\ti \domq \to [0,\infty]$,
$\mathsf{x} \in \{\mathsf{u}, \mathsf{z}\}$ via
\begin{equation}
 \label{def:GeneralSlope}
 \begin{aligned}
 & \slov utq: = \inf\bigset{\,\disv u^* ({-}\mu)\, }{(\mu,\zeta) \in \frsubq qtq}
 \quad \text{and}
 \\
 & \slov ztq: =\inf\bigset{\conj z({-}\zeta)}{(\mu,\zeta) \in \frsubq qtq},
 \end{aligned}
\end{equation}
 where the infimum over the empty set is always $+\infty$. These functionals
play the same key role as (the square of) the metric slope for metric gradient
systems, hence  from now on we shall refer to $\slovname u$ and $\slovname z$ as
\emph{generalized slopes}.  Clearly, energy balance \eqref{enid.a} entails the
validity of the following energy-dissipation estimate featuring the slopes
$\slovname u$ and $\slovname z$:
\begin{equation}
  \label{enid-ineq}
  \begin{aligned}  
    \eneq t{q(t)}+ \int_s^t \!\!\Big(  \disve u{\eps^\alpha}(  u'(r)) 
    {+}  \calR(z'(r)) {+}  \disve z\eps(  z'(r)) +
    \frac{\slov  ur{q(r)}}{\eps^\alpha}   +  \frac{\slov  zr{q(r)}}\eps
      \Big)    \dd r  \\   \leq 
    \eneq s{q(s)}
    + \int_s^t\pl_t \eneq r{q(r)}
    \dd r \qquad \text{for all } 0 \leq s \leq t \leq T\,.
  \end{aligned}
\end{equation}
Note that \eqref{enid-ineq} is weaker than \eqref{enid.a}, but it has the
advantage that the selections  $\xi=(\mu,\zeta)$  in \eqref{enid.b} are
no longer needed. Moreover, \eqref{enid-ineq} will be still strong enough to
handle the limit passage $\eps \to 0^+$.  For this, we will assume that the
infima in \eqref{def:GeneralSlope} are attained. We set
\begin{align*}
  &\mathrm{dom}(\pl_q \calE):= \bigset{(t,q) \in [0,T]\ti \Spq}{
    \pl_q\calE(t,q)\neq \emptyset } \quad \text{and} \\
  &\mathrm{dom}(\pl_q \calE(t,\cdot)):= \bigset{q\in \Spq}{
    \pl_q\calE(t,q)\neq   \emptyset } .
\end{align*}
In fact, it can be checked (e.g.\ by resorting to \cite[Prop.\,4.2]{MRS2013}),
that $\mathrm{dom}(\pl_q \calE(t,\cdot))$ is dense in $\domq$.

\begin{hypothesis}[Attainment and lower semicontinuity] 
\label{hyp:Sept19}
For every  $(t,q) \in \mathrm{dom}(\pl_q \calE)$ the
infima   in \eqref{def:GeneralSlope} are attained, namely  
\begin{equation}
  \label{not-empty-mislo}
    \argminSlo utq: = \mathop{\mathrm{Argmin}}\limits_{(\mu,\zeta) \in \frsubq
      qtq} \disv u^* ({-}\mu) \neq \emptyset
 \quad \text{and} \quad \
    \argminSlo ztq: = \mathop{\mathrm{Argmin}} \limits_{(\mu,\zeta) \in
       \frsubq qtq}  \conj z({-}\zeta) \neq \emptyset,
\end{equation} 
where $\conj z$ is defined in \eqref{def:conj}.  Furthermore, for all sequences
 $(t_n,q_n)_n\subset [0,T] \ti \Spq$  with $t_n\to t$, $q_n \weakto q$
in $\Spq$, and $\sup_{n\in \N} \mfE(q_n) \leq C<\infty$ there holds
\begin{align}
\label{liminf-diss-V-W}
\liminf_{n\to\infty} \slov x {t_n}{q_n}\geq \slov x tq 
\qquad \text{for } \mathsf{x} \in \{\mathsf{u}, \mathsf{z}\}\,. 
\end{align}
\end{hypothesis}

 We are going to show in Lemma \ref{l.4.13} below that 
a sufficient condition for Hypothesis \ref{hyp:Sept19} 
is that \eqref{just-inclusion} improves to an equality, namely
\begin{equation}
\label{it-is-product}
\frsubq qtq = \frsubq utq \ti \frsubq ztq  \quad \text{for all } (t,q) =
(t,u,z)\in [0,T]\ti\domq. 
\end{equation}
Observe that \eqref{it-is-product} does hold if, for instance, $\calE$ is of
the form
\begin{align*}
    & \en tq : = \calU(t,u) + \calZ(t,z) + \calF(t,u,z) \quad \text{for all }
    (t,q) = (t,u,z) \in [0,T]\ti \Spq  
   \\ 
   & \text{with } 
  \calU(t,\cdot): \Spu \to (-\infty,\infty] \text{ and } \calZ(t,\cdot): \Spz
  \to (-\infty,\infty] \text{ proper and lsc},
    \\
    &\text{and } \ \, \calF(t,\cdot): \Spu\ti \Spz \to \R \text{ Fr\'echet differentiable}. 
\end{align*}
\begin{lemma}
\label{l.4.13}
Assume Hypotheses \ref{hyp:diss-basic}, 
\ref{hyp:1}, \ref{h:closedness}, as well as  \eqref{it-is-product}.  Then, 
\begin{equation}
  \label{it-is-equality}
  \slov  utq = \inf_{\mu \in \frsubq utq}  \disv u^*  ({-}\mu) 
  \ \text{ and } \   
  \slov  ztq = \inf_{\zeta \in \frsubq ztq}    \conj z({-}\zeta) 
\end{equation}
for all $(t,q) \in [0,T]\ti  \mathrm{dom}(\pl_q \calE) $, and properties
\eqref{not-empty-mislo} and \eqref{liminf-diss-V-W} hold. 
\end{lemma}
\begin{proof}
 Obviously, for $(t,q)\in  \mathrm{dom}(\pl_q \calE)$ we have
\eqref{it-is-equality} as a consequence of 
\eqref{it-is-product}. 
We will just check the attainment \eqref{not-empty-mislo} and the lower
semicontinuity \eqref{liminf-diss-V-W} for $\slovname z$, as the properties for
$\slovname u$ follow by the same arguments.

Suppose that  $(t_n,q_n) \weakto (t,q)$ and 
$\liminf_{n\to\infty} \slov z{t_n}{q_n} <\infty$. Using
\eqref{it-is-equality}, up to a subsequence, there exist
$(\zeta_n) \subset \Spz^*$ with $\zeta_n \in \frsubq z{t_n}{q_n} $
and $(\sigma_n)_n \subset \pl \calR(0) \subset 
\Spz^*$  for all $n$ with 
\[
\lim_{n\to\infty}  \disv z^*({-}\zeta_n{-}\sigma_n)
= \lim_{n\to\infty}    \slov z{t_n}{q_n} \leq C\,.
\] 
It follows from \eqref{hyp:visc-diss} that the sequence
$(\sigma_n{+}\zeta_n)_n$ is bounded in $\Spz^*$. Since, in view of
\eqref{eq:l:classic}, $(\sigma_n)_n$ is bounded in $\Spz^*$, $(\zeta_n)_n$
turns out to be bounded in $\Spz^*$, too.  Then, up to a subsequence we have
$\sigma_n\weakto \sigma $ in $\Spz^*$ and $\zeta_n\weakto \zeta$ in
$\Spz^*$. Since $\pl  \calR(0)$ is sequentially weakly closed in $\Spz^*$,
we find $\sigma \in \pl  \calR(0)$. By Hypothesis
\ref{h:closedness} we also have $\zeta \in \frsubq ztq $, hence
\begin{align*}
  \lim_{n\to\infty} \slov z{t_n}{q_n} 
& 
  = \lim_{n\to\infty}  \disv z^*({-}\zeta_n{-}\sigma_n) 
\geq \disv z^*({-}\zeta{-}\sigma) 
\\
&  \geq \conj z({-}\zeta)  \geq 
  \inf_{\widetilde\zeta \in \frsubq ztq} \conj z({-}\widetilde\zeta) = \slov ztq,
\end{align*}
 which is the desired lsc \eqref{liminf-diss-V-W} for $\slovname z$. 

With similar arguments we deduce the attainment \eqref{not-empty-mislo}. 
\end{proof}

In the above proof we have used in an essential way that
$\pl  \calR (0)$ is bounded in $\Spz^*$ by our assumption \eqref{R-coerc}.
Without this property, the argument still goes through provided that, given a
sequence $(q_n)_n\subset \Spq $ as in  Hypothesis \ref{hyp:Sept19},   all
sequences $(\zeta_n)_n$ with $\zeta_n\in \argminSlo z{t_n}{q_n}$ for all
$n\in\N$ happen to be bounded in $\Spu^* \ti \Spz^*$, which can be, of course,
an additional property of the subdifferential $\frname z \calE$.\medskip

Throughout the rest of this paper, we will always tacitly assume the validity
of Hypotheses \ref{hyp:setup}, \ref{hyp:diss-basic}, \ref{hyp:1},
\ref{h:closedness},  \ref{h:ch-rule},  and \ref{hyp:Sept19} and omit
 any explicit mentioning of  them in most of the upcoming results (with
the exception of our main existence theorems).

\Subsection{A priori estimates for the viscous solutions}
\label{su:AprioViscSol}

Let $(q_\eps)_\eps $  be
a family of solutions to the viscously regularized systems
\eqref{van-visc-intro} in the stricter sense of \eqref{enid}, which includes
the energy-dissipation balance   \eqref{enid.a}. 
By Theorem \ref{th:exist} the existence of solutions $q_\eps =(u_\eps,z_\eps)$
is guaranteed, and in this subsection we discuss some a priori
estimates on $(u_\eps,z_\eps)_\eps$ that are uniform with respect to the
parameter $\eps$ and that form the core of our vanishing-viscosity analysis.

The starting point is the energy-dissipation estimate \eqref{enid-ineq} that
follows directly from   \eqref{enid.a}.   Recalling the constant $C_\#$ from
\eqref{h:1.3d} in Hypothesis \ref{hyp:1} and $c_\calR$ in Hypotheses
\ref{hyp:diss-basic}, we see that  the following \emph{basic a priori
  estimates}, are valid under the \emph{sole} assumptions of Hypotheses
\ref{hyp:diss-basic} and  \ref{hyp:1}.

\begin{lemma}[Basic a priori estimates]
  \label{l:1}
 For all $\eps>0$ and all solutions
$q_\eps=(u_\eps,z_\eps):[0,T] \to \Spq=\Spu \ti \Spz$  of \eqref{enid} with
$\calE(0,q_\eps(0))< \infty$ we have
the a priori estimates
\begin{subequations}
  \label{est-quoted-later}
  \begin{align}
   \label{est-quoted.a}
    &  \int_0^T\! \Big( \frac{1}{\eps^\alpha} \disv u
      (\eps^\alpha u_\eps'(t)) + \calR(z_\eps'(t)) + \frac{1}{\eps} \disv z
      (\eps z_\eps'(t)) \\
& \hspace{7em} \nonumber
      + \frac{\slov u {t}{q_\eps(t)}}{\eps^\alpha}  + 
      \frac{\slov z {t}{q_\eps(t)}} \eps  \Big) \dd t \leq \mathrm
      e^{C_\#T}\calE(0,q_\eps(0)), 
    \\
   \label{est-quoted.b}
    & 0 \leq  \eneq t{q_\eps(t)} \leq \mathrm e^{C_\# t} \calE(0,q_\eps(0)) \text{
      for all } t\in [0,T]. 
  \end{align}
\end{subequations}
whence, in particular,
\begin{equation}
  \label{est1}
   \| z_\eps'\|_{\rmL^1(0,T; \Spy)} \leq \frac{\mathrm e^{C_\# T}} 
  {c_\calR}\,   \calE(0,q_\eps(0)) 
  \quad \text{and} \quad 
  \sup_{t\in [0,T]} \mfE(q_\eps(t)) 
  \leq  \mathrm e^{2C_\# T}   \calE(0,q_\eps(0))\,. 
\end{equation}
\end{lemma}
\begin{proof} The proof follows  as in  the purely rate-independent
  case treated in \cite[Cor\,3.3]{Miel05ERIS}.  We  start from
  \eqref{enid.a} and  drop the nonnegative
  dissipation to obtain
\[
\calE(t,q_\eps(t))\leq \calE(0,q_\eps(0))+ \int_0^t \pl_s\calE(s,q_\eps(s))\dd
s \leq \calE(0,q_\eps(0))+ \int_0^t C_\# \calE(s,q_\eps(s))\dd s ,
\]
where we used \eqref{h:1.3d}.  
Thus, Gr\"onwall's estimate gives \eqref{est-quoted.b} and this we find 
\[
\calE(0,q_\eps(0))+ \int_0^T \pl_s\calE(s,q_\eps(s))\dd
s \leq \calE(0,q_\eps(0))+ \int_0^T C_\# \mathrm e^{C_\# s} \calE(0,q_\eps(0))\dd
s = \mathrm e^{C_\# T} \calE(0,q_\eps(0))
\] 
and \eqref{est-quoted.a} is established as well, as $\calE(T;q_\eps(T))\geq C_0>0$
by \eqref{h:1.2}.

Since $\disv x$ and $\slovname x$ are nonnegative, assumption \eqref{R-coerc} leads
to the first estimate in \eqref{est1}. The last assertion follows from
\eqref{est-quoted.b} and applying  \eqref{h:1.3d} once again. 
\end{proof}

Clearly, \eqref{est1} provides a uniform bound on the total variation of the
solution component $z_\eps$ in the space $\Spy$.  A similar bound cannot be
expected for the components $u_\eps$, unless we add further assumptions.  To
see the problem consider $\Spu = \R^2$ and the ordinary differential
equation
\[
\eps^\alpha u'_\eps(t) + \rmD \varphi(u_\eps(t)) = z_\eps(t)= a\binom{\cos(\omega
  t)}{\sin(\omega t )},  \quad \text{where } \varphi(u)=\frac{\lambda}2|u|^2+
\frac12\max\{|u|{-}1,0\}^2
\]
with $\lambda\geq 0$. Note that $\varphi$ is uniformly coercive for all
$\lambda \geq 0$. However, the equation is linear for $|u|\leq 1$ and
has an exact periodic solution of the form 
\[
 u(t) = \big(\mathrm{Re}\, U(t), \mathrm{Im}\, U(t)\big)  \quad \text{with }
 U(t)= \frac a{\lambda {+}\mathrm i\,\omega \eps^\alpha} \,\mathrm e^{\mathrm
   i\,\omega t}\in \mathbb C,
\]
as long as $|U(t)|\leq 1$, i.e.\ $a^2\leq \lambda^2{+}\omega^2\eps^{2\alpha}$.  
In this case, the derivatives satisfy the following $\rmL^1$-estimates 
\[
\|u'_\eps\|_{\rmL^1(0,T)} = |\omega|T \,\| u_\eps\|_{\rmL^\infty}= 
\big|\frac{a\omega}{\lambda {+}\mathrm i\,\omega
  \eps^\alpha}\big| \,T = \frac1{(\lambda^2 {+}\omega^2 \eps^{2\alpha})^{1/2}}
  \,\|z'_\eps\|_{\rmL^1(0,T)} . 
\]
For $\lambda>0$ we thus obtain a bound on $\|u'_\eps\|_{\rmL^1(0,T)}$ from a 
bound on $\|z'_\eps\|_{\rmL^1(0,T)}$ as in \eqref{est1}. However, in the case
$\lambda=0$ the value $\|u'_\eps\|_{\rmL^1(0,T)}$ may blow up while
$\|z'_\eps\|_{\rmL^1(0,T)}$ remains bounded (or even tends to $0$) and $a^2 \leq
\omega^2\eps^{2\alpha}$, e.g.\ choosing $\omega=\eps^{-\alpha/2}$ and
$a=\eps^{2\alpha/\alpha}$. 

In the main part of this subsection,  we provide 
sufficient conditions for the validity of a  uniform bound on
$\|u'_\eps\|_{\rmL^1(0,T;\Spu)} $. In the spirit of the above ODE example we
assume that $u \mapsto \en t {u,z}$ is uniformly convex (i.e.\ $\lambda>0$) and
that $z \mapsto \rmD_u\en t {u,z}$ is Lipschitz. Moreover, we need to assume
that $\disv u$ is quadratic. More precisely, we have to  confine the
discussion to a special setup given by conditions \eqref{structure-diss}
and \eqref{eq:E=E1+E2cond}: 

\noindent (1) the dissipation potential $\disv u $ is quadratic: 
\begin{equation}
\label{structure-diss}
 \Spu \text{ is a Hilbert space \quad  and }\disv u(v): =
\frac12\|v \|_{\Spu}^2= \frac12\langle \mathbb V_\sfu v,v\rangle ,
\end{equation}
where $\mathbb V_\sfu:\Spu\to \Spu^*$ is Riesz' norm isomorphism; 

\noindent (2) the energy functional $\calE$ 
 has domain $\domq = \domene u \ti \domene z$  and 
admits the decomposition
\begin{subequations}
 \label{eq:E=E1+E2cond}
\begin{align}
  &\ene tuz = \calE_1(u) + \calE_2(t,u,z) \ \text{ with }
 \\[0.3em] 
  \label{E1-unif-cvx}
  &\exists\,\Lambda>0:\quad \calE_1 \text{ is  $\Lambda$-convex},
 \\[0.3em]
 & \label{fr-diff-E2}
  \forall\, (t,z ) \in [0,T]\ti \domene z:\quad  u\mapsto
   \calE_2(t,u,z) \text{ is Fr\'echet differentiable on } \domene u,
 \\[0.3em] \label{fr-Lip-E2}
   &\begin{aligned}  
   &\exists\, C_\sfu \in (0,\Lambda) \  \forall\, E>0 \ 
   \exists\, C_E>0 \  \forall\, t_1, t_2 \in [0,T] \ \forall\, (u_1,z_1), 
   (u_2,z_2) \in \subl{E} :
    \\
   & \hspace{6em} \norm{\mathrm{D}_\sfu \calE_2(t_1,u_1,z_1){-}
     \mathrm{D}_\sfu \calE_2(t_2,u_2,z_2)}{\Spu^*}
   \\ &\hspace{9em}  \leq C_E \left(
     |t_1{-}t_2| + \norm{ z_1{-}z_2}{\Spy} \right)+ C_\sfu
   \|u_1{-}u_2\|_\Spu
 \end{aligned}
\end{align}
\end{subequations}
where $\subl E$ denotes the sublevel of $\mfE$, cf.\ \eqref{Esublevels}.

Hence, the possibly nonsmooth, but \emph{uniformly convex} functional $\calE_1$
is perturbed by the smooth, but possibly nonconvex, functional
$u\mapsto \calE_2 (t,u,z)$. However, by $C_\sfu<\Lambda$ the mapping
$u\mapsto \calE(t,u,z)$ is still uniformly convex.

Unfortunately condition \eqref{eq:E=E1+E2cond} is rather restrictive,
because in concrete examples  the driving  energy functional features  a coupling between
the variables $u$ and $z$ that is more complex. Nevertheless the desired a
priori estimate derived in Proposition \ref{l:3.2} 
may still be valid. Indeed, for 
our delamination model examined in Section \ref{s:appl-dam} we establish 
the corresponding estimate via an \emph{ad hoc} approach for the specific
system.

The proof of the following results follows the technique for the a priori
estimate developed in \cite[Prop.\,4.17]{Miel11DEMF}. We emphasize that the two
additional assumptions \eqref{structure-diss} and \eqref{eq:E=E1+E2cond} yield
 that the solution $u_\eps$ for $\mathbb V_\sfu u' + \pl 
\calE_1(u) + \rmD \calE_1(t,u,z_\eps(t)) \ni 0$ is unique as long as
$z_\eps$ is kept fixed, since it is a classical Hilbert-space gradient flow for
a time-dependent, convex functional. 

\begin{proposition}[$\rmL^1$ bound on $u'_\eps$]
\label{l:3.2}
In addition to 
 Hypotheses \ref{hyp:diss-basic} and \ref{hyp:1} 
assume \eqref{structure-diss} and
\eqref{eq:E=E1+E2cond} and consider initial conditions
 $(q_\eps^0)_\eps$ 
 such that 
\[
\exists\, C_\mathrm{init}>0\ \forall\, \eps\in (0,1):\quad
\calE(0,q_\eps^0)+\eps^{-\alpha} \| \pl ^0_\sfu
\calE(0,q_\eps^0)\|_{\Spu^*} \leq C_\mathrm{init} <\infty,
\] 
where $\pl ^0_\sfu \calE(0,q_\eps^0)\subset \Spu^*$ denotes the unique element
of minimal norm in $\pl_\sfu \calE(0,q_\eps^0)\subset \Spu^*$.  Then, 
there exists a  $C>0$ such that for all $\eps \in (0,1)$ all
solutions $q_\eps = (u_\eps, z_\eps)$ of system \eqref{enid} with
$q_\eps(0)=q_\eps^0$ satisfy 
\begin{equation}
  \label{est2}
\begin{aligned}           \| u_\eps'\|_{\rmL^1(0,T; \Spu)} 
  &\leq \frac1{\Lambda{-}C_\sfu} \Big(
    C_\mathrm{init} + C_ET+ C_E \|z'_\eps\|_{\rmL^1(0,T;\Spy)} \Big) \\
&\leq 
\frac1{\Lambda{-}C_\sfu} \Big(
  C_\mathrm{init} + C_ET+ \frac{C_EC_\mathrm{init}}{c_\calR} \,\mathrm
  e^{C_\#T}  \Big). 
\end{aligned}
\end{equation}
\end{proposition}
\begin{proof} By Lemma \ref{l:1} all curves $q_\eps:[0,T] \to \Spq$ lie in
$\calS_E=\bigset{q\in \Spq}{\mfE(q)\leq E}$ for $E=\mathrm e^{2C_\# T} 
C_\mathrm{init}$. Throughout the rest of this proof we drop the subscripts 
$\eps$ at $q_\eps=(u_\eps,z_\eps)$, but keep all constant explicitly to 
emphasize that they do not depend on $\eps$.

Setting $\kappa= \Lambda{-}C_\sfu>0$, the uniform convexity of
$\calE(t,\cdot, z)$ gives $\langle \mu_1{-}\mu_2 , u_1{-}u_2 \rangle 
\geq \kappa \|u_1{-}u_2\|_\Spu^2 $ for all $\mu_j \in \pl_\sfu
\calE(t,u_j,z)$.  We write the equation for $u$ in the
form $0=\eps^\alpha\mathbb V_\sfu u'(t)+ \mu(t) $ with
$\mu(t)\in \pl_\sfu \calE(t,u(t),z(t))$. For small $h>0$ and
$t\in [0,T{-}h]$ we find
\begin{align*}
\frac{\eps^\alpha}2\frac{\rmd}{\rmd t} &\| u(t{+}h){-}u(t)\|^2_\Spu
 = \big\langle \eps^\alpha\mathbb V_\sfu (u'(t{+}h)-u'(t)), u(t{+}h)-u(t)
 \big\rangle 
\\
&= - \langle \mu(t{+}h) - \mu(t), u(t{+}h)-u(t)\rangle 
\\
&\leq  - \langle \widetilde \mu_h(t) - \mu(t), u(t{+}h)-u(t)\rangle  +
\|\widetilde \mu_h(t) {-} \mu(t{+}h) \|_{\Spu^*} \| u(t{+}h)-u(t)\|_\Spu,  
\end{align*}
where $\widetilde \mu_h(t) \in \pl_\sfu \calE(t, u(t{+}h),z(t))$. The
uniform convexity and \eqref{fr-Lip-E2} give
\[
 \frac{\eps^\alpha}2\frac{\rmd}{\rmd t} \varrho_h(t)^2 \leq 
- \kappa \varrho_h(t)^2 + C_E\big( h +
\|z(t{+}h){-}z(t)\|_{\Spy}\big) \varrho_h(t),
\]
where $\varrho_h(t):= \| u(t{+}h){-}u(t)\|_\Spu$. Choosing $\delta>0$
and setting $\nu_h(t):=\varrho_h(t)^2{+}\delta$ yields
\begin{align*}
\eps^\alpha \dot \nu_h &= \frac{\eps^\alpha \tfrac{\rmd}{\rmd t}\varrho_\rmH^2}{2
  \nu_h} \leq - \kappa \frac{\nu_n^2- \delta}{\nu_h} + C_E\big(h + 
\|z(\cdot\,{+}h){-}z\|_{\Spy}\big) \frac{\varrho_h}{\nu_h}\\
&\leq - \kappa \nu_h + \kappa \delta^{1/2} + C_E\big(h + 
\|z(\cdot\,{+}h){-}z\|_{\Spy}\big). 
\end{align*} 
Integrating this inequality in time we arrive at 
\[\textstyle
  \kappa \int\limits_0^{T-h} \varrho_h(t) \dd t 
  \leq \kappa \int\limits_0^{T-h} \nu_h(t) \dd t 
  \leq \eps^\alpha \nu_h(0) + \delta^{1/2} T + C_E
 h T + C_E\int_0^{T-h}\|z(t{+}h){-}z(t)\|_{\Spy} \dd t.
\]
Taking the limit $\delta \to 0^+$, dividing by $h>0$, and using
$\|z(t{+}h){-}z(t)\|_{\Spy}\leq \int_t^{t+h} \|z'(s)\|_{\Spy} \dd s$ gives
\[
\kappa \int_0^{T-h} \!\! \big\|\frac1h\big(u(t{+}h){-}u(t)\big) \big\|_\Spu \dd
t \leq \eps^\alpha\big\|\frac1h\big(u(0{+}h){-}u(0)\big) \big\|_\Spu + 
C_E T + C_E \int_0^T \| z'(t) \|_{\Spy} \dd t. 
\] 
Since the equation for $u$ is a Hilbert-space gradient flow we can apply
\cite[Thm.\,3.1]{Brez73OMMS}, which shows that $\frac1h(u(h){-}u(0)) \to
\pl ^0_\sfu \calE(0,u(0),z(0))$ for $h\to 0^+$. Thus,  in the limit
$\eps\to 0^+$  we find 
\[
\textstyle
\kappa \int\limits_0^T \|u'(t)\|_\Spu\dd t = \lim\limits_{h\to 0^+} 
\kappa \int\limits_0^{T\!-h}\!
\big\|\frac1h\big(u(t{+}h){-}u(t)\big) \big\|_\Spu \dd t \leq C_\mathrm{init} +
C_E T + C_E \int_0^T \| z'(t) \|_{\Spy} \dd t,
\]
which is the desired result, when recalling $\kappa = \Lambda - C_\sfu$. 
\end{proof}

The above result is valid for all solutions of the viscous system \eqref{enid},
but it relies on the rather strong assumptions \eqref{structure-diss} and
\eqref{eq:E=E1+E2cond}.  While the uniform convexity of $u \mapsto \en tuz$ in
\eqref{eq:E=E1+E2cond} seems to be fundamental, it is expected that 
the rather strong assumption that $\disv u $ is the square of a
Hilbert space norm, see \eqref{structure-diss}, can be relaxed, but then the
solution $u_\eps$ may no longer be uniquely determined for fixed $z_\eps$. In
that case it may be helpful to restrict the analysis to specific solution
classes satisfying better a priori estimates, e.g.\ to minimizing movements
obtained via time-incremental minimization problems as in
\cite[Thm.\,3.23]{MRS13} or to solutions obtained as limit of Galerkin
approximations as in \cite[Def.\,4.3 \& Thm.\,4.13]{Mielke-Zelik}.
 We also refer to our delamination model in Section \ref{s:appl-dam} for a
derivation of the  additional a priori estimate \eqref{est2} in a more difficult
case.

\Section{Parametrized Balanced-Viscosity solutions} 
\label{s:4+}

In this section we will give the definition of Balanced-Viscosity solution to
the rate-independent system $\RIS$ in a \emph{parametrized version}.  For this,
we study instead of the viscous solutions $q_\eps:[0,T] \to \Spq$ suitable
reparametrizations
$(\sft_\eps,\sfq_\eps):[0,\mathsf S_\eps] \to [0,T]\ti \Spq$, i.e.,
$\sfq_\eps(s)=q_\eps(\sft_\eps(s))$, see Section \ref{su:ReJoMFcn}. While quite
general reparametrizations are possible, we will perform the
vanishing-viscosity limit $\eps\to 0^+$ for the one given in terms of the
energy-dissipation arclength $s=\sfs_\eps(t)$ defined in terms of the \RJMF\
$\mename \eps\alpha$ arising from the rescaled joint B-function
$ \mfB^\alpha_\eps$, see \eqref{arclength-est1-2}.  The $\Gamma$-limit
$\mename 0\alpha$ of $\mename \eps\alpha$, which is called the \emph{limiting
  \RJMF}, will then be used, to introduce the concept of \emph{admissible
  parametrized curves}, see Definition \ref{def:adm-p-c} in Section
\ref{ss:4.1bis}.  This is the basis of our notion of for \emph{parametrized
  Balanced-Viscosity} ($\pBV$) \emph{solutions}, defined in Section
\ref{ss:4.2}.  Theorem \ref{thm:existBV} states our main existence result for
$\pBV$ solutions, which is based on the convergence  in  the
vanishing-viscosity limit $\eps\to 0^+$.  However, we emphasize that the notion
of `$\pBV$ solutions' is independent of the limiting procedure.  Finally, in
Section \ref{ss:6.3-diff-charact} we provide a characterization of (enhanced)
$\pBV$ solutions showing that they are indeed solutions of  the
time-rescaled  generalized gradient system \eqref{e:diff-char-intro}.

\Subsection{Reparametrization and rescaled joint M-functions}
\label{su:ReJoMFcn}

This subsection revolves around the concept and the properties of the limiting
\RJMF\ $\mename 0\alpha$ that will be introduced in the Definition
\ref{defM0}. First, we will prove that $\mename 0\alpha$ is the $\Gamma$-limit
of the family of M-functions $(\mename \eps\alpha)_\eps$ that appear naturally
in the reparametrized version of the energy-dissipation estimate
\eqref{enid-ineq} and that are given  by  a composition of the rescaled joint
B-function $\mfB_\eps^\alpha$ and the slopes $\mathscr S_\sfx^*$. Namely, the
\emph{\RJMF s} are defined by
\begin{equation}
  \label{def-Me}
  \begin{aligned}
    & \mename \eps \alpha: [0,T] \ti \domq \ti [0,\infty) \ti \mathbf{Q} \to
    [0,\infty],
    \\
    & \mathfrak{M}_\eps^{\alpha}(t,q,t',q'): =
    \begin{cases}
      \mfB_\eps^\alpha(t', u',z', \slov utq, \slov ztq)
      & \text{ for } \frsubq qtq \neq \emptyset,\\
      \infty & \text{otherwise};
    \end{cases}
  \end{aligned}
\end{equation}
where $ \mfB_\eps^\alpha$ is the rescaled joint B-function from
\eqref{eq:def.B.al.eps} associated with the dissipation potentials
$\psi_\sfu = \disv u$ and $\psi_{\mathsf{z}} =\calR{+} \disv z$.

The basis for the construction of parametrized BV solutions is the
reparametrization of the the viscous solutions $q_\eps:[0,T]\to \Spq$ in the
form $\sfq_\eps(s)=q_\eps(\sft_\eps(s))$ such that the behavior of the function
$(\sft_\eps,\sfq_\eps):[0,\sfS_\eps]\to [0,T]\ti \Spq$ is advantageous. In
particular, the formation of jumps in $q_\eps$ with
$\|q'_\eps(t)\|\approx 1/\eps$ can be modeled by a plateau-like behavior of
$\sft_\eps$ with $\sft'_\eps(s)\approx \eps$ and a soft transition of
$\sfq_\eps$ with $\|\sfq'_\eps(s)\|\approx 1$. The first usage of such
reparametrizations for the vanishing-viscosity limit goes back to
\cite{EfeMie06RILS}, but here we stay close to \cite[Sec.\,4.1]{MRS13} in using
an `energy-based time reparametrization'.  Hence, for a family
$(q_\eps)_\eps =(u_\eps, z_\eps)_\eps$ of solutions to \eqref{van-visc-intro}
for which the estimates from Lemma \ref{l:1} hold, as well as  the
additional a priori estimate \eqref{est2} on $\int_0^T\|u'_\eps\|_\Spu \dd t$,
  we reparametrize the functions $q_\eps$ using the
\emph{energy-dissipation arclength}
$\mathsf{s}_\eps: [0,T]\to [0,\mathsf{S}_\eps]$ with
$\mathsf{S}_\eps:= \mathsf{s}_\eps(T)$ (cf.\ \cite[(4.3)]{MRS13}) defined by
\begin{equation}
\label{arclength-est1-2}
\begin{aligned}
  \mathsf{s}_\eps(t): = \int_0^t \Big( 1 & +
  \disve u{\eps^\alpha}( u_\eps'(t)) + \calR(z_\eps'(t)) +\disve z \eps(z_\eps'(t))
  \\
  & + \frac{\slov u {t}{q_\eps(t)} }{\eps^\alpha}  +
  \frac{\slov z {t}{q_\eps(t)}} \eps   + \|u_\eps'(t)\|_{\Spu} \Big) \dd t\,,
 \end{aligned}
\end{equation}
such that estimates  \eqref{est-quoted-later} and \eqref{est2} yield
that $\sup_{\eps>0} \mathsf{S}_\eps \leq C$.  Below we consider the
reparametrized curves
$(\mathsf{t}_\eps, \mathsf{q}_\eps) : [0,\mathsf{S}_\eps] \to [0,T]\ti \Spq $
defined by $\mathsf{t}_\eps: = \mathsf{s}_\eps^{-1}$,
$\mathsf{q}_\eps : = q_\eps \circ \sft_\eps$ and show in Section \ref{ss:4.2}
that they have an absolutely continuous limit $(\sft,\sfq)$, up to choosing a
subsequence.

We first remark that the quantities involved in the definition
of $\mathsf{s}_\eps$ rewrite as
\begin{equation}
    \label{just4clarity}
    \calR(z_\eps')  +  \disve u{\eps^\alpha} (  u_\eps') +  
     \disve z\eps (  z_\eps')
    + \frac{\slov u {t}{q_\eps}}{\eps^\alpha}  + 
    \frac{\slov z{t}{q_\eps}} \eps  \:=\: \mathfrak{M}_\eps^{\alpha}(t,q_\eps,1,q_\eps').
\end{equation}
With this, the energy-dissipation estimate \eqref{enid-ineq} 
can be rewritten in terms of the parametrized curves
$(\mathsf{t}_\eps, \mathsf{q}_\eps) $ in the form (for all $0\leq s_1< s_2 \leq
\mathsf S_\eps$) 
\begin{equation}
  \label{reparam-enineq}
  \begin{aligned}
    & \eneq {\sft_\eps(s_2)}{\sfq_\eps(s_2)} +\int_{s_1}^{s_2}
    \mathfrak{M}_\eps^\alpha (\sft_\eps(\sigma), \sfq_\eps(\sigma),
    \sft_\eps'(\sigma), \sfq_\eps'(\sigma)) \,\dd \sigma \\ & \leq \eneq
    {\sft_\eps(s_1)}{\sfq_\eps(s_1)} +\int_{s_1}^{s_2} \pl_t \eneq
    {\sft_\eps(\sigma)}{\sfq_\eps(\sigma)} \,\sft'_\eps(\sigma) \,\dd \sigma\,.
  \end{aligned}
\end{equation}    
Moreover, the definition of $\sfs_\eps$ in \eqref{arclength-est1-2} is
equivalent to the normalization condition
\begin{equation}
  \label{normal-cond}
  \sft_{\eps}'(s) + \mathfrak{M}_\eps^{\alpha}(\sft_\eps(s), 
  \sfq_\eps(s),\sft_\eps'(s),\sfq_\eps'(s))
  + \|  \sfu_\eps'(s)\|_{\Spu} =1 \qquad \foraa\, s \in (0,\mathsf{S}_\eps)\,.
\end{equation}
Of course, the reparametrized solutions $\sfq_\eps$ inherit the energy estimate
\eqref{est1}, namely 
\begin{equation}
  \label{en-est-param}
  \sup_{s\in [0,\mathsf{S}_\eps]} \mfE (\sfq_\eps(s))
  \leq \mathrm e^{2C_\# T} \sup_{\eps\in (0,1)}\calE(0,\sfq_\eps(0))\,.
\end{equation}

The a priori estimates \eqref{normal-cond} and \eqref{en-est-param} for the
reparametrized curves $(\sft_\eps,\sfq_\eps)_\eps $ will be strong enough to
ensure their convergence along a subsequence, 
as $\eps \to 0^+$, to a curve $(\sft,\sfq) : [0,\mathsf{S}]\to [0,T]\ti \Spq $,
with $\mathsf{S}= \lim_{\eps \to 0^+}\mathsf{S}_\eps$.  The basic properties of
$(\sft,\sfq)$ are fixed in the concept of \emph{admissible parametrized
  curve}, see Definition \ref{def:adm-p-c}.

For studying the limit $\eps \to 0^+$, we need to bring into play the
limiting \RJMF\ $ \mename{0}{\alpha}$.

\begin{definition}
\label{defM0}
We define $ \mename{0}{\alpha}
: 
[0,T]\ti \domq \ti [0,\infty) \ti \Spq \to [0,\infty] $ via
\begin{equation}
 \label{mename-0}
 \mename{0}{\alpha}(t,q,t',q') : =
 \begin{cases}
   \mfB_0^\alpha(t',u',z', \slov u tq, \slov ztq)\hspace*{-2em} & 
   \hspace*{2em}\text{ for } \frsubq
   qtq \neq \emptyset, 
   \\[0.3em]
   0 & \text{for } t'=0, \, q'=0 \text{ and } 
   \\ & (t,q) \in   \overline{\mathrm{dom}(\frname q\calE)}^{\mathrm{w,S}}
    {\setminus}\mathrm{dom}(\frname q\calE), 
   \\[0.3em]
   \infty & \text{otherwise},
 \end{cases}
\end{equation}
where  $\mfB_0^\alpha$ is defined in Proposition \ref{pr:Mosco.Beps}  and 
 $\overline{\mathrm{dom}(\frname q\calE)}^{\mathrm{w,S}} $ is the
weak closure of $ \mathrm{dom}(\frname q\calE)$ confined to energy sublevels:
\begin{equation}
  \label{closure-domain-subdifferential}
  \!\overline{\mathrm{dom}(\frname q\calE)}^{\mathrm{w,S}} \! := \bigset{ (t,q) }{
  \exists\, (t_n,q_n)_n \subset \mathrm{dom}(\frname q\calE){:} \ (t_n,q_n)
  \weakto (t,q), \ \sup_{n} \mfE (q_n) <\infty }\,. 
\end{equation}
\end{definition}
It follows from Proposition \ref{pr:Mosco.Beps} that 
\begin{equation}
\label{important-for-later}
(t',q') \mapsto  \mename{0}{\alpha}(t,q,t',q') \text{ is convex and $1$-homogeneous for all $(t,q) \in [0,\infty) \ti \Spq $}.
\end{equation}

Relying on Proposition \ref{pr:Mosco.Beps} and Hypothesis \ref{hyp:Sept19}, we
are ready to prove the following  $\Gamma$-convergence  result, which
straightforwardly gives that $ \mename{0}{\alpha}$ is (sequentially) lower
semicontinuous with respect to the weak topology of $\R\ti\Spq \ti \R\ti\Spq$
along sequences with bounded energy.

\begin{proposition}[Weak $\Gamma$-convergence of M-functions]
\label{pr:Mosco.Meps}
The limiting M-function \\
$\mename 0\alpha : [0,T]\ti \domq \ti [0,\infty) \ti \Spq \to [0,\infty]$ is
the $\Gamma$-limit of the M-functions $(\mename{\eps}{\alpha})_\eps$, with
respect to the weak topology, \emph{along sequences with bounded energy},
namely the following assertions hold: \smallskip

\noindent 
 (a)  $\Gamma$-$\liminf$ estimate: 
\begin{subequations}
\label{Gamma-convergence-concrete}
\begin{equation}
\label{Gamma-liminf}
\begin{aligned}
&
\begin{aligned}
\Big(     (t_\eps,q_\eps,{t_\eps'}, q_\eps')\weakto (t,q,t', q') \text{ in } 
\R\ti\Spq \ti \R\ti\Spq  \text{ as } \eps \to 0^+ \ \text{ with }
   \sup_{\eps>0} \mfE(q_\eps)< \infty \Big)
  \end{aligned}
\\
&\qquad  \Longrightarrow \quad 
    \meq{0}{\alpha} tq{t'}{q'} \leq \liminf_{\eps \to 0^+} 
        \meq{\eps}{\alpha} {t_\eps}{q_\eps}{t_\eps'}{q_\eps'};
        \end{aligned}
\end{equation}

\noindent 
 (b) 
$\Gamma$-$\limsup$ estimate: 
  \begin{equation}
    \label{Gamma-limsup}
    \begin{aligned}
      & \forall\, (t,q,t',q') \in [0,T]\ti \domq \ti [0,\infty) \ti \Spq \
      \exists\, (t_\eps,q_\eps,t_\eps',q_\eps')_\eps \text{ such that }
      \\
      & \quad \text{\upshape\;\ (i)}\quad (t_\eps,q_\eps, t_\eps', q_\eps') \weakto  (t,q,t',q') \text{ in }
      \R\ti\Spq \ti \R\ti\Spq \text{ as } \eps \to 0^+, 
      \\
      & \quad \text{\upshape\ (ii)}\quad   \sup\nolimits_{\eps>0}
      \mfE(q_\eps) < \infty, \text{ and }
       \\
     & \quad \text{\upshape(iii)}\quad   \meq{0}\alpha tq{t'}{q'} \geq
     \limsup\nolimits_{\eps \to 0^+} \meq{\eps}\alpha 
      {t_\eps}{q_\eps}{t_\eps'}{q_\eps'}\,.
    \end{aligned}
  \end{equation}
\end{subequations}
\end{proposition}

\noindent
\begin{proof} Concerning (a), let
$ (t_\eps,q_\eps,{t_\eps'}, q_\eps')_\eps$ be a sequence as in
\eqref{Gamma-liminf}. Of course we can suppose that 
$\liminf_{\eps \to 0^+} \meq{\eps}{\alpha}
{t_\eps}{q_\eps}{t_\eps'}{q_\eps'}<\infty$, and thus that
$\sup_\eps \meq{\eps}{\alpha} {t_\eps}{q_\eps}{t_\eps'}{q_\eps'}<\infty$. Then,
there exists $\bar\eps>0$ such that for all $\eps \in (0,\bar\eps)$ there holds
$t_\eps' > 0$, 
the Fr\'echet subdifferential 
$\frsubq q{t_\eps}{q_\eps} $ is non-empty,   and
$\meq {\eps}{\alpha} {t_\eps}{q_\eps}{t_\eps'}{q_\eps'} =
\mfB_\eps^\alpha(t_\eps', u_\eps',z_\eps', \slov u{t_\eps}{q_\eps}, \slov
z{t_\eps}{q_\eps}) $. 
In order to apply  Proposition \ref{pr:Mosco.Beps} we now need to discuss the boundedness of the slopes 
$ (\slov u{t_\eps}{q_\eps})_\eps$,  $(\slov
z{t_\eps}{q_\eps})_\eps$. Indeed,  
If $t'>0$, then $t_\eps'\geq c>0$ for all $\eps \in (0,\bar\eps)$ (up to choosing a smaller $\bar\eps$), so that, 
by the definition \eqref{eq:rescaled.B.eps} of $\mfB_\eps^\alpha$ we infer that $ \slov u{t_\eps}{q_\eps} \leq C \eps^\alpha$ and 
$ \slov z{t_\eps}{q_\eps} \leq C \eps$ for all  $\eps \in (0,\bar\eps)$. In the case $t'=0$, suppose e.g.\ that 
$\liminf_{\eps \to 0}\slov u{t_\eps}{q_\eps}=+\infty$ while $\liminf_{\eps \to 0}\slov z{t_\eps}{q_\eps}<+\infty$.  Then, from 
the coercivity estimate \eqref{eq:1LowBo.Bae} we deduce (up to extracting a not relabeled subsequence) that 
$u_\eps'\to 0$. Thus,
\[
\begin{aligned}
\liminf_{\eps\to 0}\mfB_\eps^\alpha(t_\eps', u_\eps',z_\eps', \slov u{t_\eps}{q_\eps}, \slov
z{t_\eps}{q_\eps})  & \geq  \liminf_{\eps\to 0}\calB_{\disv z}\big( \frac{t_\eps'}{\eps}, z_\eps', \slov
z{t_\eps}{q_\eps}\big)  \\ & \geq \mfB_0^\alpha (0,0,z', \slov utq, \slov ztq)\,, 
\end{aligned}
\]
with the latter estimate due to Proposition \ref{pr:Mosco.Beps}, Hypothesis
\ref{hyp:Sept19}, and the monotonicity of
$\mfB^\alpha_0(\tau,q',\sigma_\sfu,\sigma_\sfz)$ in $\sigma_\sfu$ and
$\sigma_\sfz$.  We may argue similarly in the case
$\liminf_{\eps \to 0}\slov u{t_\eps}{q_\eps}<+\infty$ and
$\liminf_{\eps \to 0}\slov z{t_\eps}{q_\eps}=+\infty$ and when both limits are
finite.

The $\Gamma$-$\limsup$ estimate (b) is trivial for all
$ (t,q,t',q') \in [0,T]\ti \domq \ti [0,\infty) \ti \Spq $ with
$\meq 0\alpha tq{t'}{q'}=\infty$.  If
$\meq 0\alpha tq{t'}{q'} = \mfB_0^\alpha(t',u',z',\slov utq, \slov
ztq)<\infty$, then the $\limsup$ estimate immediately follows via the constant
recovery sequence $(t_\eps, q_\eps, t_\eps', q_\eps') \equiv (t,q,t',q')$ with
the same arguments as in the proof of Proposition \ref{pr:Mosco.Beps}. Let us
now suppose that $(t',q')=(0,0) $ with
$ (t,q) \in \overline{\mathrm{dom}(\frname q\calE)}^{\mathrm{w,S}} {\setminus}
\mathrm{dom}(\frname q\calE)$, so that $\meq 0\alpha tq{t'}{q'} =0$. We observe
that there exists a sequence $(t_n,q_n)_n \subset \mathrm{dom}(\frname q\calE)$
with $(t_n,q_n) \weakto (t,q)$ and $\sup_{n} \mfE(q_n)<\infty$.  We will show
that for every null sequence $(\eps_k)_{k\in \N}$ there exists a recovery
sequence for $(t,q,0,0)$. For this, we first fix $n\in \N$ and associate with
$(t_n,q_n, 0, 0)$ the recovery sequence
$(t_{\eps_k,n}, q_{\eps_k,n}, t_{\eps_k,n}', q_{\eps_k,n}') =
(t_n,q_n,t_{\eps_k,n}', 0)$, where we choose $ t'_{\eps_k,n}>0$ such that
\[
t'_{\eps_k,n} \leq \eps_k \, ,\quad 
 \frac{t_{\eps_k,n}'}{\eps_k^\alpha} \slov u{t_n}{q_n} \leq \eps_k \, ,
  \quad \text{and} \quad
 \frac{t_{\eps_k,n}'}{\eps_k} \slov z{t_n}{q_n} \leq \eps_k\,.
\]
Setting $n=k$ we obtain the sequence 
$(\tilde{t}_{\eps_k}, \tilde{q}_{\eps_k}, \tilde{t}_{\eps_k}',
\tilde{q}_{\eps_k}')= (t_k,q_k, t'_{\eps_k,k}, 0)  \weakto  (t,q,0,0)$, which gives
(i). By construction we also have $\sup_{k\in \N}\mfE( \tilde{q}_{\eps_k})\leq
\sup_{n\in \N}\mfE( q_n) <\infty$, which gives (ii). Moreover, because of $
\tilde{t}_{\eps_k}'> 0$ and $\tilde q'_{\eps_k}=0$ we have 
\begin{align*}
&\meq{\eps_k}\alpha
{\tilde{t}_{\eps_k}}{\tilde{q}_{\eps_k}}{\tilde{t}_{\eps_k}'}{\tilde{q}_{\eps_k}'}
= \mfB^\alpha_{\eps_k}(\tilde{t}_{\eps_k}', 0 ,  \slov  u{t_k}{q_k}, \slov
z{t_k}{q_k})  \\
&=   \frac{t_{\eps_k,k}'}{\eps_k^\alpha} \slov u{t_k}{q_k} +
\frac{t_{\eps_k,k}'}{\eps_k} \slov z{t_k}{q_k} \leq 2\eps_k\to 0  = \meq 0\alpha
tq{0}{0}\, .   
\end{align*} 
Thus,  condition (iii) in  \eqref{Gamma-limsup} holds as well. With this, Proposition
\ref{pr:Mosco.Meps} is established. 
\end{proof}

For later use we also introduce the \emph{`reduced' \RJMF} 
\begin{equation}
\label{decomposition-M-FUNCTION}
  \mredname{0}{\alpha} : [0,T]\ti \domq \ti [0,\infty)  \ti \Spq \to
  [0,\infty], \quad \mredq 0\alpha tq{t'}{q'} : = \meq  0\alpha
  tq{t'}{q'}-\calR(z'). 
\end{equation}
We observe that the dissipation potentials $\psi_\sfu: = \disv u$ and
$\psi_{\mathsf{z}}: =\calR{+} \disv z$ have rate-independent parts null and
equal to $\calR$, respectively, and that
$\mfb_{\psi_{\mathsf{z}}} = \calR + \mfb_{\disv z}$ thanks to
\eqref{later-added}. Thus, from \eqref{mename-0} and  Proposition \ref{pr:Mosco.Beps} we infer that
the following representation formula for $  \mredname{0}{\alpha}$:
\begin{subequations}
\label{l:partial}
\begin{align}
\nonumber
\text{for } \frsubq qtq \neq \emptyset \text{ we have }\hspace*{-2em}
\\
\qquad 
 \nonumber
   t'>0:\qquad &  \mredq 0\alpha tq{t'}{q'} = 
  \begin{cases} 
    0 & \text{ for } \slov utq=\slov ztq=0, \\
    \infty& \text{ otherwise};
  \end{cases}  
\\ \nonumber
 t'=0, \ \alpha>1{:}\ \
  & \mredq 0\alpha tq{0}{q'}
   =\begin{cases}  \mfb_{\disv z}
            (z',\slov ztq)  &\text{for }\slov utq=0, \\
      \mfb_{\disv u}(u',\slov utq)& \text{for } \slov utq>0, \ z'=0,\\
     \infty& \text{otherwise};
 \end{cases}
\\  \label{l:partial-1}
 t'=0, \ \alpha=1{:}\ \
   & \mredq 0{\,1\,} tq{0}{q'}= \mfb_{\disv u \oplus
    \disv z}(q',\slov utq{+}\slov ztq)
\\ \nonumber
 t'=0,\ \alpha<1{:}\ \
& \mredq 0\alpha tq{0}{q'}
   =\begin{cases} \mfb_{\disv u} (u',\slov utq) 
        &\text{for } \slov ztq=0, \\
      \mfb_{\disv z}(z',\slov ztq)& \text{for } \slov ztq>0, \  u'=0,\\
     \infty& \text{otherwise},
 \end{cases}
\\
 \nonumber
\text{for } \frsubq qtq = \emptyset  \text{ we have }\hspace*{-2em}
\\
  &\hspace*{-6em}
\label{l:partial-2}
 \mredq 0\alpha tq{t'}{q'} = 
 \begin{cases}
   0 & \text{for } t'=0, \, q'=0 \text{ and }   (t,q) \in   \overline{\mathrm{dom}(\frname q\calE)}^{\mathrm{w,S}}
    {\setminus}\mathrm{dom}(\frname q\calE), 
   \\[0.3em]
   \infty & \text{otherwise}.
 \end{cases}
\end{align}
\end{subequations} 
The expressions in \eqref{l:partial} reflect the fact that $\mename
\eps\alpha$  only depends on the three cases given by  $\alpha \in (0,1)$,
$\alpha=1$, or $\alpha>1$.  

We emphasize that  $ \mredname 0\alpha$ depends on $\calR$ as well, namely
through $\slovname z$ which is defined via $\conj z$. 
In particular, for $t'>0$ finiteness of $ \mredq 0\alpha {t}q{t'}{q'}$ enforces
that $0=\slov u tq = \slov ztq $ and hence, taking into account Hypothesis
\ref{hyp:Sept19},
\begin{equation}
\label{station+loc-stab-forced}
\left\{
\begin{array}{llc}
 \text{the stationarity of $u$:}
& \exists\, (\mu,\zeta) \in \frsubq qtq\, : & 
\mu =0 \, ,
\\[0.3em]
\text{the local stability of  $z$:} 
& \exists\,  (\widetilde\mu,\widetilde\zeta) \in \frsubq qtq\, : & 
\widetilde\zeta \in \pl  \calR (0)\,.
\end{array}
\right.
\end{equation}
In the specific cases of dissipation potentials $\disv u$ and $\disv z$
considered in Example \ref{ex:VVCP}, we even have the explicit expression of
the respective contact potentials $\mfb_{\disv u}$ and $\mfb_{\disv z}$, and
thus of the (reduced) \RJMF\  $\mredname 0\alpha$. In particular, let us
revisit the $p$-homogeneous case:

\begin{example}[The $p$-homogeneous case]\slshape
\label{ex:p-homog}
Suppose that the dissipation potentials $\disv u$ and $\disv z$ are
positively $p$-homogeneous with the \emph{same} $p\in (1,\infty)$. Then,
combining \eqref{p-homo-mfb} with \eqref{l:partial} we conclude that for  $t'=0$ and $\frsubq qtq \neq \emptyset$ 
we have (where $\hat{c}_p = p^{1/p} (p')^{1/p'}$)
\begin{align}
\nonumber
 \alpha>1{:}\ \
  & \mredq 0\alpha tq{0}{q'}
   =\begin{cases} 
   \hat{c}_p  \left( \disv z(z')\right)^{1/p} \left( \slov ztq \right)^{1/p'}
   &\text{for }\slov utq=0, \\ 
       \hat{c}_p  \left( \disv u(u')\right)^{1/p} \left( \slov utq
       \right)^{1/p'} & \text{for } \slov utq>0 , \ z'=0,\\ 
     \infty& \text{otherwise};
 \end{cases}
\\ \label{RJMF-p-hom}
\alpha=1{:}\ \
   & \mredq 0\alpha tq{0}{q'}=   \hat{c}_p   \left( \disv u(u'){+} \disv z(z')
   \right)^{1/p} \left(\slov utq {+} \slov ztq  \right)^{1/p'} 
\\ \nonumber
\alpha<1{:}\ \
& \mredq 0\alpha tq{0}{q'}
   =\begin{cases}  \hat{c}_p  \left( \disv u(u')\right)^{1/p} \left( \slov utq
     \right)^{1/p'} &\text{for } \slov ztq=0, \\
      \hat{c}_p  \left( \disv z(z')\right)^{1/p} \left( \slov ztq
      \right)^{1/p'} & \text{for } \slov ztq>0, \ u'=0,\\
     \infty& \text{otherwise}.
 \end{cases}
\end{align}
\end{example}

The  M-functions $ \mename \eps{\alpha}$ enjoy
suitable coercivity properties that will play a key role in the compactness
arguments for proving the existence of $\BV$ solutions.   These estimates
are direct consequences of the the lower bounds on $\mfB^\alpha_\eps$ derived
in Lemma \ref{le:LoBo.Bae} and the definition of $\mename \eps\alpha$. The
importance here is the uniformity in $\eps\in [0,1]$.

We also emphasize that we are stating a result that is focusing on $z'$ and
ignoring $u'$, which reflects the fact that we always assume the bound on
$\|u'_\eps\|_{\rmL^1(0,T;\Spu)}$ whereas for $z'_\eps$ we only have a bound in
$\rmL^1(0,T;\Spy)$, but we need the derivative $\sfz'(s)\in \Spz$ at least in
points where $\slov z{\sft(s)}{\sfq(s)}>0$.

\begin{lemma}
  \label{new-lemma-Alex}
The following estimates hold for all $c>0$ and
$\eps\in [0,1]$ with  $\varkappa$ from Lemma \ref{le:LoBo.Bae}: 
\begin{subequations}
 \label{est-Alex}
 \begin{align}
 \label{est-Alex-1}
 \qquad\alpha\in (0,1):\quad & \slov ztq \geq c \quad &\Longrightarrow \quad 
  \|z'\|_{\Spz}\: \leq \: \frac{\mename \eps\alpha(t,q,t',q')}{\varkappa(c)}, \qquad
\\
 \label{est-Alex-1-bis}
\alpha\geq 1:\quad   &\slov utq{+} \slov ztq \geq c &\Longrightarrow \quad 
  \|z'\|_{\Spz} \: \leq \: 
\frac{\mename \eps{\alpha}(t,q,t',q')}{\varkappa(c)}. \qquad
\end{align}
\end{subequations}
\end{lemma}

\noindent
 The proof of \eqref{est-Alex-1} and \eqref{est-Alex-1-bis} follows directly
from the definition of $\mename \eps\alpha$ and the corresponding estimates
\eqref{eq:1LowBo.Bae} and 
\eqref{eq:1LowBo.Bag1} for $\mfB^\alpha_\eps$ in Lemma \ref{le:LoBo.Bae},
respectively.

The following result is an immediate consequence of the definition of
$\mename \eps\alpha$ and of  Proposition \ref{pr:VVCP}(b5), if we 
recall the definitions of $\mathfrak A_\sfx^{*,0}$ from
\eqref{not-empty-mislo}.

\begin{lemma}
 \label{new-lemma-Ricky}
For all $\alpha>0$ and all $\eps\in [0,1]$ we have that
\begin{equation}
  \label{coerc-ric}
    \mename \eps\alpha (t,q,t',q') \geq - \pairing{}{\Spu}{{  \mu }}{ u'}
     - \pairing{}{\Spz}{\zeta}{z'}
\end{equation}
for all $(t,q,t',q') \in [0,T]\ti \Spq \ti [0,\infty) \ti \Spq$ and all
$\xi  = (\mu,\zeta) \in \argminSlo utq, \, \zeta \in \argminSlo ztq$.
\end{lemma}

\Subsection{Admissible parametrized curves}
\label{ss:4.1bis}
 
The concept of \emph{admissible parametrized curve} is tailored in such a way
that it is able to describe limiting curves
$(\sft,\sfq):[a,b]\to [0,T]\ti \Spq$ of a family of parametrized viscous curves
$(\sft_\eps,\sfq_\eps)_{\eps}$ satisfying
\[
\sup_{\eps\in (0,1)} \int_a^b  \meq\eps{\alpha}
{\sft_\eps(s)}{\sfq_\eps(s)} {\sft'_\eps(s)} {\sfq'_\eps(s)} \dd s \ < \ \infty.
\]
Since Proposition \ref{pr:Mosco.Meps} guarantees that 
$\mename0{\alpha}$ is the $\Gamma$-limit of $\mename\eps{\alpha}$ it seems natural
that such curves can be characterized by the condition 
\begin{equation}
  \label{eq:M0a.sft.sfq}
 \int_a^b  \calR(z'(s)) \dd s +
 \int_a^b  \mredq 0{\alpha}
{\sft(s)}{\sfq(s)} {\sft'(s)} {\sfq'(s)} \dd s  <  \infty. 
\end{equation}
However, this expression is not well-defined, since we are not able to define
the derivatives $\sfq'(s)=(\sfu'(s),\sfz'(s))$ almost everywhere. To 
 reformulate \eqref{eq:M0a.sft.sfq} in a proper way, 
 we take advantage of the special form of
$\mathfrak{M}^{\alpha,\mathrm{red}}_0$ given in \eqref{l:partial} by observing that
$\sfz'(s)$ is only needed on the special sets $\SetG\alpha\sft\sfq$ to be
defined below.  Hence, condition   \eqref{eq:M0a.sft.sfq} can be replaced by
\eqref{summability} in Definition \ref{def:adm-p-c}   ahead, which relies on 
the fact
that absolutely continuous curves $z$ with values in (the possibly
non-reflexive space) $\Spy$ need not be differentiable with respect to time.
Therefore, the pointwise derivative $z'$ is replaced by a scalar surrogate,
cf.\ \eqref{calR-mder} below, whose definition involves the dissipation
potential $\calR$ and generalizes the concept of \emph{metric derivative} from
the theory of gradient flows in metric spaces \cite{AGS08}.
\begin{enumerate}
\item We say that a curve $z: [a,b]\to \Spy$ is $\calR$-absolutely continuous
  if there exists a nonnegative function $m \in \rmL^1(a,b)$ such that
  \[
    \calR( z(s_2) {-} z(s_1) ) \leq \int_{s_1}^{s_2} m(s) \dd s \qquad \text{for
      every } a\leq s_1\leq s_2 \leq b,
  \]
  and we denote by $ \mathrm{AC}([a,b]; \Spy, \calR)$ the space of
  $\calR$-absolutely continuous curves.
\item
  For a curve $z \in \mathrm{AC}([a,b]; \Spy, \calR)$   we set 
  \begin{equation}
    \label{calR-mder}
    \calR [z'](s): = \lim_{h\to 0} \calR \Big( \frac1h \big( z(s{+}h)\;\!{-} \!\;
        z(s)\big) \Big)  \qquad \foraa\, s \in (a,b).
  \end{equation}
\end{enumerate}

We are now in a position to give our definition of admissible parametrized
curve, which adapts \cite[Def.\ 4.1]{MRS13} to the present multi-rate
system. We recall that the slope functions $\slovname x$ are lsc according to
Hypothesis \ref{hyp:Sept19}. Hence, along continuous curves
$(\sft,\sfq):[a,b]\to [0,T]\ti \Spq$ the following sets are relatively  open: 
\begin{equation}
  \label{setGalpha}
  \SetG\alpha\sft\sfq : = 
    \begin{cases}
      \bigset{ s\in [a, b] }{    \slov  u{\sft(s)}{\sfq(s)} {+}
      \slov  z{\sft(s)}{\sfq(s)} >0 }  & \text{for } \alpha \geq 1,
      \\
      \bigset{ s\in [a, b] }{ \slov  z{\sft(s)}{\sfq(s)}
      >0 }     & \text{for } \alpha\in (0,1).
     \end{cases}
\end{equation}
The difference  between the cases  $\alpha > 1$  and
$\alpha \in (0,1)$ in the definition of the set $ \SetG\alpha\sft\sfq$ is
commented  after the following definition.

\begin{definition}[${\mathscr{A}([a,b];[0,T]\ti \Spq)}$] 
\label{def:adm-p-c}
A curve $(\sft,\sfq) = (\sft,\sfu,\sfz): [a,b] \to [0,T]\ti \Spq$ is called an
\emph{admissible parametrized curve} if
\begin{enumerate}
\item $\sft$ is non-decreasing,  $\sft \in \AC([a,b];\R)$, 
  $\sfu \in \AC([a,b];\Spu)$ and $\sfz \in \mathrm{AC}([a,b];\Spy,\calR)$;
\item $\slov u{\sft}{\sfq}  = 0$ and $\slov z {\sft}{\sfq}  = 0$ 
   on the   set $\bigset{ s \in (a,b) }{ \sft'(s) >0 } $;
\item $\sfz$ is locally $\Spz$-absolutely continuous on the open set
  $ \SetG\alpha {\sft}{\sfq}$, and $\sft$ is constant on every connected
  component of $\SetG\alpha \sft\sfq$;
\item $  \sup_{s\in [a,b]}\mfE (\sfq(s))  \leq E$ for some $E>0$;
\item there holds
  \begin{equation}
   \label{summability}
   \int_{a}^{b}  \calR [\sfz'](s) \dd s + \int_{\SetG\alpha\sft\sfq}
   \mredname 0\alpha (\sft(s),\sfu(s),\sfz(s),0,\sfu' (s),\sfz' (s)) 
   \dd s <\infty\,.
 \end{equation}
\end{enumerate}

We will denote by $\mathscr{A}([a,b];[0,T]\ti \Spq)$ the collection of all
admissible parametrized curves from $[a,b]$ to $[0,T]\ti \Spq$.  Furthermore,
we say that $(\sft,\sfq) \in \mathscr{A}([a,b];[0,T]\ti \Spq)$ is
\begin{itemize}
\item[\textbullet] \emph{non-degenerate}, if
  \begin{equation}
    \label{non-degeneracy}
   \sft'(s) +  \calR [\sfz'](s) + \| \sfu'(s)\|_\Spu >0 \quad \text{ for a.a.\
     $s\in (a,b)$;} 
   \end{equation}
\item[\textbullet] \emph{surjective}, if $\sft(a) =0$ and $\sft(b)=T$.
\end{itemize}

Finally, in the case in which the function $\sft$, defined on the canonical
interval $[0,1]$, is constant with $\sft(s) \equiv t$ for some $t\in [0,T]$, we
call $\sfq$ an \emph{admissible transition curve between $q_0:=q(0)$ and
$q_1: = q(1)$  at time $t$}, and we will use the notation 
\[
\admtcq t{q_0}{q_1}: = \{ \sfq: [0,1]\to \Spq\, : \ (\sft, \sfq) \in
\mathscr{A}([0,1];[0,T]\ti \Spq), \ \sft(s) \equiv t, \ \sfq(0) = q_0, \,
\sfq(1) = q_1 \}. 
\]
\end{definition}

The requirement that $\sfz$ has to be locally $\Spz$-absolutely
continuous on the set $\SetG\alpha \sft\sfq$, and the different definition of
$\SetG\alpha \sft\sfq$ in the cases $\alpha \geq 1$ and $\alpha \in (0,1)$, are
clearly motivated by properties \eqref{est-Alex} (which, in turn, derive from
Lemma \ref{le:LoBo.Bae}).  Indeed, in the case $\alpha \in (0,1)$, in view of
\eqref{est-Alex-1}, once $\meq 0\alpha {\sft}{\sfq}{\sft'}{\sfq'}$ is estimated
and $\slov z{\sft}{\sfq}$ is strictly positive, then
$\meq 0\alpha {\sft}{\sfq}{\sft'}{\sfq'}$ provides a control on
$\|\sfz'\|_{\Spz}$.  Because of this, parametrized curves are required to be
absolutely continuous on the set $\slov z{\sft}{\sfq}>0$. In the case
$\alpha \geq 1$, in view of estimate \eqref{est-Alex-1-bis}, the $z$-component
of admissible parametrized curves is expected to be absolutely continuous on
the larger set where $\slov u{\sft}{\sfq}+\slov z{\sft}{\sfq}>0$.

Hence, on the one hand, in \eqref{summability} we integrate only over the set
$\SetG\alpha\sft\sfq$, because it is in $\SetG\alpha\sft\sfq$ where the
pointwise derivative \ $\sfz' \in \Spz$ exists,  which  makes the term
$\mredq 0\alpha{\sft(s)}{\sfq(s)}{0}{\sfq'(s) }$ well defined.  On the
other hand, the specific form of $\mredname{0}{\alpha}$ in \eqref{l:partial}
and the fact that $\mfb_\psi(v,0)=0$ for all $v$ show that $z' \in \Spz$ is
only needed on the set $\SetG\alpha\sft\sfq $. 

 Hereafter,  along an admissible parametrized curve $(\sft,\sfq)$ we shall use
the notation
\begin{equation}
  \label{short-hand-M0}
  \mathfrak{M}_0^\alpha [\sft,\sfq,\sft',\sfq'] (s): = \calR[\sfz'](s) +
  \indic_{\SetG\alpha\sft\sfq}(s) \mredq  0\alpha{\sft(s)}{\sfq(s)}{0}{\sfq'(s)} 
\end{equation}
with $ \calR[\sfz'] $ from \eqref{calR-mder}, and 
$\indic_{\SetG\alpha\sft\sfq}$ is the indicator function of the set
$\SetG\alpha\sft\sfq$.  Let us stress that the above notation makes sense only
along an admissible curve.  If the admissible curve $(\sft,\sfq)$ has the
additional property $\sfz \in \AC([a,b];\Spz)$ 
and thus $\sfz'(s)$ is well
defined as an element of $\Spz\subset\Spy$ for almost all $s\in (a,b)$, then
$ \calR[z'](s)= \calR(z'(s)) $ a.e.\ in $(a,b)$. Hence, for an admissible curve with
$\sfz \in \AC([a,b];\Spz)$ we have
$ 
  \mathfrak{M}_0^\alpha [\sft,\sfq,\sft',\sfq'] (s) =   \mathfrak{M}_0^\alpha 
  (\sft(s),\sfq(s),\sft'(s),\sfq'(s)) \  \foraa s \in (a,b).
$ 

\Subsection{Definition of parametrized Balanced-Viscosity solutions}
\label{ss:4.2}

We are now in a position to precisely define parametrized Balanced-Viscosity
($\pBV$) solutions to the rate-independent system $\RIS$, see Definition
\ref{def:pBV}.  At the core of this concept there lies a (parametrized)
chain-rule inequality, cf.\ Hypothesis \ref{h:ch-rule-param} that will be
imposed as an additional property of the rate-independent system,  while
 Proposition \ref{prop:better-chain-rule-MOexpl} will provide sufficient
conditions for the validity of Hypothesis \ref{h:ch-rule-param}.

We will also introduce an \emph{enhanced} version
of the $\pBV$ concept, in which we additionally require $z$ to be absolutely
continuous with values in $\Spz$.  In \cite[Sec.\,4.2]{MRS13} this notion had
been already introduced, using a different terminology that might create slight
confusion in the present multi-rate context and has thus been changed here.  We
believe the enhanced concept to be significant as well because, for some
examples (cf.\ e.g.\ the applications discussed in Section \ref{s:appl-dam}), the
vanishing-viscosity analysis will directly lead to enhanced BV solutions.

The definition of $\pBV$ solutions relies  on the validity of the following
assumption on the rate-independent system $\RIS$.

\begin{hypothesis}[Chain rule along admissible parametrized curves]
\label{h:ch-rule-param}
For every admissible parametrized curve
$(\sft,\sfq) \in \mathscr{A} ([a,b];[0,T]\ti \Spq)$
 \begin{equation}
 \label{ch-rule}
 \begin{aligned}
   &\text{the map $s\mapsto \eneq {\sft(s)}{\sfq(s)} $ is absolutely continuous
     on $[a,b]$ and}
   \\
   &\frac{\rmd}{\rmd s} \eneq{\sft(s)}{\sfq(s)} - \pl_t
   \eneq{\sft(s)}{\sfq(s)} \sft'(s) \geq - \mathfrak{M}_0^\alpha
   [\sft,\sfq,\sft',\sfq'](s) \ \foraa\, s \in (a,b).
\end{aligned}
\end{equation}
\end{hypothesis}

\begin{remark}
\label{rmk:ptfw-chain-rule}
\slshape In general, the chain-rule inequality \eqref{ch-rule} along a given
admissible parametrized curve $(\sft,\sfq)$ does not follow from
the chain rule of Hypothesis \ref{h:ch-rule}, because for these 
curves the pointwise derivative $\sfz'$ exists as an
element in $\Spz$ only on the set $\SetG\alpha\sft\sfq$ from \eqref{setGalpha}.
That is why, Proposition \ref{prop:better-chain-rule-MOexpl} provides a
sufficient condition under which Hypothesis \ref{h:ch-rule} ensures the 
validity of Hypothesis \ref{h:ch-rule-param}, albeit restricted to admissible curves
satisfying additionally $\sfz \in \AC ([a,b];\Spz)$.  
\end{remark} 

We are now  ready to introduce  the exact notion of $\pBV$  solutions.

\begin{definition}[$\pBV$  and enhanced $\pBV$ solutions]
\label{def:pBV}
In addition to Hypotheses \ref{hyp:setup}, \ref{hyp:diss-basic},
\ref{hyp:1}, \ref{h:closedness}, and \ref{hyp:Sept19}, let the rate-independent
system $\RIS$ satisfy Hypothesis \ref{h:ch-rule-param}. We call a curve
$(\sft,\sfq) \in \mathscr{A} ([a,b];[0,T]\ti \Spq)$ a \emph{parametrized
  Balanced-Viscosity} ($\pBV$) solution to the rate-independent system $\RIS$
if $(\sft,\sfq)$ satisfies the \emph{parametrized energy-dissipation balance}
\begin{equation}
\label{def-parBV}
\begin{aligned}
  \eneq{\sft(s_2)}{\sfq(s_2)} +\int_{s_1}^{s_2} \mathfrak{M}_0^\alpha
  [\sft,\sfq, \sft',\sfq'](s) \dd s = \eneq{\sft(s_1)}{\sfq(s_1)}
  +\int_{s_1}^{s_2} \pl_t \eneq{\sft(s)}{\sfq(s)} \sft'(s) \dd s
\end{aligned}
\end{equation}
for every $ a \leq s_1 \leq s_2\leq b$,  where $\mathfrak{M}_0^\alpha$ is
defined in  \eqref{short-hand-M0}.

A $\pBV$ solution $(\sft,\sfq)=(\sft,\sfu,\sfz)$ is called \emph{enhanced $\pBV$
  solution},  if 
additionally $\sfz\in \AC ([a,b];\Spz)$.
\end{definition}

For an enhanced $\pBV$ solution $(\sft,\sfq)$ we have
$\sfq \in \AC ([a,b];\Spq)$, since $\sfq \in \mathscr{A} ([a,b];[0,T]\ti \Spq)$
already implies $\sfu \in \AC ([a,b];\Spu)$.  As a consequence of the
chain-rule inequality \eqref{ch-rule} from Hypothesis \ref{h:ch-rule-param}, we
have the following characterization.

\begin{lemma}[Characterization of $\pBV$ solutions]
  \label{l:characterizBV}  Let Hypothesis \ref{h:ch-rule-param} hold
  additionally. Then for  an admissible parametrized curve
  $(\sft,\sfq) \in \mathscr{A} ([a,b];[0,T]\ti \Spq)$, the following three
  properties are equivalent:
\begin{enumerate}
\item $(\sft,\sfq) $ is a $\pBV$ solution of the rate-independent system
  $\RIS$;
\item $(\sft,\sfq) $ fulfills the \emph{upper energy estimate} $\leq$ in
  \eqref{def-parBV} on for $s_1=a$ and $s_2=b$;
\item $(\sft,\sfq) $ fulfills the pointwise identity for a.a.\ $s\in (a,b)$
  \begin{equation}
    \label{ptw-ident}
    \begin{aligned}
      \frac{\rmd}{\rmd s} \eneq{\sft(s)}{\sfq(s)} - \pl_t
      \eneq{\sft(s)}{\sfq(s)} \sft'(s) =- \mathfrak{M}_0^\alpha
      [\sft,\sfq,\sft',\sfq'](s) \,. 
    \end{aligned}
  \end{equation}
\end{enumerate}
\end{lemma}
   
\noindent
The proof is a simple adaptation of the arguments for
\cite[Prop.\,5.3]{MRS12} and \cite[Cor\,3.5]{MRS13} and is thus omitted.

\Subsection{Existence results  for $\pBV$ solutions}
\label{su:Exist.pBV}

Our first main result states that any family
$(\sft_\eps,\sfu_\eps,\sfz_\eps)_\eps$ obtained by suitably rescaling (cf.\
Remark \ref{rmk:arbitrary-param} ahead) a family of solutions to the viscous
system \eqref{van-visc-intro} converges along a subsequence, as $\eps\to 0^+$,
to a parametrized solution of the rate-independent system $\RIS$.

\begin{theorem}[Existence of $\pBV$ solutions]
\label{thm:existBV}
Under Hypotheses \ref{hyp:setup}, \ref{hyp:diss-basic}, \ref{hyp:1},
\ref{h:closedness}, \ref{hyp:Sept19}, and \ref{h:ch-rule-param}, let
$(q_{\eps_k})_k = (u_{\eps_k}, z_{\eps_k})_k \subset \AC ([0,T]; \Spq)$ be a
sequence of solutions to the generalized viscous gradient system
\eqref{van-visc-intro} with $(\eps_k)_k \subset (0,\infty)$ a null sequence.
Suppose that
\begin{equation}
    \label{init-data-cv}
    q_{\eps_k}(0) \to q_0 \text{ in } \Spq \ \text{ and }  \ \eneq
    0{q_{\eps_k}(0)} \to \eneq 0{q_0} \qquad \text{as $k\to\infty$},   
\end{equation}
for some $q_0=(u_0,z_0) \in \domq$. Let $\sft_{\eps_k}: [0,\sfS]\to
[0,T]$ be non-decreasing surjective time-rescalings such that 
$\sfq_{\eps_k} = (\sfu_\epsk,\sfz_\epsk)$ defined via
$\sfq_\epsk(s)=q_{\eps_k}(\sft_{\eps_k}(s))$ satisfies 
\begin{equation}
\label{condition-4-normali}
\begin{aligned}
  & \exists\, C>0 \ \forall\, k \in \N \ \ \foraa s \in (0,\sfS): \\ &
  \sft_\epsk'(s) + \calR (\sfz_\epsk'(s)) + \mredq {\epsk} \alpha
  {\sft_\epsk(s)}{\sfq_\epsk(s)} {\sft_\epsk'(s)}{\sfq_\epsk'(s)} +
  \|\sfu_\epsk'(s)\|_{\Spu} \leq C.
 \end{aligned}
\end{equation}
Then, there exist a (not relabeled) subsequence and a curve $(\sft,\sfq) \in
\mathscr{A}([0,\sfS]; [0,T]\ti \Spq)$ such that
\begin{enumerate}
\item
 \mbox{}\vspace{-1em}
\begin{equation}
\label{continuity-properties}
\begin{aligned}
  & \sft \in \rmC_{\mathrm{lip}}^0 ([0,\sfS]; [0,T]), \quad \sfq = (\sfu,\sfz)
  \in \rmC_{\mathrm{weak}}^0([0,\sfS]; \Spw\ti \Spx) ,
  \\
  &\sfu \in \rmC_{\mathrm{lip}}^0 ([0,\sfS];\Spu), \quad \sfz \in 
  \rmC_{\mathrm{lip}}^0 ([0,\sfS];\Spy) \cap  \rmC^0([0,\sfS];\Spz);
 \end{aligned}
\end{equation}
\item
the following convergences  hold as $k\to\infty$
\begin{subequations}
\label{cvs-eps}
\begin{align}
&
\label{cvs-eps-t}
\sft_{\eps_k} \to \sft \text{ in }   \rmC^0([0,\sfS]), 
\\
\label{cvs-eps-u-z-added}
\mbox{}\quad&  \sfu_\epsk \weaksto \sfu \text{ in }  W^{1,\infty} (0,\sfS;\Spu),
&& 
\sfz_\epsk \to \sfz \text{ in }  \rmC^0 ([0,\sfS];\Spz), 
\\
\label{cvs-eps-u-z}
&  \sfu_\epsk(s) \weakto \sfu(s) \text{ in } \Spw \ \text{and}  && 
\sfz_\epsk(s) \weakto \sfz(s) \text{ in } \Spx  \qquad \text{ for all } s \in
[0,\sfS]; \quad \mbox{}
\end{align}
\item $(\sft,\sfq)$ fulfills the upper energy-dissipation estimate $
  \leq$ in \eqref{def-parBV} on $[0,\sfS]$, hence $(\sft,\sfq)$ is a
  $\pBV$ solution to the rate-independent system $\RIS$.
\end{subequations}
\end{enumerate}
Moreover, $(\sft,\sfu,\sfz)$ is surjective and there hold the additional convergences
\begin{subequations}
\label{eq:E.M.cvg.PBV}
\begin{align}
&
\label{cvs-eps-energy}
\eneq{\sft_{\eps_k}(s)}{\sfq_{\eps_k}(s)} \to \eneq{\sft(s)}{\sfq(s)} \qquad
\text{for all } s \in [0,\sfS],  
\\
&
\label{cvs-eps-M}
\int_{s_1}^{s_2} \mathfrak{M}_{\eps_k}^\alpha (\sft_{\eps_k}(\sigma)
,\sfq_{\eps_k}(\sigma),\sft'_{\eps_k}(\sigma), \sfq'_{\eps_k}(\sigma)) \dd
\sigma \to 
\int_{s_1}^{s_2}  \mathfrak{M}_{0}^\alpha [\sft_,\sfq,\sft',\sfq'](\sigma)  \dd\sigma 
\end{align}
\end{subequations}
for all  $0\leq s_1\leq s_2 \leq \sfS$.
\end{theorem} 

We postpone the proof of Theorem \ref{thm:existBV} to Section
\ref{ss:8.2},  but point out here that the  core of the limit passage in the
parametrized energy-dissipation estimate 
\eqref{reparam-enineq}, leading to \eqref{def-parBV},  lies in  the following
straightforward consequence of Ioffe's theorem \cite{Ioff77LSIF} (see
also \cite[Thm.\,21]{Valadier90}). A `metric version' of Proposition
\ref{prop:Ioffe} below was proved in \cite[Lemma 3.1]{MRS09}.

\begin{proposition}
\label{prop:Ioffe}
Let $\Iof$ be a weakly closed subset of $\Spq$, and let
$(\mathscr{M}_\eps)_{\eps}, \, \mathscr{M}_0: \R \ti \Iof \ti \R \ti \Spq \to
[0,\infty]$ be measurable and weakly lower semicontinuous functionals
fulfilling the $\Gamma$-$\liminf$ estimate
\begin{equation}
\label{Gamma-liminf-weak}
\begin{aligned}
\Big(     (t_\eps,q_\eps,{t_\eps'}, q_\eps')\weakto (t,q,t', q') \text{ in } 
\R\ti\Iof \ti \R\ti\Spq  \text{ as } \eps \to 0^+ \Big)
 \Longrightarrow  
\mathscr{M}_0(t,q,t',q') \leq \liminf_{\eps \to 0^+} 
\mathscr{M}_\eps(t_\eps,q_\eps, t_\eps', q_\eps')\,.
        \end{aligned}
\end{equation}
Suppose that, for $\eps \geq 0$, the functional $\mathscr{M}_\eps(t,q,\cdot,\cdot)$ is convex for every $(t,q) \in \R \ti \Iof$. 
Let $(\sft_\eps, \sfq_\eps), \, (\sft,\sfq) \in \AC([a,b]; \R\ti\Iof)$ fulfill 
\begin{equation}
\label{ideal-convergences}
\sft_\eps(s) \to \sft(s), \quad \sfq_\eps(s) \weakto \sfq(s) \text{ for all } s\in [a,b], \qquad (\sft'_\eps, \sfq_\eps') \weakto (\sft',\sfq') \text{ in } \rmL^1(a,b; \R \ti \Spq)\,. 
\end{equation}
Then,
\begin{equation}
\label{thesis-Ioffe}
\liminf_{\eps \to 0^+} \int_a^b \mathscr{M}_\eps(\sft_\eps(s), \sfq_\eps(s),
\sft_\eps'(s),\sfq_\eps'(s)) \dd s \geq \int_a^b \mathscr{M}_0(\sft(s),
\sfq(s), \sft'(s),\sfq'(s)) \dd s\,. 
\end{equation}
\end{proposition}
\begin{proof}
  It is sufficient to introduce the functional
  $\bar{\mathscr{M}} :\R \ti \Iof \ti \R \ti \Spq \ti [0,\infty] \to
  [0,\infty]$ defined by
\[
\bar{\mathscr{M}}(t,q,t',q',\eps): = \left\{
\begin{array}{ll}
\mathscr{M}_\eps(t,q,t',q') & \text{if } \eps>0,
\\
\mathscr{M}_0(t,q,t',q') & \text{if } \eps=0,
\end{array}
\right.
\]
and to observe that $\mathscr{M}$ is lower semicontinuous with respect to the weak
topology of $\R \ti \Iof \ti \R \ti \Spq \ti [0,\infty]$ and convex for every
$(t,q) \in \R \ti \Spq$ and $\eps\geq 0$. Then, by Ioffe's theorem we conclude
that
\[
\liminf_{\eps \to 0^+} \int_a^b \bar{\mathscr{M}}(\sft_\eps(s), \sfq_\eps(s),
\sft_\eps'(s),\sfq_\eps'(s),\eps) \dd s \geq \int_a^b
\bar{\mathscr{M}}(\sft(s), \sfq(s), \sft'(s),\sfq'(s),0) \dd s\,, 
\]
i.e., \eqref{thesis-Ioffe}  is established. 
\end{proof}

\begin{remark}
\label{rmk:non-deg}
\slshape
 Theorem \ref{thm:existBV}  does not guarantee that the $\pBV$ solution is
\emph{non-degenerate} even if the quantity in \eqref{condition-4-normali} has a
uniform positive lower bound.   Nonetheless, any (possibly degenerate) 
solution $(\sft,\sfq)$ can be reparametrized to a \emph{non-degenerate} one
$(\widetilde\sft,\widetilde\sfq): [0,\widetilde\sfS]\to [0,T]\ti \Spq$, fulfilling
\begin{equation}
\label{normalization-gained}
\widetilde\sft'(\sigma) +  \calR [\widetilde\sfz'](\sigma) + \|
\widetilde\sfu'(\sigma)\|_\Spu =1  \qquad \foraa\, \sigma \in (0,\widetilde\sfS). 
\end{equation}
For this, we proceed as in 
\cite{EfeMie06RILS} and associate with $(\sft,\sfq)$ the rescaling
function $\widetilde\sigma$ 
defined by
$\widetilde\sigma(s) = \int_0^s ( \sft'(r) {+} \calR [\sfz'](r) {+} \|
\sfu'(r)\|_\Spu) \dd r $ and set $\widetilde\sfS = \sigma(\sfS)$.  We then
define $(\widetilde\sft(\sigma),\widetilde\sfq(\sigma)): = (\sft(s),\sfq(s)) $
for $\sigma = \widetilde\sigma(s)$.  The very same calculations as in 
\cite[Lem.\,4.12]{Miel11DEMF}  (or based on the reparametrization result
\cite[Lem.\,1.1.4]{AGS08}), yield \eqref{normalization-gained}.
\end{remark}

Our next result, whose proof is omitted (cf.\ also Remark
\ref{rmk:arbitrary-param}), addresses the existence of \emph{enhanced} $\pBV$
solutions.

\begin{theorem}[Existence of  enhanced $\pBV$ solutions]
\label{thm:exist-enh-pBV}
Assume Hypotheses \ref{hyp:setup}, \ref{hyp:diss-basic}, \ref{hyp:1},
\ref{h:closedness}, \ref{hyp:Sept19}, and \ref{h:ch-rule-param}. Suppose that
there exist rescaled solutions $(\sft_\epsk,\sfq_\epsk)_k$ to the viscous
system \eqref{van-visc-intro}$_{\eps_k}$ such that, in addition to
\eqref{condition-4-normali}, there also holds the estimate
\begin{equation}
\label{condition-4-normali-enhn}
\exists\, C>0 \ \forall\, k \in \N \quad  \foraa\, s \in (0,\sfS)\,: \qquad
\|\sfz_\epsk'(s)\|_{\Spz} \leq C\,. 
\end{equation}
Then, up to a (not relabeled) subsequence the curves $(\sft_\epsk,\sfq_\epsk)_k$
converge to an admissible parametrized curve
$(\sft,\sfq) \in \mathscr{A}([0,\sfS]; [0,T]\ti \Spq)$ such that
\eqref{continuity-properties}, \eqref{cvs-eps}, \eqref{eq:E.M.cvg.PBV}
 hold and additionally $ \sfz \in \rmC_{\mathrm{lip}}^0 ([0,\sfS];\Spz)$,
 i.e.,  $(\sft,\sfq)$ is an \emph{enhanced} $\pBV$ solution
to the rate-independent system $\RIS$.
\end{theorem}

\begin{remark}
\label{rmk:arbitrary-param}
\slshape In the statement of Theorem \ref{thm:existBV}, the reparametrization
$t=\sft_\epsk(s)$ yielding the rescaled solutions $\sfq_\epsk$ can be chosen
arbitrarily, provided it guarantees the Lipschitz bound
\eqref{condition-4-normali}. Under Hypotheses \ref{hyp:setup},
\ref{hyp:diss-basic} and \ref{hyp:1}, \emph{all} viscous solutions
$(u_\epsk,z_\epsk)$ satisfy the \ uniform bound
$\|z'_\epsk\|_{\rmL^1(0,T;\Spy)}\leq C$, see \eqref{est1}. If, additionally
 $\|u'_\epsk\|_{\rmL^1(0,T;\Spy)}\leq C$ holds (i.e.\ \eqref{est2} from
Corollary \ref{l:3.2}), then a reparametrization yielding
\eqref{condition-4-normali} is easily obtained, for instance, by  using the
energy-dissipation arclength in \eqref{arclength-est1-2}.

Similarly,  under the stronger a priori estimate 
\begin{equation}
  \label{enhanced-z-estimate}
  \exists\, C>0 \ \forall\, k \in \N \, : \qquad \big\| z'_\epsk
  \big\|_{\rmL^1(0,T;\Spz)} = \int_0^T 
  \|z_\epsk'(t)\|_\Spz \dd t \leq C,
\end{equation}
one easily obtains rescaled solutions  satisfying the stronger Lipschitz
bound \eqref{condition-4-normali-enhn}. Hence, one gains enhanced
compactness information for the sequence $(\sfz_\epsk)_k$, and the  proof
of Theorem \ref{thm:existBV} immediately yields a proof of Theorem 
\ref{thm:exist-enh-pBV}.
\end{remark}

We conclude this section with  some sufficient conditions  for the 
validity of (a stronger form of) the parametrized chain rule in Hypothesis
\ref{h:ch-rule-param}.  It will be derived  as a consequence  of the
non-parametrized chain rule in Hypothesis \ref{h:ch-rule}.

\begin{proposition}[Sufficient conditions for parametrized chain rule]
\label{prop:better-chain-rule-MOexpl}
Assume that \linebreak[3] Hypothesis \ref{h:ch-rule} holds and that 
the vanishing-viscosity contact potentials associated with 
$\disv u$
and $\disv z$ satisfy
\begin{equation}
\label{coercivity-VVCP} \exists\, c_{\mathsf{x}}>0 \, \ \forall\,  (v,\eta)  \in \mathbf{X}\times \mathbf{X}^*\, : \qquad 
\mfb_{\disv x}(v,\disv x^*(\eta)) \geq c_{\mathsf{x}} \| v\|_{\mathbf{X}} \| \eta\|_{\mathbf{X}^*} 
\end{equation}
for $\mathsf{x}\in \{ \sfu,\sfz\}$ and $\mathbf{X}\in \{ \Spu,\Spz\}$.

Then, the parametrized chain rule \eqref{ch-rule} holds along all admissible
curves $ (\sft,\sfq) \in \mathscr{A} ([a,b];[0,T]\ti \Spq)$ with
$\sfq \in \AC ([a,b];\Spq)$. In particular, we have
\begin{equation}
\label{better-chain-rule-MOexpl}
 \frac{\rmd}{\rmd s} \eneq{\sft}{\sfq}-  \pl_t \eneq{\sft}{\sfq} \sft'   =
 \pairing{}{\Spu}{\mu}{\sfu'} + \pairing{}{\Spz}{\zeta}{\sfz'} \geq  -
 \mathfrak{M}_0^\alpha (\sft,\sfq,\sft',\sfq') \quad  \aein (a,b) 
\end{equation}
for  all measurable selections $ (a,b) \ni s  \mapsto (\mu(s),\zeta(s) 
) \in \Spu^*\ti \Spz^*$ satisfying
for almost all $s\in (a,b)$
 $ (\mu(s),\zeta(s) 
) \in \argminSlo u{\sft(s)}{\sfq(s)} \ti \argminSlo z{\sft(s)}{\sfq(s)} $.
\end{proposition}

\noindent  
The proof will be carried out in  Appendix  \ref{s:app-CR}.

\Subsection{Differential characterization of enhanced $\pBV$ solutions} 
\label{ss:6.3-diff-charact} 

The main result of this section  is Theorem \ref{thm:diff-charact}, which
 provides a further characterization of \emph{enhanced $\pBV$ solutions} in
terms of solutions of a system of subdifferential inclusions,  see
\eqref{param-subdif-incl}. This differential form  has the very same
structure as the viscous system \eqref{DNE-system}, except that the small
parameters $\eps^\alpha$ and $\eps$ multiplying the viscous terms are replaced
by coefficients $\thn{u}$ and $\thn z$ satisfying the switching conditions
\eqref{eq:SwitchCond}.  For this, we use the optimality in the
energy-dissipation balance. 

In Lemma \ref{new-lemma-Ricky} we have established the estimate 
\begin{equation}
  \label{eq:M.a.0.ineq}
   \mename 0\alpha (t,q,t',q') \geq - \pairing{}{\Spu}{\mu}{u'} -
 \pairing{}{\Spz}{\zeta}{z'} \quad \text{for all } 
(\mu,\zeta) \in \argminSlo utq {\times} \argminSlo ztq, 
\end{equation}
which is  valid for all $(t,q,t',q') \in [0,T]\ti \Spq \ti [0,\infty) \ti
\Spq$ and  which is a generalization of the classical Young--Fenchel
inequality $\psi(v)+ \psi^*({-}\xi) \geq - \pairing{}{}{\xi}{v}$. 
With the first result of this section  we will show that, 
in analogy to
the characterization of generalized gradient-flow equations via the
energy-dissipation principle, we are able to characterize $\pBV$ solutions via
the optimality condition that estimate \eqref{eq:M.a.0.ineq} holds as an
equality. Thus, we define the \emph{contact set} $\Ctc_\alpha$ (cf.\
\cite[Def.\,3.6]{MRS2013}) via 
\begin{align}
  \label{ctc-set}
    \Ctc_\alpha: = \Big\{\, (t,q,t',q') \in [0,T] \ti \Spq \ti [0,\infty) \ti
    \Spq \; \Big|\; &  \exists\, (\mu,\zeta) \in \argminSlo utq {\times}
    \argminSlo ztq:
 \\ \nonumber
 & \meq{0}\alpha tq{t'}{q'} = -\pairing{}{\Spu}{\mu}{u'} - 
    \pairing{}{\Spz}{\zeta}{z'} \, \Big\}. 
\end{align}
Proposition \ref{pr:char-eBV} below makes the relation between enhanced $\pBV$
solutions and the contact set $ \Ctc_\alpha$ rigorous. We emphasize here that
we need to exploit the stronger version \eqref{better-chain-rule-MOexpl} of the
parametrized chain rule from Hypothesis \ref{h:ch-rule-param}, in addition to
Hypotheses \ref{hyp:setup}, \ref{hyp:diss-basic}, \ref{hyp:1},
\ref{h:closedness}, and \ref{hyp:Sept19}, always tacitly assumed. Recall that a
sufficient condition for such a chain rule is provided by Proposition
\ref{prop:better-chain-rule-MOexpl}.

\begin{proposition}[Enhanced $\pBV$ solutions lie in $\Ctc_\alpha$]
\label{pr:char-eBV}
Suppose that the parametrized chain rule \eqref{better-chain-rule-MOexpl} holds
along all admissible curves
$ (\sft,\sfq) \in \mathscr{A} ([a,b];[0,T]\ti \Spq)$ with
$\sfq \in \AC ([a,b];\Spq)$.  Then, a curve
$(\sft,\sfq) \in \mathscr{A} ([a,b];[0,T]\ti \Spq)$ is an enhanced\/ $\pBV$
solution of\/ $\RIS$ if and only if\/ $\sfq \in \AC([a,b];\Spq)$ and
$(\sft,\sfq,\sft',\sfq')\in \Ctc_\alpha$ a.e.\ in $(a,b).$
\end{proposition}
\begin{proof}
  Let us consider an admissible parametrized curve
  $(\sft,\sfq) \in \mathscr{A} ([a,b];[0,T]\ti \Spq)$ with
  $\sfq \in \AC ([a,b];\Spq)$.  By the characterization provided in Lemma
  \ref{l:characterizBV}, $(\sft,\sfq)$ is a $\pBV$ solution if and only if
  $ - \mathfrak{M}_0^\alpha (\sft,\sfq,\sft',\sfq') = \frac{\rmd}{\rmd s}
  \eneq{\sft}{\sfq}- \pl_t \eneq{\sft}{\sfq} \sft' $ almost everywhere in
  $(a,b)$.  Combining this with the chain-rule inequality
  \eqref{better-chain-rule-MOexpl} we in fact conclude that
\[
  \frac{\rmd}{\rmd s} \eneq{\sft}{\sfq}- \pl_t \eneq{\sft}{\sfq} \sft' =
  \pairing{}{\Spu}{\mu}{\sfu'}  + \pairing{}{\Spz}{\zeta}{\sfz'} = -
  \mathfrak{M}_0^\alpha (\sft,\sfq,\sft',\sfq') \quad \aein (a,b),
\]
\emph{for all} measurable selections $  \xi =(\mu,\zeta): (a,b)  \to 
 \argminSlo u{\sft(s)}{\sfq(s)} \ti \argminSlo z{\sft(s)}{\sfq(s)} $, hence
$(\sft,\sfq,\sft',\sfq')\in \Ctc_\alpha$ a.e.\ in $(a,b).$
The converse implication follows by the same argument. 
\end{proof}

The final step in relating enhanced $\pBV$ solutions to the solutions of the
subdifferential system \eqref{param-subdif-incl} is obtained by 
 analyzing the structure of $\Ctc_\alpha$.   For this, we exploit   the
exact form on $\mename0\alpha$ and use   the definition of the set
$\argminSlo x t q $ in terms of the Fr\'echet subdifferential $\frsubq x t q$,
$\mathsf{x}\in \{\mathsf{u}, \mathsf{z}\}$.  To formulate this properly, we
recall the definition of the rescaled viscosity potentials
$\disve x \lambda$ and their subdifferentials $\pl  \disve x \lambda$ from
\eqref{eq:Def.Vx.la} for $\lambda \in [0,\infty]$. In particular, we have
\begin{equation}
\label{convention-recall}
\pl  \disve x \lambda(v)= 
\pl  \disv x (\lambda v) \text{ for all  $\lambda\in [0,\infty)$, 
 and  }
 \pl 
\disve x \infty(v)=
\left\{
\begin{array}{ll}
\mathbf{X}^* &  \text{for $v=0$},
\\ 
\emptyset & \text{otherwise}. 
\end{array}
\right.
\end{equation}
Observe that, thanks to \eqref{later-added} we have $\pl \disv x (0) = \{ 0 \}$
for $\mathsf{x}\in \{\sfu,\sfz\}$.

We now consider the system of subdifferential inclusions for the quadruple
$(t,q,t',q')= (t,u,z, t',u',z')$  including the two parameters
$\thn u, \thn z \in [0,\infty]$: 
\begin{subequations}
\label{static-tq}
\begin{align}
\label{subdiff-stat.u}
\pl  \disve u{\thn u } (u')  &+\frsubq u{t}{q} \ni 0 && \text{in } \Spu^*,
\\
\label{subdiff-stat.z}
\qquad \qquad\pl \calR(z')  +  \pl  \disve z{\thn z} ( z') &+\frsubq
z{t}{q} \ni 0 && 
\text{in } \Spz^*, \qquad \qquad
\\
\label{switch-1-A}
t'\,\frac{\thn u}{1{+} \thn u} &= t'\, \frac{\thn z}{1{+}\thn z }  =0 \,.
\end{align}
\end{subequations}
Here we use the usual convention $\infty/(1{+}\infty)=1$ and emphasize that, at
this stage, system \eqref{static-tq} is not to be  understood  as a
system of subdifferential inclusions.  Instead, 
$(t',q')\in [0,\infty){\times} \Spq$ are treated as independent variables.
With this we are able to introduce the following subsets of
$[0,T]\ti \Spq \ti [0,\infty) \ti \Spq$, called  \emph{evolution regimes},
 thus providing a basis for the 
informal discussion at the end of Section \ref{s:EDI}:  
\begin{equation}
\label{regime-sets}
\begin{aligned}
\rgs Eu  & :=  \bigset{  (t,q,t',q')  }{ \exists\, \thn z \in [0,\infty]{:}
  \text{ \eqref{static-tq} holds with } \thn u =0 },
\\
\rgs Rz  & := \bigset{  (t,q,t',q') }{ \exists\, \thn u \in [0,\infty]{:} 
   \text{ \eqref{static-tq} holds with } \thn z =0 },
\\
\rgs Vu  & := \bigset{  (t,q,t',q') }{ \exists\, \thn z \in [0,\infty]{:}
  \text{ \eqref{static-tq} holds with } \thn u  \in (0,\infty)}  ,
\\
\rgs Vz  & := \bigset{  (t,q,t',q') }{ \exists\, \thn u \in [0,\infty]{:}
  \text{ \eqref{static-tq} holds with } \thn z  \in (0,\infty)}  ,
\\
\rgs V{uz} & : = \bigset{  (t,q,t',q') }{ \text{ \eqref{static-tq} holds with }
     \thn u = \thn z   \in (0,\infty)},  
\\
\rgs Bu &  := \bigset{  (t,q,t',q') }{  \exists\, \thn z \in [0,\infty]{:} 
            \text{ \eqref{static-tq} holds with }\thn u =\infty} ,
\\
\rgs Bz &  := \bigset{  (t,q,t',q') }{  \exists\, \thn u \in [0,\infty]{:} 
            \text{ \eqref{static-tq} holds with }\thn z =\infty} .
\end{aligned}
\end{equation}
The letters $\mathrm{E}, \, \mathrm{R},\, \mathrm{V},\, \mathrm{B} $, stand for
\emph{Equilibrated}, \emph{Rate-independent}, \emph{Viscous}, and
\emph{Blocked}, respectively.  We will discuss the meaning of the names of
 the evolution  regimes below. It will be efficient to use the notation 
\[
\rgs Au \rgs Cz  := \rgs Au \cap \rgs Cz\quad \text{ for } 
\mathrm{A \in \{E,V,B\}} \text{ and }  \mathrm{C \in \{R,V,B\} };
\]
nonetheless, note that the set $ \rgs V{uz} $ is different from (indeed,
strictly contained in) $\rgs Vu \rgs Vz $.  We also remark that any set
involving `V' of `B' is necessarily restricted to the subspace with $t'=0$
because of \eqref{switch-1-A}. With this, we are now in a position to state our
result for the contact sets $\Ctc_\alpha$, under the additional condition
\eqref{it-is-product} on the product form of the Fr\'echet subdifferential
$\pl_q\calE$. Proposition \ref{pr:charact-Ctc-set} below will be proven in
Section \ref{su:pr:char-Ctc-set}.

\begin{proposition}[$\Ctc_\alpha$ and evolution regimes]
\label{pr:charact-Ctc-set}
If, in addition, the Fr\'echet subdifferential $\pl_q \calE$ has the product
structure \eqref{it-is-product}, then we have the following inclusions for the
contact set $ \Ctc_\alpha $:
\begin{subequations} 
\label{eq:CtcSet.Incl}
\begin{align}
\alpha>1:\quad  &\label{eq:CtcSet.Incl.g1}
\Ctc_\alpha \ \subset \ 
  \rgs Eu \rgs Rz\ \cup \ \rgs Eu\rgs Vz \ \cup \ \rgs Bz ,
\\
\alpha=1:\quad  &\label{eq:CtcSet.Incl.e1}
\Ctc_{1 \,}\  \subset \  
\rgs Eu \rgs Rz \ \cup \ \rgs V{uz} \ \cup \ \rgs Bu \rgs Bz,
\\
\alpha\in (0,1):\ \ &\label{eq:CtcSet.Incl.l1}
\Ctc_\alpha \ \subset  \ 
\rgs Eu\rgs Rz \ \cup \ \rgs Vu \rgs Rz \ \cup \  \rgs Bu  \,,
\end{align}
\end{subequations}
where in all cases the three sets on the right-hand side are disjoint. 
\end{proposition}

\begin{remark}\slshape 
\label{rm:???}
In the characterization of (enhanced) $\pBV$ solution provided by Proposition
\ref{pr:char-eBV}, the contact condition
$\meq{0}\alpha tq{t'}{q'} = -\pairing{}{\Spu}{\mu}{u'} -
\pairing{}{\Spz}{\zeta}{z'}$ holds \emph{for all}
$ (\mu,\zeta) \in \argminSlo utq {\times} \argminSlo ztq$.  Hence, it seems
possible to define a smaller contact set $\widetilde\Sigma_\alpha$ by replacing
``$\exists$'' in \eqref{ctc-set} by ``$\forall$''. Because of
$\widetilde\Sigma_\alpha \subset \Ctc_\alpha$  inclusions
\eqref{eq:CtcSet.Incl}  would remain true. However, using our  larger
set $\Ctc_\alpha$ is sufficient to deduce that $\pBV$ solutions satisfy
the system of subdifferential inclusions \eqref{param-subdif-incl} ahead.
\end{remark}

The different evolution regimes characterized by the right-hand sides in
\eqref{eq:CtcSet.Incl} can be visualized by considering the three real
parameters $(t',\thn u, \thn z)\in [0,\infty)\ti [0,\infty]^2$, since the
rate-independent regimes $\rgs Eu$ and $\rgs Rz$ are given by $\thn u=0$ and
$\thn z=0$ respectively. Similarly, the viscous regimes $\rgs Vx$,
$\mathsf{x} \in \{ \mathsf{u}, \mathsf{z}\}$, are defined via
$\thn x\in (0,\infty)$, and the blocking regime $\rgs Bx$ is determined by
$\thn x=\infty$. The sets on the right-hand sides in \eqref{eq:CtcSet.Incl} are
then defined in terms of the switching conditions
\begin{equation}
\label{eq:Switch11} 
\text{\eqref{switch-1-A} holds \quad  and } \quad  
\begin{cases}  
 \thn u=0 \text{ or } \thn z=\infty&\text{ for } \alpha>1,\\
\thn u=\thn z \in [0,\infty] & \text{ for }\alpha=1,\\
\thn u=\infty \text{ or }\thn z =0 & \text{ for }\alpha \in (0,1).
\end{cases} 
\end{equation}
The corresponding sets in the $(t',\thn u,\thn z)$ space are depicted in Figure
\ref{fig:EqBlViRa}.

\begin{figure}
\mbox{}\hfill\begin{minipage}{0.38\textwidth}
\caption{} \label{fig:EqBlViRa}\raggedright
The switching conditions and the different regimes are displayed in the space
for $(t',\thn u,\thn z) \in [0,\infty]^3$. 

For $\sft'>0$  the only admissible regime is given by 
the intersection $\rgs Eu \rgs Rz=\rgs Eu \cap \rgs Rz$.

For $\sft'=0$ the different admissible regimes depend on $\alpha>0$:
 
$\ \ \ \alpha>1$: \ \  $\rgs Eu \cup \rgs Bz$

$\ \ \ \alpha=1$: \  \ $\rgs Eu \rgs Rz \cup \rgs V{uz}\cup \rgs Bu \rgs Bz$

$\alpha\in (0,1)$: $\rgs Vu \rgs Rz \cup \rgs Bu$ 
\end{minipage}
\hfill
\begin{minipage}{0.5\textwidth}
\newlength{\AMlength}
\setlength{\AMlength}{2.95em}
\begin{tikzpicture}[scale =1.60]
\draw[fill=gray!20] (-1,-0.5)--(0,0)--(0,2)--
                    (-1,1.5) node[left]{$\rgs Rz$}--(-1,-0.5);
\draw[fill=gray!20] (-1,-0.5)--(0,0)--(2,0)-- 
                  (1,-0.5) node[below]{$\rgs Eu$}--(-1,-0.5);
\draw[->, thick] (-0.4,0)--(2.5,0)
   node[right]{$\frac{\lambda_\sfz}{1{+}\lambda_\sfz}$}; 
\draw[->, thick] (0,-0.4)--(0,2.5)
   node[left]{$\frac{\lambda_\sfu}{1{+}\lambda_\sfu}$}; 
\draw[->, thick] (0.4,0.2)--(-1.3,-0.65)
   node[pos=1.1, above]{$\frac{\sft'}{1{+}\sft'}$}; 
\draw[fill=black] (2,2) circle (0.05);
\draw (0.2,2)--(-0.6,2) node[left] {$\rgs Bu $};
\draw (2,0.2)--(2,-0.4) node[below] {$\rgs Bz $};
\node[right] at (2.1,1) 
  {$\left.\rule[-\AMlength]{0pt}{2\AMlength}\right\} \rgs Vu$};
\draw (0.0,0.0)--(1.96,1.96) node[above, rotate=45,pos=0.5] {$\rgs V{uz} $};
\node[above] at (1,2.2) 
  {$\overbrace{\rule{2\AMlength}{0pt}}^{\displaystyle \rgs Vz}$};
\draw[line width=3 , opacity=0.9, color=green!70!blue] 
   (-1,-0.44) -- (0,0.06) -- (0,2)--(1.95,2)
 node[pos=0.4,below]{$\alpha \in (0,1)$};
\draw[line width=3 , opacity=0.9, color=red] 
 (-1,-0.5) -- (0.0,0.0) --(1.97,1.97)
 node[pos=0.5,below, rotate=45]{$\alpha =1$};
\draw[line width=3 , opacity=0.9, color=blue] 
  (-0.9,-0.5) -- (0.1,0.0) -- (2,0) node[pos=0.75,above]{$\alpha>1$} --(2,1.95);
%
\fill[color = white] (-1,-0.5) circle (0.1);
\draw (-1,-0.5) circle (0.1);
\end{tikzpicture}
\end{minipage}
\hfill\mbox{}
\end{figure}

The inclusions \eqref{eq:CtcSet.Incl} that relate  the contact sets to
the different  evolution regimes  $\rgs Au \rgs Cz$ have a clear and immediate
interpretation in terms of the  evolutionary behavior  of an enhanced
$\pBV$ solution $(\sft, \sfq)$:
\begin{compactitem}
\item[\textbullet] $\rgs Eu$ encodes the regime where
  $u=\sfu(s)$ stays in \emph{equilibria},  which may depend on $s$. Indeed, 
   inserting $\thn u(s)=0$ in \eqref{subdiff-stat.u} leads to the equilibrium
  condition $0 \in \frsubq u{\sft(s)}{\sfq(s)}$. This means that $\sfu(s)$ follows
  $\sfz(s)$ that may evolve rate-independently when $t'>0$, and may follow a
  viscous jump path, or may be blocked, when  $\sft'(s)=0$.

\item[\textbullet] $\rgs Rz$ denotes the rate-independent evolution  for $\sfz(s)$,
  where $\thn z(s)=0$. The component $\sfu(s)$  either follows staying in
  equilibria, evolves viscously, or is blocked. 
  
\item[\textbullet] In the case $t'>0$ only the rate-independent regime
  $\rgs Eu\rgs Rz$ is admissible. This is the regime with $\thn u=\thn z=0$
  where the viscous dissipation potentials $\disv u$ and $\disv z$ do not come
  into action.

\item[\textbullet] In the regime $\rgs Vx$, the variable $\sfx(s)$ evolves
  viscously with $\thn x(s)\in (0,\infty)$, and necessarily $\sft'(s)=0$. 

\item[\textbullet] $\rgs V{uz}$ is the special case occurring only for
  $\alpha=1$, where $\thn u(s)=\thn z(s)\in (0,\infty)$, i.e.\ both components
  have a synchronous viscous phase. 

\item[\textbullet] The blocked regime $\rgs Bx$, occurring when $\sft'(s)=0$,
  encodes the situation that $\thn x(s)=\infty$, which means that on the
  given time scale the viscosity is so strong that the $\mathsf{x}$-component
  cannot move, i.e.\ it is blocked with $\sfx'(s)=0$.
  
\item[\textbullet] $\rgs B{uz} = \rgs Bu \rgs Bz$ means that both components
  are blocked, namely $\sfq'(s)=0$. This can occur, for instance, if we set
  $(\sft(s), \sfq(s))=(t_*,q_*)$ for $s\in (s_1,s_2)$. Then,
  $\thn u(s)=\thn z(s)=\infty$ still gives a trivial, constant solution. 
  Such a behavior may occur after taking a limit like $\eps\to 0^+$, but  of
  course the interval can be cut out by defining a $\pBV$ solution on
  $[0,\sfS{-}s_2{+}s_1]$.  
\end{compactitem}

We are now in a position to prove a characterization of \emph{enhanced} $\pBV$
solutions in terms of the following system of subdifferential inclusions
\begin{equation}
  \label{param-subdif-incl}
  \begin{aligned}
    \pl \disve u{\thn u(s)} ( \sfu'(s)) + \frsub u{\sft(s)}{\sfu(s)}{\sfz(s)}
    \ni 0 & \quad\text{in } \Spu^*,
    \\
    \pl \calR(\sfz'(s)) + \pl \disve z{\thn z(s)} ( \sfz'(s)) + \frsub
    z{\sft(s)}{\sfu(s)}{\sfz(s)} \ni 0 &\quad \text{in } \Spz^*,
  \end{aligned}
\end{equation}  
where the balanced interplay of viscous and rate-independent behavior in the
equations for $u$ and $z$, respectively, is determined by the
(arclength-dependent) parameters $\thn u(s)$ or $\thn z(s)$.  We emphasize that
the so-called \emph{switching conditions} for $t'\geq 0$ and
$\thn u, \,\thn z \in [0,\infty]$, cf.\ \eqref{eq:SwitchCond} below, are
different for the three cases $\alpha>1$, $\alpha=1$, and $\alpha\in (0,1)$.

\begin{theorem}[Differential characterization of enhanced $\pBV$ solutions]
\label{thm:diff-charact} 
Assume Hypotheses \ref{hyp:setup}, \ref{hyp:diss-basic}, \ref{hyp:1},
\ref{h:closedness}, and \ref{hyp:Sept19} and let the parametrized chain rule
\eqref{better-chain-rule-MOexpl} hold.  In addition, suppose that the Fr\'echet
subdifferential $\pl_q \calE$ has the product structure from
\eqref{it-is-product}.  Let
$(\sft,\sfq) \in \mathscr{A}([0,\sfS];[0,T]\ti \Spq)$ be an admissible
parametrized curve with $\sfq \in \AC ([0,\sfS];\Spq)$.
\begin{enumerate}
\item If $(\sft,\sfq):(0,\sfS)\to \Spq$ is a enhanced $\pBV$ solution of $\RIS$, then
  there exist measurable functions $(\thn u,\thn z): (0,\sfS)\to [0,\infty]^2 $
  and $\xi=(\mu,\zeta): (0,\sfS)\to \Spu^*\ti \Spz^*$ with 
\begin{subequations}
  \label{e:diff-char}
  \begin{equation}
  \label{exist-meas-select}
  \mu(s) \in  \frsub u{\sft(s)}{\sfu(s)}{\sfz(s)}
  \ \text{ and } \ 
  \zeta(s) \in  \frsub z{\sft(s)}{\sfu(s)}{\sfz(s)}
  \quad \foraa\, s \in (0,\sfS)
\end{equation}  
satisfying for almost all $s\in (0,\sfS)$ the subdifferential inclusions
\begin{equation}
  \label{param-subdif-incl-selections}
  \begin{aligned}
    -\mu(s) \in \pl \disve u{\thn u(s)} ( \sfu'(s))\hspace{5em} 
    & \quad\text{in } \Spu^*,
    \\
    - \zeta(s) \in \pl \calR(\sfz'(s)) + \pl \disve z{\thn z(s)} ( \sfz'(s))
    &\quad \text{in } \Spz^*,
  \end{aligned}
\end{equation} 
and the switching conditions
\begin{equation}
  \label{eq:SwitchCond}
  \sft'(s)\,\frac{\thn u(s)}{1{+}\thn u(s) } = 0 = 
  \sft'(s)\,\frac{\thn z(s)} {1{+}\thn z(s)}  
  \quad \text{and } \quad 
     \begin{cases} 
      \thn u(s)\,\frac1{1{+}\thn z(s)} =0 & \text{for }\alpha>1, 
      \\
       \thn u(s) = \thn z(s)& \text{for }\alpha=1, 
      \\ 
       \frac1{1{+}\thn u(s)}\, \thn z(s)=0& \text{for } \alpha\in (0,1). 
      \end{cases}
    \end{equation}
  \end{subequations}
\item Conversely, if there exist measurable functions
  $(\thn u,\thn z): (0,\sfS)\to [0,\infty]^2 $  and
  $\xi=(\mu,\zeta): (0,\sfS)\to \Spu^*\ti \Spz^*$  satisfying
  \eqref{e:diff-char} and, in addition,
  \begin{equation}
    \label{hyp-4-chr}
    \sup_{s \in (0,\mathsf{S})} |\calE(\sft(s),\sfq(s))|<\infty,  \quad \text{and} \quad
    \int_0^{\mathsf{S}} \!\big(  \| \mu(s)\|_{\Spu^*}  \| \sfu'(s)\|_{\Spu} {+}
      \| \zeta(s)\|_{\Spz^*}  \| \sfz'(s)\|_{\Spz}  \big) \dd s <\infty, 
  \end{equation}
  then $(\sft,\sfq)$ is an enhanced $\pBV$ solution.
\end{enumerate}
\end{theorem}
\begin{proof} Part (1) basically follows from combining the characterization 
of enhanced $\pBV$ solutions from Proposition \ref{pr:char-eBV} in terms of 
the contact set, with Proposition \ref{pr:charact-Ctc-set}.  Only the
measurability of the coefficients $\thn u,\, \thn z: [0,\sfS]\to [0,\infty] $
and of the selections $\xi=(\mu,\zeta):(0,\sfS) \to \Spu^* \ti \Spz^*$ deserves
some discussion that is postponed to Appendix \ref{appendix-measurability}.

Let us now carry out the proof of Part (2).  After cutting out possible
intervals where $(\sft,\sfq)$ may be constant (i.e.\ in the blocking regime
$\rgs Bu \rgs Bz$), we may  suppose that the admissible parametrized curve
$(\sft,\sfq)$ fulfills the non-degeneracy condition \eqref{non-degeneracy}.  In
what follows, we will use the short-hand notation
\begin{equation}
  \label{short-hand-regimes}
  (0,\sfS) \cap \rgs Au \rgs Cz: = \{ s \in (0,\sfS)\, : \  (\sft(s),\sfq(s),
  \sft'(s), \sfq'(s) ) \in  \rgs Au \rgs Cz  \} 
\end{equation}
for $\mathrm{A \in \{E,V,B\}} $ and $ \mathrm{C \in \{R,V,B\} }$. We will
discuss at length the case $\alpha>1$; the very same arguments yield the thesis
also in the cases $\alpha=1$ and $\alpha \in (0,1)$.  It follows from the
switching conditions \eqref{eq:SwitchCond} that the integral
$I:=\int_0^\sfS \big( \pairing{}{\Spu}{{-}\mu}{\sfu'}{+}
\pairing{}{\Spu}{{-}\zeta}{\sfz'}\big) \dd s$ decomposes as
\begin{align}
& \label{initial-step}
I= I_1+I_2+I_3 \quad \text{with }
I_1:= \int_{(0,\sfS) {\cap}  \rgs Eu \rgs Rz}  \!\! \big(
   \pairing{}{\Spu}{{-}\mu(s)}{\sfu'(s)}{+}
   \pairing{}{\Spz}{{-}\zeta(s)}{\sfz'(s)}  \big) \dd s\,, 
\\
& \nonumber
  I_2:= \int_{(0,\sfS) {\cap}   \rgs Eu\rgs Vz}  \!\! \big(
    \pairing{}{\Spu}{{-}\mu}{\sfu'}{+}
    \pairing{}{\Spz}{{-}\zeta}{\sfz'}  \big) \dd s, \text{ and }
  I_3:= \int_{(0,\sfS) {\cap}  \rgs Bz}  \!\!\big(
    \pairing{}{\Spu}{{-}\mu}{\sfu'}{+}
    \pairing{}{\Spz}{{-}\zeta}{\sfz'}  \big) \dd s\,,
\end{align}
 where we use that the three regimes  $ \rgs Eu \rgs Rz$,
$ \rgs Eu\rgs Vz$, and $\rgs Bz$ are disjoint.  Now, on
$(0,\sfS) \cap \rgs Eu \rgs Rz$ we have that $\mu(s) \equiv 0$, while
$\zeta(s) \in \pl \calR(\sfz'(s))$, so that
\begin{equation*}
I_1=  \int_{(0,\sfS) {\cap}  \rgs Eu \rgs Rz} \calR(\sfz'(s))  =
\int_{(0,\sfS) {\cap}  \rgs Eu \rgs Rz} \meq 0\alpha
{\sft(s)}{\sfq(s)}{\sft'(s)}{\sfq'(s)} \dd s  
\end{equation*}
where we used \eqref{decomposition-M-FUNCTION} and \eqref{l:partial}, taking
into account 
$\slov u {\sft(s)}{\sfq(s)} = \slov z {\sft(s)}{\sfq(s)} \equiv 0$ on
$(0,\sfS) \cap \rgs Eu \rgs Rz$.  On $(0,\sfS) \cap \rgs Eu\rgs Vz$ we 
have  $\slov u {\sft(s)}{\sfq(s)} \equiv 0$  and the $z$-equation in
 \eqref{param-subdif-incl} holds with $\thn z(s)>0$, so that
\[
\begin{aligned}
I_2    & =   \int_{(0,\sfS) {\cap}   \rgs Eu\rgs Vz}   \frac1{\thn
  z(s)}\pairing{}{\Spu}{{-}\zeta(s)}{\thn z(s) z'(s)}  \dd s
\\
 & \stackrel{(1)}{=} 
  \int_{(0,\sfS) {\cap}   \rgs Eu\rgs Vz}  
\frac1{\thn z(s)} \left( 
  \calR \big( \thn z (s) \sfz'(s) \big) 
            {+} \disv z \big( \thn z (s) \sfz'(s)\big) 
            {+}   \conj z({-}\zeta(s))  \right) \dd s 
      \\
 &    \stackrel{(2)}{\geq}          
   \int_{(0,\sfS) {\cap}   \rgs Eu\rgs Vz}  
    \frac1{\thn z(s)}    \left( 
  \calR \big( \thn z (s) \sfz'(s) \big) 
            {+} \disv z \big( \thn z (s) \sfz'(s)\big) 
            {+}   \slov z {\sft(s)}{\sfq(s)}  \right) \dd s 
\\
& 
  \stackrel{(3)}{\geq}     
    \int_{(0,\sfS) {\cap}   \rgs Eu\rgs Vz}  \!\!\left(   
   \calR \big(  \sfz'(s) \big)  {+}    \mfb_{\disv z}(z'(s),\slov z
   {\sft(s)}{\sfq(s)})   \right)  \dd s  
   \stackrel{(4)}{=}    \int_{(0,\sfS) {\cap}   \rgs Eu\rgs Vz} \!\!\meq 0\alpha
   {\sft(s)}{\sfq(s)}{\sft'(s)}{\sfq'(s)} \dd s , 
    \end{aligned}        
\]
where {\footnotesize (1)} follows from \eqref{param-subdif-incl-selections} via
Fenchel-Moreau conjugation, {\footnotesize (2)} is a consequence of the
definition of $\slov z{\sft}{\sfq}$, {\footnotesize (3)} is due to the
definition of $ \mfb_{\disv z}$, and {\footnotesize (4)} again ensues from
\eqref{decomposition-M-FUNCTION} and \eqref{l:partial}.  Finally, with the very
same arguments we find that
\[
  \begin{aligned}
    I_3 & = \int_{(0,\sfS) {\cap} \rgs Bz}
    \pairing{}{\Spu}{{-}\mu(s)}{\sfu'(s)}
 \ = \ \int_{(0,\sfS) {\cap} \rgs Bz} \frac1{\thn u(s)}
    \pairing{}{\Spu}{{-}\mu(s)}{\thn u(s)\sfu'(s)} 
\\ 
& \geq \int_{(0,\sfS) {\cap}
      \rgs Bz} \mfb_{\disv u}(\sfu(s),\slov u {\sft(s)}{\sfq(s)}) \dd s
\  = \  \int_{(0,\sfS)  {\cap} \rgs Bz} \meq 0\alpha
 {\sft(s)}{\sfq(s)}{\sft'(s)}{\sfq'(s)} \dd  s\,.
   \end{aligned}
\]
Combining the above estimates with \eqref{initial-step} and with the chain-rule
\eqref{eq:48strong} (which applies thanks to \eqref{hyp-4-chr}), we ultimately
conclude that
\[
  \begin{aligned}
 \eneq {\sft(0)}{\sfq(0)}+  \!\int_0^\sfS \! \pl_t \calE(\sft(s),\sfq(s))\dd s
    & \geq \eneq {\sft(\sfS)}{\sfq(\sfS)} + \int_{0}^\sfS \!\big(
      \pairing{}{\Spu}{{-}\mu(s)}{\sfu'(s)}{+}
      \pairing{}{\Spz}{{-}\zeta(s)}{\sfz'(s)} \big) \dd s 
    \\
    & \geq \eneq {\sft(\sfS)}{\sfq(\sfS)} + \int_0^\sfS \meq 0\alpha
    {\sft(s)}{\sfq(s)}{\sft'(s)}{\sfq'(s)} \dd s \,,
   \end{aligned}
\]
namely $(\sft,\sfq)$ fulfills the upper energy-dissipation estimate.
Therefore, by Lemma \ref{l:characterizBV} we conclude that
$(\sft,\sfq)$ is an (enhanced) $\pBV$ solution.
\end{proof}

\Section{True Balanced-Viscosity solutions} 
\label{ss:4.3}
 
This section is devoted to the the concept of true Balanced-Viscosity ($\BV$)
solutions, i.e.\ solutions defined on the original time interval $[0,T]$
instead via the artificial arc length $s \in [0,\sfS]$.  This concept will be
introduced in Section \ref{ss:5.1} in Definition \ref{def:trueBV}. The central
ingredient in this notion is a Finsler-type  transition  cost that
measures the energy dissipated at jumps of the curve $(u,z)$, see Definition
\ref{def:jump}.  In Section \ref{ss:5.2} we will gain further insight into the
fine properties of true $\BV$ solutions,  while Section \ref{su:Exist.BV}
states  our two existence results, Theorems \ref{thm:exist-trueBV} and
\ref{thm:exist-nonpar-enh}, in which $\BV$ solutions to the rate-independent
system $\RIS $ are obtained by taking the vanishing-viscosity limit of system
\eqref{van-visc-intro} in the real process time, \emph{without
  reparametrization}.  Section \ref{ss:5.3} addresses the non-parametrized
counterpart of enhanced $\pBV$ solutions  called \emph{enhanced $\BV$
  solutions}, and Section \ref{su:Comp.pBV.BV} provides how parametrized and
true $\BV$ solutions are related. 

We start with some notations for functions having well-defined jumps. 

\begin{notation}[Regulated functions]
\label{not:1.1}
\slshape
Given a Banach space $\mathbf{B}$, we denote by
\begin{equation}
\label{def:regulated}
\begin{aligned}
\Reg 0T{\mathbf{B}}: =\Big\{\: f: [0,T] \to \mathbf{B}\; \Big| \; \forall\,
t\in [0,T]: \  &\llim ft: =
\lim_{s\to t^-} f(s) \text{ exists in }\mathbf{B},  \\ 
& \rlim ft: = \lim_{r\to t^+}f(r) \text{ exists in }\mathbf{B}  \; \Big\} 
\end{aligned}
\end{equation}
the space of (everywhere defined) \emph{regulated functions} on $[0,T]$ with
values in $\mathbf{B}$, where we use $\llim f0:= f(0)$
and $\rlim fT: = f(T)$. The symbol $\rmB\rmV([0,T];\mathbf{B})$ denotes the space
of everywhere defined functions of bounded $\mathbf{B}$-variation such
that   $\rmB\rmV([0,T];\mathbf{B}) \subset \Reg 0T{\mathbf{B}}$ with continuous
embedding. 
\end{notation}

\noindent  Note that for $f \in \Reg 0T{\mathbf{B}}$ the three values
$f(t^-)$, $f(t)$, $f(t^+)$ may all be different for $t\in (0,T)$, and that
distinguishing these values will be crucial for our notion of $\BV$
solutions.  

For a given $z\in \rmB\rmV([0,T];\Spy)$ we also  introduce the $\calR$-variation
\begin{equation}
\label{R-variation}
\Var_{\calR}(z;[a,b]): = \sup\Bigset{ \sum_{i=1}^N \calR(z(t_i){-}z(t_{i-1}))}
   {N\in \N,\   a=t_0<t_1<\ldots<t_{N}=b }  
\end{equation}
for $ [a,b]\subset [0,T]$, and we  observe that 
\begin{equation}
\label{cited-VarR-later}
\Var_{\calR}(z;[a,b]) = \int_a^b \calR[z'](t) \dd t 
   \quad \text{ for } z \in \AC ([a,b];\Spy,\calR). 
\end{equation}
We mention in advance that true $\BV$ solutions are curves $q=(u,z)$, with
$u\in \rmB\rmV([0,T];\Spu)$ and $z\in\Reg 0T{\Spz} \cap \rmB\rmV([0,T];\Spy)$.  For such
$q=(u,z)$ we introduce the \emph{jump set}
\begin{equation}
  \label{def:jumpset}
  \mathrm{J}[q] = \mathrm{J}[u] \cup  \mathrm{J}[z] \text{ with } \mathrm{J}[w]
  : = \bigset{ t \in [0,T] }{ \llim wt \neq w(t) \text{ or } \rlim wt \neq w(t) };
\end{equation}
we record that $  \mathrm{J}[q]$ consists of at most countably many points. 
Note that for $ \mathrm{J}[z]$ the left and the right limits are considered with
respect to the norm topology of $\Spz$. For later use, we finally  observe that
  \begin{equation}
  \label{4later-use-embed}
  \rmL^\infty(0,T;\Spx) \cap \rmB\rmV([0,T];\Spy) \subset \Reg 0T{\Spz},
  \end{equation}
  which can be easily checked exploiting
 the (compact) embeddings $\Spx\Subset \Spz \subset \Spy$.

\Subsection{Definition of true $\BV$ solution}
\label{ss:5.1}

The (possibly asymmetric) Finsler cost function is obtained by minimizing an
`infinitesimal cost', depending on the fixed process time $t\in [0,T]$ and
defined in terms of the \RJMF\ $ \mename{0}{\alpha}$, along \emph{admissible
  transition curves} $\sfq : [0,1] \to \Spq$.  From now on, for better clarity
we will denote a generic transition curve by $\serifTeta$ in place of $\sfq$.

\begin{definition}[Admissible transition curves, Finsler cost]
\label{def:jump}
For given $t\in [0,T]$ and $q_0=(u_0,z_0), q_1=(u_1,z_1) \in \Spu\ti \Spz$, we
define the Finsler cost induced by $\mename 0\alpha$ by
\begin{equation}
  \label{Finsler-cost}
 \costq{\mename 0\alpha}{t} {q_0}{q_1}
  : = \inf_{\serifTeta \in  \admtcq t{q_0}{q_1}}
 \int_0^1 \mename 0\alpha [t, \serifTeta, 0, \serifTeta'] \dd r 
\end{equation}
with the short-hand notation $\mename 0\alpha [\cdot,\cdot,\cdot,\cdot]$ from
\eqref{short-hand-M0} and $\admtcq t{q_0}{q_1}$ the set of all admissible
transition curves at time $t$ between $q_0$ and $q_1$,  see Definition
\ref{def:adm-p-c}. 
\end{definition}
Thanks to the $1$-positive homogeneity of the functional 
$\serifTeta' \mapsto \mename 0\alpha [t,\serifTeta,0,\serifTeta']$, we
observe that it is not restrictive to suppose that all transition curves are
defined on $[0,1]$.

We are now ready to define a new variation called the $\mename 0 \alpha$-total
  variation of a curve $q=(u,z):[0,T]\to \Spq$. It consists, cf.\
\eqref{Finsler-variation} below, of the $\calR$-variation of $z$ as
defined in \eqref{R-variation} plus extra contributions at jump points
$t_*\in \mathrm{J}(q)$ that may arise through rate-independent or viscous
transition costs between $\llim q{t_*}$, $q(t_*)$, and $\rlim q{t_*}$.  These
extra contributions are given by the Finsler cost \eqref{Finsler-cost}, from
which the $\calR$-variation is subtracted to avoid that it is counted twice in
the $\mename 0 \alpha$-variation.  The resulting terms are positive because we
always have
$\costq{\mename 0\alpha}{t} {(u_0,z_0)}{(u_1,z_1)} \geq \calR(z_1{-}z_0)$ since
$ \mename 0\alpha [t,q,0,q'] \geq \calR(z')$ (using
$\mredq 0\alpha t{q}{0}{q'}\geq 0$).

\begin{definition}[$\mename 0\alpha$-variations]
\label{def:m.a.Variation}
Let $q =(u,z):[0,T] \to \Spq$ with $u \in \rmB\rmV([0,T];\Spu) $ and 
$z\in \mathrm{R}([0,T];\Spz) {\cap} $ $  \rmB\rmV([0,T];\Spy) )$
be a curve with $\sup_{t\in [0,T]} \mfE (q(t)) \leq E <\infty$ and jump set
$\mathrm{J}[q]$.  For closed subintervals $[a,b]\subset [0,T]$ we  define
\begin{enumerate}
\item the   \emph{extra Viscous Jump Variation} of $q$ induced by
  $\mename 0\alpha$
  on $[a,b]$ via
  \begin{equation}
    \label{eq:def.EVJC}
    \begin{aligned}
      \mathrm{eVJV}_{\mename 0\alpha}(q;[a,b])& : = \big(
     \costq{\mename0\alpha}{a}{q(a)}{\rlim q a} -\calR(\rlim z a{-}
      z(a))\big)   \\[0.3em]
      &\ \quad + \!\! \sum_{t \in \mathrm{J}[q]\cap (a,b)} \!\!
            \big(\costq{\mename 0\alpha}{t}{\llim q t}{q(t)}
      -\calR(z(t){-} \llim zt) \\[-0.7em]
      &\ \qquad \qquad \qquad +\costq{\mename0\alpha}{t}{q(t) }{\rlim q t }
      -\calR(\rlim z t{-} z(t))\big)
      \\[0.3em]
      &\ \quad + \big(\costq{\mename0\alpha}{b}{\llim qb}{q(b)}
      -\calR(z(b){-} \llim zb)\big) \, ;
    \end{aligned}
  \end{equation}
\item the $\mename 0\alpha$-total variation
  \begin{equation}
  \label{Finsler-variation}
    \Var_{\mename 0\alpha}(q;[a,b]): =    
    \mathrm{Var}_{\calR}(z;[a,b]) + \mathrm{eVJV}_{\mename 0\alpha}(q;[a,b])\,.
  \end{equation}
\end{enumerate}
\end{definition}

\noindent With slight abuse of notation, here we will use the symbol
$ \Var_{\mename 0\alpha}$ for the $\mename 0\alpha$-total variation, although this is
not a standard form of total variation, cf.\ \cite[Rem.\,3.5]{MRS12}.
    
Just like for its parametrized counterpart, our definition of (true) $\BV$
solutions will rely on a suitable chain-rule requirement, enhancing Hypothesis
\ref{h:ch-rule} to curves $q=(u,z)$ having just a $\mathrm{BV}$-time
regularity.  For consistency, we will formulate this $\mathrm{BV}$-chain rule
as a {hypothesis}.

\begin{hypothesis}[Chain rule in BV]
\label{hyp:BV-ch-rule}
For every curve  $q=(u,z):[0,T]\to \Spq$ with $u \in \rmB\rmV([0,T];\Spu)$
and $z \in \mathrm{R}([0,T];\Spz) {\cap}$ 
$ \rmB\rmV([0,T];\Spy)$ and satisfying 
\[
 \slov ut{q(t)} + \slov zt{q(t)} = 0  \quad \text{for all } \ t \in
 [0,T]\setminus \mathrm{J}[q]
\]
the following chain-rule estimate holds, for all closed subset
$[t_0,t_1] \subset [0,T]$:
\begin{equation}
  \label{BV-ch-rule}
  \begin{aligned}
    & \text{the map  } t\mapsto \eneq t{q(t)} \text{ belongs to } 
      \rmB\rmV([0,T]) \text{ and}
    \\
    & \eneq {t_1}{q(t_1)} - \eneq {t_0}{q(t_0)} - \int_{t_0}^{t_1}\pl_t
    \eneq{s}{q(s)} \dd s \geq - \Variq{\mename 0\alpha}q{t_0}{t_1} .
  \end{aligned}
 \end{equation}
\end{hypothesis} 

In Lemma \ref{l:nice-implication} in Appendix \ref{s:app-CR} we 
to show that the parametrized chain rule from Hypothesis \ref{h:ch-rule-param}
also guarantees the validity of Hypothesis \ref{hyp:BV-ch-rule}. Hence,
subsequently we will directly assume Hypothesis \ref{h:ch-rule-param}.

Let us now give our definition of $\BV$ solutions 
$q:[0,T]\to \Spu\ti \Spz$, i.e.\ $\BV$ solutions without parametrization.  We
sometimes use the word `true $\BV$ solution' to distinguish $\BV$ solutions from
`parametrized $\BV$ solutions', hence there is no difference between $\BV$
solutions and true $\BV$ solutions. Definition \ref{def:trueBV} below is a natural extension of the concept of
$\BV$ solutions  introduced in \cite[Def.\,3.10]{MRS13}, now taking care of the
equilibrium condition \eqref{stationary-u} for $u$ corresponding to the regime $\rgs
Eu$, the local stability condition \eqref{loc-stab} for $z$ corresponding to
the regime $\rgs Rz$, and an energy-dissipation balance
\eqref{enid-strictBV}. Hence, all  jump behavior is compressed into the
definition of the Finsler cost $\costname{\mename 0\alpha}$, the total
$\mename 0\alpha$-variation, and the validity of the energy-dissipation balance.

\begin{definition}[$\BV$ solutions]
\label{def:trueBV}
Let the rate-independent system $\RIS$ fulfill Hypothesis
\ref{hyp:BV-ch-rule}. A curve
$q=(u,z):[0,T]\to \Spq$ is called a \emph{Balanced-Viscosity solution} to $\RIS$
if satisfies the following conditions:
\begin{subequations}
\label{eq:def.BVsol}
\begin{itemize}
\item[\textbullet] $u \in \mathrm{BV}([0,T];\Spu)$ and $z \in  \mathrm{R}([0,T];\Spz) \cap
                  \rmB\rmV([0,T];\Spy) $; 
\item[\textbullet] the stationary equation
  \begin{equation}
   \label{stationary-u}
   \slov ut{q(t)}=0  \quad \text{for all } t \in [0,T]\setminus \mathrm{J}[q];
  \end{equation}
\item[\textbullet] the  local stability condition 
  \begin{equation}
    \label{loc-stab}
    \slov zt{q(t)}=0  
    \quad \text{for all } t \in [0,T]\setminus \mathrm{J}[q];
  \end{equation}
\item[\textbullet] the energy-dissipation balance
  \begin{equation}
    \label{enid-strictBV}
    \eneq t{q(t)} +    \Variq{\mename 0\alpha}qst = \eneq s{q(s)}
    + \int_s^t \pl_t \eneq r{q(r)}  \dd r \quad \text{for }
    0\leq s\leq t  \leq T\,.
  \end{equation}
\end{itemize}
\end{subequations}
\end{definition}

We postpone to Section \ref{ss:5.3} a result comparing parametrized and true
$\BV$ solutions.  With the exception of our existence results Theorems
\ref{thm:exist-trueBV} and \ref{thm:exist-nonpar-enh}, in the following
statements we will omit to explicitly recall the assumptions of Section
\ref{s:setup}; we will only invoke the chain rule from Hyp.\
\ref{h:ch-rule-param}.

\Subsection{Characterization and fine properties of $\BV$ solutions}
\label{ss:5.2}

In the same way as for their parametrized version, thanks to the chain rule
\eqref{BV-ch-rule} we have a characterization of $\BV$ solutions in terms of
the upper energy estimate $\leq $ in \eqref{enid}, on the \emph{whole} interval
$[0,T]$.  We also have a second characterization in terms of  a simple
energy-dissipation balance  like for energetic solutions as in 
\cite{MieThe99MMRI,DM-Toa2002,Miel05ERIS,MieRouBOOK},   combined with jump
conditions that balance  the different dissipation mechanics that may be
active  at a jump point. The proof of Proposition \ref{prop:BV-charact}
follows, with minimal changes, from the arguments for \cite[Cor.\,3.14,
Thm.\,3.15]{MRS13}, to which the reader is referred.

\begin{proposition}
\label{prop:BV-charact}
Let the rate-independent system $\RIS$ fulfill Hypothesis
\ref{h:ch-rule-param}.  For a curve
$q=(u,z) \in \rmB\rmV([0,T];\Spu) \ti ( \mathrm{R}([0,T];\Spz) {\cap}
\rmB\rmV([0,T];\Spy) ) $ fulfilling the stationary equation \eqref{stationary-u},
and the local stability \eqref{loc-stab}, the following  three assertions are
equivalent:
\begin{enumerate}
\item $q$ is a $\BV$ solution of system $\RIS$;
\item $q$ fulfills 
  \begin{equation}
    \label{enineq-sBV0T}
    \eneq T{q(T)} +    \Var_{\mename 0\alpha}(q;[0,T]) \leq \eneq 0{q(0)} +
    \int_0^T \pl_t \eneq r{q(r)} \dd r; 
  \end{equation}
\item $q $ fulfills the $\calR$-energy-dissipation inequality
  \begin{equation}
    \label{enineq-ene-st}
    \eneq t{q(t)} +    \Var_{\calR}(q;[s,t]) \leq \eneq s{q(s)} + \int_s^t
    \pl_t \eneq r{q(r)} \dd r 
  \end{equation}   
  with $\Var_\calR$ from \eqref{R-variation}, \emph{and} the jump conditions at
  every $t\in   \mathrm{J}[q]\, :$
  \begin{equation}
    \label{jumpBV}
    \begin{aligned}
       \eneq t{\llim qt} - \eneq t{q(t)}  & = 
    \costq{\mename0\alpha}{t}{\llim qt}{q(t)},
      \\
        \eneq t{q(t)}  - \eneq t{\rlim qt}& =  
     \costq{\mename0\alpha}{t}{q(t)}{\rlim qt}.
    \end{aligned}
  \end{equation}
\end{enumerate}
\end{proposition}

Conditions \eqref{jumpBV} provide a fine description of the behavior of $\BV$
solutions $(u,z)$ at jumps. However, the $\inf$ in the definition of
$\costname{\mename 0\alpha}$ need not be attained, as the functional
$\mename 0\alpha$ does not control the norm of the space where we look for the
$\tht u$-component of admissible transition curves.  Nonetheless, in certain
situations (cf.\ the proof of Theorem \ref{th:pBV.v.BVsol} below) the existence
of transitions attaining the optimal cost will play a key role. In fact, it
will be sufficient to require the existence of these curves in cases in which
the Finsler cost equals the energy release, which happens at the jump points of
a true $\BV$ solution as in \eqref{jumpBV}. That is why, hereafter we will
refer to such transitions as \emph{optimal jump transitions}, a notion
that will be made precise in  Definition \ref{def-OT}. Therein we restrict
to transition curves, defined on $[0,1]$, connecting points $q_-=(u_-,z_-)$ and
$q_+=(u_+,z_+)$ such that the $u$-components $u_-$ and $u_+$ are at
equilibrium, and the $z$-components $z_-$ and $z_+$ are locally stable.

\begin{definition}
\label{def-OT}
Given $t\in [0,T]$ and $q_-=(u_-,z_-), \, q_+=(u_+,z_+) \in \Spq$
fulfilling $\slov ut{q_\pm} = \slov zt{q_\pm}=0$, we call an admissible curve
$\serifTeta \in  \admtcq t{q_-}{q_+} $ an \emph{optimal transition}
between $q_-$ and $ q_+ $ at time $t$ if it fulfills
\[
  \eneq t{q_-} - \eneq t{q_+} = \costq{\mename 0\alpha}{t}{q_-}{q_+} = \mename
  0\alpha [t,\serifTeta, 0, \serifTeta'] \quad \aein\, (0,1)\,. 
\]
Furthermore, we say that $\serifTeta =  (\db\serifteta u,\db\serifteta z)$ is of
\begin{itemize}
\item[\textbullet] \emph{sliding type} if
  $\slov ut{\serifTeta(r)}=\slov zt{\serifTeta(r)}=0$ for all $r\in [0,1]$;
\item[\textbullet] \emph{viscous type}
  $\slov ut{\serifTeta(r)}+\slov zt{\serifTeta(r)}>0$ for all $r\in (0,1)$.
\end{itemize}
\end{definition}

\noindent
Observe that an optimal transition of \emph{viscous} type can be governed by
viscosity either in $u$, or in $z$, or in both variables. With
the very same argument as for the proof of \cite[Prop.\,3.19]{MRS13}, to which
we refer for all details, we can also show that every optimal transition can be
decomposed in a canonical way into an (at most) countable collection of
\emph{sliding} and \emph{viscous} transitions.  We also refer to
\cite[Sec.\,2.3]{RiScVe21TSSN} for the  concept   of so-called \emph{two-speed
  solutions},  which are defined in terms of slow rate-independent parts
connected by jumps which themselves a concatenation of at most countable 
`jump resolution maps'.

\subsection{Existence of $\BV$ solutions}
\label{su:Exist.BV}
A most interesting feature of $\BV$ solutions, already observed in
\cite{MRS13}, is that it is possible to prove their existence by directly
taking the vanishing-viscosity limit of the viscous system \eqref{dne-q},
\emph{without} reparametrization.  In the following result, we take a
slightly different viewpoint and in fact prove that every limit point $q$ (in
the sense of pointwise weak convergence) of a sequence of viscous solutions
$(q_\epsk)_k = (u_\epsk,z_\epsk)_k$, starting from well-prepared initial data
and such that the $\rmB\rmV([0,T];\Spu)$-norm of $(u_\epsk)_k$ is \emph{a priori
  bounded} (cf.\ \eqref{estBV} below), is in fact a true $\BV$ solution.  In
fact, the existence of limit points can be proved, based on the energy
estimates from Lemma \ref{l:1} and on \eqref{estBV}, via a standard compactness
argument and the Helly Theorem.

The statement of Theorem \ref{thm:exist-trueBV} below mirrors that of Theorem
\ref{thm:existBV}: 
\begin{compactitem}
\item First, \eqref{estBV} corresponds exactly to the a priori estimate
  for $\|\sfu_{\epsk}'\|_{\Spu} $ in \eqref{condition-4-normali}, and to
  estimate \eqref{est2} established in Proposition \ref{l:3.2}.  Sufficient
  conditions for this estimate have been discussed in Section
  \ref{su:AprioViscSol}; alternatively, in concrete examples this estimate
  could be verified by direct calculations.
\item Secondly, in the same way as with \eqref{eq:E.M.cvg.PBV} for parametrized
  solutions, with \eqref{cvs-BV-b}--\eqref{cvs-BV-c} ahead we are stating the
  convergence of the left-hand side terms in the viscous energy-dissipation
  estimate \eqref{enid-ineq} - in particular, \eqref{cvs-BV-c} ensures the
  convergence for $\eps_k \to 0^+$  of
\begin{equation}
\label{expl-me-1}
\begin{aligned}
  & \int_s^t \meq \epsk{\alpha}{r}{q_\epsk(r)}1{q_\epsk'(r)} \dd r
  \\
  & = \int_s^t \left( \disve {u}{\eps_k^\alpha}( u_\epsk'(r)) {+}
    \calR(z_\epsk'(r)) {+} \disve z\epsk( z_\epsk'(r)) {+} \frac{\slov
      ur{q_\epsk(r)}}{\epsk^\alpha} {+} \frac{\slov zr{q_\epsk(r)}}\epsk
  \right) \dd r
\end{aligned}
\end{equation}
to the corresponding terms in the energy-dissipation balance
\eqref{enid-strictBV}. We emphasize here that, for \eqref{cvs-BV-c} to hold it
is crucial that the definition of the total variation functional
$\Var_{\mename 0\alpha}$, in the general closed subinterval
$[s,t] \subset [0,T]$, takes into account the appropriate contributions at the
jump points. In particular, we point out that, by \eqref{eq:def.EVJC}, also the
jumps occurring at the extrema $s$ and $t$ are  taken into account exactly,
in the sense that $\costq{\mename0\alpha} s {q(s)}{q(s^+)}=\lim_{\sigma \to
  s^+} \lim_{\eps_k\to 0} \int_0^\sigma \mename{\eps_k}\alpha(\cdot) \dd r $. 
\end{compactitem}

\begin{theorem}[Convergence to $\BV$ solutions]
\label{thm:exist-trueBV} 
Let the rate-independent system $\RIS$ fulfill Hypotheses \ref{hyp:setup},
\ref{hyp:diss-basic}, \ref{hyp:1}, \ref{h:closedness}, \ref{hyp:Sept19}, and
\ref{h:ch-rule-param}. For any null sequence $(\eps_k)_k$ let
$(q_\epsk)_k = (u_{\eps_k},z_{\eps_k})_k \subset \AC ([0,T]; \Spq)$ be a
sequence of solutions to the generalized gradient system \eqref{dne-q}, such
that convergences \eqref{init-data-cv} to a pair $(u_0,z_0)\in \domq$ hold at
the initial time $t=0$, and such that, in addition,
\begin{equation}
  \label{estBV}
\widehat S =  \sup_k \| u_\epsk\|_{\rmB\rmV([0,T];\Spu)} <\infty.
\end{equation}
Let $q:[0,T] \to \Spq$ be such that, along a not relabeled subsequence, there
holds as $k\to\infty$
\begin{equation}
\label{ptw-weak-limit}
q_\epsk(t) \weakto q(t) \quad \text{in } \Spq \quad \text{for all } t \in [0,T]
\end{equation} 
(every sequence in the above conditions possesses at least one limit point in the
sense of \eqref{ptw-weak-limit}).  Then,
\begin{enumerate}
\item $q=(u,z) \in \rmB\rmV([0,T];\Spu) \ti (\Reg 0T{\Spz} \cap \BV ([0,T];\Spy))$,
  and $q$ is a true $\BV$ solution to the rate-independent system $\RIS$;
\item there hold the additional convergences as $k\to\infty$
  \begin{subequations}
    \label{cvs-BV}
    \begin{align}
      &
      \label{cvs-BV-a}
      u_\epsk(t)\weakto u(t) \text{ in } \Spw, \qquad z_\epsk(t)\weakto z(t)
      \text{ in } \Spx  && \text{for all } t \in [0,T],
      \\
    \label{cvs-BV-b}
    &
    \eneq t{q_\epsk(t)} \to \eneq t{q(t)}  \quad \text{for all } t \in [0,T],
    \\
    \label{cvs-BV-c}
    & \lim_{k\to\infty}  \int_s^t    
       \meq \epsk{\alpha}{r}{q_\epsk(r)}1{q_\epsk'(r)}  \dd r 
    = \Var_{\mename 0\alpha}(q;[s,t])  
    && \text{for all  $0\leq s
      \leq t \leq T$.} 
    \end{align}
  \end{subequations}
\end{enumerate}
\end{theorem}

\noindent The {proof} will be carried out in Section \ref{ss:8.2}.

\begin{remark}[Vanishing-viscosity approximation versus $\BV$ solutions]
\label{rm:VVAvsBV} We emphasize that the concept of $\BV$ solutions  enjoys    better
closedness properties than defining solutions simply as all the limiting points
in the vanishing-viscosity approximation. Such solutions are called
`\emph{approximable}' in \cite{Miel11DEMF} and there, in  Examples 2.5 and 2.6,
it is shown in a simple model with $\Spz=\R$ that there  are  more $\BV$ solutions
than approximable solutions.  It is also made apparent that,    for systems with $\delta$-dependent energy
$\calE_\delta$,  approximable solutions $q^\delta:[0,T]\to \R$ may have a
limit $q^{\delta_*} $for $\delta\to \delta_*$ that is no longer an approximable
solution, but $q^{\delta_*}$ is still a $\BV$ solution. Thus, $\BV$ solutions
seem to have better stability properties, see e.g.\ \cite[Thm.\,4.8]{MRS2013}.    
\end{remark}

\begin{remark}[Existence of $\BV$ solutions by time discretization]
\label{rm:BVSol.viaTD}\slshape
Another interesting features of true $\BV$ solutions is that they
can be obtained as limits of discrete solutions of the time-incremental scheme
\begin{equation}
\label{tim-q}
q_{\tau,\eps}^n \in
\mathop{\mathrm{Argmin}}\limits_{q\in \mathbf{Q}} 
{\Big\{\tau  \Psi_{\eps,\alpha}\Big(\frac{q{-}q_{\tau,\eps}^{n-1}}{\tau} \Big)
  + \calE(t_\tau^n,q)\Big\}}, \quad n =1,\,\ldots \,, N_\tau
\end{equation}
with $\Psi_{\eps,\alpha}$ from \eqref{eq:def.Psi.e.a}, as the viscosity
parameter $\eps$ \emph{and} the time-step $\tau$ jointly tend to $0$. (Of
course, \emph{fixing} $\eps>0$ and letting $\tau\to 0^+$ in \eqref{tim-q} gives
rise to solutions  $q_\eps:[0,T]\to \Spq$  of the generalized gradient
system \eqref{GGS-structure}).  This alternative construction of $\BV$
solutions in the joint discrete-to-continuous and vanishing-viscosity limit of
the time-incremental scheme for viscous solutions was carefully explored in
 \cite[Thm.\,4.10]{MRS12} and \cite[Thm.\,3.12]{MRS13}. Following these
lines it is possible to show convergence to $\BV$ solutions along (a
subsequence of) any sequence $(\tau_k,\eps_k)$ as long as $\tau_k$ tends to $0$
faster than the time scales in our system, i.e.\ 
\begin{equation}
   \label{cond-alpha-eps/tau} 
      \lim_{k\to\infty} \frac{\tauk}{\min\{ \eps_k^\alpha,\epsk\}}=0\,.
\end{equation}
To avoid overburdening of the exposition here we refrain from  giving 
a precise convergence statement, but refer  to \cite[Thm.\
3.12]{MRS13}, which can adapted to our setup using
condition \eqref{cond-alpha-eps/tau}. The same applies to the convergence of
time-discrete solutions to {enhanced} $\BV$ solutions  introduced below. 
\end{remark}

\Subsection{Enhanced $\BV$ solutions}
\label{ss:5.3}   

This solution concept is to be compared with the notion introduced in
\cite[Def.\ 3.21]{MRS13} and, of course, with enhanced $\pBV$ solutions. In
particular, recall that for an enhanced $\pBV$ solution
$(\sft,\sfq) = (\sft, \sfu,\sfz)$ we required the additional regularity
$\sfz \in \AC ([0,\sfS];\Spz)$.  Accordingly, an \emph{enhanced $\BV$ solution}
$q=(u,z)$ is required to fulfill $z\in \rmB\rmV([0,T];\Spz)$. Moreover,
enhanced $\BV$ solutions enjoy the additional regularity property that at all
jump points the left and right limits are connected by optimal transitions with
finite length in $\Spu{\ti}\Spz$, such that the total length of the connecting
paths $\teta =(\tht u,\tht z)$ is finite. In contrast, for general $\BV$
solutions it is only required that length of the $\tht u$-component of
an optimal jump transition is finite in $\Spu$.
    
\begin{definition}[Enhanced $\BV$ solutions]
\label{def:enhanced-strict-BVsols}
A curve $q:(u,z):[0,T]\to \Spq$ is called an \emph{enhanced  $\BV$ solution} of
$\RIS$, if  it is a $\BV$ solution and it satisfies the following additional
properties: 
\begin{itemize}
\item[\emph{(i)}] $q\in \rmB\rmV([0,T];\Spq) $; 
\item[\emph{(ii)}] for all $t \in \mathrm{J}[q]$ there exists an optimal jump
  transition $\teta^t = (\tht u^t,\tht z^t) \in \admtcq t{\llim qt}{\rlim qt}
  $ such that $\teta^t \in \AC ([0,1];\Spq)$ and $q(t)=\teta^t(\hat
  r_t)$ for some $\hat r_t \in [0,1]$; 
\item[\emph{(iii)}]
  $\sum_{t\in \mathrm{J}[q]} \int_0^1 \|(\teta^t)'(r)\|_{\Spq} \dd r = \sum_{t\in
    \mathrm{J}[q]} \int_0^1 \left(\|(\tht u^t)'(r)\|_{\Spu} {+} \|(\tht
    z^t)'(r)\|_{\Spz} \right) \dd r <\infty$.
\end{itemize}
\end{definition}

Our existence result for enhanced $\BV$ solutions can be again proved by our
vanishing-viscosity approach without reparametrizing the trajectories, by
taking the vanishing-viscosity limit of viscous solutions that satisfy an
additional estimate on $\sup_{k \in \N} \| z_\epsk\|_{\rmB\rmV([0,T];\Spz)} $.

\begin{theorem}[Convergence of viscous solutions to enhanced $\BV$ solutions]
\label{thm:exist-nonpar-enh}
Assume Hypotheses \ref{hyp:setup}, \ref{hyp:diss-basic}, \ref{hyp:1},
\ref{h:closedness}, \ref{hyp:Sept19}, and \ref{h:ch-rule-param}.  Let
$(q_\epsk)_k 
\subset \AC ([0,T]; \Spq)$ be a sequence of solutions to the generalized
gradient system \eqref{van-visc-intro} such that convergences
\eqref{init-data-cv} hold at $t=0$, as well as
\begin{equation}
  \label{estBV-u}
  \exists\, S>0 \ \forall\, k \in \N \,: \qquad \| q_\epsk\|_{\rmB\rmV([0,T];\Spq)}
  \leq \widehat S.
\end{equation}
Let $q:[0,T]\to \Spq$ be a limit point for $(q_\epsk)_k$ in the sense of
\eqref{ptw-weak-limit}. Then, $q$ is an \emph{enhanced $\BV$ solution} of
$\RIS$, and  the additional convergences  \eqref{cvs-BV} hold.
\end{theorem}

\noindent Since the proof of Theorem \ref{thm:exist-nonpar-enh} follows from
combining the argument for Theorem \ref{thm:exist-trueBV} with that developed
for \cite[Thm.\,3.22]{MRS13}, it is omitted.

\Subsection{Comparing $\pBV$ and true $\BV$ solutions}
\label{su:Comp.pBV.BV}
In this final subsection we explore the relations 
between parametrized and true $\BV$ solutions, also in the enhanced case. 
Indeed, there is a very natural transition between parametrized and true $\BV$
solutions. The converse passage will be obtained by `filling the graph' of a
true $\BV$ solution at its jump points, by means of an optimal jump transition,
under the \emph{additional} assumption that it exists.  This condition is
codified in the following 

\begin{hypothesis}
\label{hyp:OJT} 
For every $t\in [0,T]$ and $q^-, q^+\in \Spq$ such that
$\slov ut{q_\pm} = \slov zt{q_\pm}=0$ and
\[
  \eneq t{q^-} - \eneq t{q^+} = \costq {\mename 0\alpha} t{q^-}{q^+} 
\]
there exists an \emph{optimal jump transition}
$\serifTeta^{\mathrm{opt}} \in \admtcq t{q^-}{q^+}$.
\end{hypothesis}

\begin{remark}
  Let us emphasize that Hypothesis \ref{hyp:OJT} plays no role in proving the
  existence of $\BV$ solutions. It only serves the purpose of showing that a
  true $\BV$ solution gives rise to a parametrized one. In this connection, let
  us mention in advance that, in the statement of Theorem \ref{th:pBV.v.BVsol},
  Hypothesis \ref{hyp:OJT} will not be required for relating \emph{enhanced}
  $\BV$ solutions to their parametrized analogues, as the definition of
  enhanced $\BV$ solutions already encompasses the information that optimal
  jump transitions exist.
\end{remark}

We are now ready to state the following  relations between true and
parametrized $\BV$ solutions. 

\begin{theorem}[$\pBV$ versus true $\BV$ solutions]
\label{th:pBV.v.BVsol}
Let $\RIS$ fulfill Hypothesis \ref{h:ch-rule-param}. \newline 
Then the following statements are true:
\begin{enumerate}
\item If $(\sft,\sfq):[0,\mathsf{S}]\to [0,T]\ti \Spq$ is a non-degenerate
  $\pBV$ solution of $\RIS$ with $\sft(0)=0$ and $\sft(\mathsf{S})=T$, then
  every $q:[0,T]\to \Spq$ satisfying
  \begin{equation}
    \label{eq:Project.pBV}
    q(t) \in \bigset{\sfq(s) }{ \sft(s)=t}
  \end{equation}
  is a (true) $\BV$ solution that enjoys, moreover, the following property: for
  every $t\in \mathrm{J}[q]$ there exists an optimal jump transition
  $\tetaopt \in \admtcq t{\llim qt}{\rlim qt}$ such that
  $q(t)=\tetaopt(\hat r)$ for some $\hat r \in [0,1]$.  
  Furthermore, 
   there holds
  \begin{equation}
    \label{equality-between-variations}
    \Variq{\mename 0\alpha}{q}{t_0}{t_1}
    = \int_{\sfs(t_0)}^{\sfs(t_1)} \mename 0\alpha [\sft, \sfq, \sft', \sfq'](s) \dd s 
     \quad \text{for all  } 0 \leq t_0\leq t_1\leq T\,.  
   \end{equation} 

\item Conversely, assume additionally Hypothesis \ref{hyp:OJT}. Then,
  for every $\BV$ solution $q:[0,T]\to \Spq$, there exists a non-degenerate,
  surjective $\pBV$ solution
  $(\sft,\sfq) \in \mathscr{A}([0,\mathsf{S}];[0,T]\ti \Spq)$ such that
  \eqref{eq:Project.pBV} and \eqref{equality-between-variations} hold. 

\item If $(\sft,\sfq):[0,\mathsf{S}]\to [0,T]\ti \Spq$ is a (non-degenerate)
  enhanced $\pBV$ solution with $\sft(0)=0$ and $\sft(\mathsf{S})=T$, then
  every $q:[0,T]\to \Spq$ given by \eqref{eq:Project.pBV} is an enhanced $\BV$
  solution, and \eqref{equality-between-variations} holds. 

\item Conversely, for any enhanced $\BV$ solution $q:[0,T]\to \Spq$, there
  exists a (non-degenerate, surjective) enhanced $\pBV$ solution
  $(\sft,\sfq) \in \mathscr{A}([0,\mathsf{S}];[0,T]\ti \Spq)$ such that
  \eqref{eq:Project.pBV} and \eqref{equality-between-variations} hold. 
\end{enumerate}
\end{theorem}

\begin{remark}[Greater generality of true $\BV$ solutions]\slshape 
\label{rm:GeneralityBV}
Theorem \ref{th:pBV.v.BVsol} seems to suggest that true $\BV$ solutions are
more general than their parametrized analogues. Indeed, while, under the
standing assumptions of Section \ref{s:setup}, parametrized solutions always give
rise to true $\BV$ ones, the converse passage is possible under the additional
Hypothesis \ref{hyp:OJT}. Hence, the set of true $\BV$ solutions is apparently
bigger.

To emphasize this, we have chosen to prove that any limit curve $q$ for a
sequence $(q_{\epsk})_k$ of (non-parametrized) viscous solutions is a true
$\BV$ solution, as stated in Theorem \ref{thm:exist-trueBV}, by resorting to
Theorem \ref{thm:existBV} for parametrized solutions. Namely, in Sec.\
\ref{ss:8.2} we will use that the graphs of a sequence $(q_{\epsk})_k$ of
viscous solutions are contained in the image sets of their parametrized
counterparts $(\sft_{\epsk}, \sfq_{\epsk})_k$ and apply Theorem \ref{thm:existBV}
to the latter curves, guaranteeing their convergence to a $\pBV$ solution
$(\sft,\sfq)$.  We will then proceed to showing that $q$ and $(\sft,\sfq)$
are related by \eqref{eq:Project.pBV} and thus conclude, by Thm.\
\ref{th:pBV.v.BVsol}(1), that $q$ is a true $\BV$ solution.
\end{remark}

\begin{proof}
\STEP{1: From $\pBV$ to $\BV$ solutions.} First, we show that, given a
$\pBV$ solution $(\sft,\sfq) = (\sft,\sfu,\sfz)$, formula
\eqref{eq:Project.pBV} defines a curve
$q=(u,z) \in \rmB\rmV([0,T];\Spu) \ti (\Reg 0T\Spz {\cap} \rmB\rmV([0,T];\Spy)$. Indeed,
let $\sfs:[0,T]\to [0,\mathsf{\mathsf{S}}] $ be any inverse of $\sft$, with
jump set $ \mathrm{J}[\sfs]$. It can be easily checked that, since
$(\sft,\sfq) = (\sft,\sfu,\sfz) $ is non-degenerate,
\[
  t\in \mathrm{J}[q] = \mathrm{J}[u] \cup \mathrm{J}[z] \qquad
  \Longleftrightarrow \qquad  
  t \in \mathrm{J}[\sfs] \text{ and } \sft(s) \equiv t \text{ for
    all } s \in [\llim \sfs t, \rlim \sfs t]\,.
\]
If  for $t\in \mathrm{J}[\sfs]$  we have $q(t) = \sfq(s_*)$ for
some $s_*\in [\llim \sfs t, \rlim \sfs t]$, then defining $\sfs(t): = s_*$
gives the identity 
\begin{equation}
\label{not-clear-use}
 q(t) = (u(t), z(t))  = \sfq(\sfs(t)) =( \sfu(\sfs(t)), \sfz(\sfs(t))) \qquad
 \text{for all } t\in [0,T].  
\end{equation}
From this, we deduce $u\in \rmB\rmV([0,T];\Spu)$ and
$z\in \rmB\rmV([0,T];\Spy)$.  Moreover, since
$\sup_{t\in [0,T]} \mfE(q(t))\leq E$ for some $E>0$ and the functional
$\mfE + \| \cdot\|_\Spu + \| \cdot\|_{\Spy} $ has sublevels bounded in
$\Spw\ti \Spx$, we also have $z\in \rmL^\infty(0,T;\Spx)$, which gives
$z\in \Reg 0T{\Spz}$ thanks to \eqref{4later-use-embed}.

From \eqref{not-clear-use} we easily deduce that
\begin{equation}
\label{ingred-stim-var-1}
 \mathrm{Var}_{\calR}(z;[t_0,t_1]) =  \int_{\sfs(t_0)}^{\sfs(t_1)}
 \calR[\sfu'](s) \dd s \qquad \text{for all } 0\leq t_0\leq t_1 \leq T\,. 
\end{equation}
Furthermore, we mimic the argument from the proof of \cite[Prop.\,4.7]{MRS13}
and observe that for every $t\in  \mathrm{J}[q]$ the curve
$\sfq = (\sfu,\sfz): [\llim \sfs t, \rlim \sfs t] \to \Spu \ti \Spz$,
reparametrized in such a way that it is defined on the interval $[0,1]$, is an
admissible transition curve between $ \llim qt$ and $\rlim qt$.  Hence,
\begin{equation*}
      \costq{\mename 0\alpha}{t}{\llim qt}{q(t)} \leq \int_{\llim \sfs
       t}^{\sfs (t)} \!\!\mename 0\alpha [\sft,\sfq, 0,  \sfq'] (s) \dd
     s,\quad 
      \costq{\mename 0{\alpha\!}}{t}{q(t)}{\rlim qt} 
      \leq \int_{\sfs (t)}^{\rlim \sfs t} \!\!
     \mename 0\alpha [\sft,\sfq,  0,  \sfq'](s) \dd s\,.
\end{equation*}
Combining this with \eqref{ingred-stim-var-1} we conclude that
\begin{equation}
  \label{est-variations}
  \Variq{\mename 0\alpha}{q}{t_0}{t_1}
  \leq  \int_{\sfs(t_0)}^{\sfs(t_1)} 
  \mename 0\alpha [\sft, \sfq, \sft', \sfq'](s)    \dd s  
\end{equation}
for all $[t_0,t_1]\subset [0,T]$. Ultimately, we infer that $q$ fulfills the
energy-dissipation estimate \eqref{enineq-sBV0T}.
  
In order to show that $q$ complies with the stationary equation
\eqref{stationary-u} and the local  stability condition \eqref{loc-stab}, we
argue in the following way.  Recalling the definition of the sets
$\mathsf{\mathscr{G}}^\alpha$ from \eqref{setGalpha}, we introduce
\[
  \mathscr{H}^\alpha[q] : = 
    \begin{cases}
      \{ t \in [0,T]\, : \ 
      \slov ut{q(t)} = \slov zt{q(t)} =0 \}  & \text{if } \alpha \geq 1, 
      \\
      \{ t \in [0,T]\, : \ 
      \slov zt{q(t)} =0 \}     & \text{if } \alpha \in (0,1).
    \end{cases}
\]
Observe that the set $ \mathscr{H}^\alpha[q]$ is dense in $[0,T]$. Indeed, its
complement
$[0,T]\setminus \mathscr{H}^\alpha[q] = \sft(\SetG \alpha{\sft}{\sfq})$ has
null Lebesgue measure, since $\sft$ is constant  on each connected
component of the open set  $\SetG \alpha{\sft}{\sfq}$. Therefore, by the
lower semicontinuity properties of $\slovname u$ and $\slovname z$ ensured by
Hypothesis \ref{hyp:Sept19}, in the case $\alpha\geq 1$ we immediately conclude
\eqref{stationary-u} and \eqref{loc-stab}.  For $\alpha \in (0,1)$, the above
argument only yields \eqref{loc-stab}, and for the validity of
\eqref{stationary-u}, we observe that for any $t\notin \mathrm{J}[q] $, then
\[
t = \sft (\bar s)  \text{ and } q = \sfq (\bar s) \quad \text{for } \bar s \in \overline{\{ s \in [0,\mathsf{S}]\, : \ \sft'(s)>0  \}}\,.
\]
Then, since $\slov u{\sft}{\sfq} \equiv 0$ on the set
$\{ s \in (0,\mathsf{S})\, : \ \sft'(s) >0 \} $ as prescribed by Definition
\ref{def:adm-p-c}, we conclude that $\slov u{t}{q(t)}=0$.

Since $q$ complies with \eqref{stationary-u}, \eqref{loc-stab}, and
\eqref{enineq-sBV0T}, by Proposition \ref{prop:BV-charact} we conclude that it
is a true $\BV$ solution. In order to conclude
\eqref{equality-between-variations}, we observe that, for all
$0\leq t_0\leq t_1\leq T$ and $ s_0\leq s_1 \in [0,\sfS]$ such that
$\sft(s_i) = t_i$ for $i\in \{0,1\}$, there holds
\begin{equation}
\label{for-eq-var}
\begin{aligned}
  & \Variq{\mename 0\alpha}{q}{t_0}{t_1} \overset{\eqref{enid-strictBV}}{=}
  \eneq {t_0}{q(t_0)} - \eneq {t_1}{q(t_1)} +\int_{t_0}^{t_1} \pl_t \eneq
  {r}{q(r)} \dd r
  \\
  & = \eneq {\sft(s_0)}{\sfq(s_0)}- \eneq {\sft(s_1)}{\sfq(s_1)} +
  \int_{s_0}^{s_1} \pl_t \eneq {\sft(s)}{\sfq(s)} \sft'(s) \dd s
  \overset{\eqref{def-parBV}}{=} \int_{s_0}^{s_1} \mathfrak{M}_0^\alpha
  [\sft,\sfq, \sft',\sfq'](s) \dd s\,.
  \end{aligned}
\end{equation}

It is immediate to see that the above arguments also yield an enhanced $\BV$
solution from any enhanced $\pBV$ solution.  Hence, assertions (1) and (3)
are proved. 
\smallskip

\STEP{2: From $\BV$ to $\pBV$ solutions} First of all, we show that, under the
additional Hypothesis \ref{hyp:OJT}, with any true $\BV$ solution
$q\in \rmB\rmV([0,T];\Spu) \ti (\Reg 0T\Spz {\cap} \rmB\rmV([0,T];\Spy))$ we
can associate a non-degenerate, surjective curve
$(\sft,\sfq) = (\sft,\sfu,\sfz) \in \mathscr{A}([0,\mathsf{S}];[0,T]\ti \Spq)$
such that \eqref{eq:Project.pBV} holds and
\begin{equation}
  \label{balance-variations}
  \Variq{\mename 0\alpha}{q}0{T}  =   \int_{0}^{\mathsf{\mathsf{S}}} 
  \mename 0\alpha [\sft,\sfq,\sft',\sfq'](s)   \dd s\,.  
\end{equation}
Indeed, along the lines of \cite[Prop.\,4.7]{MRS13} we introduce the
parametrization $\sfs$, defined on $[0,T]$ by
\[
  \begin{aligned}
    & \sfs(t): = t+  \Variq{\mename 0\alpha}{q}{0}{t}, \qquad \sfS:
    = \sfs(T) \quad \text{with } \\ & \mathrm{J}[\sfs]= \mathrm{J}[u] \cup
    \mathrm{J}[z] =(t_m)_{m\in M} \text{ and $M$ a countable set}.
  \end{aligned}
\]
We set $I := \cup_{m\in M} I_m$ with $I_m = (\sfs(t_{m}^-), \sfs(t_{m}^+) )$.
Hence, we define $(\sft,\sfq) = (\sft,\sfu,\sfz)$ on $[0,\sfS]\setminus I$ by
$ \sft: = \sfs^{-1} : [0,\sfS]\setminus I \to [0,T] $ and
$ \sfq : = q {\circ} \sft $.  In order to extend $\sft$ and $\sfq$ to $I$, we
need to use the fact that, by Hypothesis \ref{hyp:OJT}, for every $m \in M$
there exists an optimal jump transition jump transition
$\tetaopt_m \in \admtcq t{\llim q{t_m}}{\rlim q{t_m}}$, defined on the
canonical interval $[0,1]$ and such that $\tetaopt_m(\hat{r}_m) = q(t_m)$ for
some $\hat r_m \in [0,1]$. We may then define $\sft$ and $\sfq$ on
$I = \cup_{m\in M} I_m$ by
\[
  \sft(s) \equiv t_m, \qquad \sfq(s) : = \tetaopt_m (\sfr_m(s))  \text{ for } s
  \in I_m, \quad  \text{where } \sfr_m(s) = 
  \tfrac{s\ -\ \sfs(t_m^-) }{\sfs(t_m^+)-\sfs(t_m^-)}\,.
\]
It can be easily checked that
$(\sft,\sfq) \in \mathscr{A}([0,\mathsf{S}]; [0,T]\ti \Spq)$.  By construction,
the curves $q$ and $(\sft,\sfq)$ satisfy \eqref{eq:Project.pBV}. Furthermore,
recalling \eqref{ingred-stim-var-1} and the fact that
$\tetaopt_m \in \admtcq t{\llim q{t_m}}{\rlim q{t_m}}$, it is not difficult to
check that \eqref{balance-variations} holds. Therefore, since $q$ is a $\BV$
solution, we infer that $(\sft,\sfq) $ is a $\pBV$ solution, and we obtain
\eqref{equality-between-variations} by repeating the argument in
\eqref{for-eq-var}.

This argument also allows us to prove that any enhanced $\BV$ solution gives
rise to an enhanced $\pBV$ solution.  Hence, the proof of Theorem
\ref{th:pBV.v.BVsol} is finished. 
\end{proof}

\Section{Proof of  major results} 
\label{s:8}

This section focuses on the proofs of our main existence results for $\pBV$ and
true $\BV$ solutions, i.e.\ Theorems \ref{thm:existBV} and
\ref{thm:exist-trueBV}.  They will be carried out in Sections \ref{ss:8.1} and
\ref{ss:8.2}, respectively.  Moreover, Section \ref{su:pr:char-Ctc-set}
provides the proof of Proposition \ref{pr:charact-Ctc-set}. 

Throughout this section and, in particular, in the statement of the various
auxiliary results, we will always tacitly assume the validity of Hypotheses
\ref{hyp:setup}, \ref{hyp:diss-basic}, \ref{hyp:1}, \ref{h:closedness},
\ref{hyp:Sept19}, and of the parametrized chain rule from Hyp.\
\ref{h:ch-rule-param}: recall that, by Lemma \ref{l:nice-implication} it
implies the $\BV$-chain rule \eqref{hyp:BV-ch-rule}.

\Subsection{Proof of Theorem  \ref{thm:existBV}}
\label{ss:8.1}
Our first result lays the ground for the vanishing-viscosity analysis of
Theorem \ref{thm:existBV} by settling the compactness properties of a sequence
of parametrized curves enjoying the a priori estimates
\eqref{condition-4-normali}.  We have chosen to extrapolate such properties
from the proof of Theorem \ref{thm:existBV}, since we believe them to be of
independent interest.

Prior to stating Proposition \ref{prop:compactness-param}, let us specify the
meaning of the third convergence in \eqref{compactness-7-2} below. Indeed, the
sequence $(\sfu_k)_k$ is contained in a closed ball
$\overline{B}_R \subset \Spu$ by virtue of estimate
\eqref{bounds-rescaled-curves} (cf.\ Hypothesis \ref{h:1}). Now, since $\Spu$
is reflexive and separable, it is possible to introduce a distance
$d_{\mathrm{weak}}$ inducing the weak topology on $\overline{B}_R$.  Hence,
convergence in $\rmC^0 ([0,\sfS]; \Spu_{\mathrm{weak}})$ means convergence
in $\rmC^0 ([0,\sfS]; (\Spu, d_{\mathrm{weak}}))$.
%
\begin{proposition}
\label{prop:compactness-param}
Let $(\sft_k, \sfq_k)_k \subset \AC([ 0,\sfS]; [0,T]\ti \Spq)$, with $\sft_k$
non-decreasing and $\sfq_k =  (\sfu_k,  \sfz_k)$, enjoy the following
bounds, along a null sequence $(\eps_k)_k$:
\begin{equation}
\label{bounds-rescaled-curves}
\exists\,  C_*\geq 1  \ \ \forall\, k \in \N \, : \quad 
\begin{cases}
  & \hspace{-0.1cm} \sup_{s\in [0,\mathsf{S}]} \mfE(\sfq_k(s)) \leq C_*,
  \smallskip
  \\
  & \hspace{-0.1cm} \sft_k'(s) + \calR (\sfz_k'(s)) + \mredq {\epsk} \alpha
  {\sft_k(s)}{\sfq_k(s)} {\sft_k'(s)}{\sfq_k'(s)} \\
  & 
  \qquad \qquad \qquad  \qquad +
  \|\sfu_k'(s)\|_{\Spu} \leq C_*
 \  \foraa\, s \in 
  (0,\mathsf{S}).
  \end{cases}
\end{equation}
Then, there exist an admissible parametrized curve
$(\sft, \sfq)=(\sft,\sfu, \sfz) \in \mathscr{A} ([0,\sfS]; [0,T]\ti \Spq)$ with
\begin{subequations}
  \label{compactness-7}
  \begin{align}
    \label{compactness-7-1}
    & \begin{aligned}
      &
      \sft \in \rmC_{\mathrm{lip}}^0 ([0,\sfS]; [0,T]), 
      \quad   \sfu \in \rmC_{\mathrm{lip}}^0 ([0,\sfS];\Spu), \ 
      \text{ and }  \sfz  \in 
      \rmC_{\mathrm{lip}}^0 ([0,\sfS];\Spy) \cap  \rmC^0([0,\sfS];\Spz),
    \end{aligned}
    \intertext{and a (not relabeled) subsequence such that  the following
      convergences  hold as $k\to\infty$:} 
    \label{compactness-7-2}
    &
    \left\{{\renewcommand{\arraystretch}{1.15}
      \begin{array}{ll}
        \sft_{k} \to \sft \text{ in }   \rmC^0([0,\sfS])  \quad \text{and}
         \quad \sft'_{k} \weaksto \sft' \text{ in }   \rmL^\infty(0,\sfS), 
        \\
        \sfu_k \weaksto \sfu \text{ in }  W^{1,\infty} (0,\sfS;\Spu),
        \\
        \sfu_k \to \sfu \text{ in }  \rmC^0 ([0,\sfS];\Spu_{\mathrm{weak}})  
        \\
        \sfz_k \to \sfz \text{ in }  \rmC^{0} ([0,\sfS];\Spz),
        \\
        \sfu_k(s) \weakto  \sfu(s) \text{ in } \Spw  \text{ and }  
        \sfz_k(s) \weakto  \sfz(s) \text{ in } \Spx  \quad \text{ for all } s \in
        [0,\sfS],
      \end{array}
     }\right.
\\
    \label{compactness-7-3}
    &
    \int_{0}^{\mathsf{S}}  \mathfrak{M}_{0}^\alpha [\sft,\sfq,\sft',\sfq'](\sigma)  \dd\sigma
    \leq \liminf_{k\to \infty}
    \int_{0}^{\sfS} \mathfrak{M}_{\eps_k}^\alpha (\sft_{k}(\sigma)
    ,\sfq_{k}(\sigma),\sft'_{k}(\sigma), \sfq'_{k}(\sigma)) \dd \sigma\,.
  \end{align}
\end{subequations}
\end{proposition}
\begin{proof}
We split the proof in  three steps. 

\STEP{1. Compactness:}
From \eqref{bounds-rescaled-curves} we infer the following compactness
information.

(1.A) By the Ascoli-Arzel\`a Theorem, there exists a non-decreasing
$\sft \in W^{1,\infty}(0,\mathsf{S})$ such that $\sft_k\to \sft$ uniformly in
$[0,\mathsf{S}]$ and weakly$^*$ in $ W^{1,\infty}(0,\mathsf{S})$.

(1.B) Since the sequence $(\sfu_k)_k$ is bounded in
$W^{1,\infty}(0,\mathsf{S};\Spu)$ we conclude that there exists $ \sfu $ with
the regularity from \eqref{compactness-7-1} such that, along a not relabeled
subsequence, the second convergence in \eqref{compactness-7-2} hold for
$(\sfu_k)_k$. The convergence in $\rmC^0([0,\sfS];\Spu_{\mathrm{weak}})$
follows from an Ascoli-Arzel\`a type theorem, see e.g.\ \cite[Prop.\,3.3.1]{AGS08}).

(1.C) From  $\sup_{s\in [0,\mathsf{S}]} \mfE(\sfq_k(s)) \leq C$ we deduce that
there exists a ball 
\begin{subequations}
  \label{compactness-Arzela}
\begin{equation}
  \label{comp-ingr-1}
  \text{$\overline{B}_M^{\Spx}\subset \Spx \Subset \Spz$ such that $\sfz_k(s) \in
    \overline{B}_M^{\Spx}$ for all $s\in [0,\mathsf{S}]$ and all $k\in \N$.} 
\end{equation}
 Using $\Spx\Subset \Spz\subset \Spy$ and the coercivity \eqref{R-coerc} of
$\calR$, Ehrling's lemma gives that 
\[
  \forall\, \omega>0 \ \exists\, C_\omega>0 \ \forall\, z\in
  \overline{B}_M^{\Spx} \quad \| z\|_{\Spz} \leq \omega +C_\omega \calR(z).
\]
Hence, defining $\Omega_M(r): = \inf_{\omega>0} ( \omega {+}C_\omega r)$
and noting that $\Omega_M(\lambda r) \leq \lambda \Omega_M(r)$ for all
 $\lambda \geq 1$,
 we find
\begin{equation}
  \label{comp-ingr-2}
  \| \sfz_k(s_1) {-}  \sfz_k(s_2)    \|_{\Spz} \leq \Omega_M(\calR(\sfz(s_1)
  {-}  \sfz_k(s_2) )) \leq  C_*  \Omega_M(|s_1{-}s_2|)
\text{  for all } 0 \leq s_1\leq s_2 \leq \mathsf{S},
\end{equation}
\end{subequations}
where the last estimate follows from the bound for $\calR(\sfz'_k)$ in
\eqref{bounds-rescaled-curves}.  We combine the compactness information
provided by \eqref{comp-ingr-1} with the equicontinuity estimate
\eqref{comp-ingr-2} and again apply, \cite[Prop.\,3.3.1]{AGS08} to deduce that
there exists $\sfz \in \rmC^0([0,\mathsf{S}];\Spz)$ such that, along a not
relabeled subsequence, $(\sfz_k)_k$ converges to $\sfz$ in the sense of
\eqref{compactness-7-2}.

Let us denote by $\sfq$ the curve $ (\sfu, \sfz)$.

\STEP{2. $\sfq$ is an admissible parametrized curve:} Combining the previously
found convergences with the first estimate in \eqref{bounds-rescaled-curves},
we obtain $ \sup_{s\in [0,\mathsf{S}]} \mfE(\sfq(s)) \leq C$.   Using the second
estimate in \eqref{bounds-rescaled-curves} and \eqref{R-coerc} we have
$\|\sfz(s_2){-}\sfz(s_1)\|_{\Spy} \leq C_*|s_2{-}s_1|/c_\calR$. With
\eqref{comp-ingr-2}  we also infer that
$\sfz \in \rmC_{\mathrm{lip}}^0([0,\sfS]; \Spy)$.

We will now show that $\sfz$ is locally absolutely continuous in the set
$ \SetG\alpha t{\sfq}$ from \eqref{setGalpha}.  Let us first examine the case
$\alpha \in (0,1)$. Since the function $s\mapsto \slov zt{\sfq(s)}$ is lower
semicontinuous thanks to Hypothesis \ref{hyp:Sept19}, for
every
$[\varsigma,\beta] \subset \SetG\alpha t {\sfq} $ there exists $ c>0$ such that
$\slov zt{\sfq(s)} \geq c$ for all $s\in [\varsigma,\beta]$. This estimate bears
two consequences:
\begin{compactenum}
\item Exploiting the \emph{uniform} convergence of  $\sfz_k$  to $\sfz$ and
  again relying on Hypothesis \ref{hyp:Sept19},
  \begin{equation}
    \label{positivity-duals}
    \exists\, \bar k \in \N \ \forall\, k \geq \bar k \, \ \forall\, s \in
    [\varsigma,\beta]\,  : \qquad  
    \slov zt{\sfq_k(s)} \geq  \frac c2.
  \end{equation}
  This implies that, for $k\geq \bar k$, the sets
  $ \SetG\alpha  t{\sfq_k} = \{ s\, : \, \slov zt{\sfq_k(s)}>0 \} $
  contain the interval $[\varsigma,\beta]$.
\item 
Since, by \eqref{bounds-rescaled-curves},
  $C_* \geq  \mredq {\epsk} \alpha
  {\sft_k(s)}{\sfq_k(s)} {\sft_k'(s)}{\sfq_k'(s)} $ 
   for almost all $s\in (\varsigma,\beta)$, we are
  in a position to apply estimate \eqref{est-Alex-1} from Lemma
  \ref{new-lemma-Alex} and deduce that
  \begin{equation}
    \label{pavia}
\exists\, \overline C>0 \ \exists\, \bar k \in \N \ \forall\, k \geq \bar k
    \ \foraa s \in (\varsigma,\beta)\, : \qquad 
       \|\sfz_k'(s)\|_{\Spz} \leq \overline{C}\,. 
  \end{equation}
\end{compactenum}
The discussion of the case $\alpha\geq1$ follows the very same lines: for every
$[\varsigma,\beta] \subset \SetG\alpha  t{\sfq}$ we find $\tilde{c}>0$ and
$\tilde k \in \N$ such that for every $k \geq \tilde k$ we have
$  \slov ut{\sfq_k(s)} {+} \slov zt{\sfq_k(s)}   \geq \frac{\tilde{c}}2$
for every $s\in [\varsigma,\beta]$.  Then, estimate \eqref{pavia} follows from
\eqref{est-Alex-1-bis} in Lemma \ref{new-lemma-Alex}.

All in all, for all $\alpha >0$ the curves $\sfz_k$ are uniformly
$\Spz$-Lipschitz on $[\varsigma,\beta]$.  This entails that $\sfz$ is ultimately
$\Spz$-Lipschitz on any subinterval
$[s_1,s_2] \subset \SetG\alpha  t{\sfq}$,  and reflexivity of $\Spz$ gives
us 
\begin{equation}
  \label{conv-teta-z-k-prime}
  \sfz_k \weaksto \sfz \text{ in } W^{1,\infty} (\varsigma,\beta;\Spz) \quad
  \text{for all } [s_1,s_2]\subset  \SetG\alpha  t{\sfq}.
\end{equation} 

\STEP{3. Proof of \eqref{compactness-7-3}:} In order to conclude that
$(\sft,\sfq) \in \mathscr{A} ([0,\sfS]; [0,T]\ti \Spq)$, it remains to show
that it fulfills property \eqref{summability}, which will a consequence of
\eqref{compactness-7-3}.  By the lower semicontinuity we have
\begin{equation}
  \label{lim-pass-disr}
  \begin{aligned}
    \liminf_{k\to\infty} \int_0^{\mathsf{S}} \calR(\sfz_k'(s)) \dd s
    \overset{(1)}{=} \liminf_{k\to\infty} \Varname{\calR}(\sfz_k;
    [0,\mathsf{S}]) & \geq \Varname{\calR}(\sfz; [0,\mathsf{S}])
    \overset{(2)}{=} \int_0^{\mathsf{S}} \calR[\sfz'](s) \dd s,
  \end{aligned}
\end{equation}
with $  \overset{(1)}{=} $ and $\overset{(2)}{=} $ due to
\eqref{cited-VarR-later}.  Furthermore, we have
\begin{equation}
\label{reduced-lsc}
\begin{aligned}
  \liminf_{k\to\infty}
  \int_{0}^{\sfS} \mredq {\eps_k}{\alpha}{\sft_{k}}
{\sfq_{k}}{\sft'_{k}}{\sfq'_{k}} \dd s
  &  \geq 
\liminf_{k\to\infty}
\int_{(0,\sfS)\cap \SetG\alpha \sft \sfq} \mredq {\eps_k}{\alpha}{\sft_{k}}
{\sfq_{k}}{\sft'_{k}}{\sfq'_{k}} \dd s
\\ & 
 \overset{(3)}{\geq} \int_{(0,\sfS)\cap \SetG\alpha \sft \sfq} 
 \mredq{0}{\alpha}{\sft}{\sfq}0{\sfq'} \dd s\,.
\end{aligned}\vspace{-0.4em}
\end{equation}
Here, $ \overset{(3)}{\geq} $ follows from Proposition \ref{prop:Ioffe},
applied to the functionals
$ \mredname{\epsk}{\alpha}$ and $\mredname 0\alpha$, which we consider
restricted to the (weakly closed, by assumption \eqref{h:1.1}) energy sublevel
$\Iof = \{ q \in \Spq\, : \ \mfE(q) \leq C\}$.  Combining \eqref{lim-pass-disr}
and \eqref{reduced-lsc}, we infer \eqref{compactness-7-3} and thus conclude the
proof of  Proposition \ref{prop:compactness-param}. 
\end{proof}

We are now in the position to conclude the

\noindent
\begin{proof}[Proof of Theorem  \ref{thm:existBV}]
Let $(\sft_\epsk,\sfq_\epsk)_k $ be a sequence of rescaled viscous
trajectories satisfying \eqref{condition-4-normali}.  We apply Proposition
\ref{prop:compactness-param} and conclude that there exist a limit
parametrized curve $(\sft,\sfq) \in \mathscr{A}([0,\sfS];[0,T]\ti \Spq)$,
fulfilling \eqref{continuity-properties}, and a (not relabeled) subsequence
along which convergences \eqref{cvs-eps} hold.

We now show that the curves $(\sft,\sfq) $ fulfill the upper energy-dissipation
estimate $\leq$ in \eqref{def-parBV} by passing to the limit as $\epsk\to 0^+$
in \eqref{reparam-enineq} for $s_1=0$ and $s_2= s\in (0,\sfS]$.  The key lower
semicontinuity estimate
\begin{equation}
  \label{keyLSC-HS}   
  \int_0^s \mathfrak{M}_0^\alpha[\sft, \sfq, \sft',\sfq'](\sigma) \dd \sigma \leq      \liminf_{k\to\infty}
  \int_{0}^{s} \calM_\eps^\alpha (\sft_\epsk(\sigma) ,\sfq_\epsk(\sigma), 
  \sft'_\epsk(\sigma), \sfq'_\epsk(\sigma))  
  \dd \sigma \quad \text{for all $s\in [0,\sfS]$}
\end{equation}
follows from \eqref{compactness-7-3} in Proposition
\ref{prop:compactness-param}.  Convergences \eqref{cvs-eps}, the 
lower semicontinuity \eqref{h:1.2} of $\calE$, and the continuity
\eqref{h:1.3e}  of $\pl_t \calE$ give for all $s\in [0,\sfS]$ that
\begin{equation}
  \label{lsc-en-usc-power}
  \liminf_{k\to\infty} \eneq {\sft_\epsk (s)}{\sfq_\epsk(s)} \geq \eneq
  {\sft(s)}{\sfq(s)}  \ \text{ and } \ 
   \int_0^{s} \pl_t \eneq {\sft_\epsk}{\sfq_\epsk}
    \sft_\epsk' \dd \sigma  \to  \int_0^{s}
    \pl_t \eneq {\sft}{\sfq} \sft' \dd \sigma\,.
\end{equation}
For the last convergence we use $\sft'_k \weaksto \sft'$ in
$\rmL^\infty(0,\sfS)$ and $|\pl_t \eneq
{\sft_\epsk(\sigma)}{\sfq_\epsk(\sigma)} | \leq C_\# \eneq
{\sft_\epsk(\sigma)}{\sfq_\epsk(\sigma)} \leq C$ by \eqref{h:1.3d} 
and \eqref{est1}, which together with \eqref{h:1.3e} gives
$\pl_t \eneq {\sft_\epsk}{\sfq_\epsk} \to \pl_t \eneq
{\sft}{\sfq} $ strongly in $\rmL^2(0,\sfS)$. 

Taking into account the convergence of the initial energies guaranteed by
\eqref{init-data-cv}, we complete the limit passage in \eqref{reparam-enineq}.
Thanks to Lemma \ref{l:characterizBV}, the validity of the upper
energy-dissipation estimate in \eqref{def-parBV} ensures that
$(\sft,\sfq) = (\sft,\sfu,\sfz)$ is a $\pBV$ solution.

The enhanced convergences \eqref{cvs-eps-energy} and
\eqref{cvs-eps-M} are a by-product of this limiting procedure.
Although the argument is standard, we recap it for the reader's convenience and later use, and introduce   the following place-holders for every $s\in [0,\sfS]$:
\[
\left\{
\begin{array}{lll}
E_{\epsk}^s :=  \eneq {\sft_\epsk(s)}{\sfq_\epsk(s)}, &&  E_{0}^s :=  \eneq {\sft(s)}{\sfq(s)}
\\
M^s_\epsk : = \int_{0}^{s}
    \mathfrak{M}_\epsk^\alpha (\sft_\epsk(\sigma), \sfq_\epsk(\sigma),
    \sft_\epsk'(\sigma), \sfq_\epsk'(\sigma)) \,\dd \sigma  
&& M^s_ 0: = \int_{0}^{s}
    \mathfrak{M}_0^\alpha [\sft, \sfq,
    \sft', \sfq'](\sigma) \,\dd \sigma 
    \\
 E_{\epsk}^0 :=     \eneq    {\sft_\epsk(0)}{\sfq_\epsk(0)} 
&&  E_{0}^0 :=  \eneq {\sft(0)}{\sfq(0)}
  \\
   P^s_\epsk: =  \int_{0}^{s} \pl_t \eneq
    {\sft_\epsk(\sigma)}{\sfq_\epsk(\sigma)} \,\sft'_\epsk(\sigma) \,\dd \sigma 
    && P^s_0: =  \int_{0}^{s} \pl_t \eneq
    {\sft(\sigma)}{\sfq(\sigma)} \,\sft'(\sigma) \,\dd \sigma\,.
    \end{array}
   \right.
\]
Hence, the parametrized energy-dissipation estimate \eqref{reparam-enineq}
rephrases as $ E_{\epsk}^s +M^s_\epsk \leq E_{\epsk}^0+ P^s_\epsk $,
%
%
and the limiting energy-dissipation balance rewrites as
$ E_{0}^s +M_0 = E_{0}^0 + P^s_0 $. So far, we have shown that
 \[
  E_{0}^s +M^s_0  \leq \liminf_{k\to\infty} \left( E_{\epsk}^s
    +M^s_\epsk\right) \leq \limsup_{k\to\infty} (E_{\epsk}^s    +M^s_\epsk) \leq
  \limsup_{k\to\infty} (E_{\epsk}^0+     P^s_\epsk) =  E_{0}^0 +     P^s_0 =
  E_{0}^s +M^s_0  \,. 
\]
Since we have $\liminf_{k\to\infty} E_{\epsk}^s \geq E_0^s$ and
$\liminf_{k\to\infty} M^s_\epsk \geq M^s_0$, we thus conclude that
$\liminf_{k\to\infty} E_{\epsk}^s = E_0^s$ and $\lim_{k\to\infty}M^s_\epsk = M^s_0$
for all $s\in [0,\sfS]$, which means \eqref{eq:E.M.cvg.PBV}. Thus, Theorem
\ref{thm:existBV} is established. 
\end{proof}

\Subsection{Proof of Theorem \ref{thm:exist-trueBV}}
\label{ss:8.2}

\begin{proof}
We split the argument in  three steps. \smallskip 

\STEP{1.  Construction of a suitable $\pBV$ solution.}  Let $(q_{\epsk})_k$,
$q$ be as in the statement of Theorem \ref{thm:exist-trueBV}.  Lemma \ref{l:1}
ensures the validity of the basic energy estimates \eqref{est-quoted-later} and
\eqref{est1} for the sequence $(q_k)_k = (u_{\epsk},z_\epsk)_k$. The additional
estimate for $(u_\epsk)_k$ in $\rmB\rmV([0,T];\Spu) $ is assumed in
\eqref{estBV}, such that the arc-length functions $\sfs_{\epsk}$ from
\eqref{arclength-est1-2} fulfill $\sup_{k\in \N} \sfs_\epsk (T) \leq C$. We
reparametrize the curves $q_\epsk$ by means of the rescaling functions
$\sft_\epsk: = \sfs_\epsk^{-1}$ by setting
$\sfq_\epsk: = q_\epsk{\circ} \sft_{\eps_k}$. Without  loss of generality
we may suppose  that $(\sft_\epsk,\sfq_\epsk)$ is surjective and 
defined on a fixed interval $[0,\sfS]$.  

Now, for the sequence $(\sft_\epsk, \sfq_\epsk)_k$ the a priori estimate
\eqref{condition-4-normali} holds.  Hence, we are in a position to apply Thm.\
\ref{thm:existBV} to the curves $(\sft_\epsk,\sfq_\epsk)_k$ and we conclude
that $(\sft_\epsk,\sfq_\epsk)_k$ convergence along a (not relabeled)
subsequence to a $\pBV$ solution $(\sft,\sfq): [0,\sfS]\to [0,T]\ti \Spq$.
%
%
In what follows, we will prove that $q$ is related to the parametrized curve
$(\sft,\sfq)$ via \eqref{eq:Project.pBV}.\smallskip

\STEP{2. Every limit point $q$ is a true $\BV$ solution:} We first prove that
\begin{equation}
  \label{proj.pBV.to.BV}
  q(t) \in \{ \sfq(s)\, : \ \sft(s) = t \}
  \text{ for all }   t \in [0,T]. 
\end{equation} 
 For this we fix $t_*\in [0,T]$ and choose $s_k\in [0,\sfS]$ such that
$\sft_\epsk(s_k)=t_*$. After choosing a (not relabeled) subsequence we may
assume $s_k\to s_*$. As  $(\sft_\epsk,\sfq_\epsk)_k$  converges uniformly to
$(\sft,\sfq)$ in $\rmC^0([0,\sfS];\R\ti \Spq_\mathrm{weak})$ we obtain
\[
\sft_\epsk(s_k)\to \sft(s_*) \quad \text{and} \quad \sfq_\epsk(s_k)\weakto \sfq(s_*).
\]    
However, by construction we have 
\[
\sft_\epsk(s_k)=t_* \quad \text{and} \quad 
\sfq_\epsk(s_k) \overset{\text{\tiny Step 1}}= q_\epsk(\sft_\epsk(s_k)) =
q_\epsk(t_*) \weakto  q(t_*),
\]
where the last convergence is the assumption in \eqref{ptw-weak-limit}. 
Hence we conclude $\sft(s_*)=t_*$ and $\sfq(s_*)=q(t_*)$ which is the desired
relation \eqref{proj.pBV.to.BV}.  

Thanks to \eqref{proj.pBV.to.BV}, we can apply Theorem
\ref{th:pBV.v.BVsol}(1), which ensures that $q$ is a true $\BV$
solution.\smallskip

\STEP{3. Proof of convergences \eqref{cvs-BV}:} Since $(q_{\eps_k})_k$ is
bounded in $\rmL^\infty (0,T;\Spw\ti \Spx)$ by estimate \eqref{est1} and
Hypothesis \ref{hyp:1}, the pointwise weak convergence in $\Spq$ improves to
convergences \eqref{cvs-BV-a}.  Next, we observe that for every
$ 0 \leq s_0 \leq s_1 \leq \mathsf{S}$ there holds
\[
  \begin{aligned}
    \lim_{k\to \infty} \int_{\sft(s_0)}^{\sft(s_1)} \meq
    \epsk{\alpha}{r}{q_\epsk(r)}1{q_\epsk'(r)} \dd r & = \lim_{k\to \infty}
    \int_{s_0}^{s_1} \meq
    \epsk{\alpha}{\sft_\epsk(\sigma)}{\sfq_\epsk(\sigma)}{\sft_\epsk'(\sigma)}
    {\sfq_\epsk'(\sigma)} \dd \sigma
    \\
    & \overset{(1)}{=} \int_{s_0}^{s_1} \mename 0{\alpha}[\sft, \sfq, \sft',
    \sfq'](\sigma) \dd \sigma \overset{(2)}{=} \Variq{\mename
      0\alpha}{q}{\sft(s_0)}{\sft(s_1)},
  \end{aligned}
\] 
with {\footnotesize (1)} due to \eqref{cvs-eps-M} and {\footnotesize (2)} due
to \eqref{equality-between-variations}. Hence, \eqref{cvs-BV-c} follows.

Finally, the lower semicontinuity of $\calE$, the  continuity
\eqref{h:1.3e} of $\pl_t \calE$,   give  that 
\[
  \begin{aligned}
    \liminf_{k\to\infty} \eneq t{q_k(t)} \geq \eneq t{q(t)} 
\ \text{  for all } t\in [0,T] \quad \text{and} \quad
    \int_0^t \pl_t \eneq s{q_k(s)} \dd s \to  \int_0^t
    \pl_t \eneq s{q(s)} \dd s.
  \end{aligned}
\]
Hence, with similar arguments as in the proof of Theorem \ref{thm:existBV}
(cf.\  the end of  Section \ref{ss:8.1}), we conclude 
\[
  \eneq t{q(t)} + \Variq{\mename 0\alpha} q0t = \lim_{k\to\infty} \eneq
  t{q_\epsk(t)} + \lim_{k\to\infty} \int_{0}^t \meq{\epsk}\alpha
  r{q_{\epsk}(r)}1{q_{\epsk}'(r)} \dd r \ \text{ for all } t \in [0,T],
\]
and \eqref{cvs-BV-b} ensues from the previously obtained \eqref{cvs-BV-c}.
This finishes the proof  of Theorem \ref{thm:exist-trueBV}. 
\end{proof}

\Subsection{Proof of Proposition \ref{pr:charact-Ctc-set}}
\label{su:pr:char-Ctc-set}%

Our task is to show inclusions \eqref{eq:CtcSet.Incl} for the contact sets
$\Ctc_\alpha$  and the flow regimes $\rgs Au \rgs Cz$  for the three
different cases for $\alpha$. We rely on the explicit form of
$\mename0\alpha=\calR+\mredname 0\alpha $ from \eqref{l:partial}.\smallskip

\begin{proof}[Proof of Proposition \ref{pr:charact-Ctc-set}]
\STEP{1: The case $t'>0$.} We start by showing that for all $\alpha>0$
  we have 
\[
  \Ctc_\alpha^{>0}:= \bigset{ (t,q,t',q')\in \Ctc_\alpha }{ t'>0} = \rgs
  Eu \rgs Rz = \rgs Eu \cap \rgs Rz . 
\]
Indeed, in the case $t'>0$ we have $\mename 0\alpha(t,q,t',q') <\infty$ if and
only if $\slov utq =\slov ztq =0$ and then
$\mename 0\alpha(t,q,t',q') = \calR(z')$.  From the former we obtain that, in
fact, every $(\mu,\zeta)\in \argminSlo utq {\times} \argminSlo ztq$ satisfies
$\mu=0$ and $-\zeta \in \pl \calR (0)$.  From the contact condition
$\calR(z') = \mathfrak{M}_0^\alpha(t,q,t',q') = {-}\pairing{}{\Spz}{\zeta}{z'}$
and the 1-homogeneity of $\calR$ we infer that $-\zeta \in \pl \calR(z')$, 
see \eqref{eq:subdiff.calR}.  
Taking into account that $\argminSlo xtq \subset \frsubq xtq$ for
$\mathsf{x} \in \{\mathsf{u}, \mathsf{z}\}$, we ultimately infer
\begin{equation}
  \label{EuRz-in-viscosity}
  \frsubq utq \ni 0, \qquad   \pl  \calR(z')+    \frsubq ztq \ni 0,
\end{equation}
namely system \eqref{static-tq} holds with $\thn u = \thn z =0$, i.e.\
$(t,q,t',q') \in \rgs Eu \rgs Rz $.  Hence, we have shown $\Ctc_\alpha^{>0}
\subset \rgs Eu \rgs Rz$. In fact,  reverting the arguments the
opposite inclusion holds as well.\smallskip

\STEP{2. The case $t'=0$.} We define  $\Ctc_\alpha^0:= \bigset{(t,q,t',q')\in
  \Ctc_{\alpha}}{ t'=0}$ and treat the three cases $\alpha=1$, $\alpha>1$,
and $\alpha \in (0,1)$, separately.\smallskip

\STEP{2.A. $t'=0$ and $\alpha=1$.} We want to show the inclusion
\begin{equation}
  \label{desired-alpha=1}
 \Ctc_{\alpha=1}^0 \subset \big(\rgs Eu \rgs Rz \cap \{t'=0\}\big) \cup \rgs V{uz}
 \cup \rgs Bu \rgs Bz\,.
\end{equation}
From \eqref{l:partial} we have
$\mename0\alpha(t,q,0,q')=\calR(z')+\mfb_{\disv u \oplus \disv z}(q',\slov
utq{+}\slov ztq)$. \newline Hence, for $\slov utq{+}\slov ztq=0$ we argue as in Step 1
and obtain $(t,q,0,q')\in \rgs Eu \rgs Rz \cap \{t'=0\}$.
 
We may now suppose that $\slov u{t}{q} {+} \slov z{t}{q}>0$  and
$q'=(u',z')=0$. Clearly,  the contact condition
$ \mename 0{\alpha=1}(t,q,t',q') = -\pairing{}{\Spu}{\mu}{u'}-
\pairing{}{\Spz}{\zeta}{z'} $ holds for all
$(\mu,\zeta) \in \argminSlo utq {\times} \argminSlo ztq$.  However,
$\slov u{t}{q} {+} \slov z{t}{q}>0$ gives
$\big(\{0\}{\times}\pl \calR(0) \big)\cap\frsubq qtq = \emptyset$, and because
of $ \argminSlo utq {\times} \argminSlo ztq \subset \frsubq qtq$ we conclude
that $(t,q,0,(0,0))$ fulfills system \eqref{static-tq} with
$\thn u = \thn z =\infty $. Hence, $(t,q,0,q')=(t,q,0,0) \in \rgs Bu \rgs Bz$
as desired.

Suppose now $ \disv z (z') {+} \disv u (u')>0$ in addition to
$\slov u{t}{q} {+} \slov z{t}{q}>0$.   According to Proposition
\ref{pr:VVCP}(b2) there exists  $\ell = \ell(t,q,q')>0$ with 
\[
\mfb_{\disv u{\oplus}\disv z}(q', \slov u{t}{q} {+} \slov z{t}{q})
 = \ell \:\Big(\disv u \big( \frac1\ell \, {u'}\big) {+}\disv z \big(
 \frac1{\ell}\,{z'}\big) {+} \slov u{t}{q}{+} \slov z{t}{q} \Big) \,.
\]
Now, $(t,q,0,q') \in \Ctc^0_1$ means that  there exists 
$(\mu,\zeta) \in \argminSlo u{t}{q} {\times} \argminSlo z{t}{q}$ fulfilling the contact condition 
\begin{align*}
\mename01(t,q,0,q')= \calR(z')+ \mfb_{\psi_\sfu{\oplus}\psi_\sfz}(q', \slov u{t}{q} {+} \slov
z{t}{q}) = - \pairing{}{\Spu}{\mu}{u'} 
 - \pairing{}{\Spz}{\zeta}{z'}.  
\end{align*}
Moreover, the definition of $\argminSlo xtq$ gives $\slov utq=\disv
u^*({-}\mu)$ and $\slov ztq=\conj z({-}\zeta)$. Together with the definition of
$\ell$ we find the identity
\begin{align*}
& \disv u \big( \frac1\ell \, {u'}\big)
            {+} \calR \big( \frac1{\ell}\,{z'} \big) 
            {+} \disv z \big( \frac1{\ell}\,{z'}\big) 
            {+} \disv u^*({-}\mu){+}   \conj z({-}\zeta) \\
&=\frac1\ell \:\mename01(t,q,0,q')= -\frac1\ell\big( \pairing{}{\Spu}{\mu}{u'} 
 {+} \pairing{}{\Spz}{\zeta}{z'} \big) =  
     \pairing{}{\Spu}{{-}\mu}{\frac1{\ell}\,{u'}}  
    {+} \pairing{}{\Spz}{{-}\zeta}{\frac1{\ell} \, {z'} }. 
\end{align*}
 Since $\disv u^*\oplus \conj z$ is is the Legendre-Fenchel dual of $\disv
u \oplus(\calR {+}\disv z)$ we conclude 
\[
-\mu \in \pl  \disv u\big(\frac1\ell\,u'\big) 
  = \pl \disve u{1/\ell}(u')
 \ \text{ and } \ 
{-} \zeta \in \pl \calR \big( \frac1\ell\,z' \big) 
            {+} \pl  \disv z  \big( \frac1\ell\,z' \big) 
 = \pl \calR(z') {+} \pl  \disve z{1/\ell}(z').  
\]
From this we see that $(t,q,0,q')$ system \eqref{static-tq} holds with
$\thn u = \thn z = 1/\ell \in (0,\infty)$, i.e.\ we have
$(t,q,0,q') \in \rgs V{uz}$, and the inclusion \eqref{desired-alpha=1} is
established.\smallskip

\STEP{2.B. $t'=0$ and $\alpha>1$.} Let us now examine the case $\alpha>1$ and
prove that
\begin{equation}
  \label{desired-alpha>1}
   \Sigma^0_\alpha \subset \big(\rgs Eu \rgs Rz \cap\{t'=0\}\big) 
   \cup \rgs Eu \rgs Vz   \cup \rgs Bz \,.
\end{equation}
Using the explicit expression for $\mredq 0\alpha tq0{q'}$ in \eqref{l:partial},
we  see that $\meq 0\alpha tq0{q'} < \infty$  implies that
 either (i) $\slov utq=0$ or (ii) $\big(\,\slov utq>0$ and
$z'=0\,\big)$.

In case (i),  which means $\rgs Eu$,  the contact condition reads
\[
 \slov utq=0 \quad \text{and} \quad \exists\, \zeta \in 
\argminSlo ztq\,{:} \ \    \calR(z')+ \mfb_{\psi_\sfz}(z', \slov ztq)
  = - \pairing{}{\Spz}{\zeta}{z'} . 
\]
If $\slov ztq=0$, we have $\mfb_{\psi_\sfz}(z', \slov ztq)=0$ and infer
$\calR(z')+ \pairing{}{\Spz}{\zeta}{z'}=0$.  Moreover, $\slov ztq=0$
implies  $\zeta \in \argminSlo ztq=\pl \calR(0)$, and we conclude
$\pl \calR(z') + \zeta \ni 0$  by \eqref{eq:subdiff.calR}. Hence,  we can
choose $\thn z=0$ in \eqref{subdiff-stat.z}  and obtain $(t,q,0,q')\in \rgs Eu
\rgs Rz$. 

If $z'=0$ holds but not $\slov ztq>0$, then \eqref{subdiff-stat.z} holds for $\thn
z=\infty$  and $(t,q,0,q')\in \rgs Bz$. 

Finally, if $z'\neq 0$ and $\slov ztq>0$, then the very same discussion as in
the last part of Step 2.A provides $\thn z \in (0,\infty)$ such that
$ \pl \calR(z') + \pl \disve z{ \thn z} (z') + \frsubq ztq \ni 0 $, 
which means $(t,q,0,q')\in \rgs Eu \rgs Vz$.  

The discussion of the case (ii) with $\slov utq>0$ and $z'=0$ proceeds
along the same lines, relying  on the contact condition 
\[
\exists\, \mu \in \argminSlo utq: \quad \meq 0\alpha
tq0{u',0}=\calR(0)+\mfb_{\psi_\sfz}(u', \slov utq) = -  
\pairing{}{\Spu}{\mu}{u'} . 
\]
For $u' \neq 0$ we find $\thn u\in (0,\infty)$ with
$\pl \disve u{\thn u}(u') + \mu \ni 0$,  which gives $(t,q,0,q')\in \rgs Vu
\rgs Bz$. For  $u'=0$ we can choose
$\thn u=\infty$ such that $\pl \disve u{\infty}(0) + \mu =\Spu^*+\mu \ni 0$.
Hence \eqref{static-tq} holds with $\thn z=\infty$ and $\thn u\in (0,\infty]$,
 i.e.\ $(t,q,0,q')\in \rgs Bz$.  

Thus,  in both cases, (i) and (ii),  we conclude
\eqref{desired-alpha>1},  and Step 2.B is completed.\smallskip 

\STEP{2.C. $t'=0$ and $\alpha \in (0,1)$.} This case works similarly as the
case $\alpha>1$ in Step 2.B, but with the roles of $u$ and $z$ interchanged,
where $\rgs Eu$ is interchanged with $\rgs Rz$. Thus, in analogy to
\eqref{desired-alpha>1} we obtain
$\Sigma^0_\alpha \subset \big(\rgs Eu \rgs Rz \cap\{t'=0\}\big) \cup \rgs Vu
\rgs Rz \cup \rgs Bu $.

This concludes the proof  of Proposition \ref{pr:charact-Ctc-set}. 
\end{proof}

\Section{Application to a model for delamination}
\label{s:appl-dam}

In this section we discuss the application of our vanishing-viscosity analysis
techniques to a PDE system modeling adhesive contact.  A previous
vanishing-viscosity (and vanishing-inertia, in the momentum balance) analysis
was carried out for a delamination model in \cite{Scala14} where, however, an
energy balance only featuring defect measures, in place of contributions
describing the dissipation of energy at jumps, was obtained in the
null-viscosity limit.

After introducing the viscous model and discussing its structure as a
generalized gradient system in Section \ref{ss:10.-1}, we are going to
state the existence of \emph{enhanced} $\BV$ and parametrized 
solutions to the corresponding rate-independent system $\RIS$ in Theorem
\ref{thm:BV-adh-cont}. This result will be proved throughout Sections
\ref{ss:10.0}--\ref{su:Dela.ExiApriGener} by showing that the `abstract' 
Theorems \ref{thm:exist-enh-pBV} and \ref{thm:exist-nonpar-enh} apply. 
 
As we will emphasize later on, our analysis crucially relies on the fact that,
in the delamination system, the coupling between the displacements and the
delamination parameter only occurs through lower order terms.

\Subsection{The `viscous'  system for delamination}
\label{ss:10.-1}

We consider two bodies located in two bounded Lipschitz domains
$\Omega^\pm \subset \R^3$ and adhering along a prescribed interface $\GC$, on
which some adhesive substance is present. We denote that part of $\Omega^\pm$
that coincides with $\GC$ by $\Gamma_\pm$, see Figure \ref{fig:delam.dom}, thus
being able to talk about one-sided boundary conditions.  In what follows, for
simplicity we will assume that $\GC$ is a `flat' interface, i.e., $ \GC$ is
contained in a plane, so that, in particular,
$\mathscr{H}^{2}(\GC) =\mathscr{L}^{2}(\GC) >0$. While the generalization to a
smooth curved interface is standard, this restriction will allow us to avoid
resorting to Laplace-Beltrami operators in the flow rule for the delamination
parameter.

The state variables in the model are indeed the displacement
$u : \Omega \to \R^3$, with $\Omega: = \Omega^+ \cup \Omega^- $, and the
delamination variable $z: \GC \to [0,1]$, representing the fraction of fully
effective molecular links in the bonding. Therefore, $z(t,x) =1$ ($z(t,x)=0$,
respectively) means that the bonding is fully intact (completely broken) at a
given time instant $t\in [0,T]$ and in a given material point $x\in \Gamma$.
We denote by $n^\pm$ the outer unit normal of $\Omega^\pm$ restricted to
$\Gamma_\pm$ and by $\JUMP{u}$ the jump of $u$ across $\GC$, namely
$\JUMP{u} = u|_{\Gamma_+} - u|_{\Gamma_-}$, but now defined as function on
$\GC$.
\begin{figure}
\begin{minipage}{0.45\textwidth}
\begin{tikzpicture}
\draw[very thick,black, fill=gray!20] (0,1)-- (4,1)--(4,2)--(3,3)--(1,3)--(0,2)--(0,1);
\node[black] at (3,2){$\Omega^+$};
\draw[very thick, red] (-0.03,1) node[left]{\raisebox{0.8em}{$\Gamma_+$}}--(4.03,1); 
\draw[very thick] (-0.03,0.9)--(4.03,0.9) node[right]{$\GC$};
\draw[very thick,black, fill=gray!20] (0,0.8)--
(4,0.8)--(4,-1.1)--(-1,-1.3) --(0,0.8);
\draw[very thick, red] (-0.03,0.8) node[left]{\raisebox{-1.6em}{$\Gamma_-$}}--(4.03,0.8); 
\node[black] at (3,-0.05){$\Omega^-$};
\draw[blue, very thick] (4,2)-- (3,3) node[pos=0.5, right]{\ $\GDir$};
\draw[blue, very thick] (4,-1.1)--(-1,-1.3) node[pos=0.05, below]{$\GDir$};
\end{tikzpicture}
\end{minipage}
\begin{minipage}{0.3\textwidth}
\caption{} \label{fig:delam.dom} The two domains $\Omega^+$ and $\Omega^-$
touch along the delamination hypersurface $\GC$.
\end{minipage}
\end{figure}
 
For simplicity, we impose homogeneous
Dirichlet boundary conditions $u=0$ on the Dirichlet part $\GDir$ of
the boundary $\pl \Omega$, 
with $\mathscr{H}^2(\GDir)>0$.  
We consider a given applied
traction $f$ on the Neumann part $\GNeu=\pl \Omega \setminus 
(\GDir{\cup} \GC) $.

All in all, we address the following \emph{rate-dependent} PDE system
\begin{subequations}
 \label{PDEadhc}
  \begin{align}
    &
    \label{PDEadhc-a}
    -\mathrm{div}(\eps^\alpha \bbD e(\dot u) + \bbC e(u))= F  && \text{ in }
    \Omega \ti (0,T), 
    \\
    & 
    \label{PDEadhc-b}
    u=0  && \text{ on } \GDir  \ti (0,T),
    \\
    \label{PDEadhc-c}
    & (\eps^\alpha\bbD e(\dot u) + \bbC e(u))|_{\GNeu} \nu = f && \text{ on }
    \GNeu  \ti (0,T),  
    \\
    & 
    \label{PDEadhc-d}
    (\eps^\alpha \bbD e(\dot u) + \bbC e(u))|_{\GC} n^\pm \pm \gamma (z) 
    \pl \psi(\JUMP{u}) \pm \beta(
    \JUMP{u} ) \ni 0 && \text{ on } \Gamma_\pm  \ti (0,T),  
    \\
    & 
    \label{PDEadhc-e}
    \pl \mathrm{R}(\dot z) +  \eps \dot z- 
    \Delta z +  \tilde{\phi}(z) +  
       \pl \gamma(z)  \psi(\JUMP{u}) \ni 0  &&
    \text{ on } \GC  \ti (0,T), 
  \end{align}
\end{subequations}
where $\dot u$ and $\dot z$ stand for the partial time derivatives of $u$ and
$z$.  Here, $F$ is a volume force,
$\bbD,\, \bbC \in \text{Lin}(\R^{d \ti d}_\mathrm{sym} ) $ the positive
definite and symmetric viscosity and elasticity tensors, $\nu$ the exterior
unit normal to $\pl (\Omega^+ {\cup} \GC {\cup} \Omega^-)$, and
$\mathrm{R}$ is given by
\begin{equation}
\label{explicit-R-adh}
  \mathrm{R}(r) = \kappa_+ \max\{r,0\} + \kappa_- \max\{{-}r,0\}
  \quad \text{ with } \kappa_\pm>0.    
\end{equation}
Hence, healing of the broken molecular links is disfavored, but not totally
blocked.  Giving up unidirectionality allows for a more straightforward
application of our abstract results. Nonetheless, we expect that, at the price
of some further technicalities our techniques could be adapted to deal with
unidirectionality by means of additional estimates (like for instance in the
application of $\BV$ solutions to \emph{unidirectional} damage
developed in \cite{KRZ13}).

The term   $\gamma(z)\pl \psi(\JUMP{u})$,  with $\gamma$ and $\psi$ nonnegative functions (we may think of 
$\gamma(z) = \max\{ z, 0\}$) and $\psi$ convex,
 in \eqref{PDEadhc-d} derives from the contribution
$\gamma(z) \psi(\JUMP{u})$ to the surface energy, cf.\ \eqref{energy-delamination} ahead,
which penalizes the constraint
$z\JUMP{u} =0$ a.e.\ on $\GC$,  typical of \emph{brittle} delamination models. Indeed,  to our
knowledge, existence results for brittle models are available only in the case
of a rate-independent evolution for $z$, cf.\ e.g.\ \cite{RoScZa09QDP,
  RosTho12ABDM}.  In fact, \eqref{PDEadhc} is rather a model for \emph{contact
  with adhesion} and will be accordingly referred to in this way. 
  Our assumptions on the constitutive functions $\gamma $, $\psi$  and $\beta$,  and on the multivalued operator  $\tilde\phi$ (indeed, on the mapping $z\mapsto \tilde\phi(z) -z$), will be specified in \eqref{eq:Del.Ass02General} ahead.

We define  the operators $\bsC,\bsD : \rmH^1(\Omega;\R^3)\to
\rmH^1(\Omega;\R^3)^*$  via 
\[
\pairing{}{\rmH^1(\Omega)}{\bsC u }{v}: = \int_{\Omega}\bbC e(u): e(v) \dd x,
\qquad \pairing{}{\rmH^1(\Omega)}{\bsD u }{v}: = \int_{\Omega}\bbD e(u): e(v)
\dd x, 
\]
while we denote by
$\bsJ:\rm\rmH^1(\Omega;\R^3)\to \rmL^4(\GC;\R^3); \; u\mapsto \JUMP{u}$ the
jump operator, by $\|\bsJ\|$ its operator norm, and by $\bsJ^*$ its adjoint. We
denote by $\bsA$ the Laplacian with homogeneous Neumann boundary
conditions 
\[
 \bsA: \rmH^1(\GC)\to \rmH^1(\GC)^* 
 \qquad \pairing{}{\rmH^1(\GC)}{\bsA z}{\omega}: = \int_{\GC} 
\big( \nabla z \nabla \omega + z \omega \big) \dd x\,.
\]
In particular, we have $\|z\|^2_{\rmH^1(\GC)}= \pairing{}{\rmH^1(\GC)}{\bsA
  z}{z}$.  Finally, we denote by 
  $\ell_u: (0,T)\to \Spu^*$  the functional encompassing the volume
and surface forces $F$ and $f$, namely
\[
  \pairing{}{\rmH^1(\Omega)}{\ell_u(t)}{u}: =
  \int_{\Omega} F(t)  u \dd x +\int_{\GNeu} f(t) u \dd S\,.
\]
Throughout, we will assume that 
\begin{equation}
  \label{bold-force-F}
  \ell_u \in \rmC^1([0,T]; \rm\rmH^1(\Omega;\R^3)^*)\,.
\end{equation}
Hence, system \eqref{PDEadhc} takes the form
\begin{subequations}
  \label{eq:DelamSyst}
\begin{align}
\label{eq:DelamSyst.a}
0& \in \eps^\alpha \bsD \dot u + \bsC u + \bsJ^*\big(\beta(\JUMP{u}) +
 \gamma(z) \pl   \psi(\JUMP{u}) \big) - \ell_u && \text{in } \rmH^1(\Omega;\R^3)^* \\
\label{eq:DelamSyst.b}
0&\in \pl  \mathrm{R}(\dot z) + \eps \dot z  + \bsA z + \pl  \wh\phi(z) +
\pl  \gamma(z)\psi(\JUMP{u}) && \aein\, \GC 
\end{align}
\end{subequations}  
almost everywhere in $(0,T)$. In \eqref{eq:DelamSyst}, 
 $\widehat{\beta}$ a primitive for $\beta$ and $\wh \phi$ a primitive for  the multivalued operator $z \mapsto \tilde\phi(z)-z$. 
\subsubsection*{\bf Structure as a  (generalized) gradient system}   
First of all, let us specify our assumptions 
 on  the constitutive functions $\widehat\beta$,
 $\gamma$, $\widehat \phi$, and $\psi$: 
 \begin{equation}
  \label{eq:Del.Ass02General}
\left. \begin{aligned}
  &\psi,\,\wh\beta:\R^3\to [0,\infty) \text{ are lsc and convex with }
  \psi(0)=\wh\beta(0)=0,\\
  &\exists\, C_\psi>0 \ \forall\, a \in \R^3\, : 
    \quad  \psi(a)\leq C_\psi (1{+}|a|^2), 
  \\
  & 
 \wh \beta \in \rmC^1(\R^3) \text{ and }  \beta = \rmD\wh\beta \text{ is globally Lipschitz}
  \\
  &  \gamma \text{ is convex, non-decreasing and $1$-Lipschitz, with } \gamma(0)=0, 
  \\
  &\wh\phi:\R \to [0,\infty] \text{ is lsc and $(-\Lambda_\phi)$-convex for some $\Lambda_\phi>0$,  with  }
  \wh\phi(z)=\infty \text{ for }z\notin[0,1].\quad  
\end{aligned} \right\}
\end{equation}
Hence, $\pl \gamma$ and $\pl \psi$ in \eqref{eq:DelamSyst} are convex analysis
subdifferentials, while $\pl \wh \phi :\R \rightrightarrows \R$ is the
Fr\'echet subdifferential of $\wh \phi$.
 
To fix ideas, prototypical choices for $\widehat\beta$, $\gamma$,
$\widehat \phi$, and $\psi$ would be:
\begin{compactitem}
\item[\emph{(i)}] $\widehat\beta$ the Yosida regularization of the indicator
  function of the cone $C = \{ v \in \R^3\, : \ v \cdot n^+ \leq 0 \} $ (cf.\
  also Remark \ref{rmk:why-no-indicator});
\item[\emph{(ii)}] $\gamma(z) = \max\{ z,0\}$;
\item[\emph{(iii)}] $\widehat \phi$ encompassing the indicator function
  $I_{[0,1]}$, which would ensure that $z \in [0,1]$;
\item[\emph{(iv)}] $\psi (\JUMP{u}) = \tfrac k2 |\JUMP{u} |^2$ with $k>0$.
\end{compactitem}
  
Observe that \eqref{eq:DelamSyst}  falls into the class of gradient systems
\eqref{dne-q}, with the ambient spaces
\begin{subequations}
\label{setup-adh-cont}
\begin{equation}
  \label{spaces-adh-cont}
    \Spu = \rmH_{\GDir}^1(\Omega;\R^3), \quad  \Spz = \rmL^2(\GC), \quad \Spy =
    \rmL^1(\GC) 
\end{equation}
where $\rmH_{\GDir}^1(\Omega;\R^3)$ the space of $\rmH^1$-functions on $\Omega$
fulfilling a homogeneous Dirichlet boundary condition on $\GDir$.  By Korn's
inequality, the quadratic form associated with the operator $\bsD$ induces on
$\Spu$ a norm equivalent to the $\rm\rmH^1$-norm; hereafter, we will in fact
use that
\[
   \|u\|_\Spu^2 :=   
    \sideset{_{}}{_{\rmH^1(\Omega)}}  {\mathop{\langle \bsD u ,u \rangle}}, 
 \quad \sideset{_{}}{_{\rmH^1(\Omega)}}  {\mathop{\langle \bsC u ,u \rangle}}  
  \geq    c_{\bsC}\|u\|_\Spu^2 . 
\]
The $1$-homogeneous dissipation potential
$\calR: \Spy \to  [0,\infty)$  is defined by
\begin{equation}
  \label{diss-pot-adhc}
  \calR(\dot z): = \int_{\GC} \mathrm{R}(\dot z) \dd x \qquad \text{with } \mathrm{R} \text{ from \eqref{explicit-R-adh}.}
\end{equation} 
The viscous dissipation potentials $\disv u : \Spu \to [0,\infty)$ and
$\disv z : \Spu \to [0,\infty)$ are
\begin{equation}
  \label{disv-adhc}
  \disv u (\dot u) :=\frac12  \sideset{_{}}{_{\rmH^1(\Omega)}}  {\mathop{\langle \bsD \dot u ,\dot u
\rangle}} \ \qquad  \disv z (\dot z): = \int_{\GC} \tfrac12 |\dot z |^2
  \dd x. 
\end{equation}
The driving
  energy functional
$\calE : [0,T]\ti \Spu \ti \Spz\to (\infty,+\infty]$ is given by 
\begin{equation}
  \begin{aligned}
  \label{energy-delamination}
   \ene tuz : =   \frac12 \pairing{}{\rmH^1(\Omega)}{\bsC u}u   
   &   - \pairing{}{\rmH^1(\Omega)}{\ell_u(t)}{u} 
       + \frac12 \pairing{}{\rmH^1(\GC)}{\bsA z}z 
      + \int_{\GC} \big(       \widehat{\beta}(\JUMP{u})
     {+}  \gamma(z) \psi(\JUMP{u}) {+}  \widehat{\phi}(z) \big) \dd x
  \\
   & \quad  \text{ if } z \in \rmH^1(\GC) \text{ and } 
   \widehat{\phi}(z) \in \rmL^1(\GC),
  \end{aligned}
\end{equation}
and $\infty$ otherwise.
\end{subequations}

As we will see in Proposition \ref{l:comprehensive}, 
under the conditions on
$\widehat\beta$,
 $\gamma$, $\widehat \phi$, and $\psi$   specified in \eqref{eq:Del.Ass02General},
 
$\calE$ complies with the coercivity conditions from Hyp.\
\ref{hyp:1} with the spaces
\begin{equation}
\label{coercivity-spaces-delam}
  \Spw = \Spu =\rmH_{\GDir}^1(\Omega;\R^3) \quad \text{and} \quad
  \Spx =\rmH^1(\GC) \Subset \Spz,
\end{equation}
and  its  Fr\'echet subdifferential
$\pl_q \calE : [0,T] \ti \Spu \ti \Spz \rightrightarrows \Spu^* \ti
\Spz^*$  is given by  
\begin{equation}
  \label{Frsub-adh-cont}
  \begin{aligned}
    &   (\mu,\zeta)  \in \pl_q \ene tuz \qquad \text{ if and only if }
        \\
    & \!  \! \! \! \!  \
    \begin{cases}
      &  \!  \! \!  \mu  =  \bsC u + \bsJ^*\big( \beta(\JUMP{u})  +
 \gamma(z)\varrho \big) - \ell_u(t)  \text{ for  some selection } 
 \GC \ni x \mapsto \varrho(x)\in \pl  \psi(\JUMP{u(x)}) 
      \\
      & \!  \! \!   \zeta = \bsA z +\omega \psi(\JUMP{u}) + \phi
     \text{ for selections }   \GC \ni x \mapsto\omega(x) \in \pl \gamma(z(x)) \text{ and } 
       \GC \ni x \mapsto\phi(x) \in \pl  \widehat{\phi}(z(x)) \text{ s.t. }  
      \\ 
      & \qquad \qquad \qquad \qquad \qquad \qquad \qquad \qquad \qquad \qquad \bsA z + \phi \in \rmL^2(\GC)
    \end{cases}
  \end{aligned}
\end{equation}
(indeed, observe that, by the growth properties of $\gamma$ and $\psi$, the term $
\omega \psi(\JUMP{u}) $ is in  $\rmL^2(\GC)$ for any selection $\omega \in \pl \gamma$).
In particular, here we have used that the Fr\'echet subdifferential of the ($({-}\Lambda_\phi)$-convex) functional
$\calF: \Spz \to [0,\infty]$
\begin{subequations}
\label{calF-stuff}
\begin{align}
&
\calF(z): =
\left\{
\begin{array}{ll}\!  \! \!
 \tfrac12 \pairing{}{\rmH^1(\GC)}{\bsA z}z  +\int_{\GC} \wh\phi(z) \dd x 
 & \text{ if $z\in \rmH^1(\GC)$ and $\wh\phi(z) \in \rmL^1(\GC)$},
 \\
\!  \! \! \infty &  \text{else},
 \end{array}
 \right.
\intertext{is given by}
&
\pl \calF(z) = \{ \bsA z + \tilde\phi\, : \, \tilde \phi(x) \in 
  \pl  \wh \phi(x) \ \foraa x \in \GC, \  \bsA z + \tilde\phi \in \rmL^2(\GC)\}\,.
\end{align}
\end{subequations}

We also point out for later use that  $\frname q\calE$ fulfills the structure condition \eqref{it-is-product}, i.e.\
$  \frsub qtuz =\frsub utuz \ti \frsub ztuz $ for every $(t,u,z) \in [0,T]\ti \Spu \ti \Spz$. 

\subsubsection*{\bf Existence for the viscous system} 
As we will check in Proposition \ref{l:comprehensive} ahead, our general
existence result, Theorem \ref{th:exist}, applies to the viscous delamination
system. Hence, for every pair of initial data
$(u_0,z_0) \in \rmH_{\GDir}^1(\Omega;\R^3) \ti \rmH^1(\GC)$ there exists a
solution
\begin{equation}
\label{regularity-viscous-solutions-delam}
u   \in \rmH^1(0,T;\rmH_{\GDir}^1(\Omega;\R^3))  \text{ and } 
  z\in  \rmL^\infty(0,T;\rmH^1(\GC)) \cap \rmH^1(0,T;\rmL^2(\GC)),
\end{equation}
to the Cauchy problem for system  \eqref{eq:DelamSyst}.

\Subsection{The vanishing-viscosity limit}
\label{ss:10.0}
We will now address the vanishing-viscosity limit as $\eps\to 0^+$ of system
\eqref{eq:DelamSyst}.  Our main result states the convergence of (a selected
family of) viscous solutions to an \emph{enhanced} $\BV$ solution to the system
$\RIS$ defined by \eqref{setup-adh-cont}, in fact enjoying the additional
regularity $z\in \BV ([0,T];\Spx)$ (with $\Spx = \rm\rmH^1(\GC)$).
Analogously, we also obtain the existence of parametrized $\BV$ solutions, to
which Theorem \ref{thm:diff-charact} applies, providing a characterization in
terms of system \eqref{diff-charact-delam} ahead.

In fact, we will be able to obtain solutions to the viscous system
\eqref{eq:DelamSyst} enjoying estimates, uniform with respect to the viscosity
parameter, suitable for the vanishing-viscosity analysis, only by performing
calculations on a version of system \eqref{eq:DelamSyst} in which the functions
$\widehat\beta$, $\gamma$, $\widehat \phi$, and $\psi$ are suitably smoothened,
cf.\ \eqref{eq:DelAss01.b}.  That is why, in Theorem \ref{thm:BV-adh-cont}
below will state:
\begin{itemize}
\item[\emph{(i)}] the existence of \emph{qualified} viscous solutions to (the
  Cauchy problem for) \eqref{eq:DelamSyst}, where by `qualified' we mean
  enjoying estimates \eqref{qualified-estimates} below;
\item[\emph{(ii)}] their convergence (up to a subsequence) to an enhanced $\BV$
  solution (we mention that, since the viscous dissipation potentials from
  \eqref{disv-adhc} are both $2$-homogeneous, the formulas in
  \eqref{decomposition-M-FUNCTION} and \eqref{RJMF-p-hom} yield an explicit
  representation formula for the functional $\mename 0\alpha$ involved in the
  definition of $\BV$ solution);
\item[\emph{(iii)}] the convergence of reparametrized (\emph{qualified})
  viscous solutions to an enhanced $\pBV$ solution for which the differential
  characterization from Theorem \ref{thm:diff-charact} holds.
\end{itemize} 
 
For simplicity, in Theorem \ref{thm:BV-adh-cont} we shall not consider a sequence
of initial data $(u_0^\eps,z_0^\eps)_{\eps}$ but confine the statement to the
case of \emph{fixed} data $(u_0,z_0)$.  We will impose that $(u_0,z_0)$ fulfill
the additional `compatibility condition' \eqref{eq:Del.IniCompati}.

\begin{theorem}
\label{thm:BV-adh-cont}
Assume conditions \eqref{bold-force-F} and \eqref{eq:Del.Ass02General}.  Let
$(u_0,z_0) \in \Spu \times \Spx$ fulfill
\begin{equation}
  \label{eq:Del.IniCompati}
  u_0\in \rmH_{\GDir}^1(\Omega;\R^3),\qquad \Delta z_0 \in \rmL^2(\GC), \quad
  \pl\wh\phi(z_0) \cap  \rmL^2(\GC) \neq \emptyset. 
\end{equation}

Then,   there exists a family  
\begin{equation}
\label{additional-regularity}
(u_\eps,z_\eps)_\eps  \subset  \rmH^1(0,T;\rmH_{\GDir}^1(\Omega;\R^3)) \ti \rmH^1(0,T;\rmH^1(\GC))
\end{equation}
solving the Cauchy problem for the 
 viscous delamination system \eqref{regularity-viscous-solutions-delam} 
 with the initial data $(u_0,z_0)$, and enjoying the following estimate
 \begin{equation}
 \label{qualified-estimates}
 \sup_{\eps>0} \int_0^T \big\{ \|\dot u_\eps\|_{\rmH^1(\Omega)} + \|\dot z_\eps\|_{\rmH^1(\GC)}\big\}  \dd t  \leq C\,.
 \end{equation}
 
Moreover, for any null sequence $(\eps_k)_k$ the sequence $(u_{\eps_k},z_{\eps_k})_k $
admits  a (not relabeled) subsequence, and there exists
a pair
\[
(u,z) \in \rmB\rmV([0,T]; \rmH_{\GDir}^1(\Omega;\R^3)) \ti \BV ([0,T];\rmH^1(\GC)),
\]
such that 
\begin{enumerate}
\item the following convergences hold as $k\to\infty$
\begin{equation}
\label{pointwise-cvg-delam}
u_{\eps_k}(t) \weakto u(t) \text{ in } \rmH_{\GDir}^1(\Omega;\R^3), \quad 
z_{\eps_k}(t) \weakto z(t) \text{ in } \rmH^1(\GC) \qquad \text{for all } t \in [0,T];
\end{equation}
\item $(u,z)$ is an \emph{enhanced}  $\BV$ solution to  the delamination system $\RIS$
from \eqref{setup-adh-cont}.
\end{enumerate}

Finally, reparametrizing the sequence $(u_{\eps_k},z_{\eps_k})_k $ in such a
way that the rescaled curves $(\sft_{\eps_k}, \sfu_{\eps_k}, \sfz_{\eps_k})_k$
enjoy estimates \eqref{condition-4-normali} and
\eqref{condition-4-normali-enhn}, up to a subsequence we have convergence of
$(\sft_{\eps_k}, \sfu_{\eps_k}, \sfz_{\eps_k})_k$, in the sense of
\eqref{cvs-eps}, to an enhanced $\pBV$ solution
$(\sft,\sfu,\sfz): [0,\mathsf{S}] \to [0,T]\ti \rmH_{\GDir}^1(\Omega;\R^3)\ti
\rmH^1(\GC)$ for which the differential characterization from Theorem
\ref{thm:diff-charact} holds. Namely, there exist measurable functions
$\thn u,\, \thn z: (0,\sfS)\to [0,\infty] $ satisfying for almost all
$s\in (0,\sfS)$ the switching conditions \eqref{eq:SwitchCond} and the
subdifferential inclusions
\begin{subequations}
  \label{diff-charact-delam} 
\begin{align}
\label{diff-charact-delam.a}
0&  \in \thn u(s) \bsD \dot{\sfu}(s) + \bsC \sfu(s) + \bsJ^*\big(\beta(\JUMP{\sfu(s)}) +
 \gamma(\sfz(s)) \pl  \psi(\JUMP{\sfu(s)}) \big) - \ell_u(\sft(s)) &&  \text{in } \rmH^1(\Omega;\R^3)^*  \\
\label{diff-charact-delam.b}
0&\in \pl  \mathrm{R}(\dot{\sfz}(s)) + \thn z(s) \dot{\sfz}(s)  + \bsA \sfz(s) +  \pl  \wh\phi(\sfz(s)) +
 \pl  \gamma(\sfz(s))\psi(\JUMP{\sfu(s)}) &&  \aein\, \GC
\end{align}
\end{subequations} 
(with convention \eqref{convention-recall} in the case $\thn x(s) = \infty$). 
\end{theorem}
\begin{proof}
It is  sufficient to check that the rate-independent system
$\RIS$ from \eqref{setup-adh-cont} complies with the assumptions of Theorems
\ref{thm:exist-enh-pBV},  \ref{thm:diff-charact}, and 
\ref{thm:exist-nonpar-enh}, and that there exist `qualified' viscous solutions enjoying 
estimates \eqref{qualified-estimates}. More precisely,
\begin{enumerate}
\item In Proposition \ref{l:comprehensive} ahead we will check that the
  rate-independent delamination system $\RIS$ complies with Hypotheses
  \ref{hyp:setup}, \ref{hyp:diss-basic}, \ref{hyp:1}, \ref{h:closedness},
  \ref{hyp:Sept19}, and \ref{h:ch-rule-param} (in fact, the parametrized chain
  rule \eqref{better-chain-rule-MOexpl} holds).
\item We will obtain the existence of viscous solutions enjoying estimates
  \eqref{qualified-estimates} by working on a smoothened version of system
  \eqref{eq:DelamSyst}, introduced in Section \ref{su:DelamSmooth}
  ahead. Therein, we will obtain estimates for the solutions to the regularized
  viscous system \emph{uniform} with respect to the regularization
  parameter. Hence, with Proposition \ref{pr:Del.ViscSolImprov} in Section
  \ref{su:Dela.ExiApriGener} we will conclude the existence of `qualified'
  solutions for which \eqref{qualified-estimates} holds, and thereby conclude
  the proof of Theorem \ref{thm:BV-adh-cont}.
\end{enumerate}
\end{proof}

In what follows, we will most often use the place-holders $\Spu$, $\Spz$,
... (cf.\ \eqref{spaces-adh-cont} and \eqref{coercivity-spaces-delam}) for the
involved function spaces.

\Subsection{Properties of the rate-independent system for delamination}
\label{ss:prop-comprehensive}

This section is centered around Proposition \ref{l:comprehensive} below, in
which we check the rate-independent system $\RIS$ from \eqref{setup-adh-cont}.
complies with the `abstract' Hypotheses from Section \ref{s:setup}.  In
particular, from the following result we gather that Theorem \ref{th:exist} is
applicable, yielding the existence of solutions as in
\eqref{regularity-viscous-solutions-delam} to the viscous delamination system.

\begin{proposition}
  \label{l:comprehensive}
  Assume \eqref{bold-force-F} and \eqref{eq:Del.Ass02General}. Then, the
  delamination system $\RIS$ from \eqref{setup-adh-cont} fulfills Hypotheses
  \ref{hyp:setup}, \ref{hyp:diss-basic}, \ref{hyp:1}, \ref{h:closedness},
  \ref{hyp:Sept19}, and \ref{h:ch-rule-param}.
\end{proposition}
\begin{proof}  The proof consists of three steps. 

\noindent
\STEP{1. Hypotheses \ref{hyp:setup}, \ref{hyp:diss-basic}, \ref{hyp:1}, and
  \ref{hyp:Sept19}:} 
The validity of Hypothesis \ref{hyp:setup},
\ref{hyp:diss-basic}, and \ref{hyp:1} is obvious.  A straightforward
calculation shows that the Fr\'echet subdifferential of $\calE$ is given by
\eqref{Frsub-adh-cont}, so that the structure condition
$\frsubq qtq = \frsubq utq \ti \frsubq ztq$ holds at every
$q=(u,z) \in \Spu\times \Spz$. Therefore, by Lemma \ref{l.4.13}, Hypothesis
\ref{hyp:Sept19} will be ensured by the validity of Hypothesis \ref{h:closedness},
which we now check.

\noindent
\STEP{2. Hypothesis \ref{h:closedness}:} Let $(t_n)_n \subset [0,T]$ and
$(u_n,z_n)_n\subset \Spu \ti \Spz$ be in the conditions of Hypothesis
\ref{h:closedness}, and let   $(\mu_n,\zeta_n)_n$, with
$\mu_n   \in \frsub u{t_n}{u_n}{z_n} $ and
$\zeta_n \in \frsub z{t_n}{u_n}{z_n} $, fulfill  $\mu_n \weakto \mu $  in $\Spu^*$
and $\zeta_n\weakto \zeta $ in $\Spz^*$. Hence,
\[
  \begin{aligned}
    &  \mu_n  = \bsC u_n +\bsJ^* (\beta(\JUMP{u_n}){+}\gamma(z_n)\varrho_n)
    -\ell_u(t_n) \qquad \text{with } \varrho_n\in \pl \psi(\JUMP{u_n}) 
    \text{ a.e.\ in } \GC,
    \\
    & \zeta_n = \bsA z_n + \omega_n \psi(\JUMP{u_n}) + \phi_n \text{ for some }
    \omega_n\in \pl \gamma(z_n) \text{ and } \phi_n \in \pl 
    \widehat{\phi}(z_n) \,.
  \end{aligned}
\]   
  
We observe that, by Sobolev embeddings and trace theorems, from the
convergences $u_n\weakto u $ in $\Spw$ and $z_n\weakto z$ in $\Spx $ we
infer that $\JUMP{u_n}\to \JUMP{u}$ in $\rmL^{q}(\GC;\R^3)$ for all
$1\leq q<4$, and $z_n\to z $ in $\rmL^p(\GC)$ for all $1\leq
p<\infty$. Furthermore, since $0 \leq z_n \leq 1 $ a.e.\ on $\GC$, we even have
$z_n\weaksto z$ in $\rmL^\infty(\GC)$.  Since $\gamma$ is Lipschitz, we gather
that $\gamma(z_n) \to \gamma(z)$ in $\rmL^p(\GC)$ for all $1\leq p<\infty$,
too.  By the growth properties of $\psi$, we have that the sequence
$(\varrho_n)_n$ with $\varrho_n \in \pl \psi(\JUMP{u_n})$ a.e.\ in $\GC$ is
bounded in $\rmL^4(\GC)$ and thus, up to a subsequence, it weakly converges in
$\rmL^4(\GC)$ to some $\varrho$.  By the strong-weak closedness of the graph of
$\pl \psi$ (or, rather, of the maximal monotone operator that
$\pl  \psi: \R^3 \rightrightarrows \R^3$ induces on $\rmL^{2}(\GC)$), we
have that $\varrho \in \pl \psi(\JUMP u)$ a.e.\ in $\GC$.  Moreover, we
find that $\gamma(z_n)\varrho_n \weakto \gamma(z) \varrho$, for instance
in $\rmL^2(\GC)$.  Since $\beta $ is Lipschitz, we also have
$\beta(\JUMP{u_n})\to \beta(\JUMP{u})$ in $\rmL^{q}(\GC;\R^3)$ for all
$ 1\leq q<4$. Also taking into account that $\ell_u \in \rmC^1([0,T];\Spu^*)$,
we then conclude the weak limit  $\mu$ of the sequence $(\mu_n)_n$  belongs to
$\frsubq utq$.
  
Let us now discuss the weak $\Spz$-limit $\zeta$ of the sequence
$(\zeta_n)_n$. First of all, from the Lipschitz continuity of $\gamma$ we
gather that the sequence $(\omega_n)_n$ is bounded in $\rmL^\infty (\GC)$.
Hence, $ (\omega_n \psi(\JUMP{u_n}) )_n$ is bounded in $\rmL^2(\GC)$ and, a
fortiori, we gather that also the terms $ (\bsA z_n + \phi_n)_n$ are bounded in
$\rmL^2(\GC)$.  By the strong-weak closedness of the graph of (the operator
induced by) $\pl  \gamma$ (on $\rmL^2(\GC)$), we infer that
$\omega \in \pl  \gamma(z)$ a.e. in $\GC$. Since $\psi$ has at most
quadratic growth, from $\JUMP{u_n}\to \JUMP{u}$ in $\rmL^{q}(\GC;\R^3)$ for all
$1\leq q<4$ we obtain that $\psi(\JUMP{u_n})\to \psi(\JUMP{u})$ in
$\rmL^{q/2}(\GC;\R^3)$ via the dominated convergence theorem. All in all, we
have that
$\omega_n \psi(\JUMP{u_n}) \weakto \omega \psi(\JUMP{u}) \in \psi(\JUMP{u})
\pl \gamma(z)$ in $\rmL^2(\GC)$.
  
Now, with arguments similar to those in the previous lines and that in
$\rmL^2(\GC)$, since $\omega_n \weaksto \omega$ in $\rmL^\infty(\GC)$ and
$\psi(\JUMP{u_n}) \to \psi(\JUMP{u})$ in $\rmL^{q/2}(\GC)$ for all $1\leq
q<4$. In turn, also taking into account that the sequence $(\bsA z_n)_n$ is
itself bounded in $\rmH^1(\GC)^*$, it is immediate to check that there exists
$\phi \in \rmH^1(\GC)^*$ such that $\bsA z_n + \phi_n\weakto \bsA z + \phi$ in
$\rmL^2(\GC)$.  Since the functional $\calF$ from \eqref{calF-stuff} $\calF$ is
also $({-\Lambda_\phi})$-convex, its Fr\'echet subdifferential has a
strongly-weakly closed graph in $\rmL^2(\GC) \ti \rmL^2(\GC)$, and we thus
infer that $\phi \in \pl  \wh \phi(z) $ a.e.\ in $\GC$. All in all, we
conclude that that $\zeta \in \frsubq ztq$.  This concludes the proof of
Hypothesis \ref{h:closedness}.

\par
\noindent
\STEP{3. Hypothesis \ref{h:ch-rule-param}:} Let us now turn to the
parametrized chain rule from Hypothesis \ref{h:ch-rule-param}.  Since the viscous
dissipation potentials are $2$-homogeneous, the associated vanishing-viscosity
contact potentials are given by \eqref{p-homo-mfb} (cf.\ Example \ref{ex:VVCP})
so that, in particular, the coercivity condition \eqref{coercivity-VVCP} holds,
and Proposition \ref{prop:better-chain-rule-MOexpl} is applicable. Therefore,
Hypothesis \ref{h:ch-rule-param} follows from the chain rule of Hyp.\
\ref{h:ch-rule}.  The latter chain-rule property can be verified by resorting
to Proposition \ref{prop:ch-ruleApp} ahead. Hence, we need to show $\calE$
complies with condition \eqref{uniform-subdiff}, which states that the
Fr\'echet subdifferential $\pl_q \calE$ can be characterized by a
\emph{global} inequality akin to that defining the convex analysis
subdifferential: for every $E>0$, and energy sublevel $\subl E$, there exists
an upper semicontinuous function
$\varpi^E : [0,T]\ti \subl E \ti \subl E\to [0,\infty]$, with
$\varpi^E(t,q,q) =0 $ for every $t \in [0,T]$ and $ q \in \subl E$, such that
\begin{equation}
\label{modulus-subdif-delam}
\eneq t{\hat q} - \eneq tq  \geq \pairing{}{\Spq}{\xi}{\hat{q}{-}q} - \varpi^E(t,q,\hat{q}) \| \hat{q}{-} q\|_{\Spq}
\ \text{for all } t \in [0,T], \  q,\, \hat{q} \in \subl E \text{ and all } \xi \in \pl_q \calE(t, q)\,.
\end{equation}
In order to check \eqref{modulus-subdif-delam},
we will resort to a decomposition for
 the energy functional from \eqref{energy-delamination}  as 
\begin{subequations}
\label{decomposition-energy}
\begin{align}
&
\ene tuz = \calE_{\mathrm{elast}}(t,u) + \calF(z) + \calE_{\mathrm{coupl}}(u,z)
\intertext{with $\calF$ from \eqref{calF-stuff}, }
&
\calE_{\mathrm{elast}}(t,u) :=   \tfrac12 \pairing{}{\rmH^1(\Omega)}{\bsC u}u     - \pairing{}{\rmH^1(\Omega)}{\ell_u(t)}{u},
\intertext{while, for later convenience, we encompass the surface term $\int_{\GC}  \widehat{\beta}(\JUMP{u}) \dd x$ in the coupling energy}
&
 \calE_{\mathrm{coupl}}(u,z) := 
       \int_{\GC}    
   \left(  \widehat{\beta}(\JUMP{u})  {+} \gamma(z)
  \psi(\JUMP{u}) \right) \dd x \,.
 \end{align}
 \end{subequations}
Now,
$\calE_{\mathrm{elast}}(t,\cdot)$ is convex  while  $\calF$  is $(-\Lambda_\phi)$-convex. Hence, they both comply with \eqref{modulus-subdif-delam}.
Hence,
 it is sufficient to check its validity for  $\calE_{\mathrm{coupl}}$, and indeed for its second contribution, only, since
 $\wh\beta$ is also convex.
 Indeed, for every $\hat u, \, u  \in \Spu$ and $\hat z,\, z \in \Spz$
 and for all selections $\GC \ni x  \mapsto \varrho(x) \in \pl  \psi(\JUMP{u(x)})$ and $
\GC \ni x \mapsto \omega(x) \in \pl \gamma(z(x)) $ there holds
 \[
 \begin{aligned}
 &
\int_{\GC} \big( \gamma(\hat z)
  \psi(\JUMP{\hat u}) {-}  \gamma(z)
  \psi(\JUMP{u})  \big) \dd x  - \int_{\GC} \gamma(z) \varrho \JUMP{\hat u{-}\hat u} \dd x 
   - \int_{\GC}\omega\psi(\JUMP{u}) (\hat{z}{-}z) \dd x 
   \\
   &
   =
   \int_{\GC} \gamma(\hat z) \big\{ \psi(\JUMP{\hat u}{-}\psi(\JUMP u) \big\} \dd x  - \int_{\GC} \gamma(z) \varrho \JUMP{\hat u{-}\hat u} \dd x 
   + \int_{\GC} \big\{ \gamma(\hat z){-}\gamma(z) {-} \omega (\hat{z}{-}z) \big\}  \psi(\JUMP u)  \dd x 
   \\
   & \stackrel{(1)}{\geq} 
    \int_{\GC} \big(\gamma(\hat z){-} \gamma(z) \big)  \big( \psi(\JUMP{\hat u}){-}\psi(\JUMP u) \big) \dd x
    + \int_{\GC} \gamma(z) \big\{ \psi(\JUMP{\hat u}){-}\psi(\JUMP u) {-} \varrho \JUMP{\hat u{-} u}  \big\} \dd x 
    \\
   & \stackrel{(2)}{\geq} 
    \int_{\GC} \big(\gamma(\hat z){-} \gamma(z) \big)  \big( \psi(\JUMP{\hat u}){-}\psi(\JUMP u) \big) \dd x
    \\
    & \stackrel{(3)}{\geq}  - \| \hat {z}{-} z\|_{\rmL^2(\GC)} \| \psi(\JUMP{\hat u}){-}\psi(\JUMP u)\|_{\rmL^2(\GC)}
   \end{aligned}
 \]
 where {\footnotesize (1)} \& {\footnotesize (2)}   follow from the convexity of $\gamma$  and $\psi$, respectively, whereas {\footnotesize (3)} is due to the $1$-Lipschitz continuity of $\gamma$. Then, estimate  \eqref{modulus-subdif-delam} follows with the function $\varpi_t^E(\hat q, q): =  \| \psi(\JUMP{\hat u}){-}\psi(\JUMP u)\|_{\rmL^2(\GC)}$. 
 We have thus checked the validity of \eqref{modulus-subdif-delam} and, a fortiori, of 
 Hypothesis  \ref{h:ch-rule-param}.  This concludes the proof. 
\end{proof}

\begin{remark}\slshape
\label{rmk:why-no-indicator}
The Lipschitz continuity of $\beta$ has played a key role in the proof that
$\calE$ complies with the closedness condition from Hyp.\
\ref{h:closedness}. In fact, we could allow for a nonsmooth $\widehat \beta$,
but with a suitable polynomial growth condition, that would still ensure that
the maximal monotone operator induced by $\beta = \pl \wh\beta$ on
$\rmL^2(\GC)$ is strongly-weakly closed. However, it would not be possible to
check Hypothesis \ref{h:closedness} in the case $\beta$ is an unbounded maximal
monotone operator, such as the subdifferential of an indicator function. That
is why, we are not in a position to encompass in our analysis the
non-interpenetration constraint between the two bodies $\Omega^+$ and
$\Omega^-$.
\end{remark}

\Subsection{A priori estimates for the smooth semilinear  system}
\label{su:DelamSmooth}

In this section we address a version of the viscous system \eqref{eq:DelamSyst}
in which the functions $\widehat\beta$, $\gamma$, $\widehat \phi$, and $\psi$,
complying with \eqref{eq:Del.Ass02General}, are also smoothened. Namely, we
will additionally suppose that they fulfill
\begin{align}
 \label{eq:DelAss01.b}
  &  \left.\begin{aligned}  &
   \gamma,\,\wh\phi \in \rmC^2(\R;\R), \quad 
     \psi,\,\wh\beta \in \rmC^2(\R^3;\R),
     \\ &
    \gamma'',\, \wh\phi'',\, \rmD^2\wh\beta \text{ are bounded},\quad 
     |\rmD\psi(a)|\leq C_\psi^{(1)} \text{ for all }a\in \R^3 .
   \end{aligned} \right\}
\end{align} 
These conditions will allow us to \emph{rigorously} perform, on the solutions
to system \eqref{eq:DelamSyst}, calculations that will ultimately lead to
bounds, uniform with respect to viscosity parameter, suitable for our
vanishing-viscosity analysis.  Such estimates will however only depend on the
constants occurring in \eqref{eq:Del.Ass02General}, and not on those in
\eqref{eq:DelAss01.b}.  For these calculations we will crucially make use of
the \emph{semilinear} structure of this regularized system and of the fact that
the coupling between the displacement equation and the flow rule for the
delamination parameter is weak enough to allow us to treat those equations
separately.

As already mentioned, for all $\eps \in (0,1)$ and all initial data
$(u_0,z_0)\in \Spu\ti \Spx$ system \eqref{eq:DelamSyst} admits finite-energy
solutions $(u_\eps,z_\eps)$ with the standard time regularity
\eqref{regularity-viscous-solutions-delam}.  We now aim to derive higher
order estimates as well, and to show that these estimates are independent of
$\eps$. We will make them as explicit as possible. Let us mention in advance
that one crucial argument involves the interpolation between the different
norms for the time derivative $\dot z$, namely
\begin{equation}
  \label{eq:Interpol.dotz}
\forall \, \dot z \in \Spx:\quad \|\dot z\|_\Spz \leq C_\text{GN} \calR(\dot z)^{1/2}
\| \dot z\|_{\Spx}^{1/2}.   
\end{equation}
Indeed, \eqref{eq:Interpol.dotz} follows by combining the lower bound $\calR(v)\geq
c_R\|v\|_{\rmL^1}$ with the classical Gagliardo-Nirenberg interpolation
$\|v\|_{\rmL^2}^2 \leq C \|v\|_{\rmL^1}  \|v\|_{\rm\rmH^1}$. This will allows us
to exploit the $\eps$-independent dissipation estimate $\int_0^T \calR(\dot
z_\eps)\dd t \leq C$. 
\medskip

\STEP{1. Basic energy and dissipation estimates:} The simple energy-dissipation
estimate stemming from the energy balance \eqref{enid.a} (cf.\ Lemma
\ref{l:1}), together with $\ell_u \in \rmC^1([0,T;\Spu^*)$ implies that
for all $E_0$ there exists $C^{E_0}_1>0$ such all solutions $(u_\eps,z_\eps)$
of \eqref{eq:DelamSyst} with $\calE(0,u_\eps(0),z_\eps(0)\leq E_0$ satisfy the
basic energy estimates
\begin{equation}
  \label{eq:DelamEst01}
  \int_0^T \!\!\big\{ \calR (\dot z_\eps(t)) + \eps^\alpha\|\dot u_\eps(t)\|_\Spu^2
  + \eps \| \dot z_\eps\|_\Spz^2\big\} \dd t \leq C^{E_0}_1 \quad \text{and} \quad
\forall\, t\in [0,T]:\ \|u_\eps(t)\|_\Spu + \| z_\eps(t)\|_{\Spx} \leq C^{E_0}_1 . 
\end{equation}
As a consequence of this a priori bound, of the fact that
$\bsJ: \Spu \to \rmL^4(\GC;\R^3)$ is a bounded operator, and of upper
estimates on $\psi$ via the constants $C_\psi$ and $C^{(1)}_\psi$, we
find another constant $C^{E_0}_2$ such that all solutions $(u_\eps,z_\eps)$ of
\eqref{eq:DelamSyst} with $\calE(0,u_\eps(0),z_\eps(0))\leq E_0$ satisfy
\begin{subequations}
  \label{eq:Del.psi.est}
  \begin{align}
    \label{eq:Del.psi.est.a}
    &\| \psi(\JUMP{u_\eps(t)})\|_{\rmL^2} \leq C_\psi C^{E_0}_2, \quad 
     \| \rmD\psi(\JUMP{u_\eps(t)})\|_{\rmL^4} \leq  C_\psi C^{E_0}_2,
  \\
  \label{eq:Del.psi.est.b}
  &\| \psi(\JUMP{u_\eps(t)})\|_{\rmL^4}\leq C^{(1)}_\psi C^{E_0}_2, \quad 
   \| \rmD\psi(\JUMP{u_\eps(t)})\|_{\rmL^\infty}\leq C^{(1)}_\psi C^{E_0}_2. \medskip
\end{align} 
\end{subequations}
Estimate \eqref{eq:Del.psi.est.b} will in fact only be used for gaining enhanced regularity of  the viscous solutions $(u_\eps,z_\eps)$, and not for the vanishing-viscosity analysis. 

\STEP{2. Estimate for $\dot u_\eps$:} Because of the smoothness of $\wh\beta$
and $\psi$, the  displacement equation \eqref{eq:DelamSyst.a} for $u_\eps$ is a semilinear
equation with a smooth nonlinearity, if we consider
$z_\eps \in \rm\rmH^1(0,T;\Spz)$ as \emph{given} datum. Thus, we can use the classical
technique of difference quotients to show that $u_\eps \in \rm\rmH^2(0,T;\Spu)$
provided that $\dot u_\eps(0)= \eps^{-\alpha}\bsD^{-1}\big(\bsC
u(0)+\bsJ^*(\cdots)-\ell_u(0)\big) \in \Spu$. Hence, it is
possible to differentiate 
\eqref{eq:DelamSyst.a}
 with respect to time, which yields
\begin{equation}
  \label{eq:DelamEqn.dotu}
  0=\eps^\alpha \bsD \ddot u_\eps + \bsC \dot u_\eps + \bsJ^*\Big(\rmD^2 \wh\beta 
  (\JUMP{u_\eps})\JUMP{\dot u_\eps} 
        + \gamma(z_\eps) \rmD^2\psi (\JUMP{u_\eps})\JUMP{\dot u_\eps} 
   + \gamma'(z_\eps)\dot z_\eps \rmD\psi(\JUMP{u_\eps})\Big) - \dot \ell_u(t).
\end{equation}
We can test \eqref{eq:DelamEqn.dotu} by  $\dot u_\eps \in \rm\rmH^1(0,T;\Spu)$  and obtain
\begin{align*}
0&=\frac{\eps^\alpha}2 \,\frac{\rmd}{\rmd t}\langle \bsD \dot u_\eps, 
                      \dot u_\eps\rangle_\Spu 
  +  \langle \bsC \dot u_\eps,\dot u_\eps\rangle_\Spu   +  \langle \rmD^2
  \wh\beta (\JUMP{u_\eps})\JUMP{\dot u_\eps}  
          {+} \gamma(z_\eps) \rmD^2\psi (\JUMP{u_\eps})\JUMP{\dot u_\eps}  , 
          \JUMP{\dot u_\eps}\rangle_\Spz\\
&\quad -\langle \dot \ell_u,\dot u_\eps\rangle_\Spu - 
   \langle \gamma'(z_\eps)\dot z_\eps \rmD\psi(\JUMP{u_\eps}) , \JUMP{\dot
     u_\eps}\rangle_\Spz 
\end{align*}
Here the last duality product in the first line  is nonnegative, because
$a\mapsto \wh\beta(a) + \gamma(z)\psi(a)$ is convex. The last duality product
can be estimated using \eqref{eq:Del.psi.est.a}.
Defining  $\bftheta^\eps_\Spu$, $\bftheta^\eps_\Spz$, and $\lambda_{\Spu^*}$  via 
\[
\bftheta^\eps_\Spu (t)^2:=\langle \bsD \dot u_\eps(t),\dot u_\eps(t)\rangle_\Spu
\quad \text{and} \quad 
\bftheta^\eps_\Spz (t)^2:=\| \dot z_\eps(t)\|_{\Spz}^2 , \quad \text{and }
\lambda_{\Spu^*}(t) = \|\dot \ell_u(t)\|_{\Spu^*}, 
\]
we have established the estimate 
\[
\frac{\eps^\alpha}2\,\frac{\rmd}{\rmd t}(\bftheta^\eps_\Spu)^2 +
c_{\bsC} (\bftheta^\eps_\Spu)^2 \leq \lambda_{\Spu^*} \bftheta^\eps_\Spu 
 + 1 \, C_\psi  C^{E_0}_2 C_{\rmH^1,\rmL^4} \|\bsJ\| \bftheta^\eps_\Spz
  \bftheta^\eps_\Spu
\] 
where we have also used that $\gamma$ is $1$-Lipschitz continuous, and
$ C_{\rm\rmH^1,\rmL^4} $ denotes the constant associated with the continuous
embedding $\Spu \subset \rmL^4(\GC;\R^3)$. Using
$\frac{\rmd}{\rmd t}(\bftheta^\eps_\Spu)^2= 2 \bftheta^\eps_\Spu
\,\dot\bftheta^\eps_\Spu$ we can divide by $ \bftheta^\eps_\Spu\geq 0$ and
obtain
\begin{equation}
  \label{eq:Delam.Est02}
  \eps^\alpha \dot\bftheta^\eps_\Spu + c_{\bsC}\bftheta^\eps_\Spu \leq
  \lambda_{\Spu^*} + C_\psi  C^{E_0}_2 C_{\rm\rmH^1,\rmL^4}  \|\bsJ\|\,  \bftheta^\eps_\Spz.\medskip
\end{equation}
Let us mention that the above estimate could be rigorously obtained by
replacing $ \bftheta^\eps_\Spu$ by $\sqrt{ (\bftheta^\eps_\Spu)^2 +\delta}$,
which satisfies the same estimate, and then letting $\delta \down 0$, cf.\
\cite[Sec.\,4.4]{Miel11DEMF}.

\STEP{3. Uniqueness and higher regularity of $\dot z_\eps$:} We first observe
that \emph{given} $u_\eps \in \rm\rmH^1([0,T];\Spu)$ and $z_0$ there is a unique solution
$z_\eps$ for \eqref{eq:DelamSyst.b}. Indeed, assuming that $z_1$ and $z_2$ are
solutions (with $\varrho_j \in \pl\rmR(\dot z_j))$ we set  $w=z_1{-}z_2$ and
test the difference of the two equations by $\dot w=\dot z_1{-}\dot z_2$,
which yields
\begin{align}
  \label{eq:Del.Uniquen}
0=& \pairing{}{\Spz}{
  \varrho_1{-}\varrho_2}{\dot z_1{-}\dot z_2}  +\eps\|\dot w\|_\Spz^2  
  + \frac12\,\frac\rmd{\rmd t} \pairing{}{\Spx}{\bsA w}{w}  + 
  \pairing{}{\Spz}{G(u_\eps,z_1)-G(u_\eps,z_2)}{\dot w}, 
\end{align}    
  where  we have set $G(u,z)=\wh\phi'(z) + \gamma'(z)\psi(\JUMP u)$.  By our strengthened
assumptions \eqref{eq:DelAss01.b} 
and Gagliardo-Nirenberg interpolation 
we
have 
\begin{align*}
\| G(u_\eps,z_1)-G(u_\eps,z_2)\|_{\Spz^*} &\leq \|\wh\phi''\|_\infty \| z_1{-}z_2\|_\Spz +
 \|\gamma'(z_1){-} \gamma'(z_2)\|_{\rmL^4}\|\psi(\JUMP u_\eps) \|_{\rmL^4} 
\\
& \leq \big( \|\wh\phi''\|_\infty \| z_1{-}z_2\|_\Spz + \| \gamma''\|_{\infty} \|z_1{-}z_2\|_{\rmL^4}  
 C^{(1)}_\psi C^{E_0}_2 \big)\leq  C_G  \|w\|_\Spz^{1/2} \| w\|_{\Spx}^{1/2}.  
\end{align*}

By using the monotonicity of $\pl\rmR$, the first term in \eqref{eq:Del.Uniquen} is nonnegative. Using
$\|w\|_{\Spx}^2 =\pairing{}{\rmH^1(\GC)}{\bsA w}{w} $ we obtain
\[
 \frac12\,\frac\rmd{\rmd t} \|w\|_{\Spx}^2 + \eps\|\dot w\|_\Spz^2
 \leq  C_G \|w\|_\Spz^{1/2} \| w\|_{\Spx}^{1/2}\|\dot w\|_\Spz \leq 
\frac{C_G^2}{4\eps}  \|w\|_\Spz\| w\|_{\Spx} + \eps\|\dot w\|_\Spz^2.
\]
Canceling the terms $\eps\|\dot w\|_\Spz^2$ and using $\|w\|_\Spz\leq \|
w\|_{\Spx}$ provides the estimate 
\begin{equation}
\label{constant-CF}
 \| z_1(t){-}z_2(t)\|_{\Spx} \leq  \rme^{{C_G}^2(t-s)/(4\eps)} 
  \| z_1(s){-}z_2(s)\|_{\Spx} \quad \text{ for } 0\leq s\leq t \leq T.
\end{equation}
We emphasize that this uniqueness result is special and relies
strongly on the semilinear structure of the flow rule for $z$ under the
strengthened assumption \eqref{eq:DelAss01.b}. It is indeed thanks to
\eqref{eq:DelAss01.b} that $G(u,\cdot):\Spx \to \Spz^*$ is globally
Lipschitz, and in fact the constant $C_G$ in \eqref{constant-CF} does
depend on $ C^{(1)}_\psi $.

This uniqueness is central to derive higher regularity as it is now possible to
use suitable regularizations such as Galerkin approximations or replacing the
nonsmooth function $\rmR$ by a smoothed version $\rmR_\delta$. We do not go
into detail here, but refer to \cite{Mielke-Zelik} and
\cite[Sec.\,4.4]{Miel11DEMF}. In particular, our problem fits exactly into the
abstract setting of \cite[Sec.\,3]{Mielke-Zelik} with $H=\Spz=\rmL^2(\GC)$,
$\calB=\bsA$, and
$\Phi(t,z)= \int_\Omega \big(\wh\phi(z)+ \gamma(z) \psi(\JUMP{u(t)}\big) \dd
x$. 

Thus, under the additional condition $\bsA z_0\in \Spz$ (or $z_0\in
\rm\rmH^2(\GC)$), the unique solution $z_\eps$ with $z_\eps(0)=z_0$ satisfies 
the following higher regularity properties:
\begin{equation}
  \label{eq:Del.HighRegul}
  \dot z_\eps \in \rmL^\infty(0,T;\Spx) \quad \text{and} \quad 
\sqrt {t\,}\, \ddot z_\eps \in \rmL^2(0,T;\Spz).
\end{equation}
Of course, at this stage we have no control over the dependence on $\eps$ of
the corresponding norms.\medskip 

\STEP{4. Identities not involving $\rmR$:} Surprisingly, there
are two identities for the solution $z_\eps$ that are completely independent of
$\rmR$, i.e.\ they look like energy estimates for a semilinear parabolic
problem:
\begin{subequations}
  \label{eq:Dela.Ident}
\begin{align}
  \label{eq:Dela.Ident.B} 
\frac\eps2\,\frac\rmd{\rmd t} \| \dot z_\eps\|_\Spz^2 + \| \dot z_\eps\|_{\Spx}^2 +  
\pairing{}{\Spz}{\rmD^2_z \Phi(u_\eps,z_\eps) \dot z_\eps}{\dot z_\eps} + \pairing{}{\Spz}{
\rmD_z\rmD_u \Phi(u_\eps,z_\eps)\dot u_\eps}{\dot z_\eps} &=0,
\\
  \label{eq:Dela.Ident.C}
\eps \| \ddot z_\eps\|_\Spz^2 + \frac12\,\frac\rmd{\rmd t} \| \dot
z_\eps\|_{\Spx}^2  + \pairing{}{\Spz}{\rmD^2_z
\Phi(u_\eps,z_\eps) \dot z_\eps}{\ddot z_\eps} + \pairing{}{\Spz}{\rmD_z\rmD_u
\Phi(u_\eps,z_\eps)\dot u_\eps}{\ddot z_\eps} &\leq 0. 
\end{align}
\end{subequations}
We refer to \cite[Eqn.\,(95) and Lem.\,4.16]{Miel11DEMF} for a rigorous
derivation based on the smoothness established in \eqref{eq:Del.HighRegul}.
Relations \eqref{eq:Dela.Ident}
can be formally derived from  equation \eqref{eq:DelamSyst.b} by forgetting the
nonsmooth term $\pl\rmR$, then  differentiating the whole equation
with respect to $t$, and finally testing with $\dot z_\eps$ or $\ddot z_\eps$
respectively. Indeed, \eqref{eq:Dela.Ident.C} will not be used below any more,
but its relevance is obvious by comparison  with \eqref{eq:Del.Uniquen} and for
deriving the ($\eps$-dependent) a priori estimate for $\sqrt{t\,}\,\ddot
z_\eps$ (via Galerkin approximations). 

It is the identity \eqref{eq:Dela.Ident.B} that will be crucial for deriving
$\eps$-independent a priori estimates. It origin can formally understood by
looking at general smooth $p$-homogeneous dissipation potentials $\bfPsi$ (i.e.\
fulfilling 
$\bfPsi(\gamma v)=\gamma^p\bfPsi(v)$ for all $v$ and $\gamma>0$). Then, Euler's
formula gives $\langle \rmD \bfPsi(v),v\rangle = p \bfPsi(v) $, and we find the identity
\[
\big\langle \frac{\rmd}{\rmd t} \big( \rmD\bfPsi(\dot z)\big),\dot z \big\rangle 
=\rmD^2\bfPsi(\dot z)[\ddot z,\dot z]= \frac{\rmd}{\rmd t}\big( \langle
\rmD\bfPsi(\dot z), \dot z\rangle- \bfPsi(\dot z)\big) = (p{-} 1)\, \frac{\rmd}{\rmd
  t}\bfPsi(\dot z) .
\]
The quadratic case $p=2$ was applied above several times. Of course, in the  case 
$p=1$ the potential $\calR$ is nonsmooth. Hence, the proof in
\cite[Lem.\,4.16]{Miel11DEMF} is different and uses simple arguments
based on the characterization of $\pl  \calR$ in the
$1$-homogeneous  case.\medskip

\STEP{5. $\rmL^1$ estimates for $\bftheta^\eps_\Spu$, $\bftheta^\eps_\Spz$, and
  $\bftheta^\eps_{\Spx}$:} 
In \eqref{eq:Dela.Ident.B}  the coupling term $ \langle
\rmD_z\rmD_u \Phi(u_\eps,z_\eps)\dot u_\eps, \dot z_\eps\rangle$ can be
estimated via the weaker assumption \eqref{eq:Del.Ass02General}, namely
\[
\begin{aligned}
\pairing{}{\Spz}
{\rmD_z\rmD_u \Phi(u_\eps,z_\eps)\dot u_\eps}{\dot z_\eps} &  \leq 1 \| \dot
z_\eps\|_\Spz \| \rmD\psi(\JUMP{u_\eps})\|_{\rmL^4} \| \JUMP{\dot
  u_\eps})\|_{\rmL^4}
  \\
   &  \leq 
   C_3
   \, \bftheta^\eps_\Spz(t)
  \,\bftheta^\eps_\Spu(t) \text{ with } C_3:=C_\psi C^{E_0}_2 C_{\rm\rmH^1,\rmL^4} 
  \|\bsJ\|\,,
  \end{aligned}
\]
where we exploited the $1$-Lipschitz continuity of $\gamma$ and
\eqref{eq:Del.psi.est.a}. Introducing the short-hand notation 
$\bftheta^\eps_{\Spx}$ via
$(\bftheta^\eps_{\Spx}(t))^2=\|\dot z_\eps(t)\|^2_{\Spx} =
\pairing{}{\rmH^1(\GC)}{\bsA \dot z_\eps(t)}{\dot z_\eps(t)}$ and exploiting
the $\Lambda_\phi$-convexity of $\wh\phi$ and the convexity of $\gamma$,
identity \eqref{eq:Dela.Ident.B} yields
\[
\eps \bftheta^\eps_\Spz \dot \bftheta^\eps_\Spz + \big(\bftheta^\eps_{\Spx}\big)^2
\leq \Lambda_\phi \big(\bftheta^\eps_\Spz\big)^2 + C_3 \, \bftheta^\eps_\Spz  \,\bftheta^\eps_\Spu\,.
\]
For the first term on the right-hand side we can now exploit the interpolation
\eqref{eq:Interpol.dotz} and after division by $\bftheta^\eps_\Spz\geq 0$ (recall
$\bftheta^\eps_\Spz \leq \bftheta^\eps_{\Spx}$) we arrive, together with \eqref{eq:Delam.Est02},
at the differential estimates 
\begin{subequations}
  \label{eq:Del.theta.syst}
  \begin{align}
 \label{eq:Del.theta.syst.a}
  \eps^\alpha \dot\bftheta^\eps_\Spu + c_{\bsC}\bftheta^\eps_\Spu &\leq
  \lambda_{\Spu^*} +C_\text{GN}C_3\,  \big(\calR(\dot
  z_\eps) \bftheta^\eps_{\Spx}\big)^{1/2},   
 \\   \label{eq:Del.theta.syst.b}
   \eps \dot \bftheta^\eps_\Spz + \bftheta^\eps_{\Spx}
   &\leq \Lambda_\phi C_\text{GN} \calR(\dot z_\eps) +C_3    \,
   \bftheta^\eps_\Spu\,.
  \end{align}
\end{subequations}
We emphasize that all the appearing coefficients, except for the leading
factors $\eps^\alpha$ and $\eps$, are independent of $\eps \in (0,1)$ and indeed depend only on 
$C_\psi$.  
From the first equation we obtain via the constants-of-variation formula (or
Gr\"onwall's lemma) the estimate
\[
\bftheta^\eps_\Spu(t)\leq K_\eps(t) \eps^\alpha\bftheta^\eps_\Spu(0) + \!\int_0^t\!\!
K_\eps(t{-}s)\big( \lambda_{\Spu^*}(s) {+} C_\text{GN}C_3 \big(\calR(\dot
  z_\eps(s)) \bftheta^\eps_{\Spx}(s)\big)^{1/2} \big)\dd s \ \text{
    with } K_\eps(t)=\frac{\rme^{-c_{\bsC} t/\eps^\alpha}}{\eps^\alpha}. 
\]
Because of $\| K_\eps\|_{\rmL^1} = \int_0^\infty K_\eps(t)\dd t =1/c_{\bsC}$ the
$\rmL^1$-convolution estimate leads to 
\[
I_U:=\int_0^T \!\!\bftheta^\eps_{\Spu}(t) \dd t \leq \frac1{c_{\bsC}} \Big(\eps^\alpha 
  \bftheta^\eps_\Spu(0) + \int_0^T\!\! \lambda_{\Spu^*}(t)\dd t + C_\text{GN}C_3 
 \int_0^T\!\!\big(\calR(\dot z_\eps(t)) \bftheta^\eps_{\Spx}(t)\big)^{1/2} 
  \dd t  \Big) .
\]
Applying the Cauchy-Schwarz inequality to the last integral and integrating 
\eqref{eq:Del.theta.syst.b}  over $[0,T]$ we obtain the estimates 
\begin{align*}
I_\Spu& \leq \frac1{c_{\bsC}} \Big(\eps^\alpha 
  \bftheta^\eps_\Spu(0) +  \int_0^T\!\! \lambda_{\Spu^*}(t)\dd t + C_\text{GN}C_3 \,
 I_R^{1/2} I_{\Spx}^{1/2}  \Big) , 
\\
I_{\Spx}&:=\int_0^T \!\!\bftheta^\eps_{\Spx}(t) \dd t \leq \eps \bftheta^\eps_\Spz(0) + \Lambda_\phi
C_\text{GN}  I_R + C_3 I_\Spu, \quad \text{where } I_R:=\int_0^T\!\! 
 \calR(\dot z_\eps(t)) \dd t. 
\end{align*}
From this it is easy to show that there exists a constant $C_*$, which only
depends on $C_3 =C_\psi C^{E_0}_2 C_{\rm\rmH^1,\rmL^4} 
  \|\bsJ\|$, $c_{\bsC}$, $C_\text{GN}$, and
$\Lambda_\phi$, such that $I_\Spu{+}I_{\Spx}$ can be estimated by
$C_*\big(\eps^\alpha \bftheta^\eps_\Spu(0) + \eps \bftheta^\eps_\Spz(0)
+\int_0^T\lambda_{\Spu^*}\dd t +I_R\big) $.
We have thus proved the following result. 

\begin{lemma}[Rate-independent a priori estimate in the semilinear case]
\label{le:ApriSemiLinCase} 
Assume \eqref{bold-force-F} and \eqref{eq:Del.Ass02General}. Additionally,
let $\widehat\beta$, $\gamma$, $\widehat\phi$, and $\psi$ satisfy
\eqref{eq:DelAss01.b} and let the initial data
$(u_0,z_0) \in \Spu \times \Spx$ comply with \eqref{eq:Del.IniCompati}.  Then,
There exists a constant $C_*>0$, only depending on the initial data and on the
constants $\Lambda_\phi$ and $C_\psi$ from \eqref{eq:Del.Ass02General}, 
such that the unique solution $(u_\eps,z_\eps)$ of \eqref{eq:DelamSyst}
satisfies the a priori estimate
\begin{equation}
  \label{eq:ApriSemLinCase}
  \int_0^T\!\Big( \| \dot u_\eps\|_\Spu+ \| \dot z_\eps\|_{\Spx} \Big) \dd t
  \leq  C_*\Big( \eps^\alpha\|  \dot u_\eps(0)\|_\Spu + \eps \| \dot
  z_\eps\|_{\Spz} + \int_0^T \!\!\big( \| \dot \ell_u\|_{\Spu^*} + \calR(\dot
  z_\eps)\big) \dd t \Big) .  
\end{equation}
\end{lemma} 

\Subsection{Existence and a priori estimates in the general case} 
\label{su:Dela.ExiApriGener}
 
We now return to the setup of Sections \ref{ss:10.-1} and \ref{ss:10.0}, in
which the constitutive functions $\wh \beta$, $\gamma$, $\wh \phi$, and $\psi$
only comply with \eqref{eq:Del.Ass02General}. We exhibit approximations of
$\wh \beta$, $\gamma$, $\wh \phi$, and $\psi$ that also satisfy
\eqref{eq:DelAss01.b}. For this, we will resort to the following general
construction.

\paragraph{\bf Smoothening the Yosida approximation}
Following, e.g., the lines of \cite[Sec.\,3]{Gilardi-Rocca}, for a given convex
function $\widehat{\chi}: \R^d \to [0,\infty]$ with subdifferential
$\chi= \pl \wh\chi: \R^d \rightrightarrows \R^d$, and for a fixed
$\delta\in (0,1)$, we define
\[
\chi^\delta : = \chi_\delta^{\mathrm{Y}} \star  \eta_\delta 
\]
where $ \chi_\delta^{\mathrm{Y}} $ is the Yosida regularization of the maximal
monotone operator $\chi$ (we refer to, e.g., \cite{Brez73OMMS}) and
\begin{equation}
\label{convol-kernel}
\eta_\delta(x): = \tfrac1{\delta^{2}} \eta \left( \tfrac x{\delta^2}\right)
\qquad \text{with }  
\left\{
  \begin{array}{ll}
    \eta \in \rmC^\infty(\R^d), 
    \\
    \|\eta\|_{1} =1,
    \\
    \mathrm{supp}(\eta) \subset B_1(0).
  \end{array}
\right.
\end{equation}
Thus, $\chi^\delta \in \rmC^\infty(\R^d)$ and it has been shown in
\cite{Gilardi-Rocca} that
\begin{subequations}
\label{properties-delta-approx}
\begin{equation}
\label{prop-delta-1}
\|\rmD\chi^\delta\|_{\infty} \leq \frac1\delta,  \qquad  |\chi^\delta(x){-} \chi_\delta^{\mathrm{Y}}(x)| \leq \delta 
\text{ for all } x \in \R^d\,.
\end{equation}
Taking into account the properties of the Yosida we deduce that 
\begin{equation}
\label{prop-delta-1-bis}
|\chi^\delta(x)| \leq  |\chi^o(x)| +\delta  \qquad  \text{ with }  |\chi^o(x)| = \inf\{|y|\, : \, y \in \chi(x) \}\,.
\end{equation}
Furthermore, $\chi^\delta $ admits a convex potential $\widehat\chi^\delta$
satisfying, as a consequence of \eqref{prop-delta-1}, (below
$\wh\chi_\delta^{\mathrm{Y}}$ denotes the Yosida approximation of $\wh \chi$):
\begin{equation}
\label{prop-delta-2}
-\delta |x|  \leq  \wh{\chi}_\delta^{\mathrm{Y}} (x) -\delta |x| \leq
\widehat{\chi}^\delta(x) \leq  \wh\chi_\delta^{\mathrm{Y}} (x) +\delta |x| \leq
\widehat\chi(x)+\delta|x|  \  \text{ and } \  \widehat\chi^\delta(x) \to
\widehat\chi(x) \qquad \text{for all } x \in \R^d\,. 
\end{equation} 
Finally, the following analogue of Minty's trick holds: given $O \subset \R^m$
and sequence $(v_\delta)_\delta\, v,\, \chi \in \rmL^2 (O;\R^d)$ such that
$v_\delta\weakto v$ and $\chi^\delta(v_\delta) \weakto \eta $ in
$ \rmL^2 (O;\R^d)$,
\begin{equation}
  \label{prop-delta-3}
  \limsup_{\delta\to 0^+} \int_O \chi^\delta(v_\delta) \cdot v_\delta \dd x
  \leq \int_O \eta  \cdot  v \quad \Longrightarrow \quad \eta \in
  \pl \widehat{\chi}(v) \text{ a.e.\ in } O. 
\end{equation}
\end{subequations}
 
We apply this construction to $\gamma$, obtaining a smooth approximation
$\gamma^\delta$.  The definition of $\wh \beta^\delta$ clearly simplifies,
since we have already required that $\wh\beta \in \rmC^1(\R)$ with $\beta$
Lipschitz. As for $\phi$, we define
\[
  \phi^\delta:\R \to \R \qquad \phi^\delta(z): = f^\delta(z) -
  \frac{\Lambda_\phi}{2}z^2
\]
with $f^\delta$ the smoothened Yosida approximation of the convex function
$z\mapsto f(z)= \wh\phi(z) +\frac{\Lambda_\phi}{2}z^2 $.  It follows from
\eqref{prop-delta-1} that $\wh \beta^\delta$, $\gamma^\delta$ and
$ \phi^\delta$ comply with \eqref{eq:DelAss01.b}.

\paragraph{\bf The construction of $\psi^\delta$.}  In smoothening $\psi$ we
also have to take care of the linear growth constraint encompassed in
\eqref{eq:DelAss01.b}.  Hence, we construct $\psi^\delta$ in two steps:

\noindent
\STEP{1. Inf-convolution}
We define $\psi_\delta^{\mathrm{ic}}: \R^3 \to [0,\infty)$ via inf-convolution with the smooth function $h:\R^3 \to [0,\infty)$, $h(a): = \sqrt{1{+}|a|^2}-1$ by setting
\begin{equation}
\label{inf-convol-psi}
\psi_\delta^{\mathrm{ic}}(a): = \inf_{x\in \R^3} \left(\frac1\delta h(x{-}a)+\psi(x) \right)\,.
\end{equation}
It turns out that $\psi_\delta^{\mathrm{ic}}$ is convex, of class $\rmC^1$, and since $h(0)=0$ we have that 
\begin{subequations}
\begin{equation}
\label{bound-from-above}
\psi_\delta^{\mathrm{ic}}(a)\leq \psi(a) \qquad \text{for all } a \in \R^3.
\end{equation}
Since $h$ is even, we also have
$\psi_\delta^{\mathrm{ic}}(a) = \inf_{x\in \R^3} \{ \tfrac1\delta
h(x)+\psi(a{-}x)\}$. Hence, recalling that $\psi(0)=0$ we find that
\begin{equation}
\label{linear-growth-ic}
\psi_\delta^{\mathrm{ic}}(a)\leq \frac1\delta h(a) \qquad \text{for all } a \in \R^3,
\end{equation}
so that, in particular, $\psi_\delta^{\mathrm{ic}}$ has linear growth. Finally,
let $a_\delta \in \mathop{\mathrm{Argmin}}\limits_{x\in \R^3} 
{\{\tfrac1\delta h(x{-}a)+\psi(x)\} }$.  Then,
$\tfrac1\delta h(a_\delta{-}a) \leq \psi_\delta^{\mathrm{ic}} \leq \psi(a)$, so
that $\lim_{\delta \to 0^+}h(a_\delta{-}a) =0$, hence
$|a_\delta{-}a| = \sqrt{(h(a_\delta{-}a){+}1)^2{-}1} \longrightarrow 0$ as
$\delta \to 0^+$.  All in all, we conclude that
\begin{equation}
\label{bound-from-below}
\psi_\delta^{\mathrm{ic}}(a) = \frac1{\delta} h(a_\delta{-}a) +\psi(a_\delta)
\geq \psi(a_\delta) \qquad \text{with } a_\delta \to a \text{ as } \delta \to
0^+\,. 
\end{equation}
\end{subequations}
\STEP{2. Smoothening} We then  
define $\psi^\delta \in \rmC^\infty(\R^3; \R)$ via
\begin{equation}
\label{final-psi-delta}
\psi^\delta: = \psi_\delta^{\mathrm{ic}} \star \eta_\delta \qquad \text{with
  $\eta_\delta$ from \eqref{convol-kernel}.}  
\end{equation}
Clearly, $\psi^\delta$ is also convex. Combining \eqref{prop-delta-2} and
\eqref{bound-from-above}, \eqref{linear-growth-ic}, and
\eqref{bound-from-below} we gather that
\begin{subequations}
\label{psi-delta-properties}
\begin{equation}
\label{psi-delta-also-linear}
-   \delta |a| \leq  \psi(a_\delta)  - \delta |a| \leq \psi^\delta(a) \leq \min\big\{  \frac1\delta h(a),\psi(a)  \big \} + \delta|a| \qquad   \text{with } a_\delta \to a \text{ as } \delta \to 0^+\,.
\end{equation}
Thus, $\psi^\delta$ has also linear growth.  Taking into account that it is convex, from \eqref{psi-delta-also-linear} we easily deduce that 
\begin{equation}
\label{linear-bound-der-psi-delta}
|\rmD \psi^\delta(a)| \leq |\pl  \psi^\circ(a)| + \delta  \qquad \text{for all } a \in \R^3,
\end{equation}
(where we have again used the notation $ |\pl  \psi^\circ(a)| = \inf\{ |\eta|\, : \ \eta \in \pl \psi(a)\}$. 
\end{subequations}
Finally, 
\begin{equation}
\label{converg-psi-delta}
\lim_{\delta \to 0^+}  \psi^\delta(a) = \psi(a) \qquad \text{for all } a \in \R^3\,.
\end{equation}

The delamination system \eqref{eq:DelamSyst} featuring $\wh \beta^\delta$,
$\gamma^\delta$, $\wh \phi^\delta$ and $\psi^\delta$ obviously has a gradient
structure in the ambient spaces \eqref{spaces-adh-cont}, with the dissipation
potentials from \eqref{diss-pot-adhc} and \eqref{disv-adhc}, and with the
driving energy (cf.\ \eqref{decomposition-energy})
\begin{subequations}
\begin{align}
&
\calE^\delta(t,u,z):= \calE_{\mathrm{elast}}(t,u) + \calF^\delta(z) + \calE_{\mathrm{coupl}}^\delta(u,z)
\intertext{with $ \calE_{\mathrm{elast}}$ from \eqref{decomposition-energy}, and}
&
 \calF^\delta(z) : =  \frac12\pairing{}{\rmH^1(\GC)}{\bsA
z}{z} +\int_{\GC}  \wh\phi^\delta(z)  \dd x \quad \text{if } z \in \rmH^1(\GC), \text{ and $\infty$ else}, 
\\
& 
\calE_{\mathrm{coupl}}^\delta(u,z) : =  \int_{\GC} \big(\wh\beta^\delta(\JUMP
u)+\gamma^\delta(z) \psi^\delta(\JUMP u)) \dd x \,.
\end{align}
\end{subequations}
which indeed Mosco converges as $\delta \to 0^+$, with respect to the topology
of $\Spu\ti\Spz$, to the energy functional $\calE$ from
\eqref{energy-delamination}.  We will pass to the limit, as $\delta\to 0^+$, in
the corresponding energy-dissipation balance \eqref{enid.a} to prove that the
solutions $(u^\eps_\delta,z^\eps_\delta)_\delta$ to the regularized
delamination system converge to a solution of the original system 
\eqref{eq:DelamSyst}, satisfying the basic energy estimate
\eqref{eq:DelamEst01} as well as the rate-independent a priori estimate
\eqref{eq:ApriSemLinCase}.

\begin{proposition}[Existence of viscous solutions with improved estimates]
\label{pr:Del.ViscSolImprov} 
Under assumptions \eqref{eq:Del.Ass02General} for $\wh\beta$, $\psi$, $\gamma$,
and $\wh\phi$ and the compatibility conditions \eqref{eq:Del.IniCompati} on the
initial data $(u_0,z_0)$, there exists a constant $C_*>0$ such that for all
$\eps>0$ there exist a solution
$(u_\eps,z_\eps) \in \rmH^1(0,T;\Spu)\ti \rmH^1(0,T;\Spx)$ satisfying the
energy estimate \eqref{eq:DelamEst01} with $C_1^{E_0}=C_*$, as well as
  the improved estimate
\[
\int_0^T \big( \|\dot u_\eps\|_\Spu + \|\dot z_\eps\|_{\Spx}\big)  \dd t  \leq C_*.
\]
\end{proposition} 
\begin{proof} Let $(\delta_k)_k$ be a null sequence and, for $\eps>0$ fixed, let 
$(q^\eps_{\delta_k})_k$ be the corresponding  sequence of solutions to the regularized system \eqref{eq:DelamSyst}; from now on, we will simply write $(q_k)_k$. Our starting point is 
the energy-dissipation balance
\begin{align}
\label{enid.a-delta}
& \calE^{\dk}(t,q_k(t))  + \int_s^t \Big( \disve u{\eps^\alpha}
 (\eps^\alpha u_k'(r)) + \calR(z_k'(r)) +  \disve z \eps
  (\eps\,z_k'(r)) 
\Big) \dd r 
\\ \nonumber
 &\quad +  \int_s^t \Big( \frac1{\eps^\alpha}  \disv
u^*({-} \mu_k (r))    +  \frac1\eps  \conj z({-}\zeta_k(r))\Big)  \dd r 
 = \calE^{\dk} (s,q_k(s)) + \int_s^t\pl_t  \calE^{\dk} (r,q_k(r)) \dd r   
  \text{ for all 
$[s,t]\subset [0,T]$}
\end{align}
with 
\[
\begin{aligned}
&
  \mu_k(t)   = \bsC u_k(t) + \bsJ^*\big(\beta^{\dk}(\JUMP{u_k(t)})+
 \gamma^{\dk}(z_k(t))\rmD \psi^{\dk} (\JUMP{u_k(t)}) \big) - \ell_u(t),
       \\
      &   \zeta_k(t) = \bsA z_k(t) +(\gamma^{\dk})'(z_k(t) ) \psi^\delta(\JUMP{u_k(t) }) + \phi^{\dk}(z_k(t))\,.
\end{aligned}
\]
Relying on the energy estimate \eqref{eq:DelamEst01}
and on well known compactness results, we infer that there exists $q_\eps = (u_\eps,z_\eps) $ such that, along a not relabeled subsequence,
\begin{equation}
\label{ptwise-q}
q_k \weakto q_\eps \text{ in } \rmH^1(0,T;\Spu\ti \Spz)  \quad \text{ and } \quad q_k (t) \weakto q_\eps(t) \text{ in } \Spu\ti \Spx \text{ for all } t \in [0,T]\,.
\end{equation}
It also follows from estimate \eqref{est-quoted.a} in Lemma \ref{l:1} that
there exist   $\mu_\eps$   and $\zeta_\eps$ such that, up to a further subsequence, 
\[
 \mu_k  \weakto   \mu_\eps  \text{ in } \rmL^2(0,T;\Spu^*) \quad \text{ and } \quad 
\zeta_k \weakto \zeta_\eps \text{ in } \rmL^2(0,T;\Spz^*)\,.
\]

In order to identify the weak limit $\zeta_\eps(t)$ as an element of $\frsub z {tt}{u_\eps(t)}{z_\eps(t)}$ for almost all $t\in (0,T)$, we observe that, by \eqref{prop-delta-1-bis},
$|(\gamma^{\dk})'(z_k ) | \leq \delta + |\pl \gamma^o(z_k) | \leq \delta +1$, taking into account that $\gamma(z) = \max\{ z,0\}$. Therefore, 
\[
\| (\gamma^{\dk})'(z_k ) \psi^\delta(\JUMP{u_k })\|_{\rmL^2} 
\stackrel{(1)}{\leq} (\delta{+}1) \left( \| \psi(\JUMP{u_k })\|_{\rmL^2}{+}\delta \| \JUMP{u_k }\|_{\rmL^2} \right)
\stackrel{(2)}{\leq}  (\delta{+}1)  \left( C_\psi^{(2)}  \| \JUMP{u_k }\|_{\rmL^4}^2  {+}\delta \| \JUMP{u_k }\|_{\rmL^2}{+}C\right)
\]
with {\footnotesize (1)} due to \eqref{psi-delta-also-linear}  and 
{\footnotesize (2)} 
to \eqref{eq:Del.Ass02General}. Since $(u_k)_k$ is bounded in $\rmL^\infty(0,T;\rmH^1(\Omega;\R^3))$, we immediately deduce that 
$((\gamma^{\dk})'(z_k ) \psi^\delta(\JUMP{u_k }) )_k$ is bounded in $\rmL^\infty(0,T;\rmL^2(\GC))$. 
A standard argument based on the fact that   $z\mapsto \phi^{\delta_k}(z)+\Lambda_\phi z$ is a non-decreasing function then yields a separate estimate in $\rmL^2(0,T;\rmL^2(\GC))$ for both $(\bsA z_k)_k$ and $( \phi^{\dk}(z_k))_k$ so that, up to a subsequence, $\phi^{\dk}(z_k) \weakto \phi$ in $\rmL^2(0,T;\rmL^2(\GC))$ for some $\phi$. Combining this with the fact that $z_k\to z_\eps$ in $\rmL^2(0,T;\rmL^2(\GC))$ 
we immediately conclude by \eqref{prop-delta-3} that $\phi \in \pl \wh\phi(z_\eps)$ a.e.\ in $(0,T)\ti \GC$. With the same arguments we find that $(\gamma^{\dk})'(z_k )\weaksto \omega$
in $\rmL^\infty((0,T)\ti\GC)$
 with $\omega \in \pl \gamma(z_\eps)$ a.e.\ in $(0,T)\ti \GC$.
 Finally, 
 again applying   \eqref{psi-delta-also-linear}   to estimate $| \psi^\delta(\JUMP{u_k })| $ and  taking into account that $\JUMP{u_k }\to \JUMP{u}$ strongly in $\rmL^\infty(0,T;\rmL^q(\GC))$ for every $1\leq q<4$, with the dominated convergence theorem we conclude that 
 $\psi^\delta(\JUMP{u_k })\to \psi(\JUMP{u_\eps})$, for instance, in $\rmL^{3/2}((0,T)\ti \GC)$. All in all, we find that
 $(\gamma^{\dk})'(z_k ) \psi^\delta(\JUMP{u_k }) \weakto \omega  \psi(\JUMP{u_\eps})$  in $\rmL^{3/2}((0,T)\ti \GC)$.
 We have thus proved that 
 \[
 \zeta_\eps  =\bsA z + \omega  \psi(\JUMP{u_\eps})+ \phi \qquad \text{with } \omega \in \pl \gamma(z_\eps), \ \phi \in \pl \wh\phi(z_\eps) \ \aein\, (0,T)\ti \GC\,,
 \]
 and thus $\zeta_\eps(t) \in \frsub  zt{u_\eps(t)}{z_\eps(t)}$. 
 
 The identification of   $\mu_\eps$  as an element of $\frsub u {\cdot}{u_\eps(\cdot)}{z_\eps(\cdot)}$ first of all  follows
  from observing 
 that, by \eqref{ptwise-q},
 $
 \bsC u_k \weaksto \bsC u$ in $\rmL^\infty(0,T;\Spu^*)$.
 Moreover, with similar arguments as in the above lines,
 based on properties \eqref{properties-delta-approx}, we find 
  that 
 $\gamma^{\dk}(z_k) \to \gamma(z_\eps)$ in $\rmL^q((0,T)\ti\GC)$
  for all $1\leq q<\infty$ and, recalling that $\beta$ is Lipschitz, that 
  there exists $\widetilde \beta \in \rmL^\infty (0,T;\rmL^4(\GC))$ such that  
  $\beta^{\delta_k}(\JUMP {u_k}) \weakto \widetilde{\beta}$  in $ \rmL^\infty (0,T;\rmL^4(\GC))$. 
  Finally, 
taking into account \eqref{linear-bound-der-psi-delta} and the fact that $\psi$  has quadratic growth we conclude that
  there exists $\varrho \in \rmL^\infty (0,T;\rmL^4(\GC))$
such that   $\rmD\psi^{\delta_k}(\JUMP{u_k}) \weaksto \varrho $ in $\rmL^\infty (0,T;\rmL^4(\GC))$. All in all, we find that 
\[
 \bsJ^*(\beta^{\delta_k}(\JUMP {u_k})+\gamma^{\delta_k}(z_k)
 \rmD \psi^{\delta_k}(\JUMP {u_k}) ) \weakto \eta =\widetilde{\beta} +\gamma(z_\eps) \varrho  \quad \text{in }  \rmL^2(0,T;\Spu^*),
\]
and it remains to show that $\eta =  \bsJ^*(\beta(\JUMP {u})+\gamma(z)
 \rmD \psi(\JUMP {u}) )$. For this, we observe that 
 the functionals $\mathcal{J}^{\delta_k} : \rmL^2(0,T; \Spu{\ti}\Spz) \to \R$ defined by 
 $
\mathcal{J}^{\delta_k} (u,z): = \int_0^T \mathcal{}\int_{\GC} \big(\wh\beta^{\delta_k}(\JUMP u)+\gamma^{\delta_k}(z)
\psi^{\delta_k}(\JUMP u)\big) \dd x \dd t,
 $
clearly fulfilling
 \[
 \rmD_u \mathcal{J}^{\delta_k} (u,z)  = \bsJ^*(\beta^{\delta_k}(\JUMP {u})+\gamma^{\delta_k}(z)
 \rmD \psi^{\delta_k}(\JUMP {u}))
\qquad  \text{for  every $(u,z) \in  \rmL^2(0,T; \Spu{\ti}\Spz)$},
  \]
 enjoy the following property:
 \[
 \left\{
 \begin{array}{ll}
 (u_k, z_k) \weakto (u,z) \text{ in }  \rmL^2(0,T; \Spu{\ti}\Spz),
 \\
 \rmD_u \mathcal{J}^{\delta_k} (u_k,z_k) \weakto \eta  \text{ in }  \rmL^2(0,T; \Spu^*{\ti}\Spz^*),
\\
\limsup_{k\to\infty} \int_0^T \pairing{}{\Spu}{ \rmD_u \mathcal{J}^{\delta_k} (u_k,z_k)}{u_k} \dd t \leq 
\int_0^T \pairing{}{\Spu}{ \eta}{u} \dd t 
 \end{array}
 \right.
 \quad \Longrightarrow \quad \eta  \in \bsJ^*\big(\beta(\JUMP{u})+
 \gamma(z)\pl  \psi (\JUMP{u}) \big)\,.
\]
Hence,  we need to prove that 
\[
  \limsup_{k\to\infty} \int_0^T \int_{\GC}
  \big\{ \beta^{\dk}(\JUMP{u_k}) \JUMP{u_k} {+}
  \gamma^{\delta_k}(z_k)
 \rmD \psi^{\delta_k}(\JUMP {u_k}) \JUMP{u_k} \big\} 
   \dd x \dd t  \leq \int_0^T \pairing{}{\rmH^1(\Omega)}{\eta}{u} \dd t\,.
\]
This follows from testing the $u$-equation \eqref{eq:DelamSyst.a} at the level
$\delta_k$ by $u_k$, taking the limit as $k\to\infty$, and using that, by the
convergence arguments in the above lines, the quadruple
$(u,z,\tilde \beta,\varrho)$ fulfills the limit equation
$ 0 = \eps^\alpha \bsD \dot u_\eps + \bsC u_\eps +
+\bsJ^*(\tilde\beta{+}\gamma(z_\eps) \varrho) - \ell_u $ in $\Spu^*$ a.e.\ in
$(0,T)$.  All in all, we conclude that
$ \bsJ^*(\tilde\beta {+} \gamma(z_\eps) \varrho) \in \bsJ^*(\beta(z_\eps){+}
\gamma(z_\eps) \pl  \psi(\JUMP{u_\eps}))$, so that
\[
  \mu_\eps   \in   \bsC u_\eps + \bsJ^*\big(\beta(\JUMP{u_\eps})+
 \gamma(z_\eps)\pl \psi (\JUMP{u_\eps}) \big) -\ell_u(t) = \frsub
 ut{u_\eps}{z_\eps}\,.
\]
 
Therefore, passing to the limit as $k\to\infty$ in \eqref{enid.a} we infer that
the quadruple $(u_\eps,z_\eps,  \mu_\eps,  \zeta_\eps)$ fulfills
$(  \mu_\eps(t),  \zeta_\eps(t)) \in \frsubq q t{q_\eps(t)}$ for almost all
$t\in (0,T)$, joint with the energy-dissipation upper estimate in
\eqref{enid.a}. Now, by Proposition \ref{l:comprehensive} the energy functional
$\calE$ from \eqref{energy-delamination} complies with the chain rule of
Hypothesis \ref{h:ch-rule}.  Hence, by Remark \ref{rmk:GS-used-later} the
validity of the energy-dissipation upper estimate is sufficient to conclude
that $(u_\eps,z_\eps)$ solve the Cauchy problem for the delamination system
\eqref{eq:DelamSyst}.
  
By lower semicontinuity arguments, the a priori estimate
\eqref{eq:ApriSemLinCase} is inherited by $(u_\eps,z_\eps)$. This concludes
the proof of Proposition \ref{pr:Del.ViscSolImprov} and, ultimately, of
Theorem \ref{thm:BV-adh-cont}.
\end{proof}

\appendix

\Section{Chain rules}
\label{s:app-CR}

In this section we first of all provide a sufficient condition for the
chain-rule Hypothesis \ref{h:ch-rule} (and, in fact, for the closedness
Hypothesis \ref{h:closedness} as well).  More precisely, we will show that its
validity is guaranteed by a sort of \emph{uniform subdifferentiability}
property of the energy $\calE$ on its sublevels, cf.\ \eqref{uniform-subdiff}
below, which we borrow from \cite{MRS2013}; as already observed in the proof of
Proposition \ref{l:comprehensive}, \eqref{uniform-subdiff} for instance holds
for $\lambda$-convex functionals.  The proof of the following result combines
the argument for \cite[Prop.\,2.4]{MRS2013} with results from \cite{AGS08}.

\begin{proposition}
\label{prop:ch-ruleApp}
Let $\calE: [0,T]\times \mathbf{Q} \to (-\infty, +\infty] $ comply with
Hypothesis \ref{hyp:1}. Assume that for every $E>0$ there exists a modulus of
subdifferentiability $\varpi^E: [0,T]\ti \subl E \ti \subl E \to [0,\infty)$
such that for all $t\in [0,T]:$
\begin{equation}
  \label{uniform-subdiff}
  \begin{gathered}
    \varpi^E(t,q,q) =0 \qquad \text{for every } q \in \subl E,
    \\
    \text{the map } (t,q,\hat{q}) \to \varpi^E (t,q,\hat q) \text{ is upper
      semicontinuous, and }
    \\
    \eneq t{\hat q} - \eneq tq \geq \pairing{}{\Spq}{\xi}{\hat{q}{-}q} -
    \varpi^E(t,q,\hat{q}) \| \hat{q}{-} q\|_{\Spq} \qquad \text{for all } q,\,
    \hat{q} \in \subl E \text{ and all } \xi \in \pl_q \calE(t, q)\,.
  \end{gathered}
\end{equation}
Then, $\calE$ complies with Hypothesis 
\ref{h:ch-rule}.
\end{proposition}
\begin{proof}
In order to show Hypothesis \ref{h:ch-rule}, let $q\in \AC([0,T];\Spq)$ and 
$\xi \in \rmL^1(0,T;\Spq^*)$ fulfill \eqref{conditions-1}, and let $E>0$ be such
that $q(t) \in \subl E$ for all $t\in [0,T]$.  Preliminarily, let us suppose
that $\calE$ is independent of time, i.e.\ $\calE(t,q) = \overline{\calE}(q)$
(with a modulus of subdifferentiability
$\varpi^E (t,\cdot,\cdot) = \varpi^E(\cdot,\cdot)$), and let us prove the
absolute continuity of $[0,T] \ni t \mapsto \overline{\calE}(q(t))$. For this,
as in \cite{MRS2013} we resort to \cite[Lemma 1.1.4]{AGS08} and reparametrize
$q$ to a $1$-Lipschitz curve $\tilde q: [0,L]\to \Spq$, with
$L = \int_0^T \| q'(t)\|_{\Spq} \dd t$, $\tilde q: = q \circ \tilde t$, and
$\tilde t: [0,L]\to [0,T]$ the left-continuous, increasing map
\[
  \tilde{t}(s): = \min \{ t\in [0,T]\, : \ \int_0^t \|q'(r)\|_{\Spq} \dd r =
  s\}\,.
\]
Let us set $\tilde\xi(s):= \xi(\tilde t(s))$.  Then, it follows from
\eqref{conditions-1} and a version of the change of variables formula (cf.,
e.g., \cite[Thm.\,5.8.30]{Bogachev07}), that
\[
  \int_0^L \| \tilde{\xi}(s) \|_{\Spq^*} \dd s = \int_0^L \| \tilde{\xi}(s)
  \|_{\Spq^*} \| \tilde q'(s)\|_{\Spq} \dd s = \int_0^T \| \xi(t) \|_{\Spq^*}
  \| q'(t)\|_{\Spq} \dd t<\infty,
\]
whence $\tilde \xi \in \rmL^1(0,L)$. Hence, we are in a position to repeat the
very same arguments from the proof of \cite[Prop.\,2.4]{MRS2013}. Namely,
relying on the validity of \eqref{uniform-subdiff} we obtain that
\[
  \overline{\calE} (\tilde{q}(s_2)) - \overline{\calE}( \tilde{q}(s_1)) \geq
  \pairing{}{\Spq}{\tilde{\xi}(s_1)}{\tilde{q}(s_2){-}\tilde{q}(s_1)} -
  \varpi^E (\tilde q(s_1), \tilde{q}(s_2))\|
  \tilde{q}(s_2){-}\tilde{q}(s_1)\|_{\Spq}
\]
for all $0\leq s_1\leq s_2 \leq L$. Exchanging the role of $s_1$ and $s_2$ and
exploiting the $1$-Lipschitz continuity of $\tilde q$ leads to
\[
  \left| \overline{\calE} (\tilde{q}(s_2)) {-} \overline{\calE}(
    \tilde{q}(s_1)) \right| \leq
  \left(\|\tilde{\xi}(s_1)\|_{\Spq^*}{+}\|\tilde{\xi}(s_2)\|_{\Spq^*}{+}\varpi^E
    (\tilde q(s_1), \tilde{q}(s_2)){+}\varpi^E (\tilde q(s_2), \tilde{q}(s_1))
  \right)|s_1-s_2|\,.
\]
From the above estimate the absolute continuity of the function
$[0,L]\ni s\mapsto \overline{\calE} (\tilde{q}(s))$ follows by repeating the
very same arguments as in the proof of \cite[Thm.\,1.2.5]{AGS08}.  Again
changing variables in the integral we infer that
$[0,T]\ni t \mapsto \overline{\calE}(q(t))$ is absolutely continuous.

In the general case, let us first of all check the absolute continuity of
$[0,T]\ni t \mapsto \eneq t{q(t)}$, namely that for every $\eps>0$ there exists
$\delta>0$ such that for every collection of pairwise disjoint intervals
$(a_i,b_i) \subset (0,T)$, $i=1,\ldots, M$, there holds
\begin{equation}
  \label{AC-2-SHOW}
  \sum_{i=1}^{M} (b_i{-} a_i)<\delta \ \Longrightarrow \   \sum_{i=1}^{M}|\eneq
  {b_i}{q(b_i)}{-} \eneq {a_i}{q(a_i)}|<\eps\,. 
\end{equation}
In order to obtain the above estimate, we use that
\begin{equation}
\label{split-AC}
|\eneq {b_i}{q(b_i)}{-} \eneq {a_i}{q(a_i)}|\leq |\eneq {b_i}{q(b_i)}{-} \eneq
{a_i}{q(b_i)}| + |\eneq {a_i}{q(b_i)}{-} \eneq {a_i}{q(a_i)}| 
\end{equation}
and estimate the first term via \eqref{h:1.3d}, so that 
\[
  |\eneq {b_i}{q(b_i)}{-} \eneq {a_i}{q(b_i)}| \leq \int_{a_i}^{b_i}
  |\pl_t \eneq r{q(b_i)} | \dd r \leq C_\# \int_{a_i}^{b_i} |\eneq
  r{q(b_i)} | \dd r \leq C_\# E |b_i-a_i|,
\]
Hence, we choose $\delta_1 =\frac{\eps}{2C_1E}$ and get
$\sum_{i=1}^{M}|\eneq {b_i}{q(b_i)}{-} \eneq {a_i}{q(a_i)}|<\tfrac\eps2 $. As
for the second term in \eqref{split-AC}, we rely on the absolute continuity of
$\overline{\calE}(q): = \eneq {a_i}q$. We thus conclude \eqref{AC-2-SHOW}.
Finally, to show the chain-rule formula \eqref{eq:48strong}, we fix a point
$t\in (0,T)$ in which $q'(t)$ and $\frac{\rmd}{\rmd t}\eneq t{q(t)}$ exist, and
derive from \eqref{uniform-subdiff} that
\begin{equation}
\label{we-re-there}
\begin{aligned}
  & \eneq {t+h}{q(t+h)} -\eneq t{q(t)}
  \\
  & = \eneq {t+h}{q(t+h)} -\eneq {t}{q(t+h)} +\eneq {t}{q(t+h)}-\eneq t{q(t)}
  \\
  & \geq \int_t^{t+h} \pl_t \eneq r{q(t+h)} \dd r +
  \pairing{}{\Spq}{\xi(t)}{q(t+h){-}q(t)} - \varpi^E (t,q(t), q(t+h))
  \\
  &
\begin{aligned}
  \geq & \int_t^{t+h} \left( \pl_t \eneq r{q(t+h)}{-} \pl_t \eneq
    r{q(r)} \right) \dd r +\int_t^{t+h} \pl_t \eneq r{q(r)} \dd r 
  \\
  & \qquad \qquad + \pairing{}{\Spq}{\xi(t)}{q(t+h){-}q(t)} - \varpi^E (t,q(t),
  q(t+h)) \| q(t+h) {-} q(t)\|_{\bfQ}
\end{aligned}
\end{aligned}
\end{equation}
We now divide the above estimate by $h>0$ and take the limit as $h\to 0^+$.
Now, recall that $q(t) \in \subl E$ for all $t\in [0,T]$.  Therefore, thanks to
Hypothesis \ref{h:closedness} the function
$[0,T]\ti [0,T]\ni (r,s) \mapsto \pl_t \eneq r{q(s)}$ is uniformly
continuous, with a modulus of continuity
$\omega: [0,T]\times [0,T]\to [0,\infty)$, so that
\[
\left|\tfrac1h  \int_t^{t+h}  \left( \pl_t \eneq r{q(t{+}h)}  {-} \pl_t
    \eneq r{q(r)} \right) \dd r   \right|  \leq\tfrac1h  \int_t^{t+h}
\omega(|t{+}h{-}r|) \dd r \leq \omega(h) \rightarrow 0 \text{ as } h \to 0^+\,. 
\]
Taking into account that
$\lim_{h\to 0^+ } \tfrac1h \int_t^{t+h} \pl_t \eneq r{q(r)} \dd r =
\pl_t \eneq t{q(t)}$, \eqref{we-re-there} thus leads to the estimate
$\geq$ in \eqref{eq:48strong}. The converse inequality follows by dividing
\eqref{we-re-there} by $h<0$ and taking the limit as $h \to 0^-$.  This
concludes the proof.
\end{proof}
  
Let us now carry out the proof of Proposition
\ref{prop:better-chain-rule-MOexpl}, which shows the validity of the
parametrized chain rule from Hypothesis \ref{h:ch-rule-param} if Hypothesis
\ref{h:ch-rule} holds and, in addition, the vanishing-viscosity contact
potentials associated with $\disv u$ and $\disv z$ satisfy the coercivity
property \eqref{coercivity-VVCP}.

\noindent
\begin{proof}[Proof of Proposition \ref{prop:better-chain-rule-MOexpl}]
Preliminarily, we observe that, if the vanishing-viscosity contact 
potentials $\mfb_{\disv u}$ and $\mfb_{\disv z}$ satisfy
\eqref{coercivity-VVCP}, then the `reduced' \RJMF\ $\mredname 0\alpha$ enjoys
the following coercivity property:
\begin{equation}
  \label{stronger-than-true40}
  \begin{aligned}
  \!  \! \! \!  \!  \!    \exists\, c>0 \  \forall\, (t,q,t',q') \in [0,T]\ti \domq \ti [0,\infty)
    \ti \Spq : 
     \ \  & \mredq 0\alpha tq{t'}{q'} \geq c \left( \|\mu \|_{\Spu^*} \|u'\|_\Spu +
      \|\zeta{+}\sigma \|_{\Spu^*} \|z'\|_\Spz \right)
    \\
    & \text{ for all } (\mu,\zeta)\in \argminSlo utq \ti\argminSlo ztq
    \text{ and  some } \sigma\in \pl \calR(0).
  \end{aligned}
\end{equation}
Clearly, the above estimate trivially holds if $t'>0$, as then
$\slov utq = \slov ztq=0$, so that $ \argminSlo utq= \{0\}$ and every
$ \zeta \in \argminSlo ztq$ fulfills $-\zeta \in \pl  \calR (0)$. To show
it for $t'=0$, we will separately discuss the cases $\alpha>1$ (the arguments
for $\alpha \in (0,1)$ are indeed specular) and $\alpha=1$.  In the case
$\alpha>1$, we have that, if $\slov utq=0$, then by \eqref{l:partial} we have
\[
  \mredq 0\alpha tq{0}{q'} = \mfb_{\disv z} (z',\slov ztq) \overset{(1)}{\geq}
 c_{\sfz}   \|z'\|_{\Spz}\|\zeta{+}\sigma\|_{\Spz^*}
\]
for all $ \zeta \in \argminSlo ztq$ and all $\sigma \in \pl \calR(0)$ with
$ \conj z(\zeta) = \disv z^*(\zeta{-}\sigma)$, where {\footnotesize (1)}
follows from \eqref{coercivity-VVCP}.  Analogously, if $\slov utq>0$, then
$\mredq 0\alpha tq{t'}{q'} = \mfb_{\disv u} (u',\slov utq) $ and we have the
analogous estimate.  If $\alpha =1$, then by \eqref{l:partial} and again
\eqref{used-later-HS} we have
\[
  \begin{aligned}
    \mredq 0\alpha tq{0}{q'} & = \mfb_{\disv u \oplus \disv z}(q',\slov
    utq{+}\slov ztq) \\ &  \geq   \mfb_{\disv u}(u',\slov utq)
    +\mfb_{\disv z}(z',\slov ztq) \\ &   \geq  
  c_{\mathsf{u}}  \|u'\|_{\Spu} \| \mu \|_{\Spu^*} + c_{\sfz}
    \|z'\|_{\Spz}\|\zeta{+}\sigma\|_{\Spz^*}
  \end{aligned}
\]
for all  $  \mu $,  $\zeta$, and $\sigma$ as in \eqref{stronger-than-true40}.

Let us now consider an admissible curve
$ (\sft,\sfq) \in \mathscr{A} ([a,b];[0,T]\ti \Spq)$ such that, in addition,
$\sfz \in \AC ([a,b];\Spz)$. Hence,
$\mathfrak{M}_0^\alpha [\sft,\sfq,\sft',\sfq'] =
\mathfrak{M}_0^\alpha(\sft,\sfq,\sft',\sfq') $ a.e.\ in $(a,b)$.  Then,
\eqref{summability} yields that
$ \mathfrak{M}_0^\alpha(\sft,\sfq,\sft',\sfq') \in \rmL^1(a,b)$.  Let us now
choose measurable selections
$(a,b) \ni s \mapsto   \mu(s)  \in \argminSlo u{\sft(s)}{\sfq(s)} $,
$(a,b) \ni s \mapsto \zeta(s) \in \argminSlo z{\sft(s)}{\sfq(s)} $, and
$(a,b) \ni s \mapsto\sigma(s) \in \pl  \calR(0)$ such that
\[
  \| \mu(s)  \|_{\Spu^*} \|\sfu'(s)\|_\Spu {+} \|\zeta(s) {+}\sigma(s)\|_{\Spz^*}
  \|\sfz'(s)\|_\Spz \leq \mredq 0\alpha {\sft(s)}{\sfq(s)}{\sft'(s)}{\sfq'(s)}
  \qquad \foraa\, s\in (a,b).
\]
Hence, we have $  \int_a^b \left( \| \mu(s)  \|_{\Spu^*} \|\sfu'(s)\|_\Spu {+} \|\zeta(s)
    {+}\sigma(s)\|_{\Spz^*} \|\sfz'(s)\|_\Spz \right) \dd s <\infty$, and using 
$\int_a^b \|\sigma(s)\|_{\Spz^*} \|\sfz'(s)\|_\Spz \dd s $ $\leq C_\calR 
\int_a^b \|\sfz'(s)\|_\Spz \dd s <\infty$ by  \eqref{eq:l:classic}, we 
ultimately deduce that
\[
  \int_a^b \left( \|  \mu(s)  \|_{\Spu^*} \|\sfu'(s)\|_\Spu {+} \|\zeta(s)
    \|_{\Spz^*} \|\sfz'(s)\|_\Spz \right) \dd s <\infty\,.
\]
We are thus in a position to apply the chain rule from Hypothesis
\ref{h:ch-rule} and conclude that $s\mapsto \eneq{\sft(s)}{\sfq(s)}$ is
absolutely continuous and that
\[
  \frac{\rmd}{\rmd s} \eneq{\sft}{\sfq} - \pl_t \eneq{\sft}{\sfq} \sft' =
  \pairing{}{\Spu}{\mu}{\sfu'}   + \pairing{}{\Spz}{\zeta}{\sfz'} \quad \aein\,
  (a,b),
\]
Then, the chain-rule estimate \eqref{better-chain-rule-MOexpl} follows from
observing that the right-hand side in the above formula estimates
$- \mathfrak{M}_0^\alpha (\sft,\sfq,\sft',\sfq') $ from above thanks to Lemma
\ref{new-lemma-Ricky}.  This concludes the proof.  
\end{proof}

We conclude this section with a result relating parametrized chain rule from
Hypothesis \ref{h:ch-rule-param} to that in Hypothesis \ref{hyp:BV-ch-rule}.

\begin{lemma}
\label{l:nice-implication}
If the rate-independent system $\RIS$ satisfies the
chain rule of Hypothesis \ref{h:ch-rule-param},  then it
also satisfies Hypothesis \ref{hyp:BV-ch-rule}.
\end{lemma} 
\begin{proof}
Consider a curve $q \in \mathrm{BV}([0,T];\Spu) \ti (\mathrm{R}([0,T];\Spz) 
{\cap} \rmB\rmV([0,T];\Spy) )$ fulfilling the stationary equation
\eqref{stationary-u} and the local stability \eqref{loc-stab}, and let us fix
$[t_0,t_1] \subset [0,T]$.  As in Theorem \ref{th:pBV.v.BVsol}(2), we associate
with $q$ a parametrized curve
$(\sft,\sfq) = (\sft,\sfu,\sfz) \in \mathscr{A} ([\sfS_0,\sfS_1]; [t_0,t_1]\ti
\Spq)$ such that \eqref{balance-variations} holds on the interval $[t_0,t_1]$.
The parametrized chain-rule inequality \eqref{ch-rule} reads for a.a.\
$s\in (\sfS_0,\sfS_1) $
\[
  \frac{\rmd}{\rmd s} \eneq{\sft(s)}{\sfq(s)} - \pl_t
  \eneq{\sft(s)}{\sfq(s)} \sft'(s) \geq - \calR[z'](s) - \mredq 0\alpha
  {\sft(s)}{\sfq(s)}0{\sfq'(s) } \,.
\]
Integrating on the interval $ (\sfS_0,\sfS_1)$ and using
\eqref{balance-variations}, we conclude the desired chain-rule inequality
\eqref{BV-ch-rule}.  With this, Lemma \ref{l:nice-implication} is proved.  
\end{proof}

\Section{Measurability in Theorem \ref{thm:diff-charact}} 
\label{appendix-measurability}

To prove the statement in Theorem \ref{thm:diff-charact} concerning the
existence of measurable selections $ \xi=(\mu,\zeta):  (0,\sfS) \to \Spu^*
\ti \Spz^*$ and $(\thn u,\thn z): (0,\sfS)\to [0,\infty]^2 $  satisfying
\eqref{e:diff-char}, we will resort to 
the following generalization of Filippov's Selection Theorem, proved in
\cite[Prop.\,B.1.2]{MieRouBOOK}.

\begin{proposition}
\label{prop:Filippov}
Let $(O,\mathfrak{O}, \mu)$ be a $\sigma$-finite complete measure space and $X$
a complete separable metric space.  Let $F: O \rightrightarrows X$ be a
measurable set-valued mapping with closed non-empty images,
$\mathrm{graph}(F): = \{ (s,x)\, : \ x \in F(s)\} $ its graph, and let
$\mathfrak{G}$ be the $\sigma$-algebra given by the restriction of
$\mathfrak{O}\times \mathfrak{B}(X)$ (with $\mathfrak{B}(X)$ the Borel
$\sigma$-algebra on $X$) to $\mathrm{graph}(F)$.  Let
$g: \mathrm{graph}(F) \to \R$ be a $\mathfrak{G}$-measurable mapping such that
\begin{equation}
\label{properties-of-g}
\forall\, s \in O\, : \qquad \left\{
\begin{array}{ll}
\exists\, x \in F(s)\, : \qquad g(s,x)=0,
\\
g(s,\cdot): F(s)\to \R \text{ is continuous.}
\end{array}
\right.
\end{equation}
Then, there exists a measurable selection $f: O \to X $ of $F$ such that
$ g(s,f(s)) =0 $ for all $s\in O$.
\end{proposition}

%

For the construction of the parameters $\thn u$ and $\thn z$ introduce the
sets $\Loptimal x{v}{\sigma}$ via 
\begin{equation}
\label{characterization-optimal}
\Loptimal{x}{v}{\sigma}: =   \mathop{\mathrm{Argmin}}\limits_{\lambda>0}
{\Bfu {\disv{x}}{\tfrac1\lambda}{v}\sigma}  
=  \mathop{\mathrm{Argmin}}\limits_{\lambda>0}   \frac1\lambda \big( 
    \disve{x}{}(\lambda v) +  \sigma\big)  \qquad \text{for  }
(v,\sigma) \in \mathbf{X} \ti [0,\infty) 
\end{equation}
with $\mathsf{x} \in \{\sfu,\sfz\}$ and $ \mathbf{X} \in \{\Spu,
\Spz\}$. Recall that Proposition \ref{pr:VVCP} guarantees that
$\Loptimal{x}{v}{\sigma} \neq \emptyset$ for all $\sigma>0$. Analogously, for
$q'=(u',z')$ we will use the notation
\[
\Loptimal{uz}{q'}{\sigma_{\sfu}{+}\sigma_{\sfz}} :=
\mathop{\mathrm{Argmin}}\limits_{\lambda>0}   {\Bfu {\disv{u}\oplus \disv
    z}{\tfrac1\lambda}{q'}{\sigma_\sfu {+} \sigma_\sfz}} =
\mathop{\mathrm{Argmin}}\limits_{\lambda>0}   \frac1\lambda \big( 
    \disve{u}{}(\lambda u') +  \disve{z}{}(\lambda z')  +
     \sigma_\sfu +  \sigma_{\sfz}\big) \,. 
\]
A close perusal of the proof of Proposition \ref{pr:charact-Ctc-set} then
reveals that, for a given $ (t,q,t',q') \in \Ctc_\alpha$ with $\alpha \neq 1$,
\[
(t,q,t',q') \in \rgs{V}x \text{ if and only if 
system \eqref{static-tq} holds with } \thn x \in \Loptimal{x}{x'}{\slov x tq},
\]
for $ \sfx \in \{ \sfu, \sfz\}$ and $ x' \in \{ u', z'\} $.  Analogously, in
the case $\alpha =1$, we have that $(t,q,t',q') \in \Ctc_1 \cap \rgs V{uz}$ if
and only if \eqref{static-tq} holds with
$\lambda \in \Loptimal{uz}{q'}{\slov u tq{+} \slov ztq} $.

We are now in a position to  the missing part of the proof of Theorem
\ref{thm:diff-charact}.  

\begin{proof}[Proof of Part (1) of Theorem \ref{thm:diff-charact}]
  Let $(\sft,\sfq) \in \mathscr{A}([0,\sfS];[0,T]\ti \Spq)$ be a enhanced
  $\pBV$ solution of the rate-independent system $\RIS$.  We will now prove the
  existence of measurable $\thn u,\, \thn z: (0,\sfS)\to [0,\infty] $, and
  $ \mu:   (0,\sfS) \to \Spu^*$ and $\zeta: (0,\sfS) \to \Spz^*$ satisfying
  \eqref{e:diff-char} in the case $\alpha>1$; with similar arguments one can
  obtain the analogous statement for $\alpha=1$ and $\alpha \in (0,1)$.  To
  avoid trivial situations, we also suppose that $(\sft,\sfq)$ is
  non-degenerate.\smallskip

\STEP{1: Existence of measurable $\thn u,\thn z: (0,\sfS)\to [0,\infty]$.}
Proposition \ref{pr:char-eBV} shows 
that $(\sft(s),\sfq(s),\sft'(s),\sfq'(s)) $ belongs to the contact set
$ \Ctc_\alpha$ for a.a.\ $s\in (0,\sfS)$ and that, in turn,
$\Ctc_\alpha \subset \rgs Eu \rgs Rz \cup \rgs Eu \rgs Vz \cup \rgs Bz$.
Recalling the short-hand notation \eqref{short-hand-regimes}, we introduce the
short-hand $((0,\sfS) \cap \rgs Bz)^\circ$ for the set
$\{(\sft,\sfq,\sft',\sfq') \in \rgs Bz \text{ and } \slov u {\sft}{\sfq}
>0\}$. Since  for $\alpha>1$ we have  $\sft'\equiv 0$ and $\sfz'\equiv 0$ on
$((0,\sfS) \cap \rgs Bz)^\circ$ and  since $(\sft,\sfq)$ is non-degenerate, we have
$\sfu' \neq 0$ on $((0,\sfS) \cap \rgs Bz)^\circ$. Therefore, thanks to
Proposition \ref{pr:VVCP}, the set
$ \Loptimal{u}{\sfu'(s)}{\slov u{\sft(s)}{\sfq(s)}} =
\mathop{\mathrm{Argmin}}_{\lambda>0} {\Bfu
  {\disv{u}}{\tfrac1\lambda}{\sfu'(s)}{\slov u{\sft(s)}{\sfq(s)}}} $ is
non-empty for every $s\in ((0,\sfS) \cap \rgs Bz)^\circ$.  We consider the
multi-valued mapping
$\Gamma_\sfu: ((0,\sfS) \cap \rgs Bz)^\circ \rightrightarrows [0,\infty)$
defined by
$\Gamma_\sfu(s) : = \Loptimal{u}{\sfu'(s)}{\slov u{\sft(s)}{\sfq(s)}}$ and
observe that its graph is a Borel subset of $(0,\infty)^2$. Indeed,
$\Gamma_\sfu$ is given by the composition of the upper semicontinuous
multi-valued mapping $\Lambda_{\sfu}$, with the Borel function $\sfu'$ and the
lower semicontinuous function $\slov u {\sft}{\sfq}$.

Hence, we are in a position to apply the von Neumann-Aumann selection theorem
(cf.\ \cite[Thm.\,III.22]{Castaing-Valadier77}) to $\Gamma_\sfu$ and conclude
that
\begin{equation}
  \label{meas-selec-lu}
  \text{there exists a measurable }   \widetilde{\lambda}_{\sfu}: ((0,\sfS)
  \cap \rgs Bz)^\circ \to (0,+\infty) \text{ with }
  \widetilde{\lambda}_{\sfu}(s) \in  \Loptimal{u}{\sfu'(s)}{\slov
    u{\sft(s)}{\sfq(s)}}\,. 
\end{equation}
Let $N$ be the negligible subset of $(0,\sfS)$ on which either $\sft'$ or
$\sfq'$ do not exist, or $(\sft,\sfq,\sft',\sfq') \notin \Ctc_\alpha$: since
$\Ctc_\alpha \subset \rgs Eu \rgs Rz \cup \rgs Eu \rgs Vz \cup \rgs Bz$, we
have that
\[
  (0,\sfS) {\setminus} N = A_1\cup A_2\cup A_3 \cup A_4 \text{ with } \left\{
  \begin{array}{ll}
   A_1  =  (0,\sfS) \cap \rgs Eu\rgs Rz,
  \\
  A_2  =  (0,\sfS) \cap \rgs Eu\rgs Vz,
  \\
  A_3   = \{ s \in  (0,\sfS) \cap \rgs Bz\, : \  \slov u {\sft(s)}{\sfq(s)} =0\},
  \\
  A_4   = \{ s \in  (0,\sfS) \cap \rgs Bz\, : \  \slov u {\sft(s)}{\sfq(s)} >0\}\,.
    \end{array}
  \right.
\]
Hence, we define $\thn u$ on $(0,\sfS){\setminus} N$ by
\begin{equation}
  \label{definition-thn-u}
  \thn u(s): =
  \begin{cases}
  0 &\text{if } s \in A_1\cup A_2 \cup A_3, 
    \\
  \widetilde{\lambda}_{\sfu}(s) & \text{if } s\in A_4\,. 
  \end{cases}
\end{equation}
Analogously, we define $\thn z$ on $(0,\sfS) {\setminus} N$ by
\begin{equation}
  \label{definition-thn-z}
  \thn z(s): = \begin{cases}
  0 &\text{if } s \in A_1,
  \\
   \widetilde{\lambda}_{\sfz}(s)    &\text{if } s \in A_2, 
    \\
    \infty & \text{if } s\in A_3{\cup}A_4\,. 
  \end{cases}
\end{equation}
Here, $\widetilde{\lambda}_{\sfz}: (0,\sfS)\cap\rgs Eu\rgs Vz \to(0,+\infty)$
is a measurable selection  with  
$\widetilde{\lambda}_{\sfz}(s) \in \Loptimal{z}{\sfz'(s)}{\slov
  z{\sft(s)}{\sfq(s)}}$, whose existence is again guaranteed by
\cite[Thm.\,III.22]{Castaing-Valadier77}.

Clearly, $\thn u$ and $\thn z$ satisfy the switching conditions
\eqref{eq:SwitchCond}. Also taking into account
\eqref{characterization-optimal}, we conclude that $(\sft,\sfq)$ solve the
subdifferential system \eqref{param-subdif-incl} with $\thn u$ and $\thn z$.\smallskip

\STEP{2: existence of measurable $  \xi=(\mu,\zeta): (0,\sfS) \to \Spu^*
  \ti \Spz^*$.} We aim to apply Filippov's theorem,  in the
form of Proposition \ref{prop:Filippov},  with
$O = (0,\sfS){\setminus} N$, $X = \Spu^* \ti \Spz^*$, and the multi-valued
mapping
\[
  F:(0,\sfS) {\setminus} N\rightrightarrows \Spu^* \ti \Spz^*, \qquad F(s): =
    \begin{cases}
      \frsubq q{\sft(s)}{\sfq(s)}  & \text{if } (\sft(s),\sfq(s)) \in
      \mathrm{dom}(\pl_q \calE), 
      \\
      \{(0,0)\} &\text{otherwise}. 
\end{cases}
\]
Observe that $F$ is measurable, with (non-empty) closed images thanks to
Hypothesis \ref{h:closedness}. We now consider the mapping 
$g: \mathrm{graph}(F) \to \R$ given by  
\[
g(s,\mu,\zeta): = 
  \begin{cases}
\disv u^*({-}\mu) + \conj z({-}\zeta)   &\text{if } s \in A_1,
  \\ 
  \disv u^*({-}\mu)+ \calR(\sfz'(s)) +  
\frac1{\thn z (s)} \disv z(\thn z(s) \sfz'(s)) + \frac1{\thn z(s)} \conj z
(-\zeta)  +\!\pairing{}{\Spz}{\zeta}{\sfz'(s)}      &\text{if } s \in A_2, 
    \\
 \disv u^*({-}\mu) + \conj z({-}\zeta)      & \text{if } s\in A_3,
    \\
\frac1{\thn u(s)} \disv u(\thn u (s) \sfu'(s)) + \frac1{\thn u(s)} \disv u^*
(-\mu)  +\! \pairing{}{\Spu}{\mu}{\sfu'(s)}  & \text{if } s\in A_4\,. 
  \end{cases}
\]
It turns out that, also in view of the discussion developed in Step $1$, $g$ is
measurable. Furthermore, for every $s\in (0,\sfS)$ the functional
$g(s,\cdot,\cdot)$ is continuous thanks to the continuity of $\calR$ and
$\disv x$ and $\disv x^*$, $\sfx \in \{\sfu, \sfz\}$.  Finally, the first of
conditions \eqref{properties-of-g} holds since, by Step $1$, $(\sft,\sfq)$
solve the subdifferential system \eqref{param-subdif-incl} with $\thn u$ and
$\thn z$, which exactly means that for every $s\in (0,\sfS) {\setminus} N$
there exist $(\mu_s, \zeta_s) \in \frsubq q{\sft(s)}{\sfq(s)} $ such that
$g(s,\mu_s,\zeta_s) = 0$.  Hence, we are in a position to apply 
Proposition \ref{prop:Filippov},  thus concluding that there exists a
measurable selection
$(0,\sfS) {\setminus} N \ni s\mapsto (\mu(s),\zeta(s)) \in F(s)$ such that
$ g(s,\mu(s),\zeta(s))\equiv 0\,.  $ This yields the desired  selection as
stated in \eqref{exist-meas-select}. 
\end{proof}

\paragraph*{Acknowledgments.} A.M. was partially supported by Deutsche
Forschungsgemeinschaft (DFG) via the Priority Program SPP\,2256
``\emph{Variational Methods for Predicting Complex Phenomena in Engineering
  Structures and Materials}'' (project no.\,441470105), subproject Mi 459/9-1 
\emph{Analysis for thermo-mechanical models with internal variables}. 
The authors are grateful to Giuseppe Savar\'e for helpful and stimulating
discussions. 

\markboth{References}{References}


\providecommand{\bysame}{\leavevmode\hbox to3em{\hrulefill}\thinspace}

\end{document}